\documentclass[english]{article}
\usepackage{geometry}
\usepackage[unicode=true,
 bookmarks=false,
 breaklinks=false,pdfborder={0 0 1},colorlinks=false]
 {hyperref}
\hypersetup{
 colorlinks,citecolor=blue,filecolor=blue,linkcolor=blue,urlcolor=blue}
 \usepackage{mathtools}
\usepackage{empheq}
\usepackage{amsfonts}
\usepackage{amsgen,amsmath,amstext,amsbsy,amsopn,amssymb,subfigure,amsthm,stmaryrd}
\usepackage{bm}
\usepackage{bbm}
\usepackage{comment}
\usepackage[dvips]{graphicx}
\usepackage{enumitem}
\usepackage{natbib}
\usepackage{booktabs}
\usepackage{blkarray}
\usepackage[linesnumbered,ruled,vlined]{algorithm2e}
\usepackage{url}
\usepackage{longtable}
\usepackage{mathtools}
\usepackage{multirow}

\usepackage[usenames]{color}

\usepackage{mathtools}
\usepackage{empheq}
\usepackage{amsfonts}

\usepackage[dvips]{graphicx}
\usepackage{enumitem}
\usepackage{natbib}
\usepackage{booktabs}
\usepackage{blkarray}
\usepackage{algpseudocode}
\usepackage{url}
\usepackage{longtable}
\usepackage{mathtools}
\usepackage{multirow}

\usepackage[usenames]{color}

\definecolor{yxc}{RGB}{255,0,0}

\allowdisplaybreaks

\newtheorem{thm}{Theorem}

\newtheorem{lemma}{Lemma}

\newtheorem{Definition}{Definition}
\newtheorem{rmk}{Remark}
\newtheorem{cor}{Corollary}

\newtheorem{assump}{Assumption}

\allowdisplaybreaks

\newcommand{\be}{\begin{equation}}
	\newcommand{\ee}{\end{equation}}
\newcommand{\bea}{\begin{eqnarray}}
	\newcommand{\eea}{\end{eqnarray}}
\newcommand{\beas}{\begin{eqnarray*}}
	\newcommand{\eeas}{\end{eqnarray*}}

\newcommand{\bbR}{\mathbb{R}}

\newcommand{\Var}{{\sf Var}}

\newcommand{\F}{{\mathrm{F}}}

\newcommand{\argmin}{\mathop{\rm arg\min}}

\newcommand{\bbP}{\mathbb{P}}
\newcommand{\bbE}{\mathbb{E}}

\setdescription{font = \normalfont}

\makeatletter
\newcommand*{\rom}[1]{\expandafter\@slowromancap\romannumeral #1@}
\makeatother

\title{Deflated HeteroPCA: Overcoming the curse of ill-conditioning in heteroskedastic PCA}

\author{Yuchen Zhou\thanks{Department of Statistics, University of Illinois Urbana-Champaign, Champaign, IL 61820, USA.} 	\and
	Yuxin Chen\thanks{Department of Statistics and Data Science, Wharton School, University of Pennsylvania, Philadelphia, PA 19104, USA.} 
	}

\date{March 2023; \quad Revised: October 2024}

\begin{document}
	\maketitle
	\begin{abstract}

This paper is concerned with estimating the column subspace of a low-rank matrix $\bm{X}^\star \in \mathbb{R}^{n_1\times n_2}$ from contaminated data.  
How to obtain optimal statistical accuracy while accommodating the widest range of signal-to-noise ratios (SNRs) 
becomes particularly challenging in the presence of heteroskedastic noise and unbalanced dimensionality (i.e., $n_2\gg n_1$). 
While the state-of-the-art algorithm  {\sf HeteroPCA} emerges as a powerful solution for solving this problem, 
it suffers from ``the curse of ill-conditioning,'' namely, its performance degrades as the condition number of $\bm{X}^\star$ grows. 
In order to overcome this critical issue without compromising the range of allowable SNRs,   
we propose a novel algorithm, called {\sf Deflated-HeteroPCA}, that achieves near-optimal and condition-number-free theoretical guarantees in terms of both $\ell_2$ and $\ell_{2,\infty}$ statistical accuracy. 
The proposed algorithm divides the spectrum of $\bm{X}^\star$ into well-conditioned and mutually well-separated subblocks, 
and applies {\sf HeteroPCA} to conquer each subblock successively. 
		 Further, an application of our algorithm and theory to two canonical examples --- the factor model and tensor PCA --- leads to 
		 remarkable improvement for each application. 
	\end{abstract}


\tableofcontents

\section{Introduction}

In a diverse array of science and engineering applications, we are asked to identify a low-dimensional subspace that best captures the information underlying a large collection of high-dimensional data points, 
a classical problem that goes by the names of principal component analysis (PCA), subspace estimation, subspace tracking, among others \citep{johnstone2018pca,balzano2018streaming,chen2021spectral}. 
A simple yet useful mathematical model is of the following form: 
imagine we have an unknown large-dimensional matrix $\bm{X}^{\star}\in \bbR^{n_1 \times n_2}$ whose columns are high-dimensional vectors embedded in a $r$-dimensional subspace (so that $\bm{X}^{\star}$ has rank $r\ll \min\{n_1,n_2\}$), 
and we seek to estimate the {\em column space} of $\bm{X}^{\star}$ from noisy observations:  
    \begin{align}\label{eq:model}
    	\bm{Y} = \bm{X}^{\star} + \bm{E} \in \bbR^{n_1 \times n_2}, 
    \end{align}
where $\bm{E}$ stands for the noise matrix that contaminates the data. 
Despite decades-long research, there remain substantial challenges to handle heteroskedastic noise in high dimension,  as we shall elaborate on below.

\subsection{Challenges:   unbalanced dimensionality and heteroskedasticity} 
How to achieve statistically efficient PCA in high dimension 
is an active research topic that has received much recent interest \citep{lounici2014high,johnstone2018pca,cai2021subspace,zhu2019high,zhang2022heteroskedastic,agterberg2022entrywise}. 
In this paper, we pay particular attention to the case where $n_1$ and $n_2$ are  both enormous but highly \emph{unbalanced} in the sense that $n_1 \ll n_2$, 
a scenario that arises frequently in, say, covariance estimation (when there are many noisy samples available) and tensor estimation (when one has to matrice the tensor before estimation).  
Such unbalanced dimensionality gives rise to unique challenges not present in the complement case:  
as the signal-to-noise ratio (SNR) keeps decreasing, one might soon enter a regime where 
consistent estimation of $\bm{X}^{\star}$ is no longer infeasible but its column subspace --- which is much smaller dimensional than the full matrix --- remains estimatable. 
This regime is often considerably more challenging than the case with $n_2=O(n_1)$, 
given that the majority of low-rank matrix estimation algorithms that directly attempt to estimate $\bm{X}^{\star}$ become completely off.

One natural strategy that comes into mind is thus to estimate the column subspace of $\bm{X}^{\star}$ by calculating the left singular subspace of the observed matrix $\bm{Y}$ \citep{cai2018rate,abbe2020entrywise,chen2021spectral}, 
which we shall refer to as the {\em vanilla SVD-based approach} throughout.  
In the case with $n_1\ll n_2$, this simple scheme has only been shown to achieve the desired statistical performance when the noise matrix $\bm{E}$ is composed of i.i.d.~entries, 
but falls short of effectiveness when handling \emph{heteroskedastic} noise (i.e., the scenario where the variances of the entries of $\bm{E}$ are location-varying) \citep{zhang2022heteroskedastic,cai2021subspace}. 
This issue presents a hurdle to transferring this scheme from theory to practice, due to the ubiquity of heteroskedastic data in applications like social networks, recommendation systems, medical imaging, etc.

To mitigate this issue,  at least two strategies have been proposed that attempt estimation by looking at the empirical covariance matrix (or gram matrix) $\bm{Y}\bm{Y}^{\top}$.  
Recognizing that large heteroskedastic noise might lead to significant bias in the diagonal of $\bm{Y}\bm{Y}^{\top}$ that distorts estimation, 
one natural remedy is to zero out (or sometimes rescale) the diagonal entries of $\bm{Y}\bm{Y}^{\top}$ before computing its eigendecomposition \citep{koltchinskii2000random,lounici2014high,florescu2016spectral,loh2012high,montanari2018spectral,elsener2019sparse,cai2021subspace,ndaoud2021improved}. 
A more refined iterative procedure called $\mathsf{HeteroPCA}$ was subsequently proposed by \citet{zhang2022heteroskedastic}, 
which starts with the solution of diagonal-deleted PCA (cf.~\eqref{eq:diagonal-deletion}) and alternates between:
\begin{itemize}
	\item imputing the diagonal entries of $\bm{X}^{\star}\bm{X}^{\star\top}$; 
	\item computing the rank-$r$ eigenspace of $\bm{Y}\bm{Y}^\top $ with its diagonal replaced by the imputed values. 
\end{itemize}
\noindent 
See Section~\ref{sec:algorithm} for precise descriptions. 
In both theory and numerical experiments, 
this iterative paradigm yields enhanced performance compared to diagonal-deleted PCA \citep{zhang2022heteroskedastic,yan2024inference}.

\subsection{The curse of ill-conditioning} 

\begin{figure}
	\centering
	\begin{minipage}[t]{\linewidth}
		\centering
		\subfigure[Noiseless case]{
			\includegraphics[width=0.48\linewidth,height=5cm]{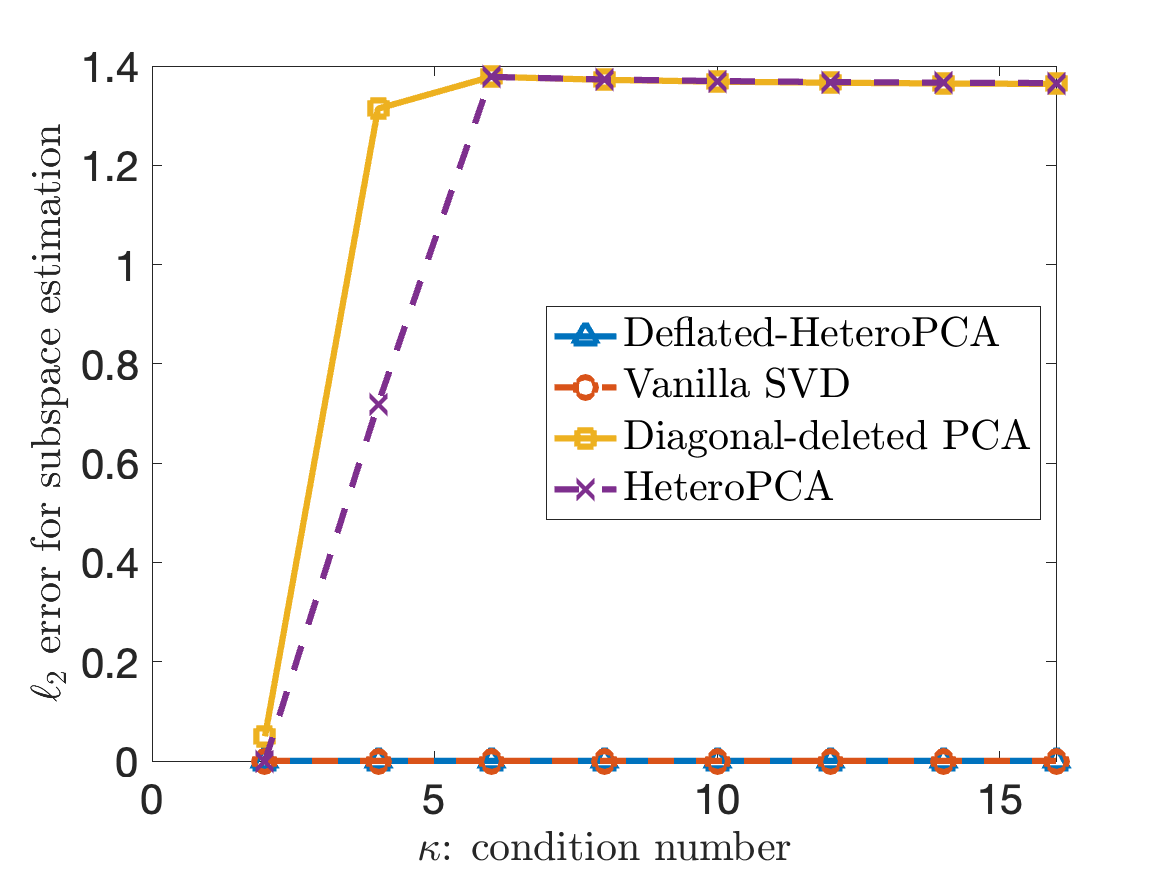}}
		\subfigure[Noisy case]{
			\includegraphics[width=0.48\linewidth,height=5cm]{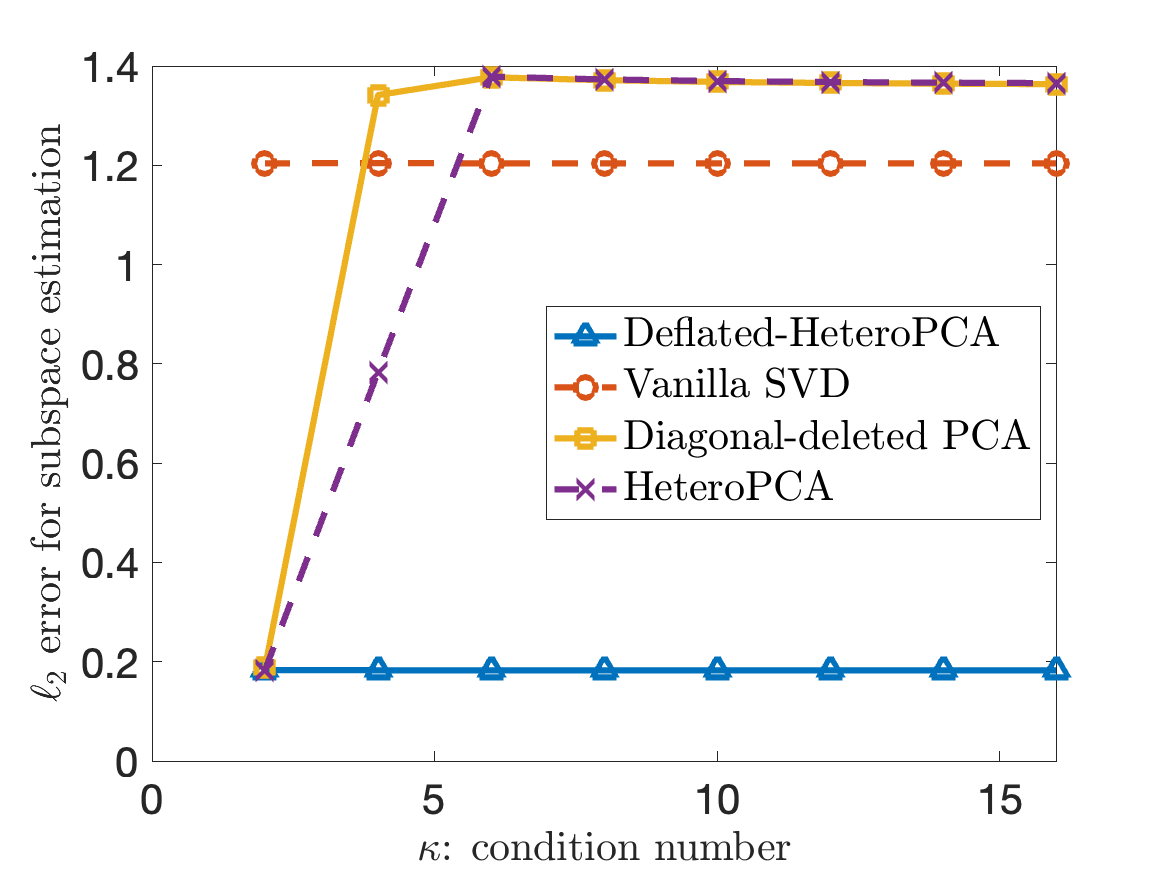}}
	\end{minipage}
	\caption{Subspace estimation error vs.~condition number $\kappa$ of $\bm{\Sigma}^{\star}$. 
	Here, we set $r = 2, n_1 = 200$ and $n_2 = 40,000$. The truth $\bm{X}^\star = \bm{U}^\star\bm{\Sigma}^\star\bm{V}^{\star\top}  $ 
	has rank 2 with $\bm{U}^\star \in \mathcal{R}^{n_1\times 2}$ and $\bm{V}^\star \in \mathcal{R}^{n_2\times 2}$ generated randomly. 
	Plot (a) represents the noiseless case ($\bm{E}=\bm{0} $).  
	In Plot (b), we choose the two singular values of $\bm{X}^{\star}$ as $\sigma_1^\star = \kappa\sigma_2^\star$ and $\sigma_2^\star = 200$,  generate $\{\omega_i\}_{1\leq i\leq n_1}$ independently from $\mathsf{Unif}([0, 2])$, and draw  the entries of  $\bm{E}=[E_{i,j}]_{1\leq i\leq n_1, 1\leq j \leq n_2}$ independently such that $E_{i,j}\sim \mathcal{N}(0, \omega_i^2)$. 
	We compare multiple subspace estimators here, where $\mathsf{HeteroPCA}$ is run with 100 iterations. 
	For each estimator $\widehat{\bm{U}}$, we compute the spectral-norm-based error $\|\widehat{\bm{U}}\bm{R}_{\widehat{\bm{U}}} - \bm{U}^{\star}\| $ as $\kappa$ varies, where $\bm{R}_{\widehat{\bm{U}}} = \argmin_{\bm{R} \in \mathcal{O}^{r, r}}\|\widehat{\bm{U}}\bm{R} - \bm{U}^\star\|_{\F}$; 
	the  results are averaged over 50 independent runs.}\label{fig:toy_example}
\end{figure}

Nevertheless, one drawback stands out when running either diagonal-deleted PCA or $\mathsf{HeteroPCA}$ in practice;  
that is, both algorithms become ineffective as the condition number of $\bm{X}^{\star}$ (when restricted to its non-zero singular values) grows. 
Let us illustrate this point more clearly via numerical experiments. 
\begin{itemize}
	\item {\bf (Numerical example)}  
		Consider the case where the unknown signal $\bm{X}^{\star}$ has rank $r=2$ and obeys $\bm{X}^\star
		= \bm{U}^\star\bm{\Sigma}^\star\bm{V}^{\star\top} ,$
		where the columns of $\bm{U}^\star\in \mathbb{R}^{n_1\times 2}$ (resp.~$\bm{V}^\star\in \mathbb{R}^{n_2\times 2}$) are the two left (resp.~right) singular vectors of $\bm{X}^{\star}$, and $\bm{\Sigma}^{\star}\in \mathbb{R}^{2\times 2}$ is a diagonal matrix composed of the two singular values $\sigma_1^{\star}\geq \sigma_2^{\star} >0$ of $\bm{X}^{\star}$.  
		Denote by $\kappa = \sigma_1^{\star} / \sigma_2^{\star}$ the condition number of $\bm{\Sigma}^{\star}$. 
		We conduct a series of experiments based on randomly generated $\bm{X}^{\star}$ with $n_2\gg n_1$, as detailed in the caption of Figure~\ref{fig:toy_example}.  
		As illustrated in Figure~\ref{fig:toy_example}, 
		when $\kappa$ is not too large, 
		both diagonal-deleted PCA and {\sf HeteroPCA} fail to return reliable estimates of the subspace $\bm{U}^\star$, 
		even in the noiseless case (i.e., $\bm{E} = 0$).

\end{itemize}
In summary, both diagonal-deleted PCA and {\sf HeteroPCA} 
suffer from the ``curse of ill-conditioning'', namely, they might lead to grossly incorrect subspace estimates 
as the largest signal component strengthens with all other signal components unchanged.  
This observation is somewhat counter-intuitive; after all, altering the signal this way 
only serves to increase the SNR and hence simplify the task from the information-theoretic perspective. 
In this sense, the aforementioned curse of ill-conditioning seems  to be algorithm-specific, 
although the two algorithms it concerns happen to be the state-of-the-art methods. 
All this naturally leads to the following question: 
{
\setlist{rightmargin=\leftmargin}
\begin{itemize}
	\item[] {\em Can we overcome the above curse of ill-conditioning  without compromising the advantages of both diagonal-deleted PCA and {\sf HeteroPCA}? }
\end{itemize}
}

\subsection{This paper}

As it turns out, we can answer the above question in the affirmative, 
which forms the main contribution of this paper.  
Our main findings are summarized as follows. 
\begin{itemize}
	\item {\em Algorithm design.}  
		In an attempt to address the above question, we propose a new algorithm --- dubbed as {\sf Deflated-HeteroPCA} --- on the basis of {\sf HeteroPCA}. In a nutshell, the proposed algorithm divides the spectrum of $\bm{X}^{\star}$ into well-conditioned yet mutually well-separated subblocks, 
		and successively applies {\sf HeteroPCA} to conquer each subblock. 
		This approach counters the adverse influence of ill conditioning via successive ``deflation'' (a term borrowed from \citet{dobriban2019deterministic}), which  gradually ``deflates'' the undesirable bias effect resulting from the diagonal deletion operation.

	\item {\em Statistical guarantees.} 
		We develop sharp theoretical guarantees, in terms of both $\ell_2$ (spectral-norm-based) and $\ell_{2, \infty}$ estimation errors,  
		for the proposed algorithm. Encouragingly, all of these statistical guarantees are 
		condition-number-free, and match the minimax lower bounds established in \cite{zhang2022heteroskedastic} and \cite{cai2021subspace} (up to some logarithmic factors).  
		To the best of our knowledge, these provide the first near-optimal results in the heteroskedastic PCA setting herein that (i) do not degrade as the condition number of the truth increases, and (ii) accommodate the widest range of SNRs. 

	\item {\em Consequences in two canonical examples.} 
		To illustrate the utility of our algorithm and theory, we develop concrete consequences of our results for two canonical examples: (a) the factor model, and (b) tensor PCA. We demonstrate that (i) {\sf Deflated-HeteroPCA} achieves rate-optimal and condition-number-free estimation under the factor model, and (ii) {\sf Deflated-HeteroPCA} followed by the HOOI algorithm improves upon the state-of-the-art performance guarantees for tensor PCA. 
		Numerical experiments are carried out to corroborate the effectiveness of the propose algorithm.

\end{itemize}

\paragraph{Paper organization.} The rest of the paper is organized as follows. We formulate the problem precisely in Section~\ref{sec:setting}, and present the proposed algorithm in Section~\ref{sec:algorithm}. The theoretical guarantees of our algorithm, along with their implications, are presented in Section~\ref{sec:theory}. We develop concrete consequences of our results in two applications in Section \ref{sec:example}. Additional numerical experiments are reported in Section~\ref{sec:simulation}, and a discussion of further related works is provided in Section \ref{sec:related_work}. The technical proofs are collected in the Appendix.

\subsection{Notation}\label{subsection:notation}
Throughout this paper, we denote $[n] := \{1, \dots, n\}$ for any positive integer $n$. 
We let  bold capital letters (e.g., $\bm{X}$) and  bold lowercase letters (e.g., $\bm{x}$) denote matrices and vectors, respectively. For any matrix $\bm{A} \in \bbR^{n_1 \times n_2}$,  $\lambda_{i}(\bm{A})$ and $\sigma_{i}(\bm{A})$ are used to represent the $i$-th largest eigenvalue (in magnitude) and the $i$-th largest singular value of $\bm{A}$, respectively. 
Let $\|\cdot\|_{\F}$ indicate the Frobenious norm and $\|\cdot\|$ the spectral norm. We denote by $\bm{A}_{i,:}$ and $\bm{A}_{:,j}$ the $i$-th column and the $j$-th row of $\bm{A}$, respectively. 
We also let $\bm{A}_{:, i:j}$ denote the submatrix of $\bm{A}$ containing those columns with indices falling in $[i,j]$.  Let $\|\bm{A}\|_{2,\infty} := \max_{i }\|\bm{A}_{i,:}\|_2$ denote the $\ell_{2,\infty}$ norm of $\bm{A}$. We use $\mathcal{O}^{n, r}:=\{\bm{U} \in \bbR^{n \times r}: \bm{U}^\top\bm{U} = \bm{I}_r\}$ to represent the set containing all $n \times r$ matrices with orthonormal columns. For any $\bm{U} \in \mathcal{O}^{n, r}$, we define the projection matrix $\mathcal{P}_{\bm{U}} = \bm{U}\bm{U}^\top$. Let $\bm{U}_{\perp} \in \mathcal{O}^{n, n-r}$ denote the orthogonal complement of $\bm{U}$. We use $\mathcal{P}_{\sf diag}(\cdot)$ to represent the projection operator that keeps all diagonal entries and sets to zero all non-diagonal entries; meanwhile, we define $\mathcal{P}_{\sf off\text{-}diag}(\bm{M}) := \bm{M} - \mathcal{P}_{\sf diag}(\bm{M})$ for any $\bm{M} \in \bbR^{n \times n}$. For any vector $\bm{a} = (a_1, \dots, a_n)$, we denote by ${\sf diag}(\bm{a}) \in \mathbb{R}^{n\times n}$ the diagonal matrix whose $(i,i)$-th entry is $a_{i}$. For any full-rank matrix $\bm{H} \in \bbR^{r \times r}$ with singular value decomposition (SVD) $\bm{U}\bm{\Sigma}\bm{V}^\top$, we define the sign matrix
\begin{align}\label{def:sign_matrix}
	{\sf sgn}(\bm{H}) := \bm{U}\bm{V}^\top.
\end{align}
We let $C, c, C_0, c_0, \dots$ denote numerical constants whose values may change from line to line.  
The boldface calligraphic letters (e.g., $\bm{\mathcal{X}}$) are used  to represent tensors. For any tensor $\bm{\mathcal{G}} \in \bbR^{r_1 \times r_2 \times r_3}$ and any matrix $\bm{V}_1 \in \bbR^{n_1 \times r_1}$, we define the multi-linear product $\times_1$ as follows: 
\begin{align*}
	\bm{\mathcal{G}} \times_1 \bm{V}_1 = \left(\sum_{j=1}^{r_1}G_{j, i_2, i_3}V_{i_1, j}\right)_{i_1 \in [n_1], i_2 \in [r_2], i_3 \in [r_3]}.
\end{align*} 
We can define $\times_2$ and $\times_3$ analogously. For any tensor $\bm{\mathcal{X}} \in \bbR^{n_1 \times n_2 \times n_3}$, let $\mathcal{M}_j(\bm{\mathcal{X}}) \in \bbR^{n_j \times (n_1n_2n_3/n_j)}$ denote the $j$-th matricization of $\bm{\mathcal{X}}$ such that for any $(i_1, i_2, i_3) \in [n_1] \times [n_2] \times [n_3]$,
\begin{align*}
	\left[\mathcal{M}_1\left(\bm{\mathcal{X}}\right)\right]_{i_1, i_2 + n_2\left(i_3 - 1\right)} = \left[\mathcal{M}_2\left(\bm{\mathcal{X}}\right)\right]_{i_2, i_3 + n_3\left(i_1 - 1\right)} = \left[\mathcal{M}_3\left(\bm{\mathcal{X}}\right)\right]_{i_3, i_1 + n_1\left(i_2 - 1\right)} = X_{i_1, i_2, i_3}.
\end{align*} 
The Frobenious norm of a tensor $\bm{\mathcal{X}} \in \bbR^{n_1 \times n_2 \times n_3}$ is defined as
\begin{align*}
	\left\|\bm{\mathcal{X}}\right\|_{\F} = \bigg(\sum_{i=1}^{n_1}\sum_{j=1}^{n_2}\sum_{k=1}^{n_3}X_{i,j,k}^2\bigg)^{1/2}.
\end{align*}

The notation $f(n_1, n_2) \lesssim g(n_1, n_2)$ or $f(n_1, n_2) = O(g(n_1, n_2))$ means that $|f(n_1, n_2)| \leq Cg(n_1, n_2)$ holds for some numerical constant $C > 0$; we let $f(n_1, n_2) \gtrsim g(n_1, n_2)$ indicate that $f(n_1, n_2) \geq C|g(n_1, n_2)|$ for some numerical constant $C > 0$; $f(n_1, n_2) \asymp g(n_1, n_2)$ means that both $f(n_1, n_2) \lesssim g(n_1, n_2)$ and $f(n_1, n_2) \gtrsim g(n_1, n_2)$ hold; we use the notation $f(n_1, n_2) \ll g(n_1, n_2)$ to represent that $f(n_1, n_2) \leq cg(n_1, n_2)$  holds for some sufficiently small constant $c > 0$, and we say $f(n_1, n_2) \gg g(n_1, n_2)$ if $g(n_1, n_2) \ll f(n_1, n_2)$. In addition, we use $f(n_1, n_2) = o(g(n_1, n_2))$ to indicate that $f(n_1, n_2)/g(n_1, n_2) \to 0$ as $\min\{n_1, n_2\} \to \infty$. For any $a, b \in \bbR$, we define $a \wedge b := \min\{a, b\}$ and $a \vee b := \max\{a, b\}$.

	\section{Problem formulation}\label{sec:setting}

\paragraph{Models and assumptions.} 
Let us present a more precise description of the problem to be studied here. 
Imagine that we have access to the following noisy data matrix:
\begin{equation}
	\bm{Y} = \bm{X}^{\star} + \bm{E} \in \mathbb{R}^{n_1\times n_2},
	\label{eq:model-setup}
\end{equation}
where $\bm{E}=[E_{i,j}]_{1\leq i\leq n_1, 1\leq j\leq n_2}$ is a zero-mean noise matrix composed of independent entries, 
and $\bm{X}^{\star}=[X^{\star}_{i,j}]_{1\leq i\leq n_1, 1\leq j\leq n_2}$ is a rank-$r$ matrix to be estimated.  
The SVD of the signal matrix $\bm{X}^\star $ is given by
\begin{align}\label{eq:svd}
	\bm{X}^{\star} = \bm{U}^{\star}\bm{\Sigma}^{\star}\bm{V}^{\star\top} = \sum_{i=1}^{r}\sigma_i^{\star}\bm{u}_i^{\star}\bm{v}_i^{\star\top} \in \bbR^{n_1 \times n_2}. 
\end{align}
Here, $\sigma_1^{\star} \geq \dots \geq \sigma_r^{\star} > 0$ denote the singular values of $\bm{X}^\star$, 
 $\bm{u}_i^{\star}$ (resp.~$\bm{v}^{\star}_i$) represents the left (resp.~right) singular vector associated with $\sigma_i^{\star}$, 
 and we introduce the matrices $\bm{\Sigma}^\star = {\sf diag}(\sigma_1^\star, \dots, \sigma_r^\star)$, 
$\bm{U}^\star = [\bm{u}_1^\star, \dots, \bm{u}_r^\star] \in \mathcal{O}^{n_1, r}$ and $\bm{V}^\star = [\bm{v}_1^\star, \dots, \bm{v}_r^\star] \in \mathcal{O}^{n_2, r}$. 
Clearly, $\bm{U}^\star $ and $\bm{V}^\star $ represent the column and row subspaces of $\bm{X}^\star$, respectively. 

Moreover, we introduce additional definitions and assumptions to be used throughout. 
\begin{itemize}
	\item To begin with, let us introduce the following incoherence condition that appears frequently in the low-rank matrix estimation literature \citep{candes2009exact,keshavan2010matrix,chen2021spectral}.
\begin{Definition}[Incoherence]\label{assump:incoherence}
	The incoherence parameters $\mu_1$ and $\mu_2$ of $\bm{X}^{\star}$ are defined as: 
	\begin{align}\label{def:incoherence}
		\mu_1 := \frac{n_1}{r}\max_{1 \leq i \leq n_1}\left\|\bm{U}_{i,:}^{\star}\right\|_2^2 \qquad \text{and} \qquad \mu_2 := \frac{n_2}{r}\max_{1 \leq j \leq n_2}\left\|\bm{V}_{j,:}^{\star}\right\|_2^2.
	\end{align}
\end{Definition}
It is self-evident that $1\leq \mu_1 \leq n_1/r$ and $1\leq \mu_2 \leq n_2/r$. 
		In words, if the incoherence parameter $\mu_1$ (resp.~$\mu_2$) is small, then the energy of of $\bm{U}^\star$ (resp.~$\bm{V}^\star$) would be more or less dispersed across all rows of $\bm{U}^\star$ (resp.~$\bm{V}^\star$). 
		Throughout this paper, for simplicity we denote
\begin{align}\label{def:n_mu}
	\mu = \max\{\mu_1, \mu_2\}  \qquad \text{and} \qquad n := \max\left\{n_1, n_2\right\}.
\end{align}

	\item Turning to the zero-mean noise matrix $\bm{E}$, we first introduce the following parameters: 
\begin{align}
	&\omega_{i,j}^2 := \Var[E_{i,j}], ~~ \omega_{\sf max}^2 := \max_{i,j}\Var\left[E_{i,j}\right], ~~
	\omega_{\sf row}^2 := \max_{i}\sum_{j=1}^{n_2}\Var\left[E_{i,j}\right],  ~~
	\omega_{\sf col}^2 := \max_{j}\sum_{i=1}^{n_1}\Var\left[E_{i,j}\right],
	\label{eq:defn-omega-set}
\end{align}
	where $\omega_{i,j}, \omega_{\sf max}, 	\omega_{\sf row}, \omega_{\sf col} \geq 0$. 
		Here, we allow the variances $\{\omega_{i,j}^2\}$ to be location-varying, 
		in order to account for {\em heteroskedasticity} of noise. 
		Moreover, we impose the following assumptions throughout: 
\begin{assump}[Noise]\label{assump:hetero}
	Suppose the noise components satisfy the following properties: 
	\begin{itemize}
		\item[1.] The $E_{i,j}$'s are statistically independent and obey $\bbE[E_{i,j}] = 0$ for all $(i,j) \in [n_1] \times [n_2]$;
		\item[2.] $\bbP(|E_{i,j}| > B) \leq n^{-12}$, 
			where the quantity $B$ satisfies $$B \leq C_{\sf b}\frac{\min\big\{\left(\omega_{\sf row}\omega_{\sf col}\right)^{1/2}, \omega_{\sf row}\big\}}{\sqrt{\log n}}$$
				for some numerical constant $C_{\sf b} > 0$.

	\end{itemize}
	
\end{assump}
\begin{rmk}
	Assumption \ref{assump:hetero}  imposes a mild condition on the tails of noise. 
	For instance, if $\omega_{i,j} \asymp \omega_{\sf max}$ for all $i,j$, then $B$ is allowed to be as large as $\min\{(n_1n_2)^{1/4}, \sqrt{n_2}\}\,\omega_{{\sf max}}$ (up to some logarithmic factor),  which can be substantially larger than the typical noise level $\omega_{{\sf max}}$. 
	In comparisons to prior works, (i) this assumption is similar to --- in fact slightly weaker than --- \citet[Assumption 2]{cai2021subspace} (in that the assumption therein requires noise distributions to be symmetric); (ii) given that Assumption~\ref{assump:hetero} is satisfied if $\{E_{i,j}\}$ are $C\omega_{\sf max}$-sub-Gaussian and $\omega_{\sf max} \lesssim \min\{(\omega_{\sf row}\omega_{\sf col})^{1/2}, \omega_{\sf row}\}/\log n$, it is less stringent than the one assumed in \citet[Theorem 4]{zhang2022heteroskedastic}.
\end{rmk}

\end{itemize}

\paragraph{Goal.} 
We seek to estimate the column subspace $\bm{U}^\star$ (up to global rotation) on the basis of $\bm{Y}$. 
Our goal is to design an estimator that satisfies the following two desirable properties simultaneously:
\begin{itemize}
	\item[1)] it allows for faithful estimation of the column subspace despite the presence of heteroskedasticity and unbalanced dimensionality; 
		  we hope to accomplish this for the widest possible range of SNRs; 

	\item[2)] it achieves the desirable statistical guarantees that do not degrade when the condition number $\kappa = \sigma_1^\star/\sigma_r^\star$ increases.
\end{itemize}

\section{Algorithms}
\label{sec:algorithm}

In this section, we proceed to describe the proposed algorithm in attempt to achieve the goal set forth in Section~\ref{sec:setting}, 
following a brief overview of previous algorithms.

\paragraph{Review: SVD, diagonal-deleted PCA and {\sf HeteroPCA}.}
Before continuing, we briefly review three popular methods that are commonly studied in the literature. 
\begin{itemize}
	\item {\em The vanilla SVD-based approach.} 
		This approach computes the leading $r$ singular vectors of $\bm{Y}$, or equivalently, the top-$r$ eigenspace of the Gram matrix $\bm{Y}\bm{Y}^\top$, namely, 
		\begin{equation}
			\text{(vanilla SVD)}\qquad  \widehat{\bm{U}}_{\mathsf{svd}} ~\leftarrow~ \mathsf{eigs}_r\big(\bm{Y}\bm{Y}^\top \big),
			\label{eq:vanilla-SVD-alg}
		\end{equation}
		where $\mathsf{eigs}_r(\cdot)$ stands for the leading rank-$r$ eigen-subspace of a matrix. 
		While this approach works well when $n_2 = O(n_1)$, it suffers from some fundamental limitations in the case with $n_2\gg n_1$ and heteroskedastic noise. 
		To illustrate this point, direct calculation reveals that
\begin{align}\label{eq:expectation_gram_matrix}
	\bbE\left[\bm{Y}\bm{Y}^\top\right] = \bm{X}^\star\bm{X}^{\star\top} + {\sf diag}\Bigg(\bigg[\sum_{j=1}^{n_2}\bbE\left[E_{i,j}^2\right]\bigg]_{1 \leq i \leq n_1}\Bigg).
\end{align}
		When $n_2\gg n_1$ and when the noise components are highly heteroskedastic, 
		the set of diagonal entries $\big\{\sum_{j=1}^{n_2}\bbE\left[E_{i,j}^2\right] \big\}_{1\leq i\leq n_1}$ might vary drastically, 
		thereby resulting in a large deviation between the top-$r$ eigenspace of $\bbE[\bm{Y}\bm{Y}^\top]$  and that of  $\bm{X}^{\star}\bm{X}^{\star\top}$ (which is the desirable $\bm{U}^{\star}$). 

	\item 	{\em Diagonal-deleted PCA.} 
		In an effort to rectify the above limitation of the vanilla SVD-based approach, 
		prior works have put forward a solution called ``diagonal-deleted PCA,'' 
		which suppresses the influence of the diagonal entries of  $\bm{Y}\bm{Y}^\top$ by suppressing them \citep{koltchinskii2000random,florescu2016spectral,cai2021subspace,ndaoud2021improved,ndaoud2022sharp,abbe2022lp};  that is, this approach outputs
\begin{align}
	\text{(diagonal-deleted PCA)}\qquad
	\widehat{\bm{U}}_{\mathsf{del}} ~\leftarrow~ \mathsf{eigs}_r\big(\bm{Y}\bm{Y}^\top - \mathcal{P}_{\sf diag}(\bm{Y}\bm{Y}^\top)\big),
	\label{eq:diagonal-deletion}
\end{align}
where $\mathcal{P}_{\sf diag}$ denotes Euclidean projection onto the set of diagonal matrices.
		When the diagonal entries of $\bm{X}^\star\bm{X}^{\star\top}$ are sufficiently small, we have
\begin{align*}
	\bbE\left[\mathcal{P}_{\sf off\text{-}diag}\left(\bm{Y}\bm{Y}^\top\right)\right]  = \bm{X}^\star\bm{X}^{\star\top} - \mathcal{P}_{\sf diag}\left(\bm{X}^\star\bm{X}^{\star\top}\right) \approx \bm{X}^\star\bm{X}^{\star\top} = \bm{U}^\star\bm{\Sigma}^{\star2}\bm{U}^{\star\top}, 
\end{align*}
		which forms the rationale of this approach. 

	\item {\em The {\sf HeteroPCA} algorithm.}  
		The above diagonal-deleted approach can be further improved.  
		Employing \eqref{eq:diagonal-deletion} as an initialization, \citet{zhang2022heteroskedastic} put forward the {\sf HeteroPCA} algorithm 
		that combines the spectral method with successively refined diagonal estimates; 
		more precisely, {\sf HeteroPCA} initializes $\bm{G}$ as $\mathcal{P}_{\sf off\text{-}diag}(\bm{Y}\bm{Y}^\top)$, 
		and alternates between the following two steps until convergence: 
		\begin{align*}
			(\mathsf{HeteroPCA})
			\qquad  \text{repeat}\quad
			&\,\text{(i)} ~~ \bm{U}\bm{\Lambda}\bm{U}^{\top} \,\leftarrow\, \text{rank-}r\text{ eigendecomposition of }(\bm{G}); \\
			&\text{(ii)} ~~ \bm{G} \,\leftarrow\, \mathcal{P}_{\mathsf{off}\text{-}\mathsf{diag}}\big(\bm{Y}\bm{Y}^{\top}\big) 
			+ \mathcal{P}_{\mathsf{diag}}\big(\bm{U}\bm{\Lambda}\bm{U}^{\top}\big). 
		\end{align*}
		See Algorithm \ref{algorithm:heteroPCA} for a complete description of this procedure, with the input matrix (or initialization) chosen to be $\bm{G}_{\mathsf{in}}=\bm{Y}\bm{Y}^{\top} - \mathcal{P}_{\sf diag}(\bm{Y}\bm{Y}^\top)$. 
		The key lies in employing the improved diagonal estimates to help alleviate the bias induced by diagonal deletion.

\end{itemize}

\begin{algorithm}[t]
	\caption{{\sf HeteroPCA($\bm{G}_{\mathsf{in}}$, $r$, $t_{\sf max}$)} ~\citep{zhang2022heteroskedastic}} \label{algorithm:heteroPCA}
	\DontPrintSemicolon
	\textbf{input:} symmetric matrix $\bm{G}_{\mathsf{in}}$, rank $r$, number of iterations $t_{\sf max}$.\\
	\textbf{initialization:} $\bm{G}^{0} = \bm{G}_{\mathsf{in}}$.\\
	\For{$t = 0, 1, \dots, t_{\sf max}$} 
	         {
	         	$\bm{U}^{t}\bm{\Lambda}^{t}\bm{U}^{t\top}$ $\,\leftarrow\,$  rank-$r$ leading eigendecompostion of $\bm{G}^t$.\\
	         	$\bm{G}^{t+1} = \mathcal{P}_{\sf off\text{-}diag}\left(\bm{G}^{t}\right) + \mathcal{P}_{\sf diag}\left(\bm{U}^{t}\bm{\Lambda}^{t}\bm{U}^{t\top}\right)$.\\
	        }
    \textbf{output:} matrix estimate $\bm{G} = \bm{G}^{t_{\sf max}}$ and subspace estimate $\bm{U} = \bm{U}^{t_{\sf max}}$.    
\end{algorithm}

\noindent 
When the condition number $\sigma_1^{\star}/\sigma_r^{\star}$ is large, however, the magnitude of the diagonal entries of $\bm{X}^\star\bm{X}^{\star\top}$ can be substantially larger than, say, the square of the least singular value  of $\bm{X}^{\star}$ (i.e., $\sigma_r^{\star2}$). 
If this is the case, then diagonal-deleted PCA might eraze a significant fraction of the useful signal, 
resulting in loss of effectiveness. 
This issue carries over to {\sf HeteroPCA}, 
as its initialization --- which is based on diagonal-deleted PCA --- might already be highly unreliable.

\paragraph{The proposed algorithm: {\sf Deflated-HeteroPCA}.}
We now describe how to alleviate the above curse of ill-conditioning. 
One lesson that we have learned from past {\sf HeteroPCA} theory \citep{zhang2022heteroskedastic,yan2024inference} is that: 
this procedure works well if (i) the condition number of the truth is well-controlled and (ii) the least singular value is not buried by noise. 
Motivated by this fact, we propose to divide the set of eigenvalues of interest into ``well-conditioned'' subblocks that are sufficiently separated from each other, and include more subblocks one by one. More precisely, the main ideas of the proposed algorithm are as follows: 
\begin{itemize}
	\item[1)] Sequentlly identify a collection of ranks $r_0 = 0 < r_1 < r_2 < \cdots < r_{k_{\sf max}} = r$, 
		which partitions the set of eigenvalues (or singular values) of interest into disjoint subblocks. 
		These points are chosen to ensure that (i)   $\sigma_{r_{k-1}+1}^\star/\sigma_{r_{k}}^\star$ is sufficiently small for each $k$, 
		and (ii) there is a sufficient gap between $\sigma_{r_k}^\star$ and $\sigma_{r_{k}+1}^\star$. 
		Given that we do not know the true signular values {\em a priori}, 
		we shall make careful use of the singular values of our running estimates instead. 

	\item[2)] In the $k$-th round, we invoke {\sf HeteroPCA} with the rank $r_k$ and the initialization $\bm{G}_{k-1}$ to impute the diagonal entries and obtain an improved estimate  $\bm{G}_k$ of the Gram matrix of interest. Here, the first iteration employs the diagonal-deleted version $\bm{G}_0 = \mathcal{P}_{\sf off\text{-}diag}(\bm{Y}\bm{Y}^\top)$. 
\end{itemize}
It then boils down to how to select the aforementioned ranks $\{r_k\}$ in a data-driven manner. 
Towards this end, we look at the following set of ranks in the $k$-th round:\footnote{The threshold $4$ in \eqref{def:R_k} can be replaced with any numerical constant $C_{\sf gap} \geq 4$.}
\begin{align}\label{def:R_k}
	\mathcal{R}_k := \bigg\{r': r_{k-1} < r' \leq r,~ \frac{ \sigma_{r_{k-1}+1}\left(\bm{G}_{k-1}\right) }{\sigma_{r'}\left(\bm{G}_{k-1}\right) } \leq 4 ~\text{ and }~ \sigma_{r'}\left(\bm{G}_{k-1}\right) - \sigma_{r'+1}\left(\bm{G}_{k-1}\right) \geq \frac{1}{r}\sigma_{r'}\left(\bm{G}_{k-1}\right)\bigg\},
\end{align}
and select $r_k$ as follows:
\begin{align}\label{def:r_1}
	r_k = \begin{cases}
		\max\mathcal{R}_k, &\text{ if } \mathcal{R}_k \neq \emptyset,\\
		r, & \text{ otherwise}.
	\end{cases}
\end{align}
Here, we remind the readers that $\sigma_i(\bm{G}_{k-1}) $ is the $i$-th singular value of $\bm{G}_{k-1}$. 
Evidently, the first condition in \eqref{def:R_k} is imposed to ensure well-conditioning of each subblock, 
whereas the second condition in \eqref{def:R_k} aims to guarantee a sufficient spectral separation between adjacent subblocks.

In a nutshell, the proposed algorithm counters the bias effect initially incurred by diagonal deletion via successive ``deflation'', 
a term that we borrow from \cite{dobriban2019deterministic} (although the problem considered therein is drastically different). 
More concretely, we first estimate the first subblock (which contains the largest eigenvalues of interest) by means of the diagonal deletion idea; 
once we finish estimating the eigen-subspace associated with this subblock, 
we can readily compensate for the contribution of this subblock in the diagonal of interest.  
This strategy is then repeated subblock by subblock in order to successively reduce --- or ``deflate'' --- the original bias in the diagonal. 
For this reason, we refer to the proposed algorithm as {\sf Deflated-HeteroPCA}, 
whose complete details are  summarized in Algorithm \ref{algorithm:sequential_heteroPCA}. The computation cost of {\sf Deflated-HeteroPCA} (Algorithm \ref{algorithm:sequential_heteroPCA}) is $\widetilde{O}(n_1^2n_2 + n_1^2r\sum_{k = 1}^{k_{\sf max}}t_k)$. Here, $\widetilde{O}(b)$ 
is equivalent to $O(b)$ except that it hides the logarithmic factors. 

The computational cost of the initialization step is $O(n_1^2n_2)$. For other steps, the main computation cost is attributed to the top-$r_k$ eigendecomposition, which amounts to $\widetilde{O}(n_1^2r)$. Numerically, by setting all $t_k$'s equal to $10$, the algorithm performs well and the computational cost simplifies to $\widetilde{O}(n_1^2n_2 + n_1^2rk_{\sf max}) = \widetilde{O}(n_1^2n_2 + n_1^2r^2)$ (recall that the number of blocks $k_{\sf max}$ is at most $r$). As a comparison, the computation cost of {\sf HeteroPCA} is $\widetilde{O}(n_1^2n_2 + n_1^2rt)$, where $t$ is the number of iterations. As a result, it can be seen that {\sf Deflated-HeteroPCA} does not incur a higher computational burden than {\sf HeteroPCA} when $r=O(\sqrt{n_2})$.

\begin{algorithm}[t]
	\caption{\sf Deflated-HeteroPCA} \label{algorithm:sequential_heteroPCA}
		\DontPrintSemicolon
	\textbf{input:} data matrix $\bm{Y}$ (cf.~\eqref{eq:model-setup}), rank $r$, maximum number of iterations $t_i$, $i =1, 2, ...$\\
	\textbf{initialization:} $k = 0, r_0 = 0, \bm{G}_0 = \mathcal{P}_{\sf off\text{-}diag}\left(\bm{Y}\bm{Y}^\top\right)$.\\
	\While{$r_k < r$}{
		$k = k+1$.\\
		select $r_k$ via Eqn.~\eqref{def:r_1}.\\
		$\left(\bm{G}_k, \bm{U}_k\right) = ${\sf HeteroPCA}$\left(\bm{G}_{k-1}, r_k, t_k\right)$.		
    }
    \textbf{output:} subspace estimate $\bm{U} = \bm{U}_k$.
\end{algorithm}

     \section{Main theory}\label{sec:theory}

In this section, we demonstrate the desirable statistical performance for the proposed algorithm, 
which enjoys substantially improved dependency on the condition number.
Before continuing, we find it helpful to introduce the following rotation matrix for any $\bm{U}\in \mathcal{O}^{n_1, r}$:
\begin{align}\label{def:R}
	\bm{R}_{\bm{U}} = \argmin_{\bm{R} \in \mathcal{O}^{r, r}}\left\|\bm{U}\bm{R} - \bm{U}^\star\right\|_{\F},
\end{align}
the one that best aligns $\bm{U}$ with $\bm{U}^\star$ in the Euclidean sense;  
after all, it is in general infeasible to resolve the ambiguity brought by global rotation.    
As is well known in the literature (e.g., \citet[Section~D.2.1]{ma2020implicit}), 
\begin{align}
	\bm{R}_{\bm{U}} = \mathsf{sgn}\left(\bm{U}^\top\bm{U}^\star\right), 
\end{align}
where ${\sf sgn}(\cdot)$ is defined in \eqref{def:sign_matrix}.

\subsection{Spectral-norm-based statistical guarantees}

Let us begin with statistical guarantees based on the spectral norm accuracy. 
The following theorem asserts that the proposed {\sf Deflated-HeteroPCA} algorithm enjoys appealing theoretical guarantees in terms of the spectral norm error $\|\bm{U}\bm{R}_{\bm{U}} - \bm{U}^\star\|$, no matter how large the condition number of $\bm{\Sigma}^\star$ is. The proof of this theorem is deferred to Section~\ref{proof:thm_spectral}.
\begin{thm}\label{thm:spectral}
	Suppose that Assumption \ref{assump:hetero} holds. 
	Assume that 
	\begin{subequations}
	\begin{align}
		\sigma_r^\star &\geq C_0r\left(\omega_{\sf col} + \sqrt{\omega_{\sf col}\omega_{\sf row}}\right)\sqrt{\log n} \label{assump:snr2}\\
		\mu &\leq c_0\frac{n_1}{r^3} \label{incoherence} \\
		0 &< \mu r\omega_{\sf max}^2 \le \omega_{\sf col}^2
	\end{align}
	\end{subequations}
	for some sufficiently large (resp.~small) constant $C_0 > 0$ (resp.~$c_0 > 0$).
	If the numbers of iterations obey
	\begin{subequations}
		\label{ineq:iter-12}
		\begin{align}
			t_k & > \log\left(C\frac{\sigma_{r_{k-1}+1}^{\star2}}{\sigma_{r_{k}+1}^{\star2}}\right), \qquad 1 \leq k < k_{\sf max} \label{ineq:iter1}\\
			t_{k_{\sf max}} &> \log\left(C\frac{\sigma_{r_{k_{\sf max}-1}+1}^{\star2}}{\omega_{\sf max}^2}\right)\label{ineq:iter2}
		\end{align}
	\end{subequations}
	for some large enough constant $C>0$, 
	then with probability exceeding $1 - O(n^{-10})$, the output returned by Algorithm~\ref{algorithm:sequential_heteroPCA} satisfies
	\begin{align}\label{ineq:svd_spectral2}
		\left\|\bm{U}\bm{R}_{\bm{U}} - \bm{U}^\star\right\| \lesssim \frac{\omega_{\sf col}\sqrt{\log n}}{\sigma_r^{\star}} + \frac{\omega_{\sf col}\omega_{\sf row}\log n}{\sigma_r^{\star2}}.
	\end{align}
	Here,  $r_0 = 0$, $r_1, \dots, r_{k_{\sf max}}$ are the ranks selected in Algorithm \ref{algorithm:sequential_heteroPCA} and $k_{\sf max}$ satisfies $r_{k_{\sf max}} = r$.
\end{thm}

We find it helpful to compare our theoretical guarantees with prior theory for this problem. 
To begin with, the prior theory \cite{zhang2022heteroskedastic} only covers the well-conditioned case; 
when $\kappa $ is a bounded constant (as assumed therein), 
our statistical error bound \eqref{ineq:svd_spectral2} matches the one in \citet[Theorem 4]{zhang2022heteroskedastic} (up to some logarithmic factors).\footnote{\cite{zhang2022heteroskedastic} establishes estimation guarantees for the $\sin\Theta$ distance $\|\sin\Theta(\widehat{\bm{U}}, \bm{U}^\star)\|$, which is (nearly) equivalent to the metric $\min_{\bm{R} \in \mathcal{O}^{r \times r}}\|\widehat{\bm{U}}\bm{R} - \bm{U}^\star\|$ (or more precisely, $\|\sin\Theta(\widehat{\bm{U}}, \bm{U}^\star)\| \asymp \min_{\bm{R} \in \mathcal{O}^{r \times r}}\|\widehat{\bm{U}}\bm{R} - \bm{U}^\star\|$). See \cite[Lemma 2.6]{chen2021spectral} for details.}
In addition, when it comes to the case where $\omega_{i,j} \asymp \omega_{\sf max}$ for all $(i,j)\in [n_1]\times [n_2]$, our error bound \eqref{ineq:svd_spectral2} simplifies to
	\begin{align*}
		\left\|\bm{U}\bm{R}_{\bm{U}} - \bm{U}^\star\right\| \lesssim \frac{\sqrt{n_1\log n}\,\omega_{\sf max}}{\sigma_r^{\star}} + \frac{\sqrt{n_1n_2\log^2 n}\,\omega_{\sf max}^2}{\sigma_r^{\star2}},
	\end{align*} 
	which matches the minimax lower bounds  \citet[Theorem 2]{cai2021subspace} and \citet[Theorem 4]{cai2018rate} (ignoring logarithmic factors). 
	It is noteworthy that when $\omega_{i,j} \asymp \omega_{\sf max}$ for all $(i,j)\in [n_1]\times [n_2]$ and $r = O(1)$, 
	the signal-to-noise ratio condition \eqref{assump:snr2} simplifies to 
	\begin{align}
		\sigma_r^\star \gtrsim \left[\left(n_1n_2\right)^{1/4} + n_1^{1/2}\right]\omega_{\sf max}\sqrt{\log n}
	\end{align}
	which is necessary to ensure ---  up to logarithmic factor --- the existence of a consistent estimator (which means the existence of an estimator $\widehat{\bm{U}}$ obeying	$\|\widehat{\bm{U}}\bm{R}_{\widehat{\bm{U}}} - \bm{U}^\star\| = o(1)$) (see \citet[Theorem 2]{cai2021subspace}).

\subsection{Fine-grained $\ell_{2,\infty}$-norm-based statistical guarantees}
Moving beyond the spectral norm bounds, 
we proceed to the fine-grained  $\ell_{2,\infty}$-norm-based error bounds for column subspace estimation, 
which further capture how well the estimation error is spread out across the rows  \citep{ma2020implicit,chen2020noisy,chen2019inference,chen2021bridging,agterberg2022entrywise,zhang2022leave,cai2022nonconvex}. 
As has been shown in the literature, 
such $\ell_{2,\infty}$-based subspace estimation guarantees 
play a crucial role in deriving performance bounds for the subsequent tasks like entrywise covariance estimation, entrywise tensor estimation,  
exact recovery in a variety of clustering and mixture models \citep{cai2021subspace,yan2024inference,abbe2020entrywise,cai2021subspace,abbe2022lp}. 

Before formally presenting our $\ell_{2,\infty}$-norm-based result, we first introduce the following assumption on the noise matrix $\bm{E}$.
\begin{assump}\label{assump:hetero2}
	Suppose that the noise components satisfy Condition 1 in Assumption \ref{assump:hetero}. In addition, we assume that
	\begin{equation}
		\bbP\left(\left|E_{i,j}\right| > B\right) \leq n^{-12} ,
	\end{equation}
	where $B$ satisfies, for some  universal constant $C_{\sf b} > 0$, that
	$$B \leq C_{\sf b}\omega_{\sf max}\frac{\min\big\{\left(n_1n_2\right)^{1/4}, \sqrt{n_2}\big\}}{\log n}.$$
\end{assump}
\begin{rmk}
	Our assumptions on the noise are very mild and they hold across a diverse array of distributions, including
	\begin{itemize}
		\item uniform distributions;
		\item $C\omega_{\sf max}$-sub-Gaussian random variables;
		\item centered Poisson random variables with parameter $\lambda_{\sf max} = \omega_{\sf max}^2 \gtrsim \frac{\log^4 n}{\min\{(n_1n_2)^{1/2}, n_2\}}$;
		\item centered Bernoulli random variables with $p_{i,j} \in [\frac{\log^2 n}{C_{\sf b}^2\min\{(n_1n_2)^{1/2}, n_2\}}, 1 - \frac{\log^2 n}{C_{\sf b}^2\min\{(n_1n_2)^{1/2}, n_2\}}]$. 
	\end{itemize}
	In addition, it is worth noting that the constant $12$ can be replaced by any other constant $c > 2$ to enusre a high-probability result. Here, we choose $12$ simply to guarantee that the final estimation error bound holds with probability exceeding $1 - O(n^{-10})$. With the logarithmic factors neglected, the only difference between Assumption \ref{assump:hetero2} and \citet[Assumption 2]{cai2021subspace} is that no symmetric distribution requirement is needed in Assumption \ref{assump:hetero2}. 
\end{rmk}
Built upon Assumption \ref{assump:hetero2}, we derive the following $\ell_{2,\infty}$-based theoretical guarantees for {\sf Deflated\text{-}HeteroPCA}, 
with the proof postponed to Section \ref{proof:thm_two_to_infty}.
\begin{thm}\label{thm:two_to_infty}
	Suppose that Assumption \ref{assump:hetero2} holds and the signal-to-noise ratio satisfies
	\begin{subequations}
		\label{assump:snr_two_to_infty-all}
	\begin{align}
		\frac{\sigma_r^\star}{\omega_{\sf max}} &\geq C_0r\left[\left(n_1n_2\right)^{1/4} + n_1^{1/2}\right]\log n \label{assump:snr_two_to_infty}\\
		\mu &\leq c_0\frac{n_1}{r^3} \label{assump:snr_two_to_infty_234}
	\end{align}
	\end{subequations}
	for some large (resp.~small) enough constant $C_0>0$ (resp.~$c_0>0$). 
	If the numbers of iterations satisfy \eqref{ineq:iter-12}, then with probability exceeding $1 - O(n^{-10})$, 
	then the estimate returned by Algorithm~\ref{algorithm:sequential_heteroPCA} satisfies
	\begin{subequations}
		\label{eq:two-to-inf-and-L2}
		\begin{align}
			\left\|\bm{U}\bm{R}_{\bm{U}} - \bm{U}^\star\right\| _{2,\infty} &\lesssim \sqrt{\frac{\mu r}{n_1}}\zeta_{\sf op},\label{ineq:two_to_infty_guarantee}\\
			\left\|\bm{U}\bm{R}_{\bm{U}} - \bm{U}^\star\right\| &\lesssim \zeta_{\sf op},\label{ineq:spectral_guarantee}
		\end{align}
	\end{subequations}
	where
	\begin{align}\label{def:zate_op}
		\zeta_{\sf op} = \frac{\sqrt{n_1n_2}\,\omega_{\sf max}^2\log^2 n}{\sigma_r^{\star2}} + \frac{\sqrt{n_1}\,\omega_{\sf max}\log n}{\sigma_{r}^\star}.
	\end{align}
\end{thm}
Encouragingly, 
both the $\ell_{2,\infty}$-based and spectral-norm-based estimation guarantees in \eqref{eq:two-to-inf-and-L2} match the minimax lower bounds 
previously established in \citet[Theorem 2]{cai2021subspace} (up to logarithmic factors), 
thus confirming the near minimax optimality of our results. 
It can also been seen from \citet[Theorem 2]{cai2021subspace} that 
the signal-to-noise ratio requirement \eqref{assump:snr_two_to_infty} is, in general, essential (ignoring logarithmic factors) 
in order to enable the plausibility of consistent estimation.

\paragraph{Comparison with prior results.} 
In order to demonstrate the utility of our algorithm and the accompanying theory, 
we compare our results with past works in the sequel. 
To ease presentation, the discussion below focuses attention on the case where $\mu, r = O(1)$. 
\begin{itemize}

	\item {\em Requirement on the condition number $\kappa$.} 
		In order to obtain a consistent estimator\footnote{Here, a column subspace estimator $\widehat{\bm{U}}$ 
		is said to be consistent if $\min_{\bm{R} \in \mathcal{O}^{r, r}}\|\widehat{\bm{U}}\bm{R} - \bm{U}^\star\| = o(1)$.}, 
		 all prior theory for both diagonal-deleted PCA (see \citet[Theorem 1]{cai2021subspace}) and {\sf HeteroPCA} (see \citet[Theorem 4]{zhang2022heteroskedastic}, \citet[Theorem 5]{yan2024inference} and \citet[Assumption 4]{agterberg2022entrywise}) assumes the condition number $\kappa$ to obey 
		 \begin{equation}
			 \text{(prior requirement on }\kappa\text{)} 
			 \qquad \kappa \lesssim n_1^{1/4}, 
		\end{equation}
		in order to control the bias incurred during the diagonal deletion step. 
		This, however, falls short of accommodating a wider range of condition numbers. 
		In contrast, our result in Theorem \ref{thm:two_to_infty} does not impose any assumptions on the condition number.

	\item {\em Statistical error bounds.} 
	We now compare our statistical error bounds with the ones obtained in \cite{cai2021subspace,agterberg2022entrywise,yan2024inference}. 
	For notational convenience, define
	\begin{equation}
		\mathcal{E}_{\sf noise} := \frac{\sqrt{n_1n_2}\,\omega_{\sf max}^2\log n}{\sigma_r^{\star2}} + \frac{\kappa\omega_{\sf max}\,\sqrt{n_1\log n}}{\sigma_{r}^\star},
	\end{equation}
	which makes it more convenient for us to describe the previous results. 
	\begin{itemize}
		\item

		Under the signal-to-noise ratio condition
	\begin{align}\label{snr:diagonal_deletion}
		\frac{\sigma_r^\star}{\omega_{\sf max}} \gtrsim \left(\kappa\left(n_1n_2\right)^{1/4} + \kappa^3n_1^{1/2}\right)\sqrt{\log n}, 
	\end{align}
	\citet[Theorem 1]{cai2021subspace} asserts that the estimate $\widehat{\bm{U}}_{\mathsf{del}}$ returned by  diagonal-deleted PCA obeys, with high probability,
	\begin{align}\label{rate:diagonal_deletion}
		\min_{\bm{R} \in \mathcal{O}^{r, r}}\big\|\widehat{\bm{U}}_{\mathsf{del}}\bm{R} - \bm{U}^\star\big\| _{2,\infty} \lesssim \kappa^2\sqrt{\frac{1}{n_1}}\big(\mathcal{E}_{\sf noise} + \mathcal{E}_{\sf diag\text{-}del}\big) ,
	\end{align}
	where $\mathcal{E}_{\sf diag\text{-}del}$ is an additional error term due to the bias resulting from diagonal deletion. 

	\item 
	Focusing on the case where $n_2 \gtrsim n_1$, \citet[Theorem 2]{agterberg2022entrywise} establishes an $\ell_{2,\infty}$ error bound for the {\sf HeteroPCA} estimate $\widehat{\bm{U}}_{\sf hpca}$ as follows: 
	\begin{align}\label{rate:heteroPCA1}
		\min_{\bm{R} \in \mathcal{O}^{r, r}}\big\|\widehat{\bm{U}}_{\sf hpca}\bm{R} - \bm{U}^\star\big\| _{2,\infty} \lesssim \sqrt{\frac{1}{n_1}}\mathcal{E}_{\sf noise},
	\end{align}
	albeit under a much more stringent SNR requirement:  
	\begin{align}\label{snr:heteroPCA1}
		\sigma_r^\star \gg \kappa\omega_{\sf max}\sqrt{n_2 \log n}.
	\end{align}

	\item 
		\citet[Theorem 5]{yan2024inference} further shows that under the same SNR condition \eqref{snr:diagonal_deletion},  {\sf HeteroPCA} yields an estimator $\widehat{\bm{U}}_{\sf hpca}$ with the following high-probability $\ell_{2, \infty}$ error bound:
	\begin{align}\label{rate:heteroPCA2}
		\min_{\bm{R} \in \mathcal{O}^{r, r}}\big\|\widehat{\bm{U}}_{\sf hpca}\bm{R} - \bm{U}^\star\big\| _{2,\infty} \lesssim \kappa^2\sqrt{\frac{1}{n_1}}\mathcal{E}_{\sf noise}.
	\end{align}
	
	\end{itemize}

	Let us compare our bounds with the above results. 
Recognizing that $\mathcal{E}_{\sf noise}$ is at least as large as $ \zeta_{\sf op}$ if we ignore logarithmic factors, 
our $\ell_{2,\infty}$ error bound \eqref{ineq:two_to_infty_guarantee} improves the theoretical guarantees \eqref{rate:diagonal_deletion} and \eqref{rate:heteroPCA2} by at least a factor of $\kappa^2$. Additionally,  our bound \eqref{ineq:two_to_infty_guarantee} outperforms the bound \eqref{rate:heteroPCA1} in terms of the dependency on $\kappa$ (ignoring logarithmic factors).

	\item 
	{\em SNR requirement.} 
	Let us also briefly make comparisons regarding the SNR required for consistent estimation. 
	To begin with, we make note that the vanilla SVD-based approach (cf.~\eqref{eq:vanilla-SVD-alg}) 
	requires the SNR to exceed \citep{cai2021subspace,zhang2022heteroskedastic}
	\begin{equation}
		\frac{ \sigma_r^\star }{ \omega_{\sf max} } \gtrsim \sqrt{n_1} + \sqrt{n_2},
	\end{equation}
	which can be substantially more stringent than the one required in \eqref{assump:snr_two_to_infty} if $n_2 \gg n_1$.
   In addition, compared with the SNR requirement imposed in the existing theory for diagonal-deleted PCA and {\sf HeteroPCA}, our condition \eqref{assump:snr_two_to_infty} is weaker than the one used in \cite{cai2021subspace} and \cite{yan2024inference} (see \eqref{snr:diagonal_deletion}) by at least a factor of $\kappa$, 
   while at the same time being weaker than the condition \eqref{snr:heteroPCA1} assumed in \cite{agterberg2022entrywise} by a factor of $\kappa(n_2/n_1)^{1/4}$ when $n_2 \gg n_1$.

\end{itemize}

\paragraph{High-level proof strategy.}

While the proofs of our main theorems are deferred to the Appendix, 
we highlight some novelty and technical challenges in our proof. 
In an attempt to obtain fine-grained $\ell_{2,\infty}$ control while remaining condition-number-free, 
we develop a new proof strategy that differs drastically from the state-of-the-art techniques based on leave-one-out decoupling arguments \citep{yan2024inference,cai2021subspace}. 
Inspired by a spectral representation lemma derived in the recent work \cite{xia2021normal} (see also Lemma \ref{lm:representation}), 
we proceed by decomposing the difference between the subspaces into an infinite sum of polynomials of the error matrix. 
With this decomposition at hand, one major part of our proof hinges upon establishing sharp $\ell_{2,\infty}$ bounds on each of the polynomials of the error matrix. 
The key challenge for this part lies in how to deal with the complicated and accumulated dependence brought by the power of the error matrix, 
for which we resort to careful induction analyses. 
We will then single out several sequences of critical quantities and develop intricate arguments to control these quantities in a recursive and inductive manner.

    \section{Consequences for specific models}\label{sec:example}
To better illustrate the effectiveness of the proposed algorithm, 
we develop concrete consequences of our theory in Section~\ref{sec:theory} for two specific models. 
In each case, we shall begin by describing the model, followed by concrete algorithms and theory tailored to the specific model.

\subsection{Factor models and spiked covariance models}\label{sec:example_pca}
\paragraph{Model.} 
A frequently studied model employed to capture low-dimensional structure in high-dimensional sample data is  the factor model, 
which finds applications numerous contexts including finance and econometrics \citep{lawley1962factor,fan2020statistical,fan2021robust}, functional magnetic resonance imaging \citep{chen2015reduced}, and signal processing \citep{zhao1986detection,kritchman2008determining,kritchman2009non},  to name just a few. For concreteness, suppose that we observe a collection of $n$ independent sample vectors in $\mathbb{R}^d$ generated as follows:
\begin{subequations}
\label{eq:factor-model}
\begin{align}
	\bm{y}_j &= \bm{B}^\star\bm{f}_j + \bm{\varepsilon}_j \in \bbR^d, 
\end{align}
where $\bm{B}^\star \in \bbR^{d\times r}$ represents the factor loading matrix with $r\ll d$, $\{\bm{f}_j\}$ stands for the latent factor vectors, 
and 
$\{\bm{\varepsilon}_j\}$ denotes the noise vectors.  We assume that 
\begin{align}
	\bm{B}^\star = \bm{U}^\star\bm{\Lambda}^{\star1/2} \in \bbR^{d\times r} 
	\qquad \text{and}\qquad
	\bm{f}_j \stackrel{\rm{i.i.d.}}{\sim} \mathcal{N}\left(\bm{0}, \bm{I}_r\right),~~ 1 \leq j \leq n, 
\end{align}
\end{subequations}
with $\bm{U}^{\star}\in \mathcal{O}^{d,r}$ and $\bm{\Lambda}^\star = {\sf diag}(\lambda_1^\star, \dots, \lambda_r^\star)$ being a diagonal matrix containing all eigenvalues of $\bm{B}^\star\bm{B}^{\star\top}$. 
 Equivalently, one can express it as the following spiked covariance model: 
\begin{align}
	\bm{y}_j &= \bm{x}_j + \bm{\varepsilon}_j,  \qquad \text{with } \bm{x}_j \overset{\text{i.i.d.}}{\sim} \mathcal{N}\left(\bm{0}, \bm{U}^\star\bm{\Lambda}^\star\bm{U}^{\star\top}\right),  \quad 1 \leq j \leq n.\label{model:PCA}
\end{align}
The noise vectors are allowed to be heteroskedastic, and it is assumed that
\begin{itemize}
	\item the $\varepsilon_{i,j}$'s are statistically independent, zero-mean, and $\omega$-sub-Gaussian,
\end{itemize}
where $\omega>0$ is an upper bound on the sub-Gaussian norm of any noise entry. 
We also assume that
\begin{equation}
	\left\|\bm{U}^\star\right\|_{2,\infty} \leq \sqrt{\frac{\mu_{\sf pc} r}{d}}.
	\label{eq:incoherence-Ustar-pc}
\end{equation}
Our goal is to estimate the subspace $\bm{U}^\star$ based on the observed vectors $\{\bm{y}_i\}_{1\leq i\leq n}$.

\paragraph{Algorithm and theoretical guarantees.} 
Taking the data matrix as $\bm{Y} = [\bm{y}_1 \ \dots \ \bm{y}_n] \in \bbR^{d \times n}$, 
we can readily invoke Algorithm~\ref{algorithm:sequential_heteroPCA} to estimate the subsapce $\bm{U}^\star$. 
The performance guarantees are stated below, whose proof is deferred to Section~\ref{proof:corollary_pca}.

\begin{cor}\label{cor:pca}
	Consider the factor model in \eqref{eq:factor-model}. Assume that 
	\begin{subequations}
	\begin{align}
		\frac{\lambda_r^\star}{ \omega^2 } &\geq C_1 r^2 \bigg[ \bigg( \frac{d}{n} \bigg)^{1/2} + \frac{d}{n} \bigg]\log^2 (n+d), \label{eq:assumption-PCA-SNR}\\
		\mu_{\sf pc} \vee \log(n+d) &\leq c_1\frac{d}{r^3} ,\label{eq:assumption-PCA-inc}\\
		r \vee \log(n+d) &\leq c_1 n 
	\end{align}
	\end{subequations}
	for some sufficiently large (resp.~small) constant $C_1>0$ (resp.~$c_1>0$). 
	Suppose that the numbers of iterations obey, for some large enough constant $C>0$, 
	\begin{subequations}
		\begin{align}
			t_k &\geq \log_2\left(C\frac{\lambda_{r_{k-1} + 1}^{\star}}{\lambda_{r_k + 1}^{\star}}\right), \quad \forall 1 \leq k \leq k_{\sf max} - 1,\label{ineq:iter_pca_1}\\
			t_{k_{\sf max}} &\geq \log\left(C\frac{n\lambda_{r_{k_{\sf max}-1} + 1}^{\star}}{\omega^2}\right),\label{ineq:iter_pca_2}
		\end{align}
	\end{subequations}
	where $k_{\sf max}$ satisfies $r_{k_{\sf max}} = r$. 
	Then with probability exceeding $1 - O\big((n+d)^{-10}\big)$, the output $\bm{U}$ returned by Algorithm~\ref{algorithm:sequential_heteroPCA} satisfies
	\begin{subequations}
		\begin{align}
			\left\|\bm{U}\bm{R}_{\bm{U}} - \bm{U}^\star\right\| _{2,\infty} &\lesssim \sqrt{\frac{\left(\mu_{\sf pc} + \log(n+d)\right)r}{d}}\left(\frac{\sqrt{d/n}\,\omega^2\log^2(n+d)}{\lambda_r^{\star}} + \frac{\sqrt{d/n}\,\omega\log(n+d)}{\sqrt{\lambda_{r}^\star}}\right),\label{ineq:pca_two_to_infty}\\
			\left\|\bm{U}\bm{R}_{\bm{U}} - \bm{U}^\star\right\| &\lesssim \frac{\sqrt{d/n}\,\omega^2\log^2 (n+d)}{\lambda_r^{\star}} + \frac{\sqrt{d/n}\,\omega\log(n+d)}{\sqrt{\lambda_{r}^\star}}\label{ineq:pca_spectral}.
		\end{align}
	\end{subequations}
\end{cor}

Let us briefly discuss the implications of our results. 
	Consider, for example, the case where $\bbE[\varepsilon_{i,j}^2] \asymp \sigma^2$ for all $(i, j) \in [d] \times [n]$. 
		The spectral norm bound \eqref{ineq:pca_spectral} matches the minimax limit (see \citet[Theorems 1 and 4]{zhang2022heteroskedastic}) modulo some logarithmic factor. In addition, recognizing that 
	\begin{align*}
		d\left\|\bm{U}\bm{R}_{\bm{U}} - \bm{U}^\star\right\| _{2,\infty}^2 \geq \left\|\bm{U}\bm{R}_{\bm{U}} - \bm{U}^\star\right\| _{\F}^2 \geq \left\|\bm{U}\bm{R}_{\bm{U}} - \bm{U}^\star\right\|^2,
	\end{align*}
	we see that the $\ell_{2, \infty}$ bound \eqref{ineq:pca_two_to_infty} is also near-optimal when $\mu_{\sf pc},r \asymp 1$. Again, our result does not rely on the condition number $\kappa_{\sf pc} = \lambda_1^\star/\lambda_r^\star$. 
		Moreover, \citet[Theorem 1]{zhang2022heteroskedastic} assumes that $\kappa_{\sf pc}$ is bounded by a numerical constant, 
		while \cite[Corollary 2]{cai2021subspace} requires $\kappa_{\sf pc} \lesssim \sqrt{\frac{d}{\mu r}}$;
		these form another aspect in which Corollary \ref{cor:pca} improves upon the prior literature.

\subsection{Tensor PCA}
\paragraph{Model.}
Another canonical example in which column subspace estimation plays a key role is tensor PCA (or low-rank tensor estimation),  
a problem that has been studied extensively in recent literature \citep{richard2014statistical,zhang2018tensor,cai2021subspace,cai2022nonconvex,han2022optimal,zhou2022optimal,han2022tensor}. 
To be presice, assume that we observe a noisy tensor as follows: 
\begin{subequations}
	\label{model:tensor_SVD-all}
\begin{align}\label{model:tensor_SVD}
	\bm{\mathcal{Y}} = \bm{\mathcal{X}}^\star + \bm{\mathcal{E}} \in \bbR^{n_1 \times n_2 \times n_3}, 
\end{align}
where $\bm{\mathcal{X}}^\star$ is an unknown low-rank tensor to be estimated, and $\bm{\mathcal{E}}$ represents the noise tensor. 
We assume that  $\bm{\mathcal{X}}^\star$ has low-Tucker-rank in the sense that  \citep{zhang2022heteroskedastic,han2022tensor,xia2022inference}
\begin{equation}
	\bm{\mathcal{X}}^\star = \bm{\mathcal{S}}^\star \times_1 \bm{U}_1^\star \times_2 \bm{U}_2^\star \times_3 \bm{U}_3^\star,
\end{equation}
\end{subequations}
 where the core tensor $\bm{\mathcal{S}}^\star$ lies in $\bbR^{r_1 \times r_2 \times r_3}$ (with small $r_1,r_2,r_3$), 
 and the tensor ``principal components'' $\bm{U}_i^\star \in \mathcal{O}^{n_i, r_i}$ ($1\leq i\leq 3$) satisfy the incoherence condition
 \begin{equation}
 	\|\bm{U}_i^\star\|_{2,\infty} \leq \sqrt{\frac{\mu r_i}{n_i}}, \qquad 1\leq i\leq 3.
	 \label{eq:incoherence-tensor-PCA}
 \end{equation}
Moreover, the noise tensor $\bm{\mathcal{E}}=[E_{i,j,k}]_{(i,j,k)\in [n_1]\times [n_2]\times [n_3]}$ is composed of independent entries such that
\begin{itemize}
	\item the $E_{i,j,k}$'s are statistically independent, zero-mean, and $\omega$-sub-Gaussian,
\end{itemize}
where $\omega>0$ is an upper bound on the sub-Gaussian norm of each noise entry. 
The aim is to compute a faithful estimate of the true tensor $\bm{\mathcal{X}}^\star$ as well as the principal components $\bm{U}_1^{\star}, \bm{U}_2^{\star}$ and $\bm{U}_3^{\star}$.

\paragraph{Additional notation.} 
Before presenting the algorithm and our theoretical results, we introduce several useful notation. For any $1 \leq i \leq 3$ and $1 \leq j \leq r_i$, we denote by $\sigma_{i,j}^{\star}$ the $j$-th largest singular value of the $i$-th matricization of $\mathcal{X}$ --- denoted by $\mathcal{M}_i(\mathcal{X})$. Define 
\begin{align*}
	\sigma_{\sf min}^\star := \min\left\{\sigma_{1, r_1}^{\star}, \sigma_{2, r_2}^{\star}, \sigma_{3, r_3}^{\star}\right\}, 
\end{align*}
and the condition number of the true tensor is then defined as
\begin{align*}
	\kappa := \frac{ \max\left\{\sigma_{1, 1}^{\star}, \sigma_{2, 1}^{\star}, \sigma_{3, 1}^{\star}\right\} }{ \sigma_{\sf min}^\star }.
\end{align*}
For any $1 \leq i \leq 3$, we also let $r_{i,1}, r_{i,2}, \dots, r_{i,k_{\sf max}^i}$ denote the ranks selected in Algorithm~\ref{algorithm:sequential_heteroPCA} 
if we apply this algorithm with the input matrix $\bm{Y} = \mathcal{M}_i\left(\bm{\mathcal{Y}}\right)$, the rank $r_i$, and the numbers of iterations $t_{i,1}, \dots, t_{i,k_{\sf max}^i}$. As usual, we choose $k_{\sf max}^i$ such that $r_{k_{\sf max}^i} = r_i$. In addition, for notational convenience we let 
\begin{align*}
	n = \max_{1 \leq i \leq 3}n_i \qquad \text{and} \qquad r = \max_{1 \leq i \leq 3}r_i,
\end{align*}
and define
$$\bm{U}_{4}^\star = \bm{U}_{1}^\star \qquad \text{ and } \qquad \bm{U}_{5}^\star = \bm{U}_{2}^\star.$$

\paragraph{Algorithm and statistical guarantees.} 
In order to apply {\sf Deflated\text{-}HeteroPCA}, let us look at the matrix $\mathcal{M}_i ( \bm{\mathcal{X}}^\star ) \in \mathbb{R}^{n_i \times (n_1n_2n_3)/n_i}$, 
the $i$-th matricization of  $\bm{\mathcal{X}}^\star$.  
Recognizing that $\bm{U}_i^\star$ is also the left singular space of $\mathcal{M}_i ( \bm{\mathcal{X}}^\star )$ since $$\mathcal{M}_i\left(\bm{\mathcal{X}}^\star\right) = \bm{U}_i^\star\mathcal{M}_i\left(\bm{\mathcal{S}}^\star\right)\left(\bm{U}_{i+2}^\star \otimes \bm{U}_{i+1}^\star\right),$$ 
we propose to apply the {\sf Deflated-HeteroPCA} algorithm to compute an initial subspace estimate $\widehat{\bm{U}}_i^{0}$ for $\bm{U}_i^\star$. 
Armed with these initial estimates, we invoke the high-order orthogonal iteration (HOOI) algorithm \citep{de2000best,zhang2018tensor} to iteratively refine the estimates. More specifically, in the $t$-th iteration, we calculate 
$$\widehat{\bm{U}}_i^t = \text{ the first } r \text{ left singular vectors of } \mathcal{M}_i\big(\bm{\mathcal{Y}} \times_{i+1} \widehat{\bm{U}}_{i+1}^{t-1} \times_{i+2} \widehat{\bm{U}}_{i+2}^{t-1}\big), \quad 1 \leq i \leq 3,$$
where $i+1$ and $i+2$ are calculated modulo 3. 
Once the above iterative procedure converges, we employ the resulting subspace estimates $\widehat{\bm{U}}_1, \widehat{\bm{U}}_2, \widehat{\bm{U}}_3$ to construct the following estimator for the true tensor:
\begin{align*}
	\widehat{\bm{\mathcal{X}}} = \bm{\mathcal{Y}} \times_1 \mathcal{P}_{\widehat{\bm{U}}_1} \times_2 \mathcal{P}_{\widehat{\bm{U}}_2} \times_3 \mathcal{P}_{\widehat{\bm{U}}_3},
\end{align*}
where we recall the notation $\mathcal{P}_{\bm{U}} = \bm{U}\bm{U}^{\top}$.

The whole procedure is summarized in Algorithm~\ref{algo:hooi}, where {\sf Deflated\text{-}HeteroPCA}$(\bm{Y}, r, t_1, \dots, t_{\sf max})$ is the output of Algorithm~\ref{algorithm:sequential_heteroPCA} with the input matrix $\bm{Y}$, the rank $r$, and the numbers of iterations $t_1, \dots, t_{\sf max}$. 
The computational cost for the  initialization step ({\sf Deflated-HeteroPCA}) is $\widetilde{O}(n^4 + n^2r\sum_{i = 1}^3\sum_{j = 1}^{k_{\sf max}^i}t_{i,j})$. For each orthogonal iteration, the computational cost is $\widetilde{O}(n^3r^2 + nr^3)$. Therefore, the total computational complexity for Algorithm~\ref{algo:hooi} amounts to $\widetilde{O}(n^4 + n^2r\sum_{i = 1}^3\sum_{j = 1}^{k_{\sf max}^i}t_{i,j} + (n^3r^2 + nr^3)t_{\sf max})$. Numerically, the algorithm achieves great performance with all $t_{i,j}$'s and $t_{\sf max}$ set to 10, 
in which case the computational cost simplifies to $\widetilde{O}(n^4 + n^3r^2)$.
Our main theory for {\sf Deflated\text{-}HeteroPCA} readily leads to the following statistical guarantees for Algorithm~\ref{algo:hooi}. 
\begin{algorithm}[t]
	\textbf{input:} $\mathcal{Y}$, ranks $r_1, r_2, r_3$, number of iterations $\left\{t_{i,j}\right\}_{1 \leq i \leq 3, 1 \leq j \leq k_{\sf max}^i}$ and $t_{\sf max}$.\\
	\textbf{initialization:} call Algorithm~\ref{algorithm:sequential_heteroPCA} to compute
	\begin{align*}
		\widehat{\bm{U}}_1^{0}&={\sf Deflated\text{-}HeteroPCA}\left(\mathcal{M}_1(\bm{\mathcal{Y}}), r_1, t_{1,1}, t_{1,2}, \dots, t_{1,k_{\sf max}^1}\right);\\
		\widehat{\bm{U}}_2^{0}&={\sf Deflated\text{-}HeteroPCA}\left(\mathcal{M}_2(\bm{\mathcal{Y}}), r_2, t_{2,1}, t_{2,2}, \dots, t_{2,k_{\sf max}^2}\right);\\
		\widehat{\bm{U}}_3^{0}&={\sf Deflated\text{-}HeteroPCA}\left(\mathcal{M}_3(\bm{\mathcal{Y}}), r_3, t_{3,1}, t_{3,2}, \dots, t_{3,k_{\sf max}^3}\right).
	\end{align*}
    \While {$t < t_{\sf max}$}{
    	$\widehat{\bm{U}}_1^t=\text{ leading }r_1 \text{ left singular vectors of } \mathcal{M}_1\big(\bm{\mathcal{Y}} \times_2 \widehat{\bm{U}}_2^{t-1} \times_3 \widehat{\bm{U}}_3^{t-1}\big)$.\\
    	$\widehat{\bm{U}}_2^t=\text{ leading }r_2 \text{ left singular vectors of } \mathcal{M}_1\big(\bm{\mathcal{Y}} \times_3 \widehat{\bm{U}}_3^{t-1} \times_1 \widehat{\bm{U}}_1^{t-1}\big)$.\\
    	$\widehat{\bm{U}}_3^t=\text{ leading }r_3 \text{ left singular vectors of } \mathcal{M}_3\big(\bm{\mathcal{Y}} \times_1 \widehat{\bm{U}}_1^{t-1} \times_2 \widehat{\bm{U}}_2^{t-1}\big)$.        
    }
compute $\widehat{\bm{\mathcal{X}}} = \bm{\mathcal{Y}} \times_1 \widehat{\bm{U}}_1^{t_{\sf max}}\widehat{\bm{U}}_1^{t_{\sf max}\top} \times_2 \widehat{\bm{U}}_2^{t_{\sf max}}\widehat{\bm{U}}_2^{t_{\sf max}\top} \times_3 \widehat{\bm{U}}_3^{t_{\sf max}}\widehat{\bm{U}}_3^{t_{\sf max}\top}$.\\
	\textbf{output:} subspace estimates $\widehat{\bm{U}}_1 = \widehat{\bm{U}}_1^{t_{\sf max}},\ \widehat{\bm{U}}_2 = \widehat{\bm{U}}_2^{t_{\sf max}},\ \widehat{\bm{U}}_3 = \widehat{\bm{U}}_3^{t_{\sf max}}$, and tensor estimate $\widehat{\bm{\mathcal{X}}}$.
	\caption{High-order orthogonal iteration (HOOI) \citep{de2000best,zhang2018tensor}}
	\label{algo:hooi}
\end{algorithm}

\begin{cor}\label{cor:tensor_svd}
	Consider the tensor PCA model in \eqref{model:tensor_SVD-all}. Suppose that $n_1 \asymp n_2 \asymp n_3 \asymp n$,  and
	\begin{subequations}
		\label{condition:tensor_snr}
	\begin{align}
		\frac{\sigma_{\sf min}^\star}{ \omega } & \geq C_2rn^{3/4}\log n \label{eq:condition:tensor-snr-123}\\
		\mu &\leq c_2\sqrt{\frac{n}{r^4}} \label{eq:condition:tensor-inc-123}
	\end{align}
	\end{subequations}
	for some sufficiently large (resp.~small) constant $C_2>0$ (resp.~$c_2>0$). 
	For any $1 \leq i \leq 3$, if one  chooses
	\begin{subequations}
		\begin{align}
			t_{i, 1} &\geq \log_2\bigg(C\frac{\sigma_{i, r_{i,k-1} + 1}^{\star2}}{\sigma_{i, r_{i,k} + 1}^{\star2}}\bigg), \quad 1 \leq k \leq k_{\sf max}^i - 1,\label{ineq:iter_tensor2}\\
			t_{i, k_{\sf max}^i} &\geq \log\bigg(C\frac{\sigma_{r_{i,k_{\sf max}^i-1} + 1}^{\star2}}{\omega^2}\bigg),\label{ineq:iter_tensor3}
		\end{align}
	\end{subequations}
	then with probability exceeding $1 - O\left(n^{-10}\right)$, the initial estimator $\widehat{\bm{U}}_i^{0}$ satisfies
	\begin{subequations}
		\begin{align}
			\big\|\widehat{\bm{U}}_i^{0}\bm{R}_{\widehat{\bm{U}}_i^{0}} - \bm{U}_i^\star\big\| _{2,\infty} &\lesssim \frac{\mu r}{\sqrt{n}}\left(\frac{n^{3/2}\omega^2\log^2 n}{\sigma_{\sf min}^{\star2}} + \frac{\sqrt{n}\,\omega\log n}{\sigma_{\sf min}^{\star}}\right),\label{ineq:tensor_initial_1}\\
			\big\|\widehat{\bm{U}}_i^{0}\bm{R}_{\widehat{\bm{U}}_i^{0}} - \bm{U}^\star\big\| &\lesssim \frac{n^{3/2}\omega^2\log^2 n}{\sigma_{\sf min}^{\star2}} + \frac{\sqrt{n}\,\omega\log n}{\sigma_{\sf min}^{\star}}\label{ineq:tensor_initial_2}.
		\end{align}
	\end{subequations}
	In addition, if the number of iterations in HOOI obeys $t_{\sf max} \geq C(\log(\frac{n}{\sigma_{\sf min}}) \vee 1)$ for some large enough constant $C>0$,
	then with probability exceeding $1 - O(n^{-10})$ one has 
	\begin{subequations}
		\label{ineq:tensor_all}
	\begin{align}\label{ineq:tensor_subspace}
		\big\|\widehat{\bm{U}}_i\bm{R}_{\widehat{\bm{U}}_i} - \bm{U}_i^\star\big\| &\lesssim \frac{\sqrt{n_i}\,\omega}{\sigma_{\sf min}^\star}, \qquad 1 \leq i \leq 3  \\
		\big\|\widehat{\bm{\mathcal{X}}} - \bm{\mathcal{X}}^\star\big\|_{\rm F}^2 &\lesssim \left(n_1r_1 + n_2r_2 + n_3r_3\right)\omega^2. 
		\label{ineq:tensor}
	\end{align}
	\end{subequations}
\end{cor}

	The bounds in \eqref{ineq:tensor_all} are rate-optimal, since they match the minimax lower bounds established for the i.i.d.~Gaussian noise case in \citet[Theorem 3]{zhang2018tensor}. 
	This confirms that the proposed {\sf Deflated-HeteroPCA} algorithm serves as an effective paradigm to initialize the HOOI algorithm. 
 	It is also noteworthy that  when $r = O(1)$, the SNR condition \eqref{condition:tensor_snr} is essential (ignoring logarithmic factor) to ensure that consistent estimation is computable within polynomial time; 
	see \citet[Theorem 4]{zhang2018tensor}.

	It is then helpful to compare our results with the prior works \citet{zhang2018tensor} and \citet{han2022optimal}. 
	Firstly, \citet[Theorem 1]{zhang2018tensor} assumes that the noise tensor has i.i.d.~Gaussian entries, 
	which is clearly much more stringent than our result. 
	Secondly, while \citet[Theorem 4.1]{han2022optimal} allows the noise to be heteroskedastic, 
	it requires the condition number of the tensor to be bounded (see the analysis for their main theorems); 
	in comparison, our theory in Corollary \ref{cor:tensor_svd} suggests that Algorithm~\ref{algo:hooi} succeeds no matter how large the condition number $\kappa$ is.

    \section{Numerical experiments}\label{sec:simulation}

In this section, we conduct additional numerical experiments to verify the practical applicability of our algorithm. All results in this section are averaged over 50 Monte Carlo runs.

\begin{figure}[t]
	\centering
	\begin{minipage}[t]{\linewidth}
		\centering
		\subfigure[$\kappa=5, n_2 = 1,000$, $\ell_2$ error]{
			\includegraphics[width=0.32\linewidth]{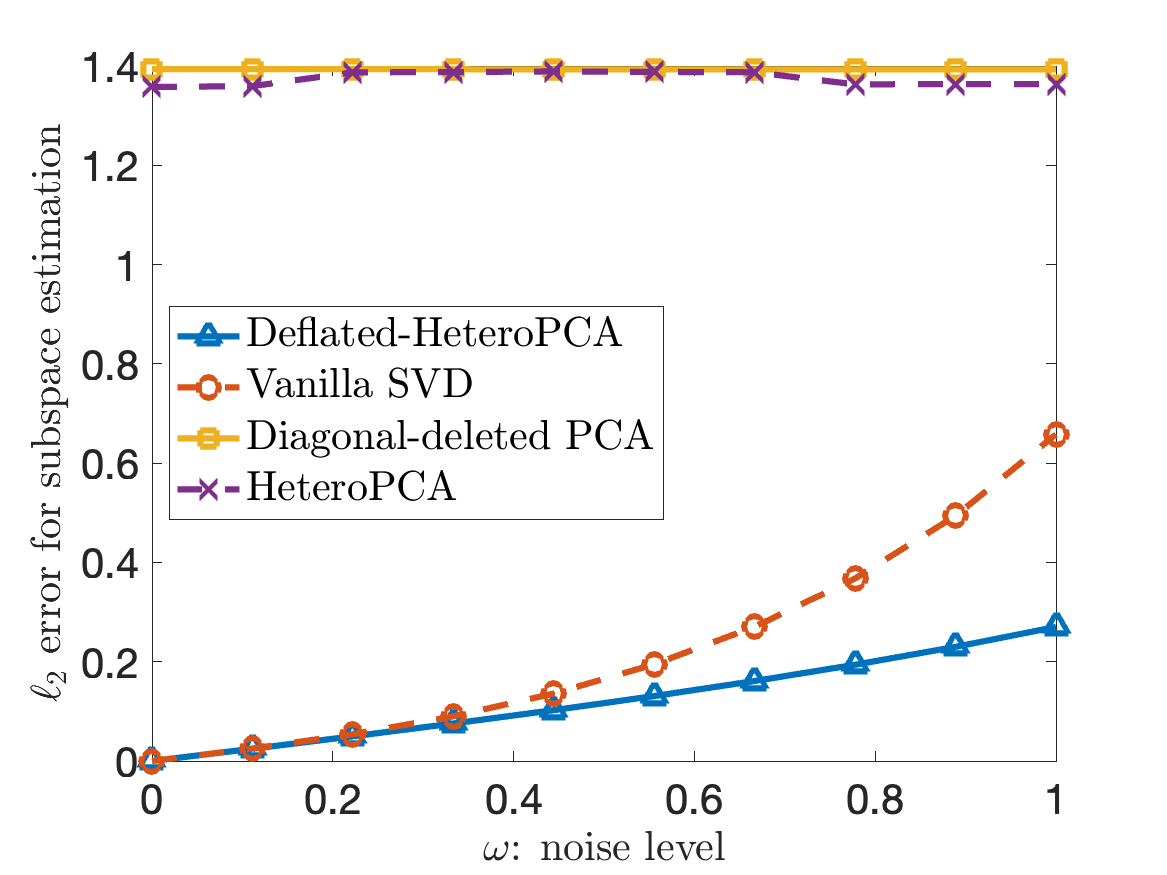}}
		\subfigure[$\kappa=5, n_2 = 1,000$, $\ell_{2,\infty}$ error]{
			\includegraphics[width=0.32\linewidth]{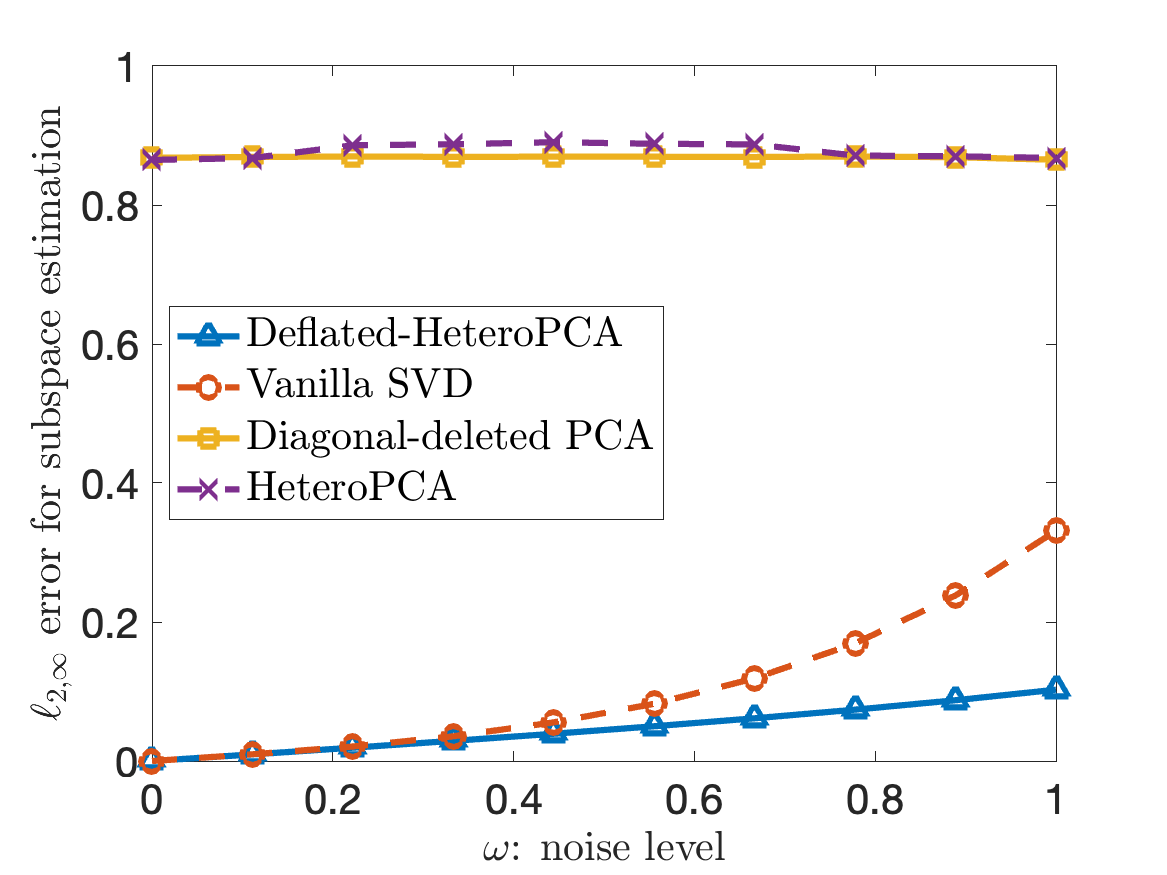}}
		\subfigure[$\omega = 1, n_2 = 1,000$, $\ell_2$ error]{
			\includegraphics[width=0.32\linewidth]{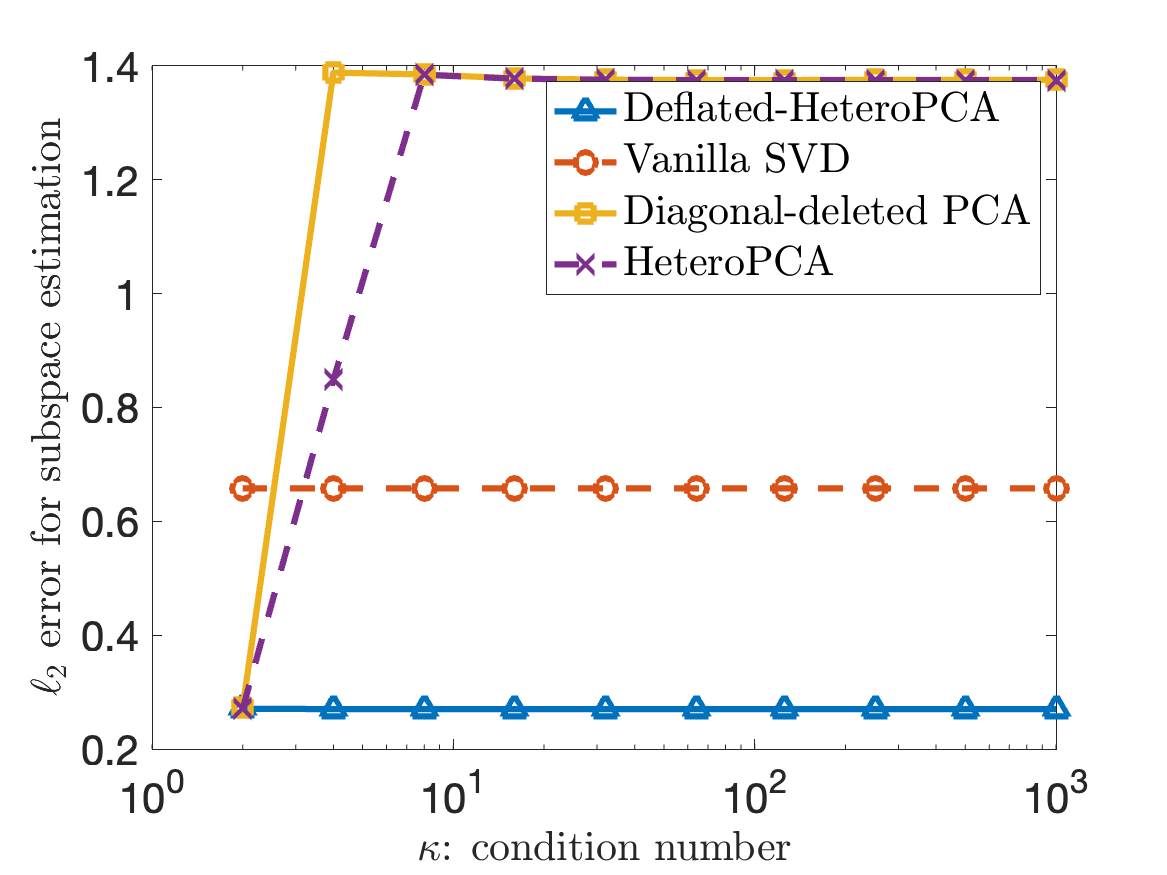}}
		\subfigure[$\omega = 1, n_2 = 1,000$, $\ell_{2,\infty}$ error]{
			\includegraphics[width=0.32\linewidth]{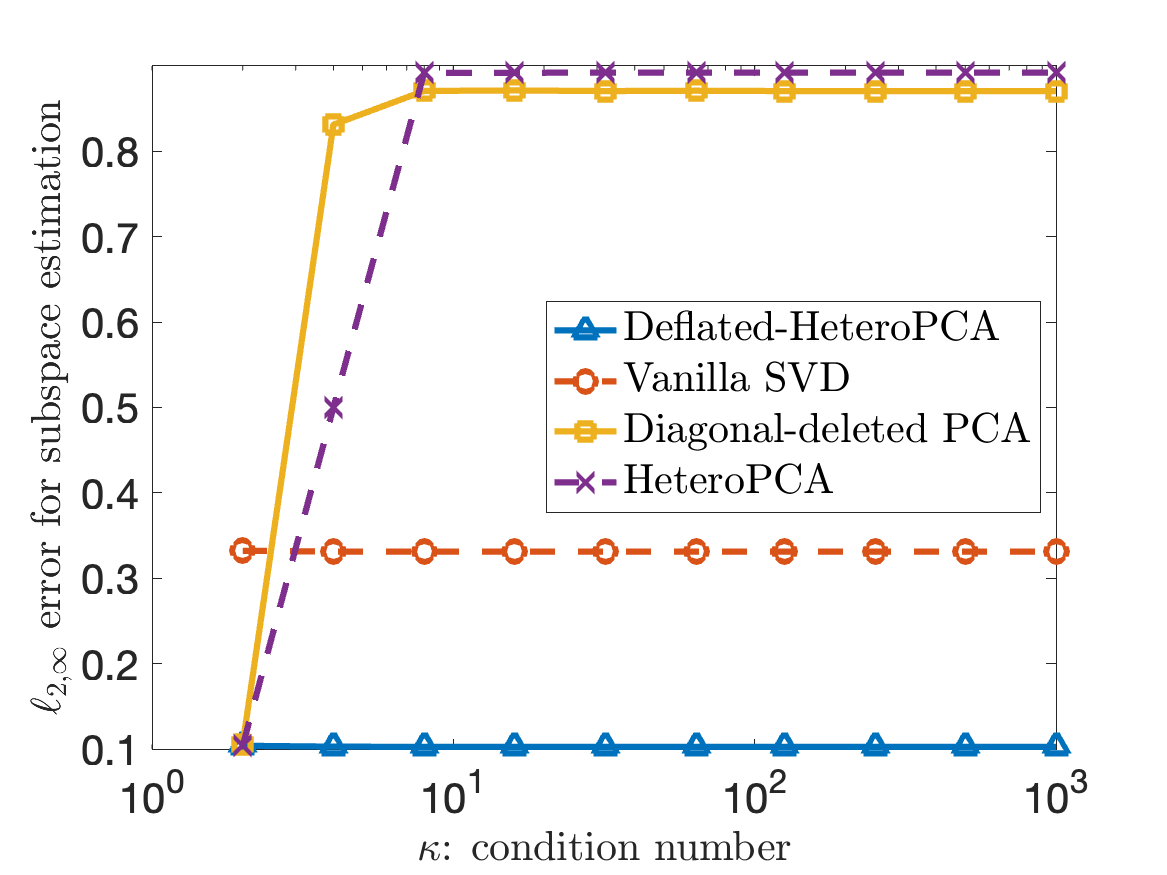}}
		\subfigure[$\omega = 1, \kappa = 5$, $\ell_{2}$ error]{
			\includegraphics[width=0.32\linewidth]{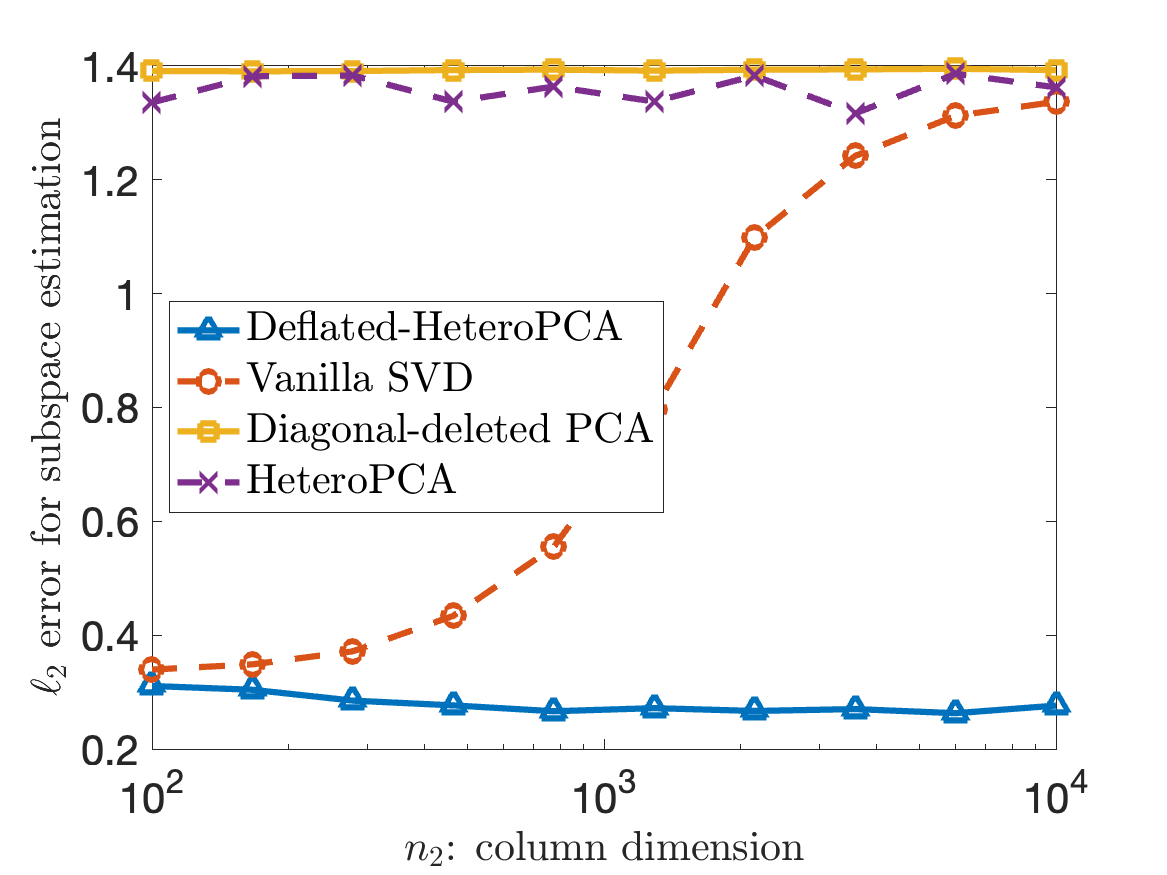}}
		\subfigure[$\omega = 1, \kappa = 5$, $\ell_{2,\infty}$ error]{
			\includegraphics[width=0.32\linewidth]{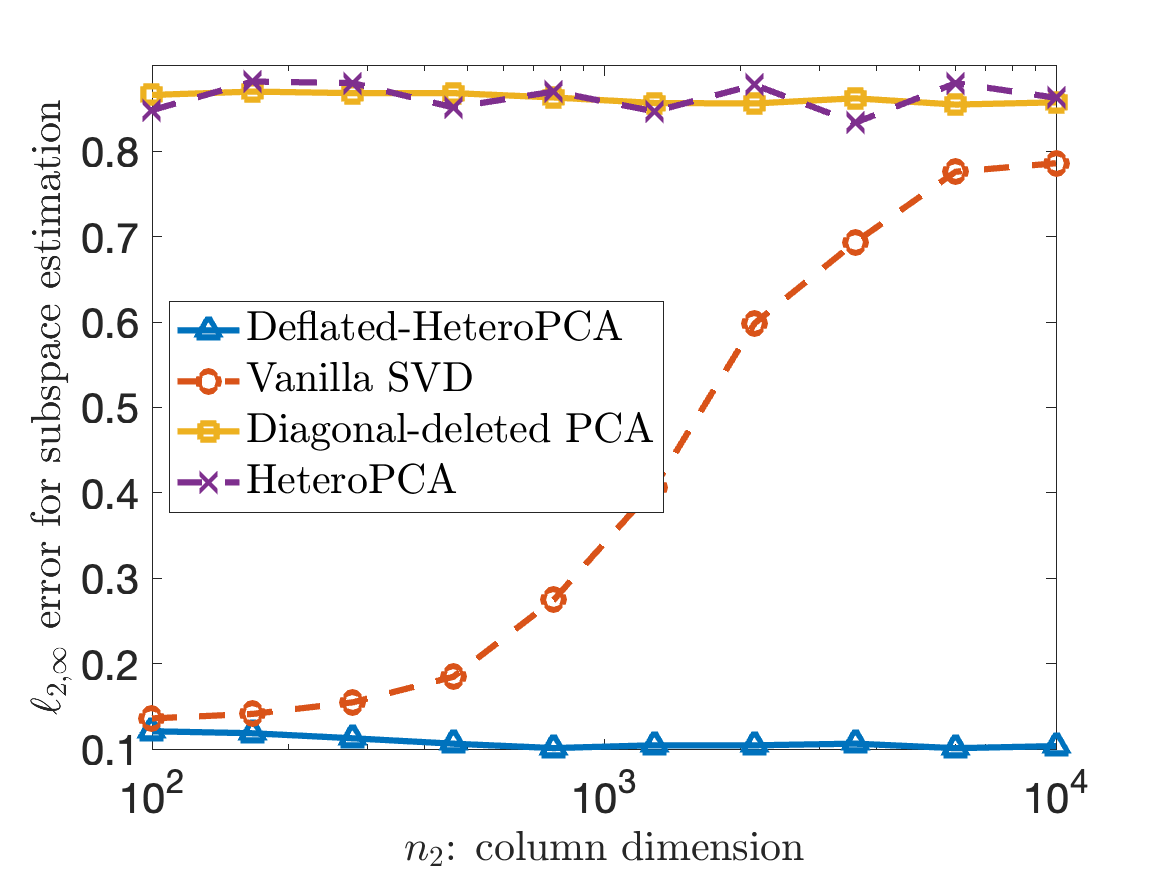}}
	\end{minipage}
	\caption{Estimation errors of $\bm{U}$ for {\sf Deflated-HeteroPCA}, {\sf Diagonal-deleted PCA}, {\sf HeteroPCA} and {\sf Vanilla SVD} for $r = 3$. Plot (a) (resp. (b)) reports the $\ell_2$ (resp.~$\ell_{2,\infty}$) error vs.~the noise level $\omega$ (where $n_1 = 100, n_2 = 1,000, \kappa = 5$). 
	Plot (c) (resp.~(d)) shows the $\ell_2$ (resp.~$\ell_{2,\infty}$) error vs.~the column dimension $\kappa$ (where $n_1 = 100, n_2 = 1,000, \omega = 1$). 
	Plot (e) (resp.~(f)) displays the $\ell_2$ (resp.~$\ell_{2,\infty}$) error vs.~the condition number $n_2$ (where $n_1 = 100, \kappa = 5, \omega = 1$).}\label{fig:simulation_r_3} 
\end{figure}

\paragraph{Low-rank subspace estimation from noisy observation.}
To begin with, 
we consider the problem of estimating the column subspace of $\bm{X}^\star$ from the noisy data \eqref{eq:model-setup}. We randomly generate $\bm{U}^\star \in \mathcal{O}^{n_1, r}$ and $\bm{V}^\star \in \mathcal{O}^{n_2, r}$, and $\bm{X}^\star =  \bm{U}^\star\bm{\Sigma}^\star\bm{V}^{\star\top}$, where $\bm{\Sigma}^\star = {\sf diag}(\sigma_1^\star, \dots, \sigma_r^\star)$. For each $i \in [n_1]$, we independently and uniformly draw $\omega_i \in [0, \omega]$, whereas the $E_{i,j}$'s are independently drawn from $\mathcal{N}(0,\omega_i^2)$. 
We fix $n_1 = 100$, set $\sigma_r^\star = (n_1n_2)^{1/4} + n_1^{1/2}$, and consider the following two settings: (i) $r = 3$, $\sigma_1^\star = \kappa\sigma_3^\star$ and $\sigma_2^\star = \sigma_3^\star$; (ii) $r = 5$, $\sigma_1^\star = \kappa\sigma_5^\star$, $\sigma_2^\star = \sigma_3^\star = \kappa^{1/2}\sigma_5^\star$ and $\sigma_4^\star = \sigma_5^\star$. We report the spectral-norm-based error$\|\bm{U}\bm{R}_{\bm{U}} - \bm{U}^\star\|$  and the $\ell_{2,\infty}$ error $\|\bm{U}\bm{R}_{\bm{U}} - \bm{U}^\star\|_{2,\infty}$ for each of the following four algorithms: (a) {\sf Deflated-HeteroPCA} in Algorithm~\ref{algorithm:sequential_heteroPCA},  where the numbers of iterations are chosen to be $t_i = 10$; (b) the diagonal-deleted PCA procedure as in \eqref{eq:diagonal-deletion}; (c) {\sf HeteroPCA} in Algorithm~\ref{algorithm:heteroPCA}, where the number of iterations is taken to be 100; (d) the vanilla SVD-based approach described in \eqref{eq:vanilla-SVD-alg}. The results for $r = 3$ and $r = 5$ are reported in Figures~\ref{fig:simulation_r_3} and \ref{fig:simulation_r_5}, respectively. 
As can be seen from the plots, the proposed {\sf Deflated-HeteroPCA} algorithm significantly outperforms the other three methods, 
and it is the only algorithm whose performance is unaffected by the condition number $\kappa$.

\begin{figure}[t]
	\centering
	\begin{minipage}[t]{\linewidth}
		\centering
		\subfigure[$\kappa=5, n_2 = 1,000$, $\ell_2$ error]{
			\includegraphics[width=0.32\linewidth]{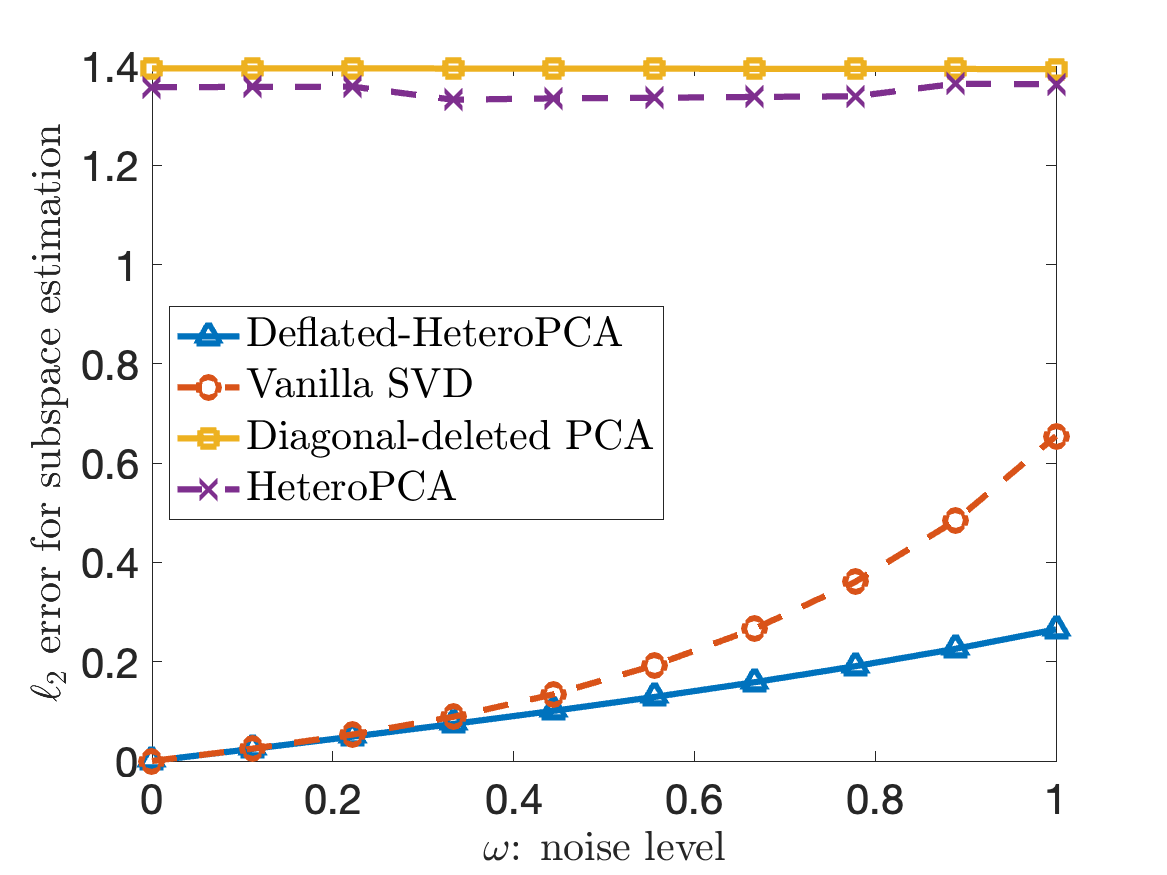}}
		\subfigure[$\kappa=5, n_2 = 1,000$, $\ell_{2,\infty}$ error]{
			\includegraphics[width=0.32\linewidth]{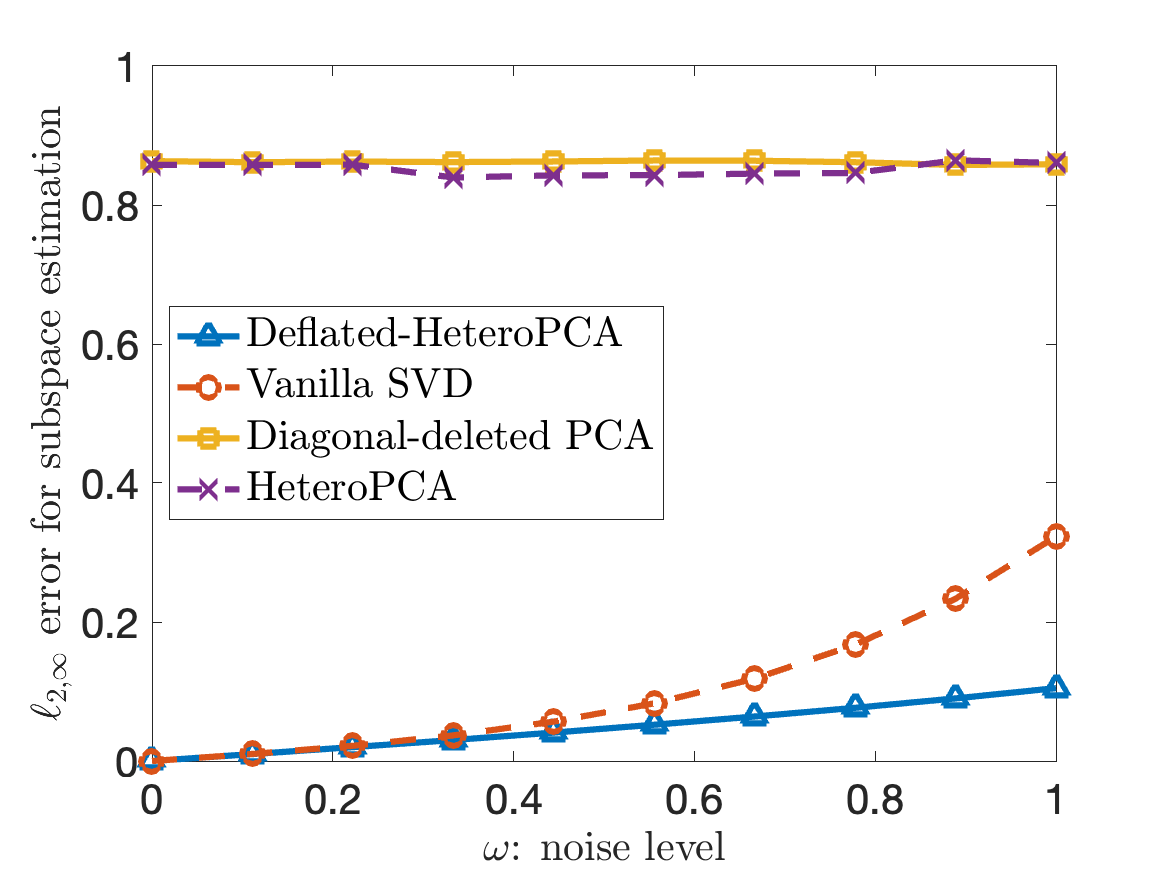}}
		\subfigure[$\omega = 1, n_2 = 1,000$, $\ell_2$ error]{
			\includegraphics[width=0.32\linewidth]{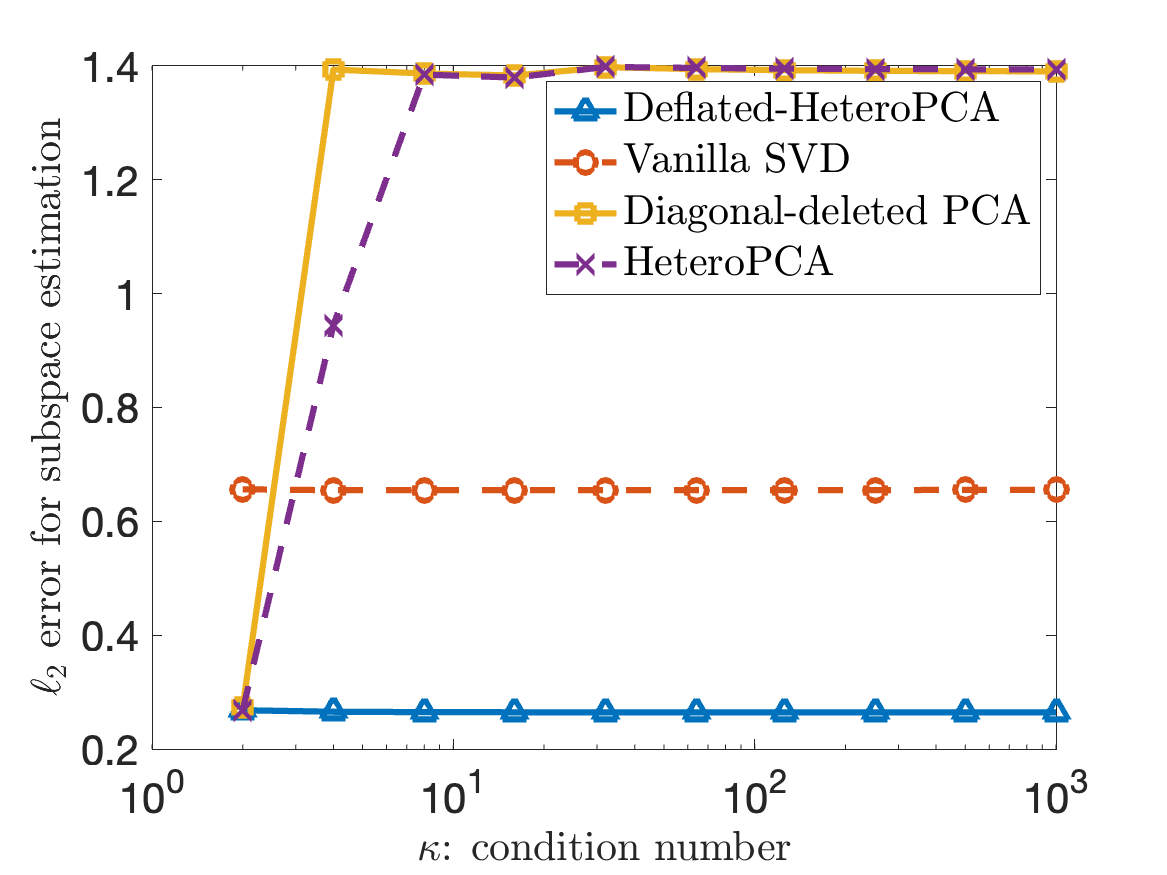}}
		\subfigure[$\omega = 1, n_2 = 1,000$, $\ell_{2,\infty}$ error]{
			\includegraphics[width=0.32\linewidth]{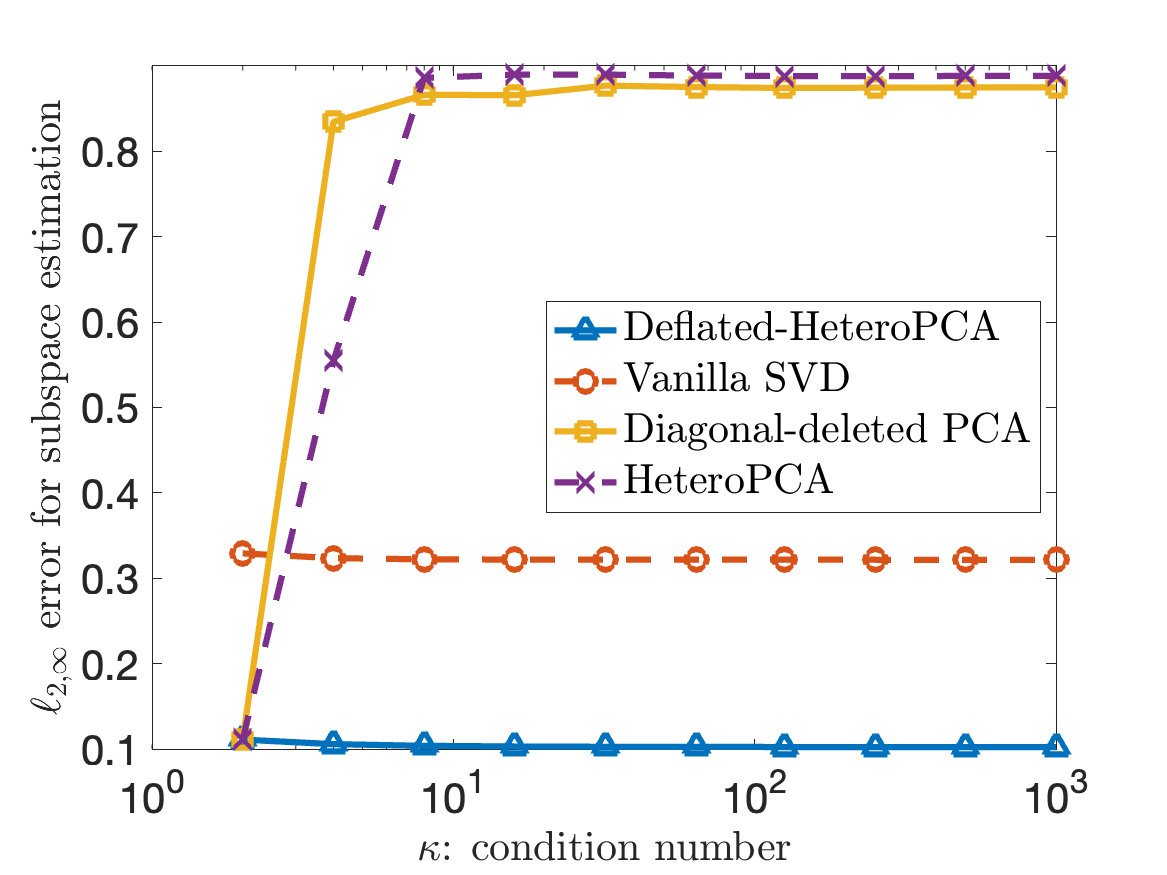}}
		\subfigure[$\omega = 1, \kappa = 5$, $\ell_{2}$ error]{
			\includegraphics[width=0.32\linewidth]{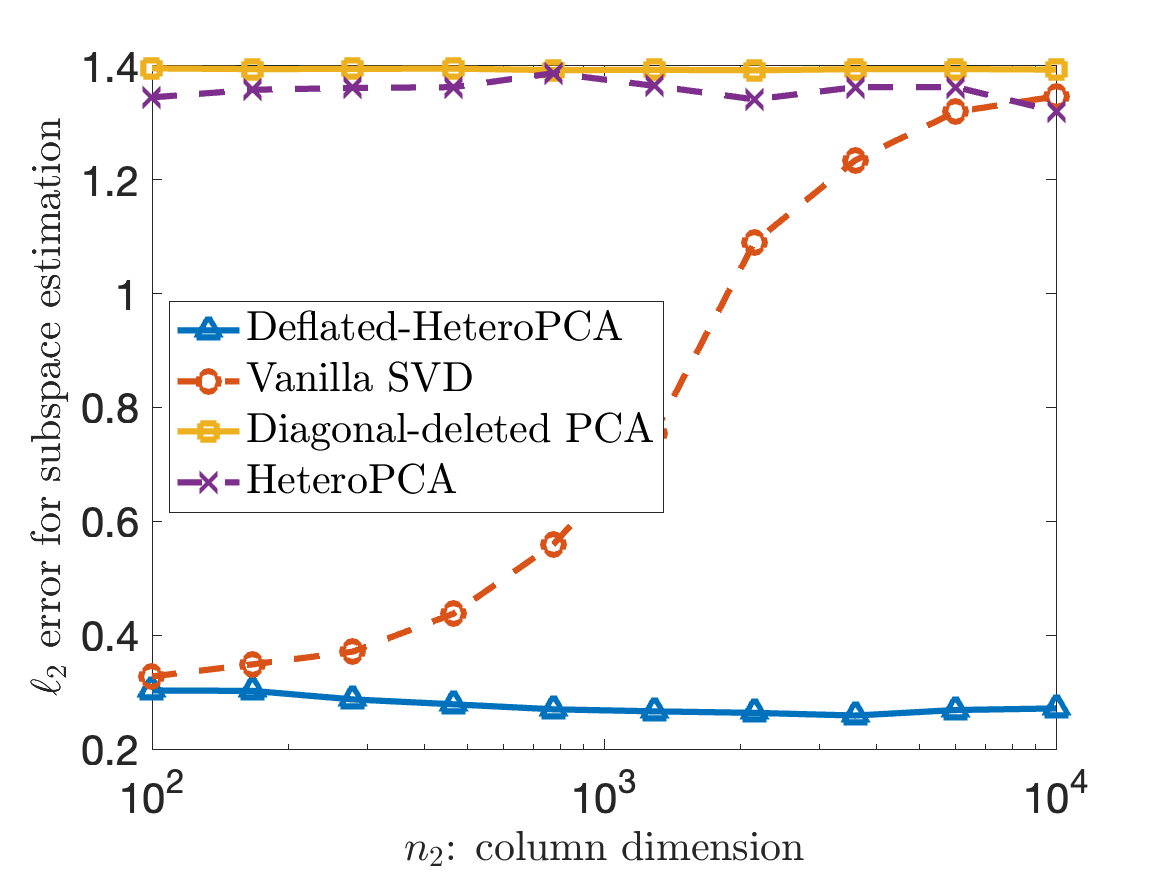}}
		\subfigure[$\omega = 1, \kappa = 5$, $\ell_{2,\infty}$ error]{
			\includegraphics[width=0.32\linewidth]{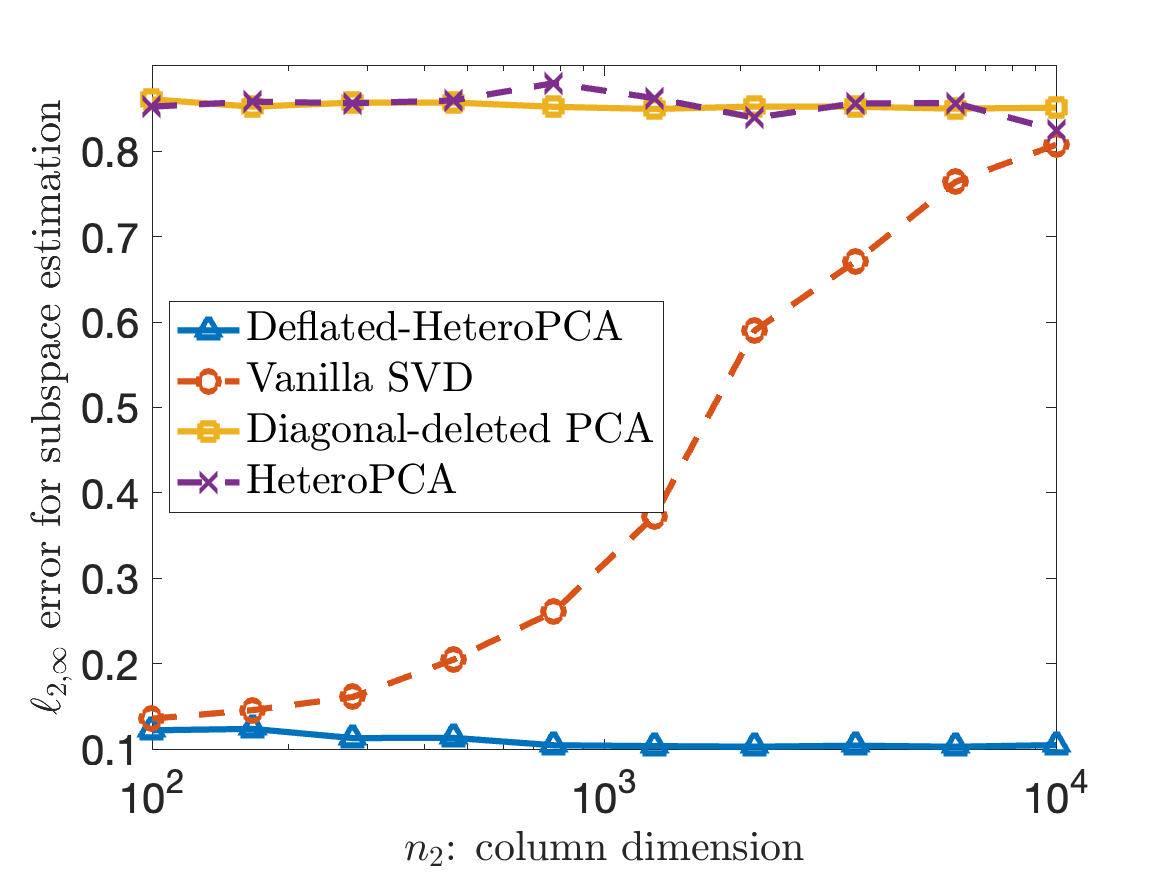}}
	\end{minipage}
	\caption{Estimation errors of $\bm{U}$ for {\sf Deflated-HeteroPCA}, {\sf Diagonal-deleted PCA}, {\sf HeteroPCA} and {\sf Vanilla SVD} when $r = 5$. 
	Plot (a) (resp.~(b)) displays the $\ell_2$ (resp.~$\ell_{2,\infty}$) error vs.~the noise level $\omega$ (where $n_1 = 100, n_2 = 1,000, \kappa = 5$). 
	Plot (c) (resp.~(d)) shows the $\ell_2$ (resp.~$\ell_{2,\infty}$) error vs.~the condition number $\kappa$ (where $n_1 = 100, n_2 = 1,000, \omega = 1$). 
	Plot (e) (resp.~(f)) diaplsys the $\ell_2$ (resp.~$\ell_{2,\infty}$) error vs.~the column dimension $n_2$ (where $n_1 = 100, \kappa = 5, \omega = 1$).}
	\label{fig:simulation_r_5}
\end{figure}

\paragraph{Factor model.} 
We then turn attention to the factor model \eqref{model:PCA}. We consider the case with $d = 100, r = 3$, and randomly generate the subspace $\bm{U}^\star \in \mathcal{O}^{d,3}$ and $\bm{F} = [\bm{f}_1\ \dots \ \bm{f}_n] \in \bbR^{3 \times n}$ with i.i.d.~standard Gaussian entries. We set the diagonal matrix $\bm{\Lambda}^\star = {\sf diag}(\lambda_1^\star, \lambda_2^\star, \lambda_3^\star)$ with $\lambda_1^\star = \kappa\lambda_3^\star$ and $\lambda_2^\star = \lambda_3^\star = (d/n)^{1/2} + d/n$. The noise matrix is generated in the same way as in the previous setting. We report in Figure \ref{fig:simulation_PCA_r_3} the $\ell_2$ and $\ell_{2,\infty}$ errors for the principal subspace for the four methods, {\sf Deflated-HeteroPCA}, {\sf Diagonal-deleted PCA}, {\sf HeteroPCA} and {\sf Vanilla SVD}. The numerical results suggest that the proposed {\sf Deflated-HeteroPCA} algorithm achieves the best performance among all these methods, 
which is not affected as $\kappa_{\sf pc}$ varies.
\begin{figure}[ht!]
	\centering
	\begin{minipage}[t]{\linewidth}
		\centering
		\subfigure[$\kappa_{\sf pc}=100, n = 1,000$, $\ell_2$ error]{
			\includegraphics[width=0.32\linewidth]{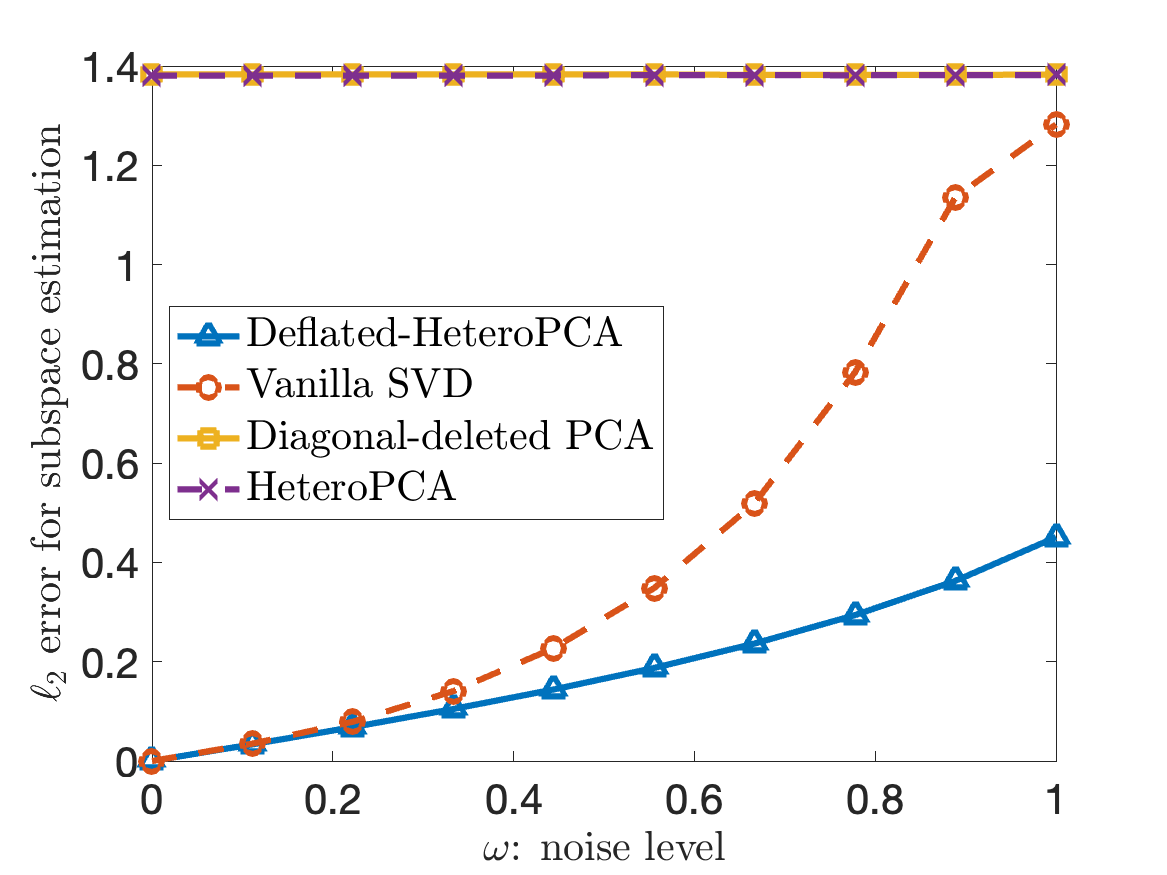}}
		\subfigure[$\kappa_{\sf pc}=100, n = 1,000$, $\ell_{2,\infty}$ error]{
			\includegraphics[width=0.32\linewidth]{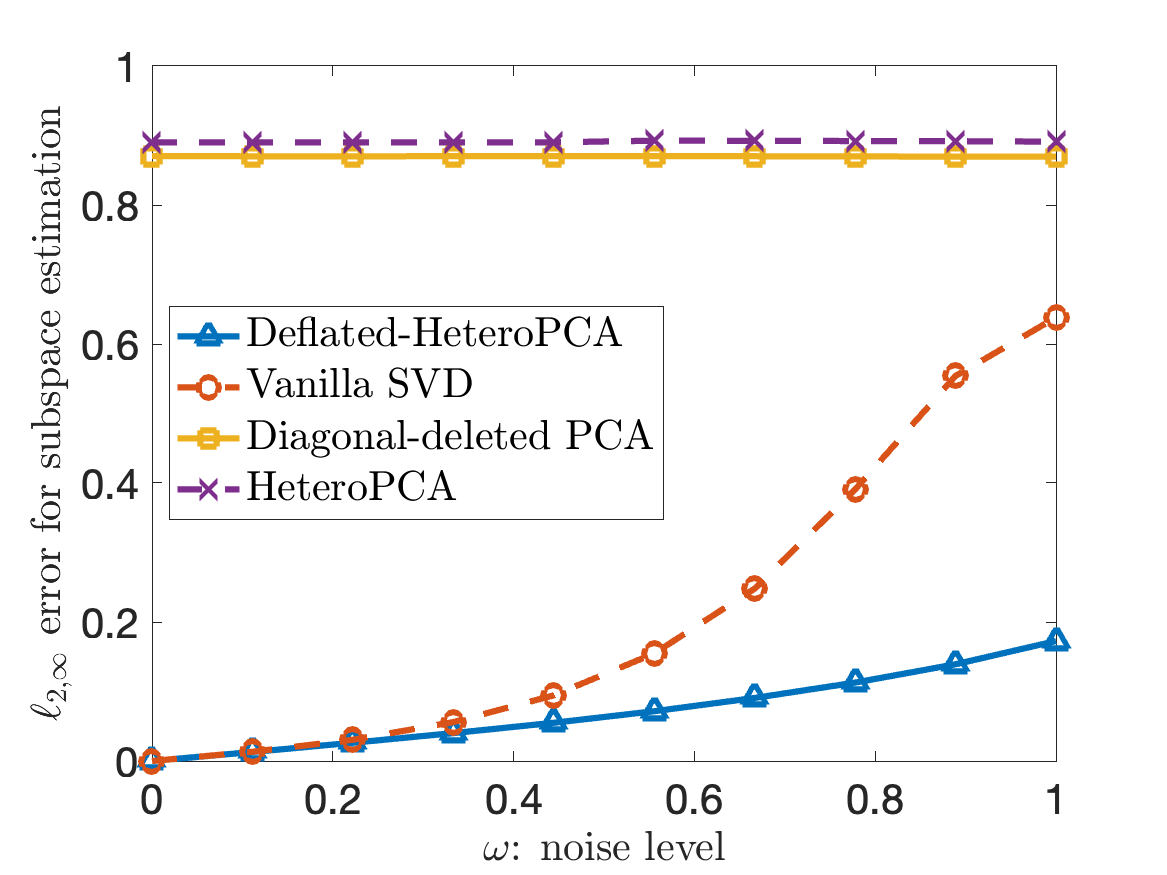}}
		\subfigure[$\omega = 1, n = 1,000$, $\ell_2$ error]{
			\includegraphics[width=0.32\linewidth]{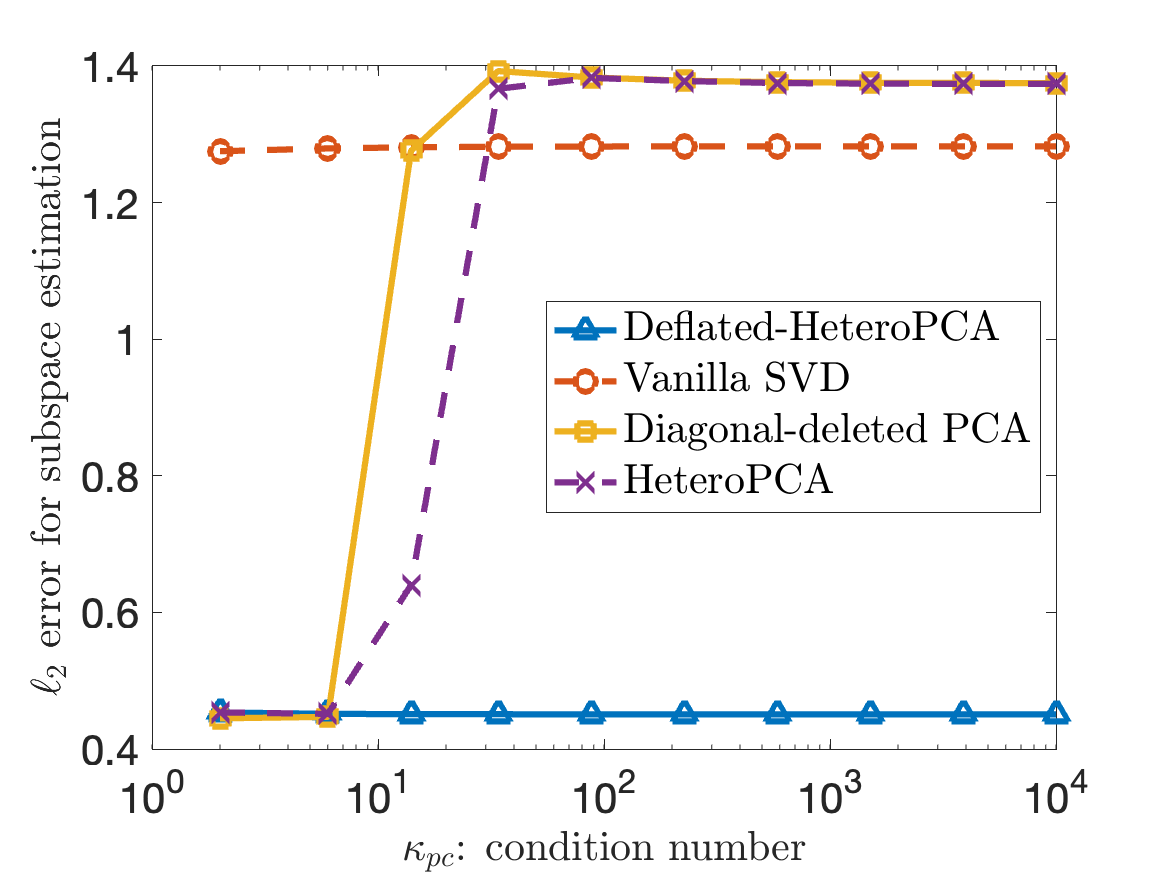}}
		\subfigure[$\omega = 1, n = 1,000$, $\ell_{2,\infty}$ error]{
			\includegraphics[width=0.32\linewidth]{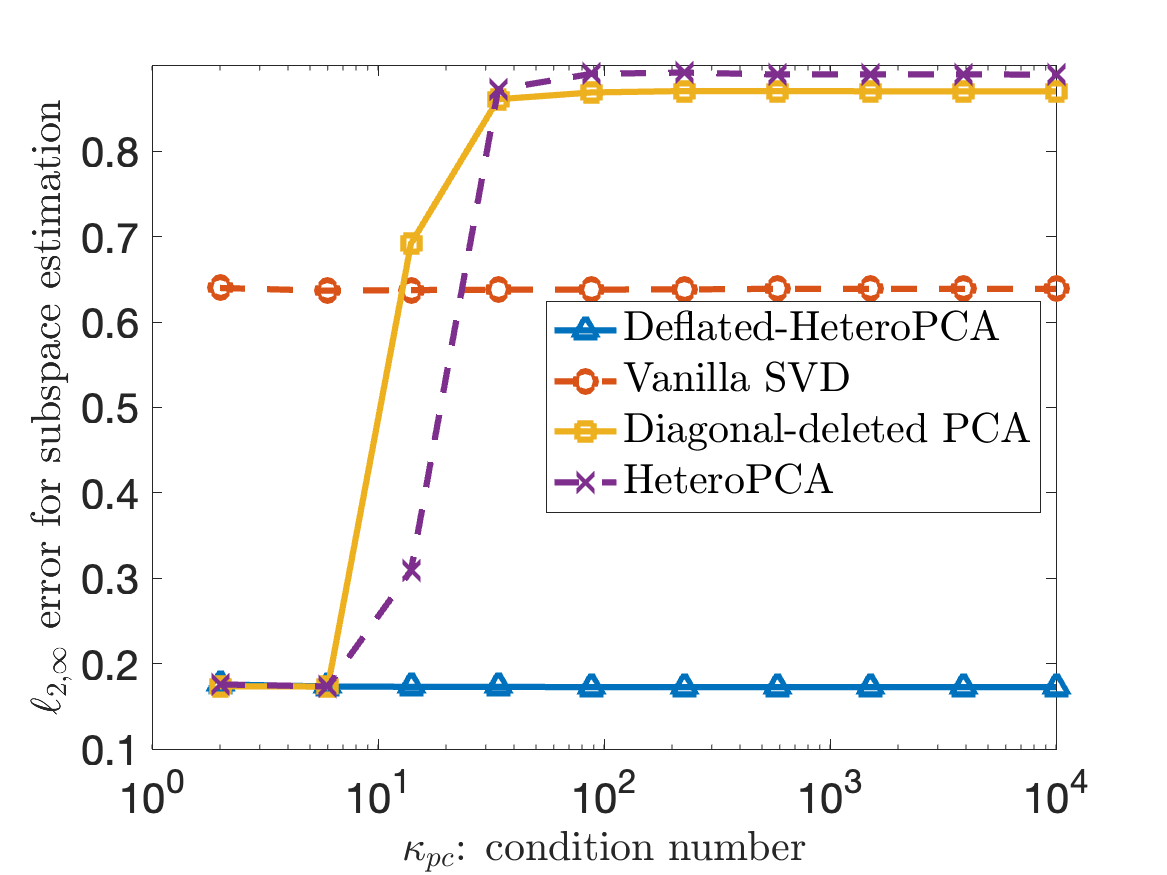}}
		\subfigure[$\omega = 1, \kappa_{\sf pc} = 100$, $\ell_{2}$ error]{
			\includegraphics[width=0.32\linewidth]{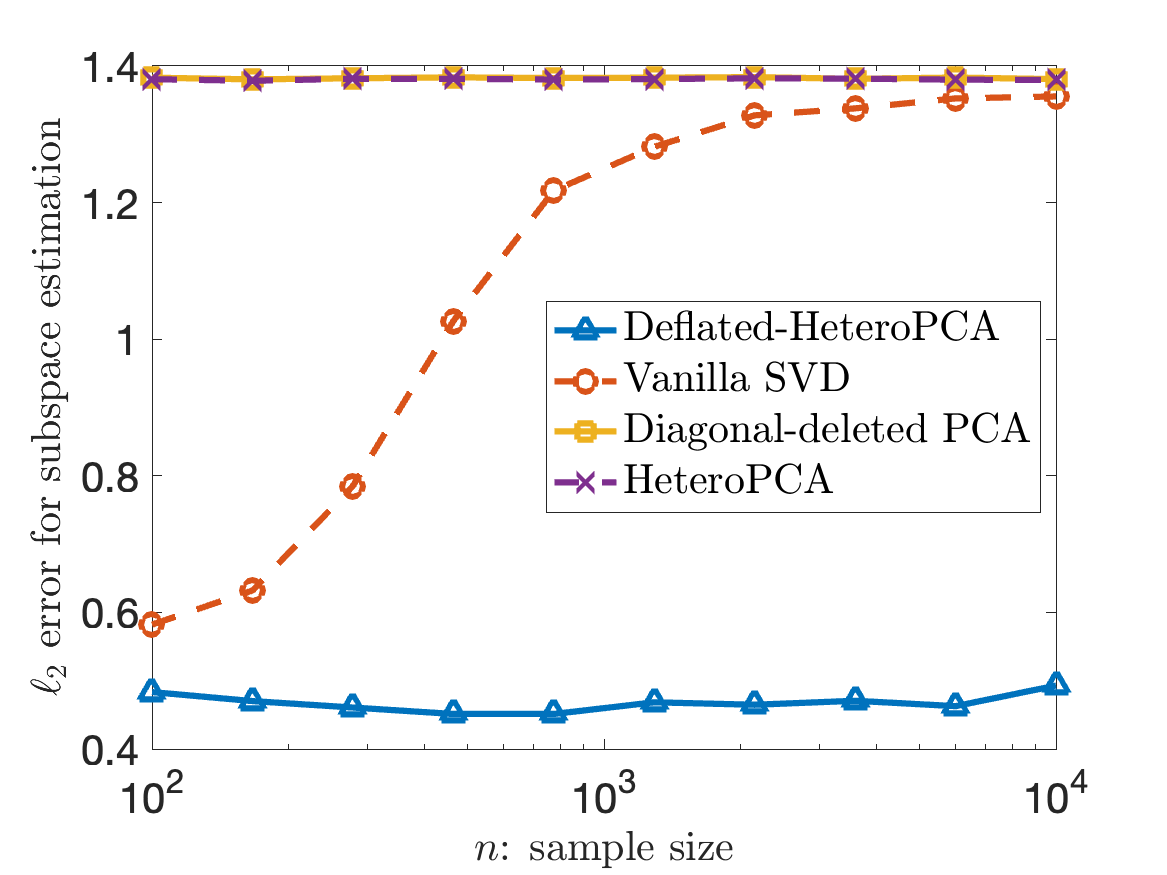}}
		\subfigure[$\omega = 1, \kappa_{\sf pc} = 100$, $\ell_{2,\infty}$ error]{
			\includegraphics[width=0.32\linewidth]{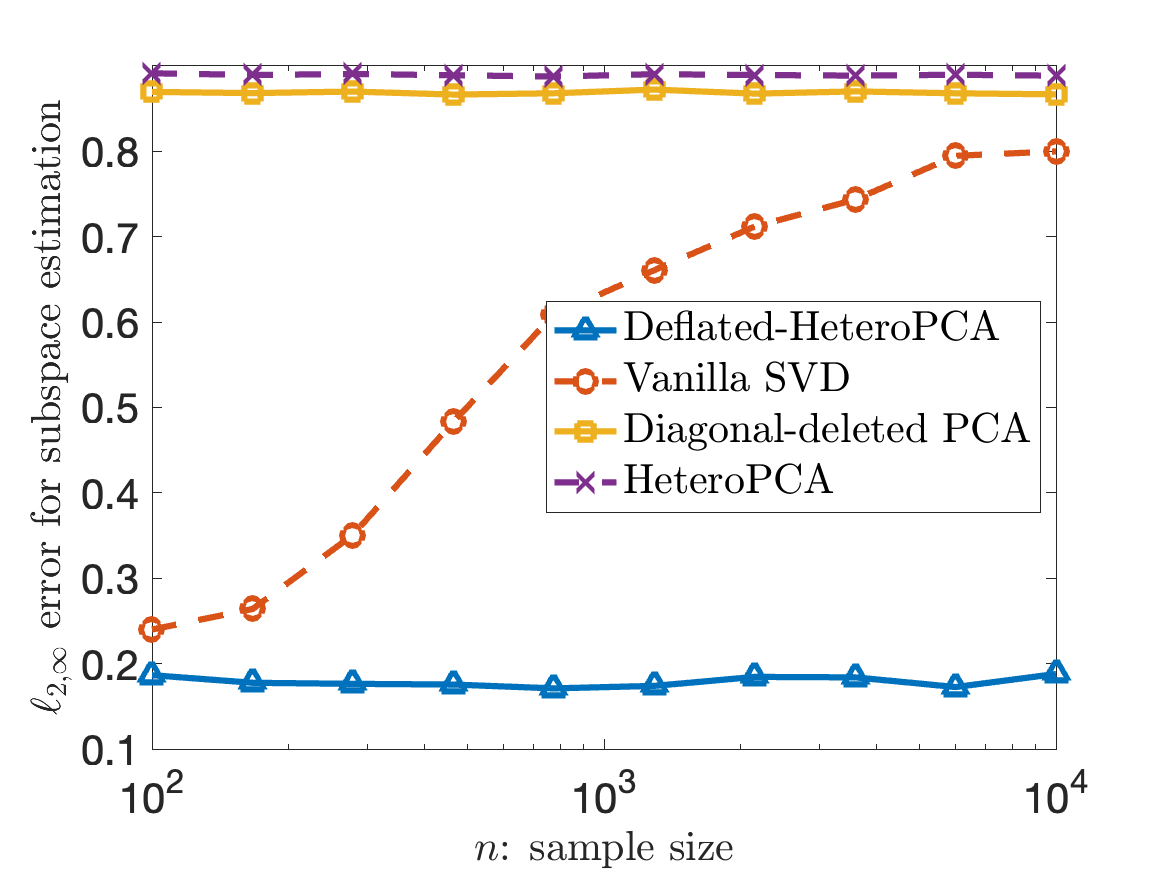}}
	\end{minipage}
	\caption{Estimation errors of $\bm{U}$ for {\sf Deflated-HeteroPCA}, {\sf Diagonal-deleted PCA}, {\sf HeteroPCA} and {\sf Vanilla SVD} under the factor model \eqref{model:PCA} when $r = 3$. 
	Plot (a) (resp.~(b)) displays the $\ell_2$ (resp.~$\ell_{2,\infty}$) error vs.~the noise level $\omega$ (where $d = 100, n = 1,000, \kappa_{\sf pc} = 100$). 
	Plot (c) (resp.~(d)) shows the $\ell_2$ (resp.~$\ell_{2,\infty}$) error vs.~the condition number $\kappa_{\sf pc}$ (where $d = 100, n = 1,000, \omega = 1$). 
	Plot (e) (resp.~(f)) displays the $\ell_2$ (resp.~$\ell_{2,\infty}$) error vs.~the sample size $n$ (where $d = 100, \kappa_{\sf pc} = 100, \omega = 1$).}\label{fig:simulation_PCA_r_3}
\end{figure}

\paragraph{Poisson PCA.} We consider the Poisson PCA problem \citep{zhang2022heteroskedastic,liu2018epca}: suppose that the truth $\bm{X}^\star = \bm{U}^\star\bm{\Sigma}^\star\bm{V}^\star \in \bbR^{n_1 \times n_2}$ is a rank-$r$ matrix with positive entries. Our goal is to estimate the column subspace $\bm{U}^\star \in \bbR^{n_1 \times r}$ based on the observations $\bm{Y} \in \bbR^{n_1 \times n_2}$, where each entry $Y_{i,j}$ of $\bm{Y}$ is an independent random variable following a Poisson distribution with mean $X_{i,j}^\star$, that is, $Y_{i,j} \sim \mathsf{Poisson}(X_{i,j}^\star)$. More specifically, we fix $n_1 = 100, n_2 = 1,000, r = 3$ and generate random matrices $\widetilde{\bm{U}} \in \bbR^{n_1 \times 3}$ and $\widetilde{\bm{V}} \in \bbR^{n_2 \times 3}$ with i.i.d.~standard Gaussian entries. We let $\overline{\bm{U}} \in \bbR^{n_1 \times 3}$ (resp.~$\overline{\bm{V}} \in \bbR^{n_2 \times 3}$) denote the matrix with entries $\overline{U}_{i,j} = |\widetilde{U}_{i,j}|$ (resp.~$\overline{V}_{i,j} = |\widetilde{V}_{i,j}|$). We define $\overline{\bm{\Lambda}} = \frac{1}{5}{\sf diag}(\lambda^2, \lambda, \lambda)$ and let $\bm{X}^\star = \overline{\bm{U}}\,\overline{\bm{\Lambda}}\,\overline{\bm{V}}^\top$. The  empirical $\ell_2$ and $\ell_{2,\infty}$ errors for the subspace estimation for the four methods, {\sf Deflated-HeteroPCA}, {\sf Diagonal-deleted PCA}, {\sf HeteroPCA} and {\sf Vanilla SVD} are illustrated in Figure~\ref{fig:simulation_Poisson_r_3}. It is clearly seen that {\sf Deflated-HeteroPCA} outperforms the other three methods.

\begin{figure}[t]
	\centering
	\begin{minipage}[t]{\linewidth}
		\centering
		\subfigure[$\kappa=5, n_2 = 1,000$, $\ell_2$ error]{
			\includegraphics[width=0.45\linewidth]{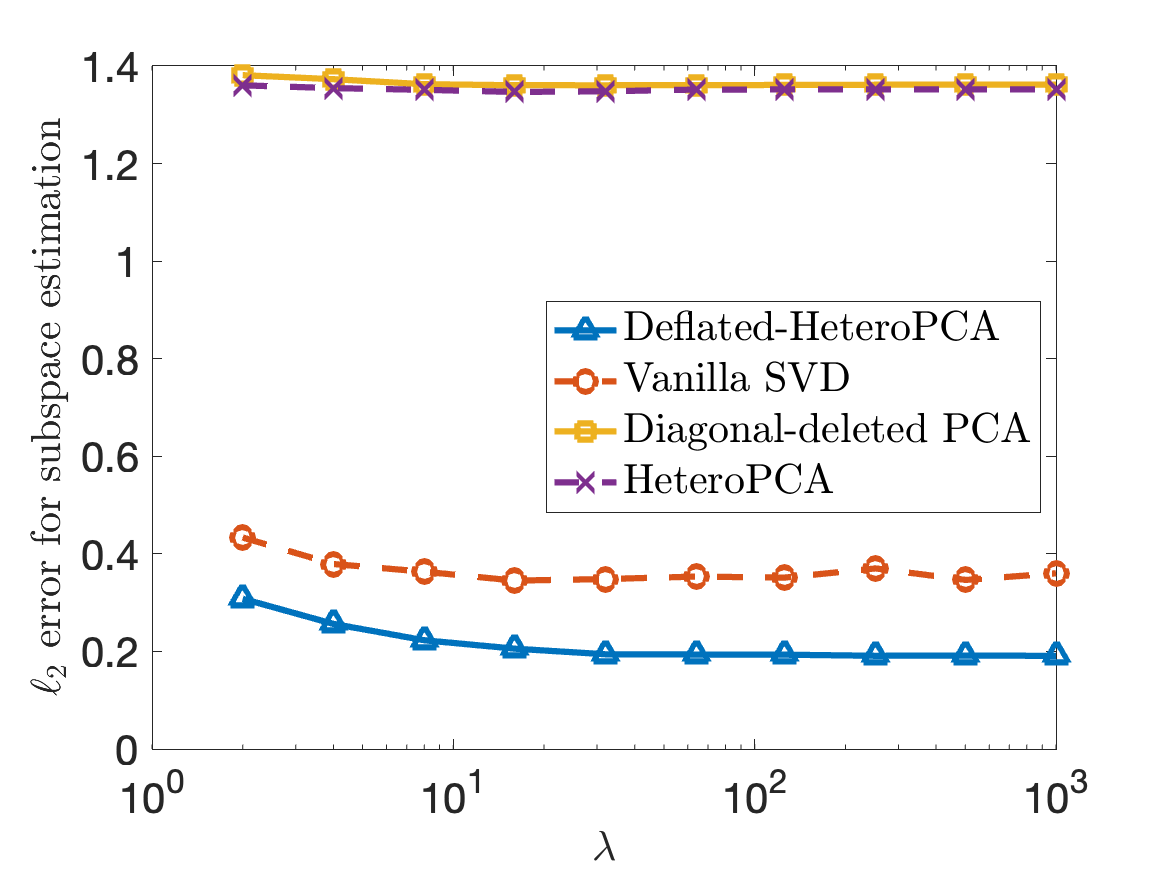}}
		\subfigure[$\kappa=5, n_2 = 1,000$, $\ell_{2,\infty}$ error]{
			\includegraphics[width=0.45\linewidth]{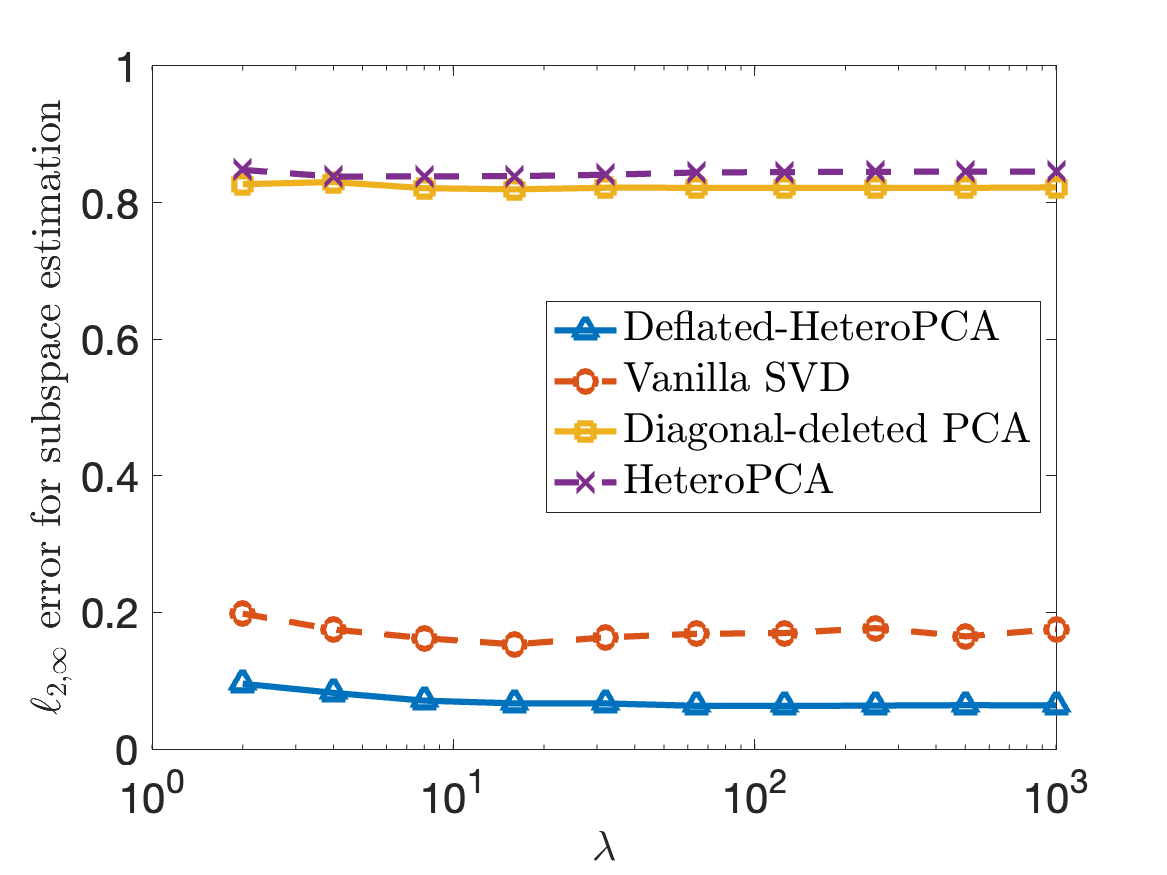}}
	\end{minipage}
	\caption{Estimation errors of $\bm{U}$ for {\sf Deflated-HeteroPCA}, {\sf Diagonal-deleted PCA}, {\sf HeteroPCA} and {\sf Vanilla SVD} under the Poisson PCA model. Plot (a) (resp. (b)) reports the $\ell_2$ (resp.~$\ell_{2,\infty}$) error vs.~$\lambda$ (where $n_1 = 100, n_2 = 1,000, r = 3$).}\label{fig:simulation_Poisson_r_3} 
\end{figure}

\paragraph{Tensor PCA.} Finally, we conduct numerical experiments for the tensor PCA model \eqref{model:tensor_SVD-all}. We fix $n = 50$ and $r = 3$, and introduce a quantity $\sigma^\star = n^{3/4}$. The subspaces $\bm{U}_1^\star \in \mathcal{O}^{100,3}$, $\bm{U}_2^\star \in \mathcal{O}^{100,3}$ and $\bm{U}_3^\star \in \mathcal{O}^{100,3}$ are generated randomly, and the core tensor $\mathcal{S}^\star \in \bbR^{3 \time 3 \times 3}$ is a diagonal tensor with entries $S_{1,1,1} = \kappa\sigma^\star$ and $S_{2,2,2} = S_{3,3,3} =\sigma^\star$. The noise tensor is generated in the following way:
we first generate three random vectors $\bm{\alpha}, \bm{\beta}$ and $\bm{\gamma}$, where $\{\alpha_i\}$, $\{\beta_j\}$, $\{\gamma_k\}$ are independently drawn from $[0, 1]$. 
We then generate each $E_{i,j,k}$ independently  from $\mathcal{N}(0, \omega^2\alpha_i^2\beta_j^2\gamma_k^2)$. The above four subspace estimation methods are applied to obtain initial subspace estimates, followed by 50 iterations of HOOI to refine the subspace estimators and construct the final tensor estimates. Figures \ref{fig:simulation_tensor_initial} and \ref{fig:simulation_tensor_final} report the initial subspace estimation errors and the final subspace/tensor estimation errors, respectively. We can see from these plots that the {\sf Deflated-HeteroPCA} algorithm produces faithful initial estimators in terms of both the $\ell_2$ and $\ell_{2,\infty}$ errors, outperforming the other three methods. Moreover, compared with the other three methods,  the {\sf Deflated-HeteroPCA} algorithm serves as a more effective initialization scheme that can help one achieve more reliable subspace and tensor estimators. 
\begin{figure}[t]
	\centering
	\begin{minipage}[t]{\linewidth}
		\centering
		\subfigure[$\kappa=6$, $\ell_2$ error]{
			\includegraphics[width=0.32\linewidth]{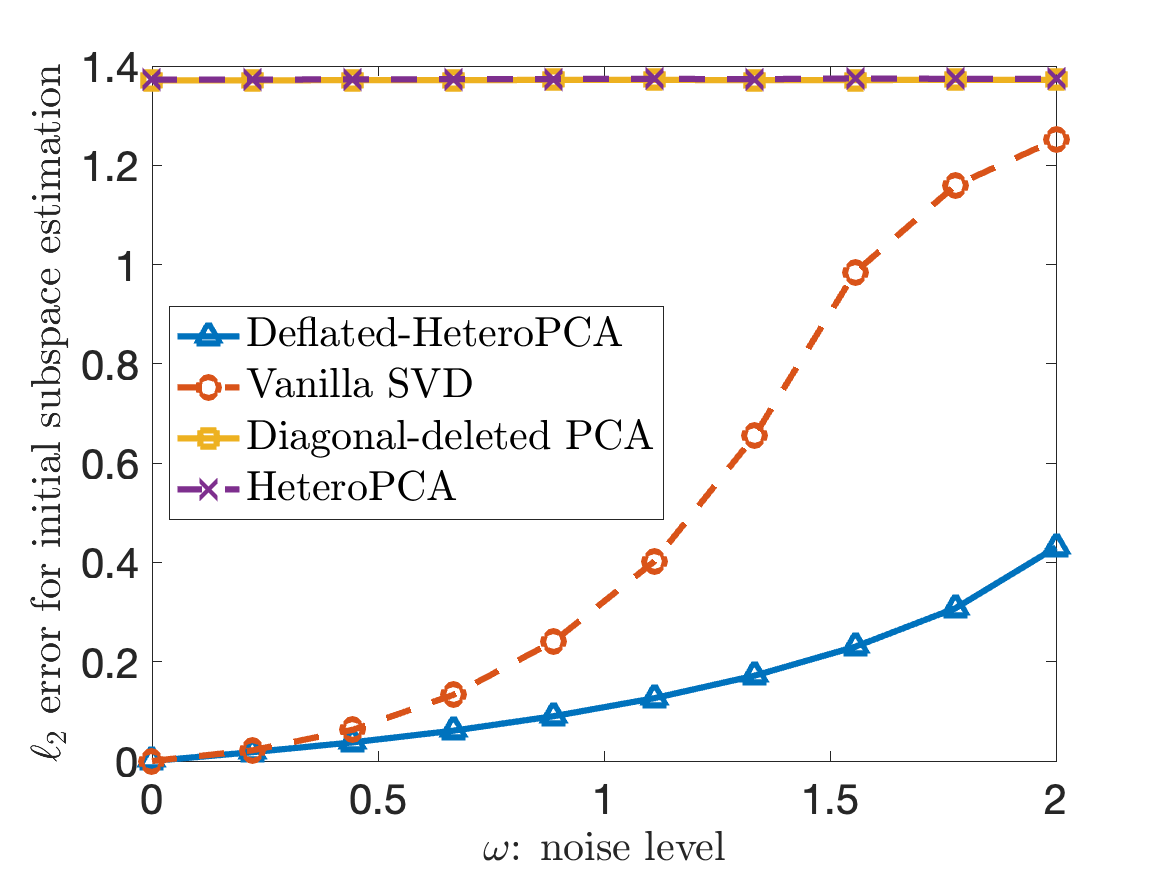}}
		\subfigure[$\kappa=6$, $\ell_{2,\infty}$ error]{
			\includegraphics[width=0.32\linewidth]{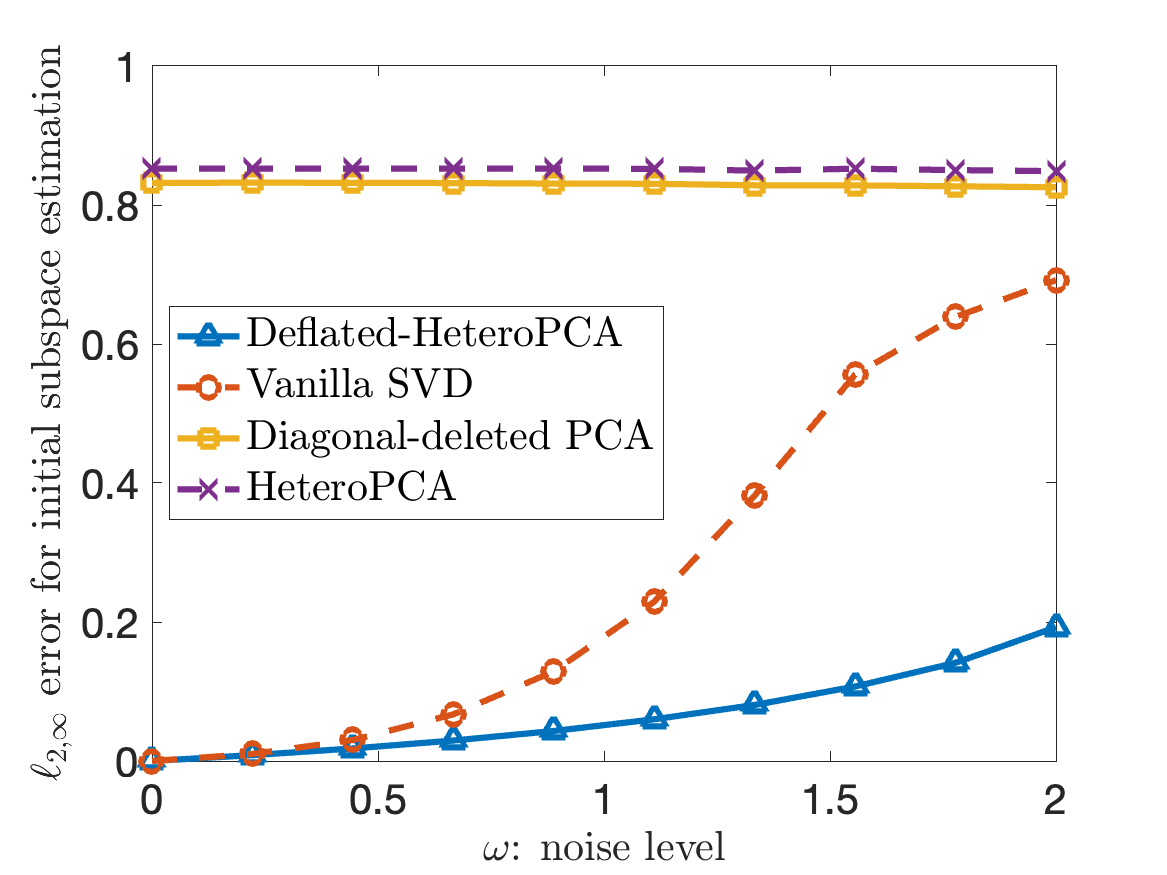}} \\
		\subfigure[$\omega = 2$, $\ell_2$ error]{
			\includegraphics[width=0.32\linewidth]{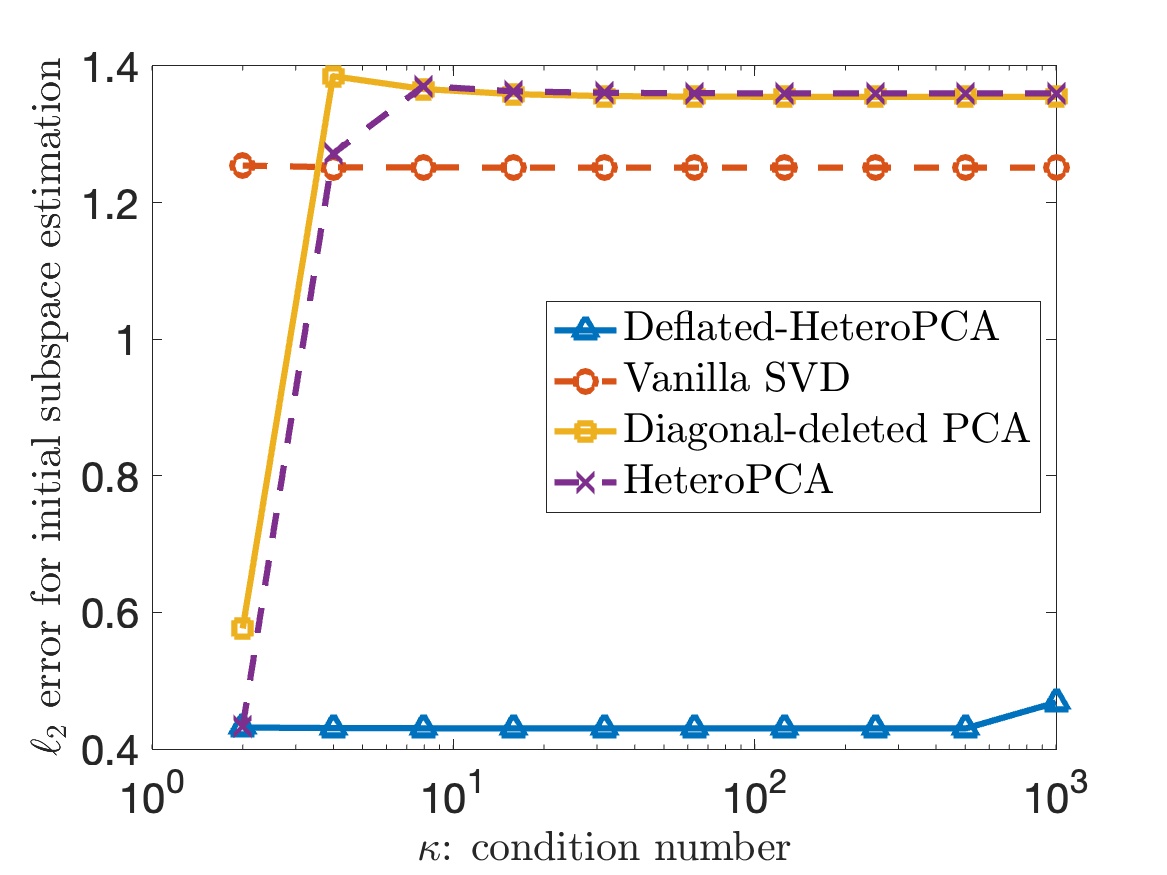}}
		\subfigure[$\omega = 2$, $\ell_{2,\infty}$ error]{
			\includegraphics[width=0.32\linewidth]{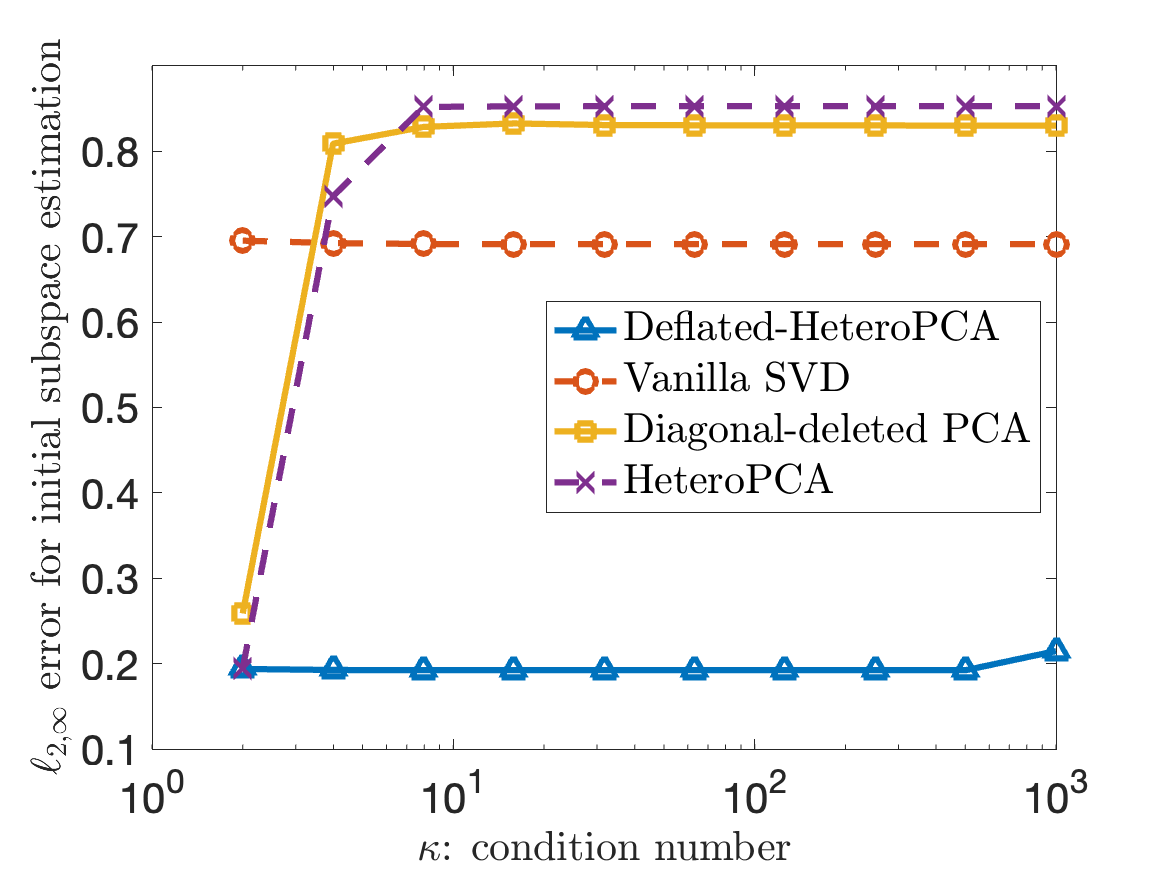}}
	\end{minipage}
	\caption{Initial estimation errors of $\widehat{\bm{U}}_1^0$ for {\sf Deflated-HeteroPCA}, {\sf Diagonal-deleted PCA}, {\sf HeteroPCA} and {\sf Vanilla SVD} under the tensor SVD model \eqref{model:tensor_SVD-all}. Plot~(a) (resp.~(b)) diplays the $\ell_2$ (resp.~$\ell_{2,\infty}$) error vs.~the noise level $\omega$ (where $n_1 = n_2 = n_3 = 50, r = 3, \kappa = 6$). 
	Plot (c) (resp.~(d)) shows the $\ell_2$ (resp.~$\ell_{2,\infty}$) error vs.~the condition number $\kappa$ (where $n_1 = n_2 = n_3 = 50, r = 3, \omega = 2$).}\label{fig:simulation_tensor_initial}
\end{figure}

\begin{figure}[t]
	\centering
	\begin{minipage}[t]{\linewidth}
		\centering
		\subfigure[$\kappa=6$, subspace estimation]{
			\includegraphics[width=0.32\linewidth]{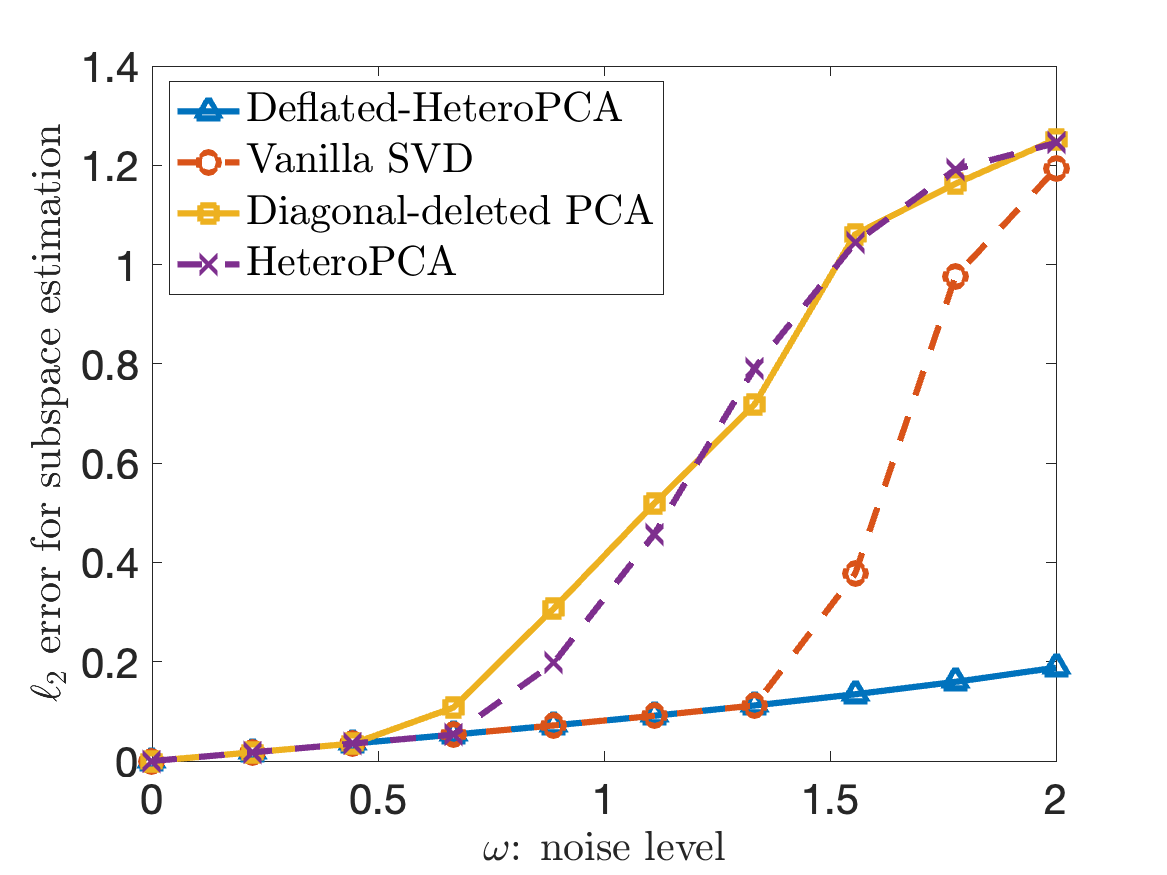}}
		\subfigure[$\kappa=6$, tensor estimation]{
			\includegraphics[width=0.32\linewidth]{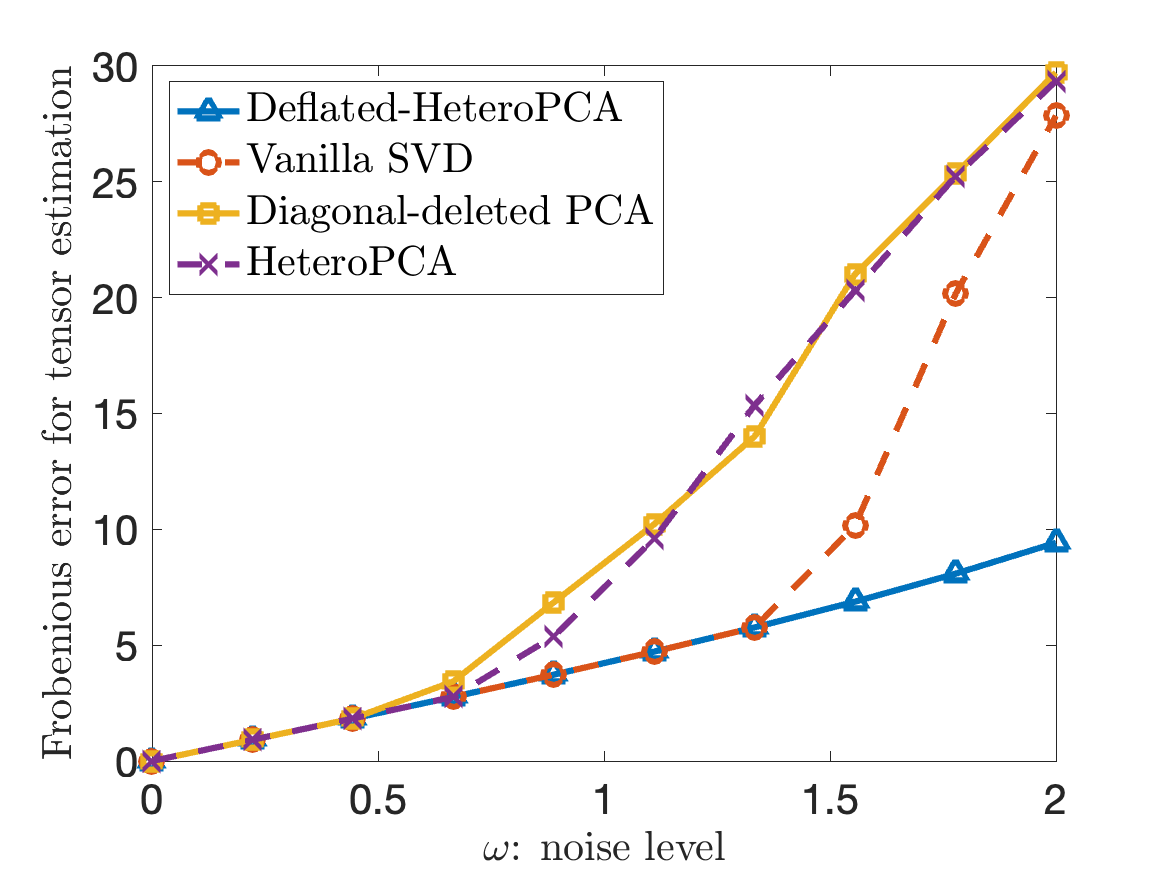}} \\
		\subfigure[$\omega = 2$, subspace estimation]{
			\includegraphics[width=0.32\linewidth]{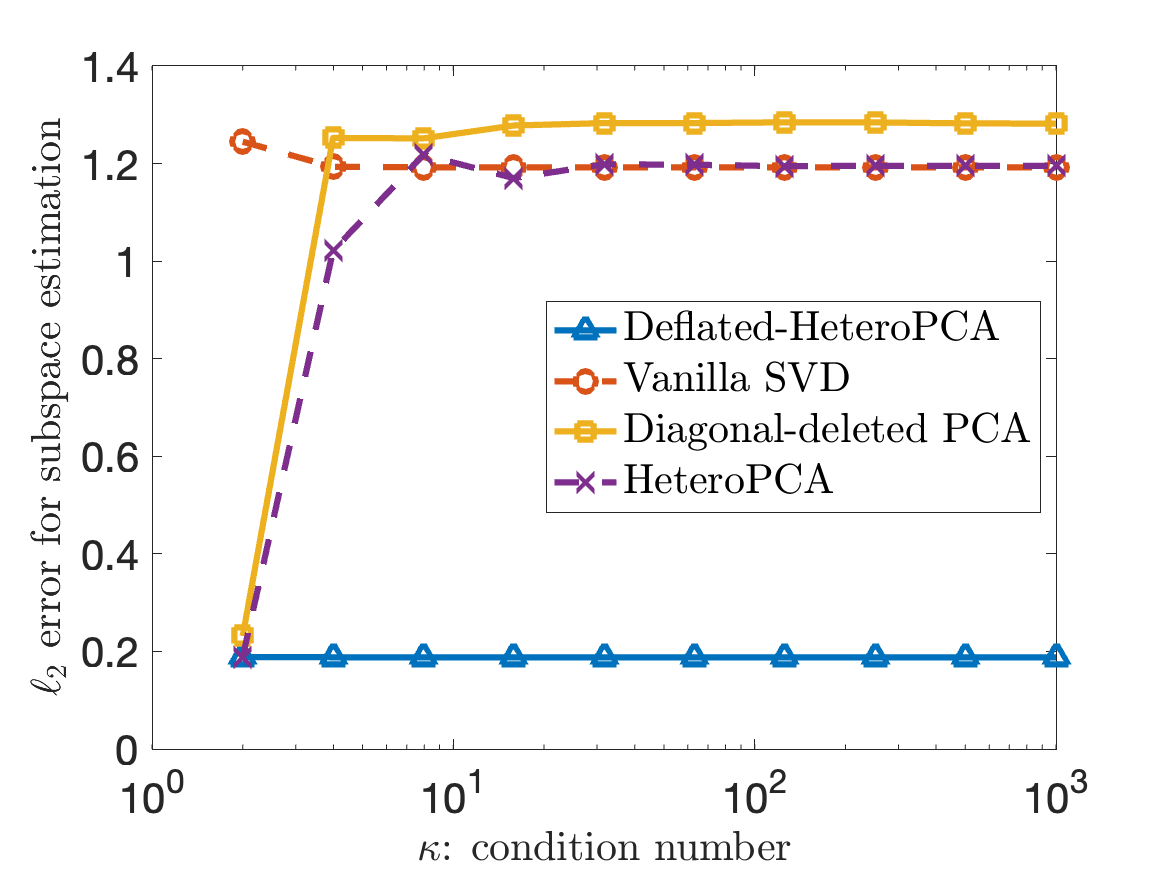}}
		\subfigure[$\omega = 2$, tensor estimation]{
			\includegraphics[width=0.32\linewidth]{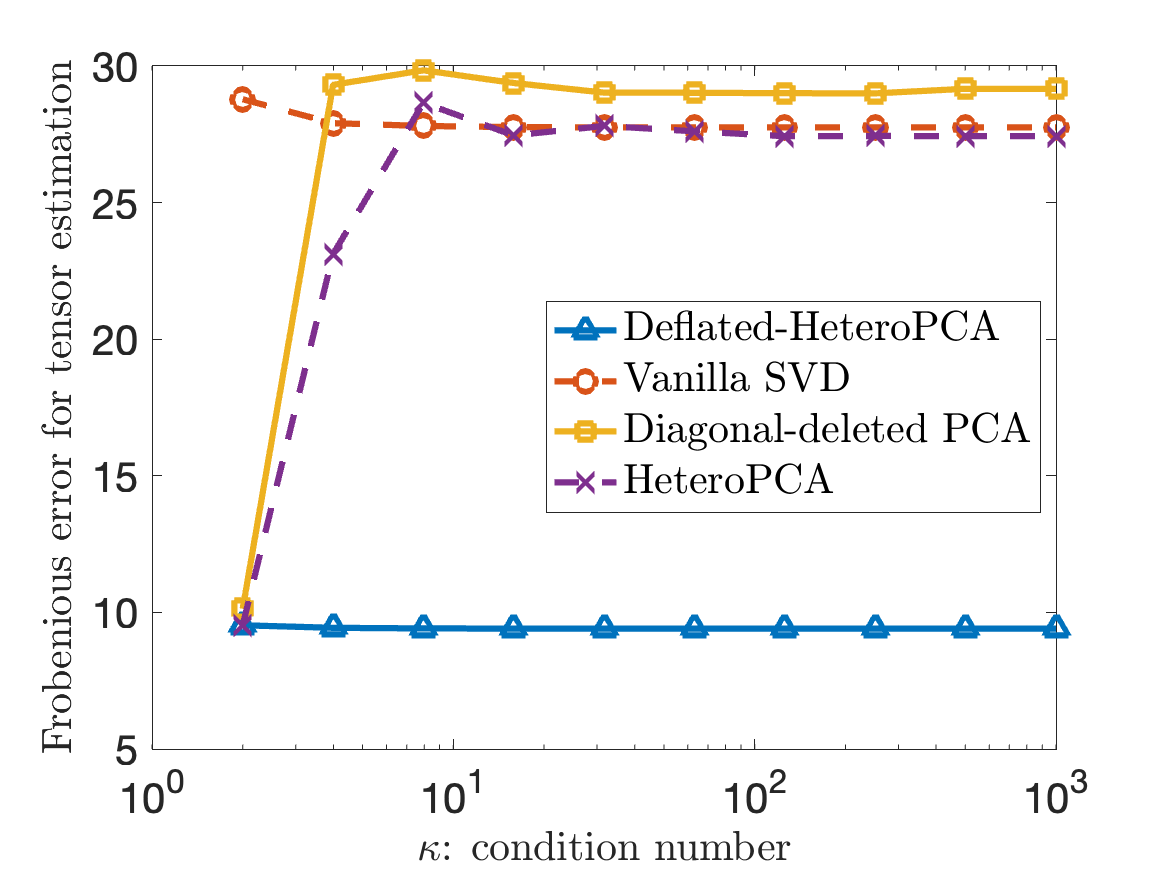}}
	\end{minipage}
	\caption{Final estimation errors of $\widehat{\bm{U}}_1$ and $\widehat{\bm{\mathcal{X}}}$ for {\sf Deflated-HeteroPCA}, {\sf Diagonal-deleted PCA}, {\sf HeteroPCA} and {\sf Vanilla SVD} under the tensor SVD model \eqref{model:tensor_SVD-all}. We report (a) (resp.~(b)) $\ell_2$ (resp.~Frobenious) error of $\widehat{\bm{U}}_1$ (resp.~$\widehat{\bm{\mathcal{X}}}$) vs. noise level $\omega$ (where $n_1 = n_2 = n_3 = 50, r = 3, \kappa = 6$); (c) (resp.~(d)) $\ell_2$ (resp.~Frobenious) error of $\widehat{\bm{U}}_1$ (resp.~$\widehat{\bm{\mathcal{X}}}$) vs. condition number $\kappa$ (where $n_1 = n_2 = n_3 = 50, r = 3, \omega = 2$).}\label{fig:simulation_tensor_final}
\end{figure}

    \section{Related works}\label{sec:related_work}

This paper is closely related to the problem of matrix denoising, which aims to estimate either a low-rank matrix or its column subspace based on noisy observations and spans a diverse array of applications \citep{chen2021spectral}. 
In addition to the examples of factor models and tensor estimation 
 \citep{cai2018rate,cai2021subspace,zhu2019high,richard2014statistical,zhang2018tensor,cai2021subspace}, 
 it can also help us understand and solve several clustering problems \citep{rohe2011spectral,florescu2016spectral,cai2021subspace,chen2022global,cai2018rate,loffler2021optimality,ndaoud2022sharp,srivastava2022robust,han2022exact,zhang2022leave}.
When it comes to the task of estimating the whole matrix, 
a number of methods have been put forward and thoroughly studied in the literature, including but not limited to singular value hard thresholding \citep{gavish2014optimal,chatterjee2015matrix}, singular value soft thresholding \citep{cai2010singular,koltchinskii2011nuclear,donoho2014minimax} and  singular value shrinkage \citep{nadakuditi2014optshrink,gavish2017optimal}. Turning to the task of subspace estimation, the vanilla SVD-based approach (see \eqref{eq:vanilla-SVD-alg}) has been commonly used and widely studied in the literature \citep{koltchinskii2016perturbation,cai2018rate,bao2021singular,xia2021normal,chen2021spectral}. 
How to perform uncertainty quantification for this approach has also been demonstrated in the previous work (see \citep{chen2021spectral}).  
In the scenario where the matrix dimensions are extremely unbalanced and the noise is heteroskedastic, however, such estimators can be highly suboptimal for subspace estimation. As already mentioned previously,  
the diagonal-deleted PCA and {\sf HeteroPCA} algorithms have been proposed to improve the performance over the vanilla SVD approach \citep{cai2021subspace,zhang2022heteroskedastic,agterberg2022entrywise,yan2024inference}. In fact, it has also been shown in \citet{yan2024inference} that the {\sf HeteroPCA} admits a non-asymptotic distributional theory, 
which paves the way to construction of fine-grained confidence regions for this problem. 
Another family of effective algorithms --- which can even accommodate the case when there is additional prior structure on the low-rank factors --- is approximate message passing \citep{montanari2021estimation,deshpande2017asymptotic,feng2022unifying,li2023approximate,li2022non,montanari2022fundamental}, 
for which the existing theory often requires more stringent assumptions on the noise components (e.g., i.i.d.~Gaussian).  
It is also worth mentioning that how to accelerate optimization-based low-rank estimation algorithms in spite of ill conditioning has been an active research topic as well, 
which oftentimes involves proper preconditioning \citep{tong2021accelerating,xu2023power}; the statistical guarantees therein, however, are still dependent on the condition number.

With regards to the factor model, one can easily find numerous works on this topic. The model \eqref{model:PCA} has been extensively studied under the names of spiked covariance models \citep{johnstone2001distribution,paul2007asymptotics,bai2012estimation,wang2017asymptotics,donoho2018optimal,perry2018optimality,bao2022statistical} and factor models \citep{lawley1962factor,bai2012statistical,fan2016projected,bai2016econometric}. Focusing on  principal component estimation under heteroskedastic noise, \cite{hong2016towards,hong2018asymptotic,hong2018optimally} investigate the case where the noise components within each noise vector $\bm{\varepsilon}_j$ are i.i.d., and develop asymptotic analysis for PCA and a variant called Weighted PCA. Turning to non-asymptotic analysis, the theoretical performances of  diagonal-deleted PCA \citep{cai2021subspace} and {\sf HeteroPCA}  have been investigated in \citep{cai2021subspace,zhang2022heteroskedastic,yan2024inference}. 
It is also worth noting that principal component estimation in the presence of missing data encounters additional challenges \citep{cai2021subspace,zhang2022heteroskedastic,zhu2019high,pavez2020covariance,yan2024inference}, which is beyond the scope of this work. 

Another important example considered in this paper is the tensor PCA or tensor SVD model \eqref{model:tensor_SVD-all}. 
Under this model, \cite{richard2014statistical,hopkins2015tensor,anandkumar2017homotopy,arous2019landscape,perry2020statistical} study the statistical and computational limits for rank-$1$ tensors. For low Tucker-rank tensors, many methods have been proposed for tensor/subspace estimation, including  high-order SVD (HOSVD, \cite{de2000multilinear}),  high-order orthogonal iteration (HOOI, \cite{de2000best,zhang2018tensor}), the sequentially truncated higher-order singular value decomposition algorithm (ST-HOSVD, \cite{vannieuwenhoven2012new}), projected gradient descent \citep{han2022optimal}, and scaled gradient descent \citep{tong2022scaling}. When the noise tensor has i.i.d.~Gaussian entries, \cite{zhang2018tensor} proves the statistical and computational limit for the tensor SVD and reveals that the HOOI achieves the optimal performance both statistically and computationally. Allowing the noise to be heteroskedastic, \cite{han2022optimal} shows that the optimal error rate can be achieved by the projected gradient descent with the initialization given by the {\sf HeteroPCA} if the condition number of the true tensor is bounded. In contrast to the prior literature, we consider the tensor and subspace estimation problem under heteroskedastic noise and aim to accommodate an arbitrarily large condition number; we show that the HOOI algorithm initialized by {\sf Deflated-HeteroPCA} yields optimal theoretical guarantees. In addition to the Tucker-rank decomposition, the tensor PCA/SVD model with the low CP-rank structure \citep{kolda2009tensor,anandkumar2014tensor,cai2021subspace,cai2022nonconvex,cai2020uncertainty} and the low tensor-train-rank structure \citep{zhou2022optimal,cai2022provable} have also received much attention in the past few years. 

In addition, recent years have witnessed much acitivity in developing $\ell_{\infty}$ and $\ell_{2,\infty}$ theoretical guarantees for singular subspaces and eigenspaces \citep{zhong2018near,fan2018eigenvector,cape2019two,agterberg2022entrywise}. 
Particularly worth noting is the leave-one-out analysis framework, which emerges as a powerful tool to derive fine-grained (e.g., entrywise or rowwise) bounds and finds applications in numerous high-dimensional estimation problems \citep{zhong2018near,ma2020implicit,chen2019spectral,abbe2020entrywise,chen2020noisy,chen2019inference,chen2021bridging,cai2021subspace,chen2021convex,cai2022nonconvex,abbe2022lp,yan2024inference,ling2022near,zhang2022leave,yang2022optimal}. However, existing $\ell_{2,\infty}$ estimation guaranteed obtained by means of the leave-one-out technique still rely on the condition number. To achieve a condition-number-free $\ell_{2,\infty}$ bound, we provide a novel analysis based on the representation theorem presented in \cite{xia2021normal}. The idea also shares similar spirit with the Neumann trick, which is commonly used in $\ell_{\infty}$ eigenvector analysis  \citep{eldridge2018unperturbed,chen2021asymmetry,cheng2021tackling}.

\section{Discussion}

This paper has studied subspace estimation from noisy low-rank matrices in the presence of unbalanced dimensionality and heteroskedastic noise. 
Recognizing a curse of ill-conditioning that appears in two cutting-edge algorithms, 
we have developed a new algorithm called {\sf Deflated-HeteroPCA} to strengthen the state-of-the-art statistical performance in the face of a large condition number, 
without compromising the range of SNRs that can be accommodated. 
We have demonstrated that the proposed estimator enjoys nearly rate-optimal statistical guarantees (in terms of both the spectral-norm error and the more fine-grained $\ell_{2,\infty}$-based error), 
which are unaffected by the underlying condition number (regardless of how large it is). 
When applied to two concrete statistical models (i.e., factor models and tensor PCA), 
our theory has led to remarkable improvement over the prior art (particularly for the ill-conditioned scenarios).

Our work suggests several potential avenues for future investigation. 
For example, the signal-to-noise ratio conditions \eqref{assump:snr2} and \eqref{assump:snr_two_to_infty} in our theory remain sub-optimal when it comes to their dependency on the rank $r$. 
How to tighten this rank dependency calls for a more refined analysis or a more powerful algorithm. 
Another direction worthy of future studies is the case with missing data (i.e., suppose we only have access to highly incomplete observations of the entries of the data matrix $\bm{Y}$ in \eqref{eq:model}).  It would be of great interest to extend our approach and develop a computationally efficient estimator that enjoys condition-number-free and rate-optimal estimation guarantees in the presence of missing data. Furthermore, note that the independent noise assumption plays an important role on our current theoretical analysis. Having said that, our method has potential to deal with more general correlated noise distributions (e.g., the one arising in network data). Our follow-up work \cite{zhou2023heteroskedastic} applied a clustering method based on {\sf Deflated-HeteroPCA} to the flight route network data, which demonstrates superior clustering performance compared to prior algorithms. We leave more extensive theoretical studies for the case with correlated data to future investigation.

\section*{Acknowledgements}

This work is supported in part by the Alfred P.~Sloan Research Fellowship,   
and the NSF grants CCF-1907661, DMS-2014279, IIS-2218713 and IIS-2218773.

\appendix

    \section{Proof of Theorem \ref{thm:spectral} ($\ell_2$ analysis for {\sf Deflated-HeteroPCA})}\label{proof:thm_spectral}
Before continuing, we introduce some notation about some intermediate objects that appear in our algorithm, which will be useful in the proofs. First,  set 
\begin{subequations}
	\begin{align}\label{def:inter_initialization}
		\bm{G}_{k+1}^{0} := \bm{G}_k, \qquad 0 \leq k \leq k_{\sf max},
	\end{align}
	where we recall that $$\bm{G}_0 = \mathcal{P}_{\sf off\text{-}diag}(\bm{Y}\bm{Y}^\top).$$
	For each $t = 0, 1, \dots, t_{k+1}$ and $k = 0, 1, \dots, k_{\sf max}$, let
	\begin{align}\label{def:inter_SVD}
		\bm{U}_{k+1}^{t}\bm{\Lambda}_{k+1}^{t}\bm{U}_{k+1}^{t\top} :=  \text{ the rank-}r_k \text{ leading eigendecompostion of } \bm{G}_{k+1}^t, 
	\end{align}
	and define 
	\begin{align}\label{def:inter_G}
		\bm{G}_{k+1}^{t+1} := \mathcal{P}_{\sf off\text{-}diag}\left(\bm{G}_{k+1}^{t}\right) + \mathcal{P}_{\sf diag}\left(\bm{U}_{k+1}^{t}\bm{\Lambda}_{k+1}^{t}\bm{U}_{k+1}^{t\top}\right), 
	\end{align}
\end{subequations}
which corresponds to the matrix computed by {\sf HeteroPCA} in the $t$-th iteration of the $(k+1)$-th round. 

In this section, we intend to prove a slightly more general version of Theorem \ref{thm:spectral} as follows.   
\begin{thm}\label{thm:spectral-general}
Suppose that Assumption \ref{assump:hetero} holds. Suppose that 
\begin{subequations}
	\begin{align}
		\sigma_r^\star &\geq C_0r\sqrt{\mu r\omega_{\sf max}^2 + \omega_{\sf col}\left(\omega_{\sf row} + \omega_{\sf col}\right)}\sqrt{\log n} \label{assump:snr}\\
		\mu &\leq c_0\frac{n_1}{r^3}
	\end{align}
\end{subequations}
for some sufficiently large (resp.~small) constant $C_0 > 0$ (resp.~$c_0 > 0$).
If the numbers of iterations obey \eqref{ineq:iter1},
then with probability exceeding $1 - O(n^{-10})$, the output returned by Algorithm~\ref{algorithm:sequential_heteroPCA} satisfies
\begin{align}\label{ineq:svd_spectral}
	\left\|\bm{U}\bm{R}_{\bm{U}} - \bm{U}^\star\right\| \lesssim \frac{\sqrt{\left(\mu r\omega_{\sf max}^2 + \omega_{\sf col}^2\right)\log n}}{\sigma_r^{\star}} + \frac{\omega_{\sf col}\omega_{\sf row}\log n}{\sigma_r^{\star2}} + e^{-t_{k_{\sf max}}}.
\end{align}
\end{thm}
Evidently, if we further have $0 < \mu r\omega_{{\sf max}}^2 \lesssim \omega_{\sf col}^2$ and if the number of iterations $t_{k_{\sf max}}$ obeys \eqref{ineq:iter2}, 
then it is easy to check that the bound \eqref{ineq:svd_spectral} (resp.~the signal-to-noise ratio condition \eqref{assump:snr}) implies \eqref{ineq:svd_spectral2} (resp.~\eqref{assump:snr2}). 
This allows us to focus attention on establishing Theorem~\ref{thm:spectral-general}.

\subsection{A key intermediate result and the proof of Theorem~\ref{thm:spectral-general}} 

Towards proving Theorem~\ref{thm:spectral-general}, 
we first single out a deterministic result that plays a crucial role in bounding $\left\|\bm{U}\bm{R}_{\bm{U}} - \bm{U}^\star\right\|$; its proof is postponed to Section~\ref{proof:thm_deterministic}.
\begin{thm}\label{thm:deterministic}
	Suppose that we observe a matrix $\bm{M} = \overline{\bm{U}}\,\overline{\bm \Lambda}\,\overline{\bm U}^\top + \bm{Z}$, 
	where $\overline{\bm \Lambda} \in \mathbb{R}^{r\times r}$ is a diagonal matrix with diagonal entries $\overline{\lambda}_1 \geq \dots \geq \overline{\lambda}_r > 0$ and $\overline{\bm U} \in \mathcal{O}^{n_1 , r}$ satisfies 
	\begin{subequations}
	\begin{equation}
		\|\overline{\bm U}\|_{2,\infty} \leq \sqrt{\frac{\overline\mu r}{n_1}}
		\qquad \text{with }~\overline{\mu} \leq c_0\frac{n_1}{r^3}
		\label{eq:spectral-det-mu-UB}
	\end{equation}
	for some sufficiently small constant $c_0 > 0$. 
	Also, assume that
	\begin{align}
		\overline\lambda_r &\geq C_0r\|\mathcal{P}_{\sf off\text{-}diag}(\bm{Z})\| .
		\label{eq:lambda-Poffdiag-thm-det}
	\end{align}
	\end{subequations}
	Then Algorithm \ref{algorithm:sequential_heteroPCA} with initialization $\bm{G}_0 = \mathcal{P}_{\sf off\text{-}diag}(\bm{M})$ 
	yields an estimate $\bm{U}$ satisfying 
	\begin{align}\label{ineq:deterministic}
		\big\|\bm{U}\bm{U}^\top - \overline{\bm{U}}\,\overline{\bm{U}}^\top\big\| \lesssim \frac{\left\|\mathcal{P}_{\sf off\text{-}diag}\left(\bm{Z}\right)\right\|}{\overline\lambda_{r}} + e^{-t_{k_{\sf max}}},
	\end{align}
	provided that the numbers of iterations obey
	\begin{subequations}
	\begin{align}
		t_1 &> \log\left(\frac{\frac{\overline{\mu} r}{n_1}\overline\lambda_1}{\sqrt{\frac{\overline{\mu} r}{n_1}}\overline\lambda_{r_1 + 1} + \left\|\mathcal{P}_{\sf off\text{-}diag}\left(\bm{Z}\right)\right\|}\right) \, \vee 0 \\
		t_k &> \log\left(\frac{7\sqrt{\frac{\overline{\mu}r}{n_1}}\overline\lambda_{r_{k-1}+1} + \left\|\mathcal{P}_{\sf off\text{-}diag}\left(\bm{Z}\right)\right\|}{\sqrt{\frac{\overline{\mu}r}{n_1}}\overline\lambda_{r_{k}+1} + \left\|\mathcal{P}_{\sf off\text{-}diag}\left(\bm{Z}\right)\right\|}\right), \quad 1 \leq k \leq k_{\sf max}-1.
	\end{align}
	\end{subequations}
\end{thm}
In a nutshell, Theorem~\ref{thm:deterministic} asserts that the subspace estimation error of {\sf Deflated\text{-}HeteroPCA} depends only on (i) the size of  $\bm{Z}$ after diagonal deletion and (ii) the $r$-th leading eigenvalue of $\overline{\bm{U}}\,\overline{\bm \Lambda}\,\overline{\bm U}^\top$,   
provided that the numbers of iterations exceed some logarithmic factors. Notably, the estimation error bound \eqref{ineq:deterministic} 
holds irrespective of the condition number of $\overline{\bm \Lambda}$ and the noise entries $\mathcal{P}_{\sf diag}(\bm{Z})$ in the diagonal (in fact, these diagonal entries of $\bm{Z}$ are never used by {\sf Deflated\text{-}HeteroPCA}). 

We now demonstrate how to invoke Theorem~\ref{thm:deterministic} to establish Theorem~\ref{thm:spectral-general}, 
which consists of several steps below. 
Before proceeding, we isolate one important matrix $\bm{U}^{\star}\bm{\Sigma}^{\star} + \bm{E}\bm{V}^{\star} \in \bbR^{n_1 \times r}$, 
and denote its SVD as 
\begin{equation}
	\widetilde{\bm{U}}\widetilde{\bm{\Sigma}}\widetilde{\bm{W}}^\top = \bm{U}^{\star}\bm{\Sigma}^{\star} + \bm{E}\bm{V}^{\star}, 
	\label{eq:SVD-Usigma-plus-EV}
\end{equation}
where  $\widetilde{\bm{U}} \in \mathcal{O}^{n_1, r}, \widetilde{\bm W} \in \mathcal{O}^{r,r}$ and $\widetilde{\bm \Sigma} = {\sf diag}(\widetilde{\sigma}_1, \dots, \widetilde{\sigma}_r)$ with $\widetilde{\sigma}_1 \geq \dots \geq \widetilde{\sigma}_r \geq 0$. 

\paragraph{Step 1: bounding the spectrum of $\widetilde{\bm{\Sigma}}^{-1}$.}
We start by controlling the spectrum of $\widetilde{\bm{\Sigma}}$. 
Taking Weyl's inequality, Assumption \ref{assump:hetero} and Lemma \ref{lm:noise} together implies that with probability exceeding $1 - O(n^{-10})$,
\begin{align}\label{ineq:singular_difference}
	\max_{1 \leq i \leq r}\left|\widetilde{\sigma}_i - \sigma_i^{\star}\right| \leq \left\|\bm{E}\bm{V}^{\star}\right\| 
	&\lesssim \sqrt{\left(r\omega_{\sf max}^2 + \omega_{\sf col}^2\right)\log n} + B\log n\sqrt{\frac{\mu_2 r}{n_2}}\notag\\
	&\lesssim \sqrt{\left(r\omega_{\sf max}^2 + \omega_{\sf col}^2\right)\log n} + \frac{\omega_{\sf row}}{\sqrt{\log n}}\log n\sqrt{\frac{\mu_2 r}{n_2}}\notag\\
	&\lesssim \sqrt{\left(r\omega_{\sf max}^2 + \omega_{\sf col}^2\right)\log n} + \sqrt{n_2}\omega_{\sf max}\sqrt{\log n}\sqrt{\frac{\mu r}{n_2}}\notag\\
	&\lesssim \sqrt{\left(\mu r\omega_{\sf max}^2 + \omega_{\sf col}^2\right)\log n},
\end{align}
where the second line relies on Assumption \ref{assump:hetero}. 
Consequently, one can deduce that
\begin{align}\label{ineq1}
	\big\|\widetilde{\bm{\Sigma}}^{-1}\big\| 
	= \frac{1}{\widetilde{\sigma}_r} \leq \frac{1}{\sigma^{\star}_r - \left\|\bm{E}\bm{V}^{\star}\right\|} \leq \frac{\sqrt{2}}{\sigma^{\star}_r} ,
\end{align}
provided that $\sigma^{\star}_r \geq C_0 \sqrt{\left(\mu r\omega_{\sf max}^2 + \omega_{\sf col}^2\right)\log n}$ for some large enough constant $C_0>0$. 
It is also seen that
\begin{align}
	\sigma^{\star}_r &\leq \widetilde{\sigma}_r + \left\|\bm{E}\bm{V}^{\star}\right\| \leq 	 \widetilde{\sigma}_r + \frac{1}{2} \sigma^{\star}_r \notag\\
	&\Longrightarrow \qquad 
	\sigma^{\star}_r \leq 2 \widetilde{\sigma}_r. 
	\label{ineq12345}
\end{align}
Repeating the same argument also reveals that
\begin{equation}
	 \frac{1}{2} \widetilde{\sigma}_i \leq \sigma^{\star}_i \leq 2 \widetilde{\sigma}_i,
	 \qquad 1\leq i\leq r. 
\end{equation}

\paragraph{Step 2: bounding $\|\widetilde{\bm{U}}\|_{2,\infty}$.}
We now move on to control $\|\widetilde{\bm{U}}\|_{2,\infty}$, 
a sort of incoherence condition needed in order to invoke Theorem~\ref{thm:deterministic} (see~\eqref{eq:spectral-det-mu-UB}).  
Towards this, we would like to first the discrepancy between $\bm{U}^{\star}\bm{U}^{\star\top}\widetilde{\bm{U}}$ and $\widetilde{\bm{U}}$, which would in turn allow us to switch attention to the $\ell_{2,\infty}$ norm of $\bm{U}^{\star}\bm{U}^{\star\top}\widetilde{\bm{U}}$. 
Recognizing that 
\begin{align}\label{ineq24}
	\big(\widetilde{\bm{U}}-\bm{U}^{\star}\bm{U}^{\star\top}\widetilde{\bm{U}}\big)\widetilde{\bm{\Sigma}}\widetilde{\bm{W}}^\top &= \mathcal{P}_{\bm{U}_{\perp}^{\star}}\widetilde{\bm U}\widetilde{\bm \Sigma}\widetilde{\bm W}^\top
	= \mathcal{P}_{\bm{U}_{\perp}^{\star}}\left(\bm{U}^{\star}\bm{\Sigma}^{\star} + \bm{E}\bm{V}^{\star}\right)\notag\\
	&= \bm{E}\bm{V}^{\star} - \bm{U}^{\star}\bm{U}^{\star\top}\bm{E}\bm{V}^{\star},
\end{align}
we can readily use $\|\widetilde{\bm{W}}\|=1$ to derive
\begin{align}\label{ineq25}
	\big\|\bm{U}^{\star}\bm{U}^{\star\top}\widetilde{\bm{U}} - \widetilde{\bm{U}}\big\|_{2, \infty} &= \left\|\left(\bm{E}\bm{V}^{\star} - \bm{U}^{\star}\bm{U}^{\star\top}\bm{E}\bm{V}^{\star}\right)\widetilde{\bm{W}}\widetilde{\bm{\Sigma}}^{-1}\right\|_{2, \infty}\notag\\
	&\leq \left(\left\|\bm{E}\bm{V}^{\star}\right\|_{2, \infty} + \left\|\bm{U}^{\star}\bm{U}^{\star\top}\bm{E}\bm{V}^{\star}\right\|_{2,\infty}\right)\big\|\widetilde{\bm{\Sigma}}^{-1}\big\|.
\end{align}
In view of Lemma \ref{lm:noise} and Assumption \ref{assump:hetero}, with probability exceeding $1 - O(n^{-10})$, one has
\begin{align*}
	\left\|\bm{E}\bm{V}^\star\right\|_{2,\infty} \lesssim \left(B\log n + \omega_{\sf row}\sqrt{\log n}\right)\sqrt{\frac{\mu_2 r}{n_2}} \asymp  \omega_{\sf row}\sqrt{\log n}\sqrt{\frac{\mu_2 r}{n_2}}
\end{align*}
and
\begin{align*}
	\left\|\bm{U}^{\star}\bm{U}^{\star\top}\bm{E}\bm{V}^{\star}\right\|_{2,\infty} &\leq \left\|\bm{U}^{\star}\right\|_{2, \infty}\left\|\bm{U}^{\star\top}\bm{E}\bm{V}^{\star}\right\|\\ &\lesssim \sqrt{\frac{\mu_1 r}{n_1}}\left(B\frac{\mu r}{\sqrt{n_1n_2}}\log n + \left(\sqrt{\frac{\mu r}{n_2}}\omega_{\sf row} + \sqrt{\frac{\mu r}{n_1}}\omega_{\sf col}\right)\sqrt{\log n}\right)\\
	&\lesssim \sqrt{\frac{\mu_1 r}{n_1}}\left(\frac{\omega_{\sf row}}{\sqrt{\log n}}\sqrt{\frac{\mu r}{n_2}}\log n + \left(\sqrt{\frac{\mu r}{n_2}}\omega_{\sf row} + \sqrt{\frac{\mu r}{n_1}}\omega_{\sf col}\right)\sqrt{\log n}\right)\\
	&\lesssim \sqrt{\frac{\mu_1 r}{n_1}}\left(\sqrt{\frac{\mu r}{n_2}}\omega_{\sf row} + \sqrt{\frac{\mu r}{n_1}}\omega_{\sf col}\right)\sqrt{\log n},
\end{align*}
where the second line has also made use of the assumption that $\mu r \lesssim n_1$. 
Putting \eqref{ineq25} and the previous two inequalities together and using the assumption $\mu r \lesssim n_1$, we arrive at
\begin{align}\label{ineq:two_to_infty_tilde_U}
	\big\|\bm{U}^{\star}\bm{U}^{\star\top}\widetilde{\bm{U}} - \widetilde{\bm{U}}\big\|_{2, \infty} &\lesssim \left(\omega_{\sf row}\sqrt{\log n}\sqrt{\frac{\mu_2 r}{n_2}} + \sqrt{\frac{\mu_1 r}{n_1}}\left(\sqrt{\frac{\mu r}{n_2}}\omega_{\sf row} + \sqrt{ \frac{\mu r}{n_1}}\omega_{\sf col}\right)\sqrt{\log n}\right)\frac{1}{\sigma^{\star}_r}\notag\\
	&\lesssim \frac{\left(\sqrt{\frac{\mu r}{n_2}}\omega_{\sf row} + \frac{\mu r}{n_1}\omega_{\sf col}\right)\sqrt{\log n}}{\sigma_r^{\star}}\notag\\
	&\ll \frac{1}{r} = \sqrt{\frac{\left(\frac{n_1}{r^3}\right)r}{n_1}} 
\end{align}
with probability exceeding $1 - O(n^{-10})$, 
provided that
\begin{align*}
	\sigma_r^\star \gg r\left[\left(\sqrt{\frac{\mu r}{n_2}}\omega_{\sf row} + \frac{\mu r}{n_1}\omega_{\sf col}\right)\sqrt{\log n} +  \sqrt{\left(\mu r\omega_{\sf max}^2 + \omega_{\sf col}^2\right)\log n}\right] \asymp r\sqrt{\left(\mu r\omega_{\sf max}^2 + \omega_{\sf col}^2\right)\log n}.
\end{align*}
As a result, with probability at least $1 - O(n^{-10})$, we reach the following upper bound:
\begin{align}
	\big\|\widetilde{\bm U}\big\|_{2, \infty} &\leq \big\|\bm{U}^{\star}\bm{U}^{\star\top}\widetilde{\bm{U}} - \widetilde{\bm{U}}\big\|_{2, \infty} + \big\|\bm{U}^{\star}\bm{U}^{\star\top}\widetilde{\bm{U}}\big\|_{2, \infty}\notag\\
	&\leq \big\|\bm{U}^{\star}\bm{U}^{\star\top}\widetilde{\bm{U}} - \widetilde{\bm{U}}\big\|_{2, \infty} + \left\|\bm{U}^{\star}\right\|_{2, \infty}\big\|\bm{U}^{\star\top}\widetilde{\bm{U}}\big\|\notag\\
	&\leq \big\|\bm{U}^{\star}\bm{U}^{\star\top}\widetilde{\bm{U}} - \widetilde{\bm{U}}\big\|_{2, \infty} + \sqrt{\frac{\mu r}{n_1}} \ll \sqrt{\frac{\left(\frac{n_1}{r^3}\right)r}{n_1}},
	\label{ineq:crude_two_to_infty_tilde_U}
\end{align}
where the last inequality holds under our assumption that $\mu r^3 \lesssim n_1$. 
With this $\ell_{2,\infty}$ bound for $\widetilde{\bm U}$ in place --- which reveals an upper bound $O\big(\frac{n_1}{r^3}\big)$ on the incoherence parameter of $\widetilde{\bm U}$ (see the requirement~\eqref{eq:spectral-det-mu-UB}) --- 
we can proceed to apply Theorem \ref{thm:deterministic} in the next step.

\paragraph{Step 3: bounding $\|\bm{U}\bm{U}^{\top} - \bm{U}^\star\bm{U}^{\star\top}\|$ and $\|\bm{U}\bm{R}_{\bm{U}} - \bm{U}^\star\|$.} 
In this step, we shall first invoke Theorem \ref{thm:deterministic} to control $\|\bm{U}\bm{U}^\top - \widetilde{\bm{U}}\widetilde{\bm{U}}^\top\|$, 
and then apply standard eigenspace perturbation theory to bound $\|  \widetilde{\bm{U}}\widetilde{\bm{U}}^\top - \bm{U}^{\star}\bm{U}^{\star\top}\|$.

 To begin with, let us write
\begin{align}\label{ineq:rewrite_model}
	\bm{Y}\bm{Y}^{\top} = \left(\bm{X}^{\star} + \bm{E}\right)\left(\bm{X}^{\star} + \bm{E}\right)^\top
	&= \left(\bm{U}^{\star}\bm{\Sigma}^{\star} + \bm{E}\bm{V}^{\star}\right)\left(\bm{U}^{\star}\bm{\Sigma}^{\star} + \bm{E}\bm{V}^{\star}\right)^\top + \big(  \bm{E}\bm{E}^\top - \bm{E}\bm{V}^{\star}\bm{V}^{\star\top}\bm{E}^\top \big).
\end{align}
Recall that $\widetilde{\bm{U}}$ represents the column subspace of $\left(\bm{U}^{\star}\bm{\Sigma}^{\star} + \bm{E}\bm{V}^{\star}\right)\left(\bm{U}^{\star}\bm{\Sigma}^{\star} + \bm{E}\bm{V}^{\star}\right)^\top$ (cf.~\eqref{eq:SVD-Usigma-plus-EV}). 
Thus, in order to apply Theorem \ref{thm:deterministic} to control $\|\bm{U}\bm{U}^\top - \widetilde{\bm{U}}\widetilde{\bm{U}}^\top\|$, 
the key lies in coping with 
$\|\mathcal{P}_{\sf off\text{-}diag}(\bm{\bm{E}\bm{E}^\top - \bm{E}\bm{V}^{\star}\bm{V}^{\star\top}\bm{E}^\top})\|$. 
By virtue of Lemma \ref{lm:off-diag} and Assumption \ref{assump:hetero}, with probability exceeding $1 - O(n^{-10})$ we have
\begin{align}
	\left\|\mathcal{P}_{\sf off\text{-}diag}\big(\bm{E}\bm{E}^\top\big)\right\| &\lesssim B^2\log^2 n + \omega_{\sf col}\left(\omega_{\sf row} + \omega_{\sf col}\right)\log n \notag\\
	&\lesssim \frac{\omega_{\sf row}\omega_{\sf col}}{\log n}\log^2 n + \omega_{\sf col}\left(\omega_{\sf row} + \omega_{\sf col}\right)\log n \notag\\
	&\asymp \omega_{\sf col}\left(\omega_{\sf row} + \omega_{\sf col}\right)\log n.
	\label{eq:P-offdiag-EE-spectral-bound}
\end{align}
Putting \eqref{ineq:singular_difference} and \eqref{eq:P-offdiag-EE-spectral-bound} together, we arrive at, with probability exceeding $1 - O(n^{-10})$,
\begin{align*}
\left\Vert \mathcal{P}_{\sf off\text{-}diag}\left(\bm{\bm{E}\bm{E}^{\top}-\bm{E}\bm{V}^{\star}\bm{V}^{\star\top}\bm{E}^{\top}}\right)\right\Vert  & \leq\left\Vert \mathcal{P}_{\sf off\text{-}diag}\left(\bm{E}\bm{E}^{\top}\right)\right\Vert +\left\Vert \bm{E}\bm{V}^{\star}\bm{V}^{\star\top}\bm{E}^{\top}\right\Vert +\left\Vert \mathcal{P}_{\sf diag}\left(\bm{E}\bm{V}^{\star}\bm{V}^{\star\top}\bm{E}^{\top}\right)\right\Vert \\
 & \leq\left\Vert \mathcal{P}_{\sf off\text{-}diag}\left(\bm{E}\bm{E}^{\top}\right)\right\Vert +2\left\Vert \bm{E}\bm{V}^{\star}\bm{V}^{\star\top}\bm{E}^{\top}\right\Vert \\	
	& \leq\left\Vert \mathcal{P}_{\sf off\text{-}diag}\left(\bm{E}\bm{E}^{\top}\right)\right\Vert +2\left\Vert \bm{E}\bm{V}^{\star}\right\Vert^2 \\
	&\lesssim \omega_{\sf col}\left(\omega_{\sf row} + \omega_{\sf col}\right)\log n + \left(\mu r\omega_{\sf max}^2 + \omega_{\sf col}^2\right)\log n\\
	&\asymp \left(\mu r\omega_{\sf max}^2 + \omega_{\sf col}\left(\omega_{\sf row} + \omega_{\sf col}\right)\right)\log n\\
	&\ll \frac{\sigma_r^{\star2}}{r} \lesssim \frac{\widetilde{\sigma}_r^2}{r}, 
\end{align*}
where the last inequality arises from our assumption \eqref{assump:snr} on $\sigma_r^{\star}$ and \eqref{ineq12345}. 
In view of Theorem \ref{thm:deterministic}, \eqref{ineq12345}, \eqref{ineq:crude_two_to_infty_tilde_U} and the previous inequality, we can easily check that: if $\{t_i\}$ satisfy \eqref{ineq:iter1}, then
one has
\begin{align}
	\big\|\bm{U}\bm{U}^\top - \widetilde{\bm{U}}\widetilde{\bm{U}}^\top\big\| &\leq \frac{\left\|\mathcal{P}_{\sf off\text{-}diag}\left(\bm{\bm{E}\bm{E}^\top - \bm{E}\bm{V}^{\star}\bm{V}^{\star\top}\bm{E}^\top}\right)\right\|}{\widetilde{\sigma}_r^2} + e^{-t_{k_{\sf max}}} \notag\\ 
	&\lesssim \frac{\left(\mu r\omega_{\sf max}^2 + \omega_{\sf col}\left(\omega_{\sf row} + \omega_{\sf col}\right)\right)\log n}{\sigma_r^{\star2}} + e^{-t_{k_{\sf max}}}
	\label{eq:UU-tildeU-tildeU-dist}
\end{align}
with probability exceeding $1 - O(n^{-10})$, 
provided that
\begin{align*}
	\sigma_r^{\star} &\gg r\left[\sqrt{\left(\mu r\omega_{\sf max}^2 + \omega_{\sf col}\left(\omega_{\sf row} + \omega_{\sf col}\right)\right)} + \sqrt{\frac{\mu r}{n_2}}\omega_{\sf row} + \frac{\mu r}{n_1}\omega_{\sf col}\right]\sqrt{\log n}\\
	&\asymp r\sqrt{\left(\mu r\omega_{\sf max}^2 + \omega_{\sf col}\left(\omega_{\sf row} + \omega_{\sf col}\right)\right)\log n}.
\end{align*}

Next, let us turn to bounding $\|  \widetilde{\bm{U}}\widetilde{\bm{U}}^\top - \bm{U}^{\star}\bm{U}^{\star\top}\|$. 
Taking \eqref{ineq:singular_difference} and the $\sin\Theta$ theorem \citep[Theorem~2.9]{chen2021spectral} together shows that
\begin{align*}
	\big\|\widetilde{\bm{U}}\widetilde{\bm{U}}^\top - \bm{U}^\star\bm{U}^{\star\top}\big\| \lesssim \frac{\sqrt{\left(\mu r\omega_{\sf max}^2 + \omega_{\sf col}^2\right)\log n}}{\sigma_r^{\star} - \left\|\bm{E}\bm{V}^{\star}\right\|} 
	\asymp \frac{\sqrt{\left(\mu r\omega_{\sf max}^2 + \omega_{\sf col}^2\right)\log n}}{\sigma_r^{\star}}
\end{align*}
with probability at least $1 - O(n^{-10})$. 
Combine this with \eqref{eq:UU-tildeU-tildeU-dist} and invoke the triangle inequality to yield
\begin{align*}
	\left\|\bm{U}\bm{U}^\top - \bm{U}^\star\bm{U}^{\star\top}\right\| 
	& \leq \big\|\widetilde{\bm{U}}\widetilde{\bm{U}}^\top - \bm{U}^\star\bm{U}^{\star\top}\big\| + 
	\big\|\bm{U}\bm{U}^\top - \widetilde{\bm{U}}\widetilde{\bm{U}}^\top\big\| \\
	&\lesssim \frac{\sqrt{\left(\mu r\omega_{\sf max}^2 + \omega_{\sf col}^2\right)\log n}}{\sigma_r^{\star}} + \frac{\left(\mu r\omega_{\sf max}^2 + \omega_{\sf col}\left(\omega_{\sf row} + \omega_{\sf col}\right)\right)\log n}{\sigma_r^{\star2}} + e^{-t_{k_{\sf max}}}\\
	&\asymp \frac{\sqrt{\left(\mu r\omega_{\sf max}^2 + \omega_{\sf col}^2\right)\log n}}{\sigma_r^{\star}} + \left(\frac{\sqrt{\left(\mu r\omega_{\sf max}^2 + \omega_{\sf col}^2\right)\log n}}{\sigma_r^{\star}}\right)^2 + \frac{\omega_{\sf col}\omega_{\sf row}\log n}{\sigma_r^{\star2}} + e^{-t_{k_{\sf max}}}\\
	&\asymp \frac{\sqrt{\left(\mu r\omega_{\sf max}^2 + \omega_{\sf col}^2\right)\log n}}{\sigma_r^{\star}} + \frac{\omega_{\sf col}\omega_{\sf row}\log n}{\sigma_r^{\star2}} + e^{-t_{k_{\sf max}}}
\end{align*}
under our assumption on $\sigma_r^{\star}$. 
Finally, using the basic inequality $\|\bm{U}\bm{R}_{\bm{U}} - \bm{U}^\star\| \leq \sqrt{2}\|\bm{U}\bm{U}^\top - \bm{U}^\star\bm{U}^{\star\top}\|$ \cite[Lemma 2.5]{chen2021spectral}
yields the desired result in Theorem~\ref{thm:spectral-general}. 

To finish up, it suffices to justify  the intermediate result in Theorem~\ref{thm:deterministic}, 
which we shall accomplish next.


\subsection{Proof of Theorem~\ref{thm:deterministic}}\label{proof:thm_deterministic}
We now present our proof of Theorem~\ref{thm:deterministic}. 
Recall the definitions of $\bm{G}_k^t$ and $\bm{U}_k^t$ in \eqref{def:inter_initialization}-\eqref{def:inter_G}. For any $k \geq 1$ and $0 \leq t \leq t_k$, 
we introduce the following convenient notation: 
\begin{equation}
\overline{\bm{M}} = \overline{\bm{U}}\,\overline{\bm \Lambda}\,\overline{\bm U}^\top, \quad D_k^t = \left\|\mathcal{P}_{\sf diag}\left(\bm{G}_k^t - \overline{\bm{M}}\right)\right\|, \quad L_k^t = \left\|\bm{G}_k^t - \overline{\bm{M}}\right\|, \quad {and} \quad \overline{\bm{U}}_k = \overline{\bm{U}}_{:,1:r_k}.
	\label{eq:def-Dkt-Lkt-Uk}
\end{equation}

\paragraph{Step 1: a basic property about $r_1$ as selected in Algorithm~\ref{algorithm:sequential_heteroPCA}.}
For $k = 1$, we first show that the rank $r_1$ selected in Algorithm~\ref{algorithm:sequential_heteroPCA} lies within
\begin{align}\label{condition:r_1}
	r_1 ~\in~ \mathcal{R}_1 := \left\{r' \leq r: \frac{\sigma_1\left(\bm{G}_0\right)}{\sigma_{r'}\left(\bm{G}_0\right)} \leq 4 \quad \text{and}\quad \sigma_{r'}\left(\bm{G}_{0}\right) - \sigma_{r' + 1}\left(\bm{G}_{0}\right) \geq \frac{1}{r}\sigma_{r'}\left(\bm{G}_{0}\right)\right\}.
\end{align}
To do so, it suffices to verify that $\mathcal{R}_1$ is non-empty, towards which we divide into two scenarios.
\begin{itemize}
	\item \emph{Case 1: $\{i \in [r-1]: \sigma_i(\bm{G}_0) \geq \frac{r}{r-1}\sigma_{i+1}(\bm{G}_0)\}$ is non-empty.} Take $1\leq \widetilde{r} \leq r-1$ to be the smallest entry in this set. Then it is seen that
	\begin{align}\label{ineq:signal_ratio1}
		\frac{\sigma_1\left(\bm{G}_0\right)}{\sigma_{\widetilde{r}}\left(\bm{G}_0\right)} = \prod_{j=1}^{\widetilde{r}-1}\frac{\sigma_j\left(\bm{G}_0\right)}{\sigma_{j+1}\left(\bm{G}_0\right)} \leq \left(\frac{r}{r-1}\right)^{r-2} \leq 4,
	\end{align}
	thus implying that $\widetilde{r} \in \mathcal{R}_1$.
	\item \emph{Case 2: $\{i \in [r-1]: \sigma_i(\bm{G}_0) \geq \frac{r}{r-1}\sigma_{i+1}(\bm{G}_0)\}$ is empty.} In this case, one necessarily has
	\begin{align*}
		\frac{\sigma_1\left(\bm{G}_0\right)}{\sigma_{r}\left(\bm{G}_0\right)} \leq \left(\frac{r}{r-1}\right)^{r-1} < e < 4.
	\end{align*}
		By virtue of  the definition $\bm{G}_1^0 = \bm{G}_0 = \mathcal{P}_{\sf off\text{-}diag}(\overline{\bm{U}}\,\overline{\bm \Lambda}\,\overline{\bm U}^\top + \bm{Z})$ (see \eqref{def:inter_initialization}), one can derive
	\begin{align}\label{ineq:L_1_0}
		L_1^0 &= \big\|\mathcal{P}_{\sf diag}\big(\overline{\bm{U}}\,\overline{\bm \Lambda}\,\overline{\bm U}^\top\big) - \mathcal{P}_{\sf off\text{-}diag}\left(\bm{Z}\right)\big\|\notag\\
		&\leq \big\|\mathcal{P}_{\sf diag}\big(\overline{\bm{U}}\,\overline{\bm \Lambda}\,\overline{\bm U}^\top\big)\big\| + \left\|\mathcal{P}_{\sf off\text{-}diag}\left(\bm{Z}\right)\right\|\notag\\
		&\leq \left\|\overline{\bm{U}}\right\|_{2,\infty}^2\left\|\overline{\bm\Lambda}\right\| + \left\|\mathcal{P}_{\sf off\text{-}diag}\left(\bm{Z}\right)\right\| \leq \frac{\overline{\mu} r}{n_1}\overline{\lambda}_1 + \left\|\mathcal{P}_{\sf off\text{-}diag}\left(\bm{Z}\right)\right\|.
	\end{align}
	Weyl's inequality then reveals that, for all $i \in [n_1]$,
	\begin{align}\label{ineq:Weyl_1}
		\left|\overline{\lambda}_i - \sigma_i\left(\bm{G}_0\right)\right| \leq L_1^0 \leq \frac{\overline{\mu} r}{n_1}\overline{\lambda}_1 + \left\|\mathcal{P}_{\sf off\text{-}diag}\left(\bm{Z}\right)\right\|,
	\end{align}
	which together with the assumptions~\eqref{eq:spectral-det-mu-UB} and \eqref{eq:lambda-Poffdiag-thm-det} immediately tells us that
	\begin{align*}
		\sigma_1\left(\bm{G}_0\right) \geq \left(1 - \frac{\overline{\mu} r}{n_1}\right)\overline{\lambda}_1 - \left\|\mathcal{P}_{\sf off\text{-}diag}\left(\bm{Z}\right)\right\| \geq \left(1 - \frac{\overline{\mu} r}{n_1}\right)\overline{\lambda}_1 - \frac{\overline{\lambda}_1}{C_0r} \geq \frac{1}{2}\overline{\lambda}_1.
	\end{align*}
	Combining \eqref{ineq:L_1_0} and \eqref{ineq:Weyl_1} with the assumptions~\eqref{eq:spectral-det-mu-UB} and \eqref{eq:lambda-Poffdiag-thm-det} also leads to
	\begin{align}\label{ineq:gap_1}
		\sigma_{r}\left(\bm{G}_0\right) - \sigma_{r+1}\left(\bm{G}_0\right) 
		&\geq \sigma_{r}\left(\bm{G}_0\right) - L_1^0\notag\\
		&\geq \sigma_{r}\left(\bm{G}_0\right) - \left(\frac{\overline{\mu} r}{n_1}\overline{\lambda}_1 + \left\|\mathcal{P}_{\sf off\text{-}diag}\left(\bm{Z}\right)\right\|\right)\notag\\
		&\geq \sigma_{r}\left(\bm{G}_0\right) -  \left(\frac{\overline{\mu} r}{n_1}\overline{\lambda}_1 + \frac{\overline{\lambda}_1}{C_0r}\right)\notag\\
		&\geq \sigma_{r}\left(\bm{G}_0\right) - \frac{1}{8}\sigma_1\left(\bm{G}_0\right)\notag\\
		&\geq \frac{1}{2}\sigma_{r}\left(\bm{G}_0\right) \geq \frac{1}{r}\sigma_{r}\left(\bm{G}_0\right).
	\end{align}
\end{itemize}
Putting \eqref{ineq:signal_ratio1} and \eqref{ineq:gap_1} for the above two cases together confirms that  $\mathcal{R}_1\neq \emptyset$, and hence \eqref{condition:r_1} is always true.

\paragraph{Step 2: bounding $L_1^t = \|\bm{G}_1^t - \overline{\bm{M}}\|$.} 
Next, we look at the difference between the iterate $\bm{G}_1^t$ (in the first round) and the low-rank matrix $\overline{\bm{M}}$. 
We will prove by induction the two properties below: for all $t \geq 0$,
\begin{subequations}
	\begin{align}
		\overline\lambda_{r_1} &\geq 18rL_1^t , \label{ineq:induction1}\\
		L_1^t - 6\sqrt{\frac{\overline{\mu} r}{n_1}}\overline{\lambda}_{r_1+1} - 4\left\|\mathcal{P}_{\sf off\text{-}diag}\left(\bm{Z}\right)\right\| &\leq \frac{1}{e^t}\left(L_1^0 - 6\sqrt{\frac{\overline{\mu} r}{n_1}}\overline{\lambda}_{r_1+1} - 4\left\|\mathcal{P}_{\sf off\text{-}diag}\left(\bm{Z}\right)\right\|\right).\label{ineq:induction2}
	\end{align}
\end{subequations}

\paragraph{Step 2.1: the base case for \eqref{ineq:induction1} and \eqref{ineq:induction2}.}
Let us start with the base case with  $t = 0$. 
Noting that \eqref{ineq:L_1_0} and \eqref{ineq:Weyl_1} hold and recalling that $\sigma_1(\bm{G}_0)/\sigma_{r_1}(\bm{G}_0) \leq 4$ and $\overline{\lambda}_{1} \geq \overline{\lambda}_{r_1} \geq \overline{\lambda}_{r} \gg r\|\mathcal{P}_{\sf off\text{-}diag}(\bm{Z})\|$, we have
\begin{align*}
	L_1^0 &\leq \frac{\overline{\mu} r}{n_1}\overline{\lambda}_1 + \left\|\mathcal{P}_{\sf off\text{-}diag}\left(\bm{Z}\right)\right\|
	\leq \frac{\overline{\mu} r}{n_1}\left(\sigma_{1}\left(\bm{G}_0\right) + L_1^0\right) + \left\|\mathcal{P}_{\sf off\text{-}diag}\left(\bm{Z}\right)\right\|\\
	&\leq \frac{\overline{\mu} r}{n_1}\left(4\sigma_{r_1}\left(\bm{G}_0\right) + L_1^0\right) + \left\|\mathcal{P}_{\sf off\text{-}diag}\left(\bm{Z}\right)\right\|\\
	&\leq \frac{\overline{\mu} r}{n_1}\left[4\overline{\lambda}_{r_1} + 5L_1^0\right] + \left\|\mathcal{P}_{\sf off\text{-}diag}\left(\bm{Z}\right)\right\|\\
	&\leq \frac{\overline{\lambda}_{r_1}}{72r} + \frac{1}{2}L_1^0 + \frac{\overline{\lambda}_{r_1}}{72r} = \frac{\overline{\lambda}_{r_1}}{36r} + \frac{1}{2}L_1^0,
\end{align*}
where the last line also makes use of  the assumptions~\eqref{eq:spectral-det-mu-UB} and \eqref{eq:lambda-Poffdiag-thm-det}. 
This further tells us that
\begin{align*}
	\overline{\lambda}_{r_1} \geq 18rL_1^0,
\end{align*}
as claimed in \eqref{ineq:induction1} when $t=0$. 
Combining Weyl's inequality, \eqref{ineq:Weyl_1}, and the previous inequality gives
\begin{align}\label{ineq:difference_true_eigengap}
	\overline{\lambda}_{r_1} - \overline{\lambda}_{r_1+1} &\geq \sigma_{r_1}\left(\bm{G}_0\right) - \sigma_{r_1+1}\left(\bm{G}_0\right) - \left|\sigma_{r_1}\left(\bm{G}_0\right) - \overline{\lambda}_{r_1}\right| - \left|\sigma_{r_1+1}\left(\bm{G}_0\right) - \overline{\lambda}_{r_1+1}\right|\notag\\
	&\geq \frac{1}{r}\sigma_{r_1}\left(\bm{G}_0\right) - 2L_1^0	\geq \frac{1}{r}\left(\overline{\lambda}_{r_1} - \left|\sigma_{r_1}\left(\bm{G}_0\right) - \overline{\lambda}_{r_1}\right|\right) - 2L_1^0\notag\\
	&\geq \frac{\overline{\lambda}_{r_1}}{r} - 3L_1^0 \geq \frac{3\overline{\lambda}_{r_1}}{4r} \vee 9L_1^0.
\end{align}
The inequality \eqref{ineq:induction2} for the base case with $t=0$ holds trivially.

\paragraph{Step 2.2: induction step for \eqref{ineq:induction1} and \eqref{ineq:induction2}.} 
Now, supposing that \eqref{ineq:induction1} and \eqref{ineq:induction2} hold for $t-1$, 
we would like to justify these two claims for $t$. 
In light of Algorithm~\ref{algorithm:heteroPCA}, we first observe that
\begin{align}\label{eq:off_diag}
	\left\|\mathcal{P}_{\sf off\text{-}diag}\left(\bm{G}_1^t - \overline{\bm{M}}\right)\right\| = \left\|\mathcal{P}_{\sf off\text{-}diag}\left(\bm{G}_0- \overline{\bm{M}}\right)\right\| = \left\|\mathcal{P}_{\sf off\text{-}diag}\left(\bm{Z}\right)\right\|
\end{align}
and
\begin{align}\label{ineq:diag}
	\left\|\mathcal{P}_{\sf diag}\left(\bm{G}_1^t - \overline{\bm{M}}\right)\right\| &= \left\|\mathcal{P}_{\sf diag}\left(\bm{P}_{\bm{U}_1^{t-1}}\bm{G}_1^{t-1} - \overline{\bm{M}}\right)\right\|\notag\\
	&\leq \underbrace{\left\|\mathcal{P}_{\sf diag}\left(\bm{P}_{\overline{\bm{U}}_1}\left(\bm{G}_1^{t-1} - \overline{\bm{M}}\right)\right)\right\|}_{=:\alpha_1} + \underbrace{\left\|\mathcal{P}_{\sf diag}\left(\bm{P}_{\left(\bm{U}_1^{t-1}\right)_{\perp}}\overline{\bm{M}}\right)\right\|}_{=:\alpha_2}\notag\\&\quad + \underbrace{\left\|\mathcal{P}_{\sf diag}\left(\left(\bm{P}_{\bm{U}_1^{t-1}} - \bm{P}_{\overline{\bm{U}}_1}\right)\left(\bm{G}_1^{t-1} - \overline{\bm{M}}\right)\right)\right\|}_{=:\alpha_3}.
\end{align}
\begin{itemize}
	\item 
In view of \citet[Lemma 1]{zhang2022heteroskedastic}, one can upper bound the first term $\alpha_1$ as
\begin{align}\label{ineq:alpha_1}
	\alpha_1 \leq \sqrt{\frac{\overline{\mu} r}{n_1}}\left\|\bm{G}_1^{t-1} - \overline{\bm{M}}\right\| = \sqrt{\frac{\overline{\mu} r}{n_1}}L_1^{t-1}.
\end{align}

\item Turning to $\alpha_2$, applying \citet[Lemma 1]{zhang2022heteroskedastic} again yields
\begin{align*}
	\alpha_2 &= \left\|\mathcal{P}_{\sf diag}\left(\bm{P}_{\left(\bm{U}_1^{t-1}\right)_{\perp}}\overline{\bm{M}}\right)\right\| = \left\|\mathcal{P}_{\sf diag}\left(\bm{P}_{\left(\bm{U}_1^{t-1}\right)_{\perp}}\overline{\bm{M}}\bm{P}_{\overline{\bm{U}}}\right)\right\| \leq \sqrt{\frac{\overline{\mu} r}{n_1}}\left\|\bm{P}_{\left(\bm{U}_1^{t-1}\right)_{\perp}}\overline{\bm{M}}\right\|\\
	&\leq \sqrt{\frac{\overline{\mu} r}{n_1}}\left(\left\|\bm{P}_{\left(\bm{U}_1^{t-1}\right)_{\perp}}\left(\bm{P}_{\overline{\bm{U}}_1}\overline{\bm{M}}\right)\right\| + \left\|\bm{P}_{\left(\overline{\bm{U}}_1\right)_{\perp}}\overline{\bm{M}}\right\|\right)\\
	&= \sqrt{\frac{\overline{\mu} r}{n_1}}\left(\left\|\bm{P}_{\left(\bm{U}_1^{t-1}\right)_{\perp}}\left(\bm{P}_{\overline{\bm{U}}_1}\overline{\bm{M}}\right)\right\| + \overline{\lambda}_{r_1+1}\right),
\end{align*}
where the second identity is valid since $\overline{\bm{M}}$ falls within the subspace $\overline{\bm{U}}$. 
Recognizing that $$\bm{G}_1^{t-1} = \bm{P}_{\overline{\bm{U}}_1}\overline{\bm{M}} + \left(\bm{G}_1^{t-1} - \bm{P}_{\overline{\bm{U}}_1}\overline{\bm{M}}\right)$$ and 
$$\left\|\bm{G}_1^{t-1} - \bm{P}_{\overline{\bm{U}}_1}\overline{\bm{M}}\right\| \leq \left\|\bm{G}_1^{t-1} - \overline{\bm{M}}\right\| + \left\|\bm{P}_{\left(\overline{\bm{U}}_1\right)_{\perp}}\overline{\bm{M}}\right\|,$$
one can invoke Lemma \ref{lm:projection} to show that
\begin{align*}
	\left\|\bm{P}_{\left(\bm{U}_1^{t-1}\right)_{\perp}}\left(\bm{P}_{\overline{\bm{U}}_1}\overline{\bm{M}}\right)\right\| \leq 2\left\|\bm{G}_1^{t-1} - \bm{P}_{\overline{\bm{U}}_1}\overline{\bm{M}}\right\| \leq 2\left(\left\|\bm{G}_1^{t-1} - \overline{\bm{M}}\right\| + \left\|\bm{P}_{\left(\overline{\bm{U}}_1\right)_{\perp}}\overline{\bm{M}}\right\|\right) = 2\left(L_1^{t-1} + \overline{\lambda}_{r_1+1}\right).
\end{align*}
Combining the previous two inequalities, we have
\begin{align}\label{ineq:alpha_2}
	\alpha_2 \leq \sqrt{\frac{\overline{\mu} r}{n_1}}\left(2L_1^{t-1} + 3\overline{\lambda}_{r_1+1}\right).
\end{align}

\item
Now, we move on to $\alpha_3$. Recall that $\bm{U}_1^{t-1}$ is the leading-$r$ eigen-subspace of $\bm{G}_1^{t-1}$. Combining \eqref{ineq:difference_true_eigengap}, the induction hypothesis $\overline\lambda_{r_1} \geq 12rL_1^{t-1}$, the $\sin\Theta$ Theorem (or more precisely, the perturbation bound (2.26a) in \cite{chen2021spectral}) and Weyl's inequality, one has
\begin{align*}
	\big\|\bm{P}_{\bm{U}_1^{t-1}} - \bm{P}_{\overline{\bm{U}}_1}\big\| &\leq \frac{2\left\|\bm{G}_1^{t-1} - \overline{\bm{M}}\right\|}{\overline{\lambda}_{r_1} - \overline{\lambda}_{r_1 + 1}} \leq \frac{2L_1^{t-1}}{3\overline{\lambda}_{r_1}/(4r)} \leq \frac{3rL_1^{t-1}}{\overline\lambda_{r_1}}.
\end{align*}
As a consequence, one can bound $\alpha_3$ as follows
\begin{align}\label{ineq:alpha_3}
	\alpha_3 \leq \big\|\bm{P}_{\bm{U}_1^{t-1}} - \bm{P}_{\overline{\bm{U}}_1}\big\|\left\|\bm{G}_1^{t-1} - \overline{\bm{M}}\right\| \leq \frac{3r\left(L_1^{t-1}\right)^2}{\overline\lambda_{r_1}}.
\end{align}
\end{itemize}
Putting \eqref{eq:off_diag}, \eqref{ineq:diag}, \eqref{ineq:alpha_1}, \eqref{ineq:alpha_2} and \eqref{ineq:alpha_3} together yields
\begin{align*}
	L_1^t &= \left\|\bm{G}_1^t - \overline{\bm{M}}\right\| 
	\leq   \left\|\mathcal{P}_{\sf diag}\left(\bm{Z}\right)\right\| +  \left\|\mathcal{P}_{\sf off\text{-}diag}\left(\bm{Z}\right)\right\|
	\leq \alpha_1+\alpha_2+\alpha_3 +  \left\|\mathcal{P}_{\sf off\text{-}diag}\left(\bm{Z}\right)\right\|\\
	&\leq 3\sqrt{\frac{\overline{\mu} r}{n_1}}L_1^{t-1} + 3\sqrt{\frac{\overline{\mu} r}{n_1}}\overline{\lambda}_{r_1+1} + \frac{3r\left(L_1^{t-1}\right)^2}{\overline\lambda_{r_1}} + \left\|\mathcal{P}_{\sf off\text{-}diag}\left(\bm{Z}\right)\right\|\\
	&\leq \frac{1}{2e}L_1^{t-1} + 3\sqrt{\frac{\overline{\mu} r}{n_1}}\overline{\lambda}_{r_1+1} + \frac{1}{2e}L_1^{t-1} + \left\|\mathcal{P}_{\sf off\text{-}diag}\left(\bm{Z}\right)\right\|\\
	&= \frac{1}{e}L_1^{t-1} + 3\sqrt{\frac{\overline{\mu} r}{n_1}}\overline{\lambda}_{r_1+1} + \left\|\mathcal{P}_{\sf off\text{-}diag}\left(\bm{Z}\right)\right\|, 
\end{align*}
where the third line holds due to the induction hypothesis \eqref{ineq:induction1} for $t-1$. 
This taken together with the induction hypothesis \eqref{ineq:induction1} for $t-1$ and the assumptions~\eqref{eq:spectral-det-mu-UB} and \eqref{eq:lambda-Poffdiag-thm-det}  implies that
\begin{align*}
	L_1^t \leq \frac{1}{e}\cdot \frac{\overline{\lambda}_{r_1}}{18r} + \frac{\overline{\lambda}_{r_1}}{72r} + \frac{\overline{\lambda}_{r_1}}{72r} \leq \frac{\overline{\lambda}_{r_1}}{18r} 
\end{align*}
and 
\begin{align*}
	L_1^t - 6\sqrt{\frac{\overline{\mu} r}{n_1}}\overline{\lambda}_{r_1+1} - 4\left\|\mathcal{P}_{\sf off\text{-}diag}\left(\bm{Z}\right)\right\| &\leq \frac{1}{e}\left(L_1^{t-1} - 6\sqrt{\frac{\overline{\mu} r}{n_1}}\overline{\lambda}_{r_1+1} - 4\left\|\mathcal{P}_{\sf off\text{-}diag}\left(\bm{Z}\right)\right\| \right)\\
	&\leq \frac{1}{e^t}\left(L_1^{0} - 6\sqrt{\frac{\overline{\mu} r}{n_1}}\overline{\lambda}_{r_1+1} - 4\left\|\mathcal{P}_{\sf off\text{-}diag}\left(\bm{Z}\right)\right\| \right).
\end{align*}
This directly concludes the proof of \eqref{ineq:induction1} and \eqref{ineq:induction2} via standard induction arguments.

\paragraph{Step 3: bounding $L_k^t = \|\bm{G}_k^t - \overline{\bm{M}}\|$ for $k > 1$.} 
Having looked at what happens in the first round, 
we now proceed to develop upper bounds for $\|\bm{G}_k^t - \overline{\bm{M}}\|$ when $k>1$. 
In view of the inequality \eqref{ineq:induction2}, 
choosing the number of iterations such that $t_1 \geq \log\Big(\frac{\frac{\overline{\mu} r}{n_1}\overline\lambda_1}{\sqrt{\frac{\overline{\mu} r}{n_1}}\overline\lambda_{r_1 + 1} + \|\mathcal{P}_{\sf off\text{-}diag}(\bm{Z})\|}\Big) \vee 0$ gives
\begin{align}\label{ineq72}
L_{2}^{0}=L_{1}^{t_{1}} & \leq6\sqrt{\frac{\overline{\mu}r}{n_{1}}}\overline{\lambda}_{r_{1}+1}+4\left\Vert \mathcal{P}_{\sf off\text{-}diag}\left(\bm{Z}\right)\right\Vert \notag\\
 & \quad+\frac{1}{e^{t_{1}}}\left(\frac{\overline{\mu}r}{n_{1}}\overline{\lambda}_{1}+\left\Vert \mathcal{P}_{\sf off\text{-}diag}\left(\bm{Z}\right)\right\Vert -6\sqrt{\frac{\overline{\mu}r}{n_{1}}}\overline{\lambda}_{r_{1}+1}-4\left\Vert \mathcal{P}_{\sf off\text{-}diag}\left(\bm{Z}\right)\right\Vert \right)\notag\\
 & \leq6\sqrt{\frac{\overline{\mu}r}{n_{1}}}\overline{\lambda}_{r_{1}+1}+4\left\Vert \mathcal{P}_{\sf off\text{-}diag}\left(\bm{Z}\right)\right\Vert +\frac{1}{e^{t_{1}}}\cdot\frac{\overline{\mu}r}{n_{1}}\overline{\lambda}_{1}\notag\\
 & \leq6\sqrt{\frac{\overline{\mu}r}{n_{1}}}\overline{\lambda}_{r_{1}+1}+4\left\Vert \mathcal{P}_{\sf off\text{-}diag}\left(\bm{Z}\right)\right\Vert +\frac{\sqrt{\frac{\overline{\mu}r}{n_{1}}}\overline{\lambda}_{r_{1}+1}+\left\Vert \mathcal{P}_{\sf off\text{-}diag}\left(\bm{Z}\right)\right\Vert }{\frac{\overline{\mu}r}{n_{1}}\overline{\lambda}_{1}}\frac{\overline{\mu}r}{n_{1}}\overline{\lambda}_{1}\notag\\
 & \leq7\sqrt{\frac{\overline{\mu}r}{n_{1}}}\overline{\lambda}_{r_{1}+1}+5\left\Vert \mathcal{P}_{\sf off\text{-}diag}\left(\bm{Z}\right)\right\Vert ,	
\end{align}
where the first inequality results from \eqref{ineq:induction2} and \eqref{ineq:Weyl_1}.

Similarly, setting the numbers of iterations as $$t_k \geq \log\left(\frac{7\sqrt{\frac{\overline{\mu}r}{n_1}}\overline\lambda_{r_{k-1}+1} + \left\|\mathcal{P}_{\sf off\text{-}diag}\left(\bm{Z}\right)\right\|}{\sqrt{\frac{\overline{\mu}r}{n_1}}\overline\lambda_{r_{k}+1} + \left\|\mathcal{P}_{\sf off\text{-}diag}\left(\bm{Z}\right)\right\|}\right), \quad 2 \leq k \leq k_{\sf max} - 1$$ and repeating similar arguments as in \eqref{condition:r_1}, \eqref{ineq:induction1}, \eqref{ineq:induction2} and \eqref{ineq72} yield that: for all $2 \leq k \leq k_{\sf max}, t \geq 0$, 
\begin{subequations}
	\begin{align}
		r_k \in \mathcal{R}_k := \bigg\{r': \frac{\sigma_{r_{k-1}+1}\left(\bm{G}_{k-1}\right)}{\sigma_{r'}\left(\bm{G}_{k-1}\right)} &\leq 4 \text{ and } \sigma_{r'}\left(\bm{G}_{k-1}\right) \geq \frac{r}{r-1}\sigma_{r'+1}\left(\bm{G}_{k-1}\right)\bigg\}, \label{ineq:induction_rank}\\
		L_k^0 = L_{k-1}^{t_{k-1}} &\leq 7\sqrt{\frac{\overline{\mu} r}{n_1}}\overline{\lambda}_{r_{k-1}+1} + 5\left\|\mathcal{P}_{\sf off\text{-}diag}\left(\bm{Z}\right)\right\|,\label{ineq:induction3}\\
		L_k^t &\leq \frac{\overline\lambda_{r_k}}{18r},\label{ineq:induction4}\\
		L_k^t - 6\sqrt{\frac{\overline{\mu} r}{n_1}}\overline{\lambda}_{r_k+1} - 4\left\|\mathcal{P}_{\sf off\text{-}diag}\left(\bm{Z}\right)\right\| &\leq \frac{1}{e^t}\left(L_k^{0} - 6\sqrt{\frac{\overline{\mu} r}{n_1}}\overline{\lambda}_{r_k+1} - 4\left\|\mathcal{P}_{\sf off\text{-}diag}\left(\bm{Z}\right)\right\| \right).\label{ineq:induction5}
	\end{align}
\end{subequations}

\paragraph{Step 4: bounding $\|\bm{U}\bm{U}^\top - \overline{\bm{U}}\,\overline{\bm{U}}^\top\|$.} 
To finish up, we still need to bound the discrepancy between $\bm{U}$ and $\overline{\bm{U}}$.
Recalling that $k_{\sf max}$ satisfies $r_{k_{\sf max}} = r$, we can invoke \eqref{ineq:induction5} and \eqref{ineq:induction3} to obtain
\begin{align*}
	L_{k_{\sf max}}^{t_{k_{\sf max}}} &\leq 4\left\|\mathcal{P}_{\sf off\text{-}diag}\left(\bm{Z}\right)\right\| + e^{-t_{k_{\sf max}}}\left(L_{k_{\sf max}}^{0} - 6\sqrt{\frac{\overline{\mu} r}{n_1}}\overline{\lambda}_{r_{k_{\sf max}}+1} - 4\left\|\mathcal{P}_{\sf off\text{-}diag}\left(\bm{Z}\right)\right\| \right)\\
	&\leq 4\left\|\mathcal{P}_{\sf off\text{-}diag}\left(\bm{Z}\right)\right\| + e^{-t_{k_{\sf max}}}\left(7\sqrt{\frac{\overline{\mu} r}{n_1}}\overline{\lambda}_{r_{k_{\sf max}-1}+1} + \left\|\mathcal{P}_{\sf off\text{-}diag}\left(\bm{Z}\right)\right\|\right)\\
	&\leq 5\left\|\mathcal{P}_{\sf off\text{-}diag}\left(\bm{Z}\right)\right\| + 7e^{-t_{k_{\sf max}}}\overline{\lambda}_{r_{k_{\sf max}-1}+1}.
\end{align*}
The sin$\Theta$ Theorem \cite[cf.][(2.26a)]{chen2021spectral} then leads to
\begin{align}
	\big\|\bm{U}\bm{U}^\top - \overline{\bm{U}}\,\overline{\bm{U}}^\top\big\| & = \left\|\bm{U}_{k_{\sf max}}^{t_{k_{\sf max}}}\bm{U}_{k_{\sf max}}^{t_{k_{\sf max}}\top} - \overline{\bm{U}}\,\overline{\bm{U}}^\top\right\| 
	\leq \frac{2 \|\bm{G}_{k_{\sf max}}^{t_{k_{\sf max}}} - \overline{\bm{M}}\| }{\overline{\lambda}_r} \notag\\ 
	&= \frac{2L_{k_{\sf max}}^{t_{k_{\sf max}}}}{\overline{\lambda}_r}  \lesssim \frac{\left\|\mathcal{P}_{\sf off\text{-}diag}\left(\bm{Z}\right)\right\|}{\overline{\lambda}_r} + e^{-t_{k_{\sf max}}}\frac{\overline{\lambda}_{r_{k_{\sf max}-1}+1}}{\overline{\lambda}_r}\label{ineq63}.
\end{align}
In addition, the definition of $k_{\sf max}$ and \eqref{ineq:induction_rank} together show that
\begin{align}
	\frac{ \sigma_{r_{k_{\sf max} - 1}+1}\left(\bm{G}_{k_{\sf max} - 1}\right) }{ \sigma_{r}\left(\bm{G}_{k_{\sf max} - 1}\right) } \leq 4.
\end{align}
In view of \eqref{ineq:induction3} and Weyl's inequality, one has
\begin{align*}
	\max_i\left|\sigma_i\left(\bm{G}_{k_{\sf max}-1}\right) - \overline{\lambda}_i\right| \leq L_{k_{\sf max}}^0 = \left\|\bm{G}_{k_{\sf max} - 1} - \overline{\bm{M}}\right\| \leq 7\sqrt{\frac{\overline{\mu} r}{n_1}}\overline{\lambda}_{r_{k_{\sf max}-1}+1} + 5\left\|\mathcal{P}_{\sf off\text{-}diag}\left(\bm{Z}\right)\right\| \leq \frac{1}{10}\overline{\lambda}_{r_{k_{\sf max}-1}+1},
\end{align*}
where the last inequality results from \eqref{eq:spectral-det-mu-UB} and \eqref{eq:lambda-Poffdiag-thm-det}. 
Combine the preceding two bounds to reach
\begin{align}\label{ineq62}
	\overline{\lambda}_{r_{k_{\sf max}-1}+1} \asymp \overline{\lambda}_r.
\end{align}
Putting \eqref{ineq63} together with \eqref{ineq62} finishes the proof of Theorem \ref{thm:deterministic}.

    \section{Proof of Theorem \ref{thm:two_to_infty} ($\ell_{2,\infty}$ analysis for {\sf Deflated-HeteroPCA})}\label{proof:thm_two_to_infty}

In this section, we present the proof of Theorem~\ref{thm:two_to_infty} that concerns $\ell_{2,\infty}$ statistical guarantees. 
For convenience, we shall continue to use the notation defined in \eqref{def:inter_initialization}-\eqref{def:inter_G}, and  again denote 
the SVD of $\bm{U}^\star\bm{\Sigma}^\star + \bm{E}\bm{V}^\star$ by
\begin{subequations}
\begin{equation}
	\widetilde{\bm{U}}\widetilde{\bm{\Sigma}}\widetilde{\bm{W}}^\top = \bm{U}^\star\bm{\Sigma}^\star + \bm{E}\bm{V}^\star,
	\label{eq:SVD-USigmaW}
\end{equation}
where $\widetilde{\bm{U}} \in \mathcal{O}^{n_1, r}$, $\widetilde{\bm{\Sigma}} = {\sf diag}(\widetilde{\sigma}_1, \dots, \widetilde{\sigma}_r)$, and $\widetilde{\bm{W}} \in \mathcal{O}^{r, r}$. 
We can then define 
\begin{equation}
	\widetilde{\bm{M}} = \widetilde{\bm{U}}\widetilde{\bm{\Sigma}}^2\widetilde{\bm{U}}^\top 
	= \big( \bm{U}^\star\bm{\Sigma}^\star + \bm{E}\bm{V}^\star \big) \big( \bm{U}^\star\bm{\Sigma}^\star + \bm{E}\bm{V}^\star \big) ^{\top}.
\end{equation}
In addition, we introduce
\begin{align}\label{eq:M_oracle-n32}
	\bm{M}^{\sf oracle}
	 = \widetilde{\bm{U}}\widetilde{\bm{\Sigma}}^2\widetilde{\bm{U}}^\top + 
	 \underset{=: \bm{Z}}{\underbrace{ \mathcal{P}_{\sf off\text{-}diag}\left(\bm{E}\bm{E}^\top - \bm{E}\bm{V}^\star\bm{V}^{\star\top}\bm{E}^\top\right) }}
	,
\end{align}
and let $\bm{U}^{\sf oracle} \in \mathcal{O}^{n_1,r}$ represent the rank-$r$ leading eigen-subspace of $\bm{M}^{\sf oracle}$. 
It is easily seen that 
\begin{align}\label{ineq85}
	\mathcal{P}_{\sf off\text{-}diag}\left(\bm{M}^{\sf oracle}\right) = \mathcal{P}_{\sf off\text{-}diag}\left(\bm{Y}\bm{Y}^\top\right) = \mathcal{P}_{\sf off\text{-}diag}\left(\bm{G}_0\right)
	\qquad \text{and} \qquad \mathcal{P}_{\sf diag}\left(\bm{M}^{\sf oracle}\right) = \mathcal{P}_{\sf diag}\big(\widetilde{\bm{M}}\big).
\end{align}
\end{subequations}
Throughout this proof, we denote by $\bm{U}_k^{\sf oracle}\in \mathbb{R}^{n_1\times r_k}$ the top-$r_k$ eigenspace of $\bm{M}^{\sf oracle}$.

\subsection{Several key results: eigenspace/eigenvalue perturbation and tail bounds}
Before embarking on the proof of Theorem~\ref{thm:two_to_infty}, 
we single out a couple of key results that play a crucial role in the proof. 
Let us begin by making note of a lemma that connects the eigenspace perturbation with a collection of polynomials of the perturbation matrix, 
originally developed by \citet{xia2021normal}.

\begin{lemma}[\cite{xia2021normal}, Theorem 1]\label{lm:representation}
	Suppose that $\bm{M} = \overline{\bm{M}} + \bm{Z} \in \bbR^{n \times n}$, where $\overline{\bm{M}}$ and $\bm{Z}$ are both symmetric matrices.
	Assume that $\overline{\bm{M}}$ is rank-$r$ with eigenvalues $\overline\lambda_1 \geq \dots \geq \overline\lambda_{r} > 0$,
	and $\overline{\bm{U}} =[\overline{\bm{u}}_1,\ \dots,\ \overline{\bm{u}}_r]$ (resp.~$\bm{U}$) represents the rank-$r$ leading eigen-subspace of $\overline{\bm{M}}$ (resp.~$\bm{M}$). If $\overline{\lambda}_r > 2\|\bm{Z}\|$, then
	\begin{align}
		\overline{\bm{U}}\,\overline{\bm{U}}^\top - \bm{U}\bm{U}^\top 
		= \sum_{k \geq 1}\sum_{ \bm{j}=[j_1,\cdots,j_{k+1}] \geq \bm{0} \atop j_1 + \cdots + j_{k+1} = k}(-1)^{\tau(\bm{j})+1}\mathfrak{P}^{-j_1}\bm{Z}\mathfrak{P}^{-j_2}\bm{Z}\cdots\bm{Z}\mathfrak{P}^{-j_{k+1}}.
	\end{align}
	Here, we define, for any $k \geq 1$, 
	\begin{subequations}
	\begin{align}
		\tau(\bm{j}) &:= \sum_{i=1}^{k+1}\mathbbm{1}\{j_i > 0\} , \\
		\overline{\bm{\Lambda}} &:= {\sf diag}\left(\bar{\lambda}_1, \dots, \bar{\lambda}_r\right) , \\
		\mathfrak{P}^0 &:= \overline{\bm{U}}_{\perp}\overline{\bm{U}}_{\perp}^\top 
		= \bm{I}- \overline{\bm{U}}\overline{\bm{U}}^\top , \\
		\mathfrak{P}^{-k} &:= \overline{\bm{U}}\,\overline{\bm{\Lambda}}^{-k}\overline{\bm{U}}^\top .
	\end{align}
	\end{subequations}
	As a consequence, we have
	\begin{align}\label{ineq:representation}
		\big\|\overline{\bm{U}}\,\overline{\bm{U}}^\top - \bm{U}\bm{U}^\top\big\|_{2,\infty} \leq 
		\sum_{k \geq 1}\sum_{ \bm{j}=[j_1,\cdots,j_{k+1}] \geq \bm{0} \atop j_1 + \cdots + j_{k+1} = k}\left\|\mathfrak{P}^{-j_1}\bm{Z}\mathfrak{P}^{-j_2}\bm{Z}\cdots\bm{Z}\mathfrak{P}^{-j_{k+1}}\right\|_{2,\infty}.
	\end{align}
\end{lemma}

Moreover, given that we are considering multiple eigen-subspaces (e.g., $\bm{U}^{\sf oracle}$, $\widetilde{\bm{U}}$, $\bm{U}^\star$), 
we isolate the following result that unveils the proximity of $\bm{U}^{\sf oracle}$ and $\bm{U}^\star$ (or $\widetilde{\bm{U}}$). 
The proof of this result is deferred to Section~\ref{proof:thm_oracle}. 
\begin{thm}\label{thm:oracle}
	Suppose that Assumption \ref{assump:hetero2} holds and 
	\begin{subequations}
	\begin{align}
		\frac{ \sigma_r^\star }{ \omega_{\sf max} } &\geq C_0r \big[(n_1n_2)^{1/4} + n_1^{1/2} \big]\log n \label{ineq75a}\\
		\mu &\leq c_0\frac{n_1}{r^3} \label{ineq75b}
	\end{align}
    \end{subequations}
	for some large (resp.~small) numerical constant $C_0 > 0$ (resp.~$c_0 > 0$). Then with probability exceeding $1 - O(n^{-10})$, one has
	\begin{subequations}
		\begin{align}
			\big\|\widetilde{\bm{U}}\widetilde{\bm{U}}^{\top} - \bm{U}^\star\bm{U}^{\star\top}\big\|_{2,\infty} 
			&\lesssim \sqrt{\frac{\mu r}{n_1}}\frac{\sqrt{n_1}\,\omega_{\sf max}\log n}{\sigma_{r}^\star},\label{ineq:two_to_infty_tilde}\\
			\left\|\bm{U}^{\sf oracle}\bm{U}^{\sf oracle\top} - \bm{U}^\star\bm{U}^{\star\top}\right\|_{2,\infty} &\lesssim \sqrt{\frac{\mu r}{n_1}}\left(\frac{\sqrt{n_1n_2}\,\omega_{\sf max}^2\log^2 n}{\sigma_r^{\star2}} + \frac{\sqrt{n_1}\,\omega_{\sf max}\log n}{\sigma_{r}^\star}\right)\label{ineq:two_to_infty_oracle},\\
			\left\|\bm{U}^{\sf oracle}\bm{U}^{\sf oracle\top} - \bm{U}^\star\bm{U}^{\star\top}\right\| &\lesssim \frac{\sqrt{n_1n_2}\,\omega_{\sf max}^2\log^2 n}{\sigma_r^{\star2}} + \frac{\sqrt{n_1}\,\omega_{\sf max}\log n}{\sigma_{r}^\star}.\label{ineq:spectral_oracle}
		\end{align}
	\end{subequations}
\end{thm}

The next two lemmas develop high-probability tail bounds on the $\ell_{2,\infty}$ norm of certain polynomials of noise matrix (with proper diagonal deletion), 
which are critical when invoking, say, the decomposition in Lemma~\ref{lm:representation}.  
The proofs of these two lemmas are postponed to Sections~\ref{proof_lm:error_moment1} and \ref{proof:lm_error_moment2}, respectively.

\begin{lemma}\label{lm:error_moment1}
	Suppose that Assumption \ref{assump:hetero2} holds. Then with probability exceeding $1 - O(n^{-10})$, one has
	\begin{align}\label{ineq:error_moment}
		\left\|\left[\mathcal{P}_{\sf off\text{-}diag}\left(\bm{E}\bm{E}^\top\right)\right]^k\bm{E}\bm{V}^\star\right\|_{2,\infty} \leq C_3\sqrt{\mu r}\left(C_3\left(\sqrt{n_1n_2} + n_1\right)\omega_{\sf max}^2\log^2 n\right)^{k}\omega_{\sf max}\log n 
	\end{align}
	for all $1 \leq k \leq \log n$. Here, $C_3>0$ is some large enough numerical constant. 
\end{lemma}
\begin{lemma}\label{lm:error_moment2}
	Suppose that Assumption \ref{assump:hetero2} holds. Then with probability exceeding $1 - O(n^{-10})$, one has
	\begin{align}\label{ineq:error_moment2}
		\left\|\left[\mathcal{P}_{\sf off\text{-}diag}\left(\bm{E}\bm{E}^\top\right)\right]^k\bm{U}^\star\right\|_{2,\infty} \leq C_3\sqrt{\frac{\mu r}{n_1}}\left(C_3\left(\sqrt{n_1n_2} + n_1\right)\omega_{\sf max}^2\log^2 n\right)^{k} 
	\end{align}
	for all $1 \leq k \leq \log n$. Here, $C_3>0$ is some large enough numerical constant.
\end{lemma}

Finally, 
recall that the eigenspace perturbation theory depends heavily on both the spectral gap and the size of the perturbation matrix, 
which we shall study in the following lemma. In addition to these two properties, this lemma also provides an upper bound concerning the incoherence of $\widetilde{\bm{U}}$. 

\begin{lemma}\label{lm:collection_event}
	Instate the assumptions in Theorem \ref{thm:oracle}. 
	Let us overload the notation here by setting $\sigma_{r+1}^\star = \widetilde{\sigma}_{r+1} = 0$,  and define
	\begin{align}
		\mathcal{R}' = \left\{r': 1 \leq r' \leq r, \left(1 - \frac{1}{2r}\right)\sigma_{r'}^{\star2} \geq \sigma_{r'+1}^{\star2}\right\}.
		\label{eq:defn-Rprime-collection-event} 
	\end{align}
	Then with probability exceeding $1 - O(n^{-10})$, we have
	\begin{subequations}
	\begin{align}
			\left|\widetilde{\sigma}_{i} - \sigma_{i}^\star\right| &\leq \left\|\bm{E}\bm{V}^{\star}\right\| \leq \sqrt{C_5}\sqrt{n_1}\omega_{\sf max}\log n \label{ineq26}\\
			\widetilde{\sigma}_{r'}^2 - \widetilde{\sigma}_{r'+1}^2 &\geq \frac{1}{2}\left(\sigma_{r'}^{\star2}-\sigma_{r'+1}^{\star2}\right), \quad \forall r' \in \mathcal{R}' 
			\label{ineq78}\\
			\left\|\mathcal{P}_{\sf off\text{-}diag}\left(\bm{E}\bm{E}^\top - \bm{E}\bm{V}^\star\bm{V}^{\star\top}\bm{E}^\top\right)\right\| &\leq 3C_5\left(\sqrt{n_1n_2} + n_1\right)\omega_{\sf max}^2\log^2 n \label{ineq23-1357} \\
			\big\|\bm{U}^{\star}\bm{U}^{\star\top}\widetilde{\bm{U}} - \widetilde{\bm{U}}\big\|_{2, \infty} &\leq \frac{4C_5\sqrt{\mu r}\omega_{\sf max}\log n}{\sigma_r^\star}\leq \sqrt{\frac{\mu r}{n_1}} ,\label{ineq34}\\
			\big\|\widetilde{\bm{U}}\big\|_{2, \infty} &\leq 2\sqrt{\frac{\mu r}{n_1}}\label{ineq27}
	\end{align}
	\end{subequations}
	for some large enough constant $C_5>0$. 
\end{lemma}
\noindent 
The proof of this lemma can be found in Section~\ref{proof: lm_collection_event}.

\subsection{Main steps for proving Theorem \ref{thm:two_to_infty}}

In what follows, we shall demonstrate how to prove Theorem~\ref{thm:two_to_infty} with the assistance of Theorem~\ref{thm:oracle}. 
Reusing some of the notation in the proof of Theorem \ref{thm:deterministic}, we define
\begin{align}
	 D_k^t = \big\|\mathcal{P}_{\sf diag}\big(\bm{G}_k^t - \widetilde{\bm{M}}\big)\big\|, \qquad L_k^t = \big\|\bm{G}_k^t - \widetilde{\bm{M}}\big\|
	 \qquad {and} \qquad \widetilde{\bm{U}}_k = \widetilde{\bm{U}}_{:,1:r_k} 
	 \label{eq:defn-Dkt-Lkt-thm2}
\end{align}
for any $k \geq 1$ and any $0 \leq t \leq t_k$. 
We find it helpful to introduce the following event: 
\begin{align}\label{eq:event_E}
	\mathcal{E} = \left\{\eqref{ineq:two_to_infty_oracle}, \eqref{ineq:spectral_oracle}, \eqref{ineq26}, \eqref{ineq23-1357} \text{ and }\eqref{ineq27} \text{ hold}\right\}.
\end{align}
The results in Lemma~\ref{lm:collection_event} and Theorem~\ref{thm:oracle} combined with the union bound give
\begin{align}\label{ineq:probability_E}
	\bbP\left(\mathcal{E}\right) \geq 1 - O\left(n^{-10}\right).
\end{align}
Throughout the remainder of this proof, we shall assume that the event $\mathcal{E}$ occurs unless otherwise noted. 
A similar argument as in the proof of \eqref{condition:r_1} also tells us that
\begin{align}\label{set:r_1}
	r_1 \in \mathcal{R}_1 = \left\{r': \frac{\sigma_1\left(\bm{G}_0\right)}{\sigma_{r'}\left(\bm{G}_0\right)} \leq 4 \quad \text{ and } \quad
	\sigma_{r'}\left(\bm{G}_{0}\right) - \sigma_{r' + 1}\left(\bm{G}_{0}\right) \geq \frac{1}{r}\sigma_{r'}\left(\bm{G}_{0}\right)\right\}.
\end{align}

\paragraph{Step 1: bounding $D_1^t = \|\mathcal{P}_{\sf diag}(\bm{G}_1^t - \widetilde{\bm{M}})\|$.} %
We now proceed to control the quantities $\{D_1^t\}$ for the first round. 
More specifically, we intend to prove, by induction, the following properties:
\begin{subequations}
	\label{eq:hypothesis-proof-thm2}
	\begin{align}
		D_1^t - \left(14\sqrt{\frac{\mu r}{n_1}}\left\|\bm{Z}\right\| + 12\sqrt{\frac{\mu r}{n_1}}\widetilde{\sigma}_{r_1 + 1}^2\right) &\leq \frac{1}{e^t}\left[D_1^0 - \left(14\sqrt{\frac{\mu r}{n_1}}\left\|\bm{Z}\right\| + 12\sqrt{\frac{\mu r}{n_1}}\widetilde{\sigma}_{r_1 + 1}^2\right)\right],\label{ineq41a}\\
		\left\|\bm{U}_1^t\bm{U}_1^{t\top} - \bm{U}_1^{\sf oracle}\bm{U}_1^{\sf oracle\top}\right\|	&\leq 2\frac{D_1^t}{\lambda_{r_1}\left(\bm{M}^{\sf oracle}\right) - \lambda_{r_1 + 1}\left(\bm{M}^{\sf oracle}\right)} \leq \frac{1}{8},\label{ineq41b}\\
		\left\|\bm{U}_1^t\right\|_{2,\infty} &\leq \left\|\bm{U}_1^t\bm{U}_1^{t\top} - \bm{U}_1^{\sf oracle}\bm{U}_1^{\sf oracle\top}\right\| + \left\|\bm{U}_1^{\sf oracle}\right\|_{2,\infty} \leq \frac{1}{4e},\label{ineq41c}
	\end{align}
\end{subequations}
where $\bm{M}^{\sf oracle}$ is defined in \eqref{eq:M_oracle-n32} and we recall that $\bm{U}_1^{\sf oracle}\in \mathbb{R}^{n_1\times r_1}$ is the top-$r_1$ eigenspace of $\bm{M}^{\sf oracle}$.

\paragraph{Step 1.1: the base case with $t=0$ for \eqref{ineq41a}-\eqref{ineq41c}.} 
The claim \eqref{ineq41a} holds trivially when $t=0$. 
Also, given that the off-diagonal entries of $\bm{G}_{k}^t$ and $\bm{G}_0$ are the same, taking \citet[Lemma 1]{zhang2022heteroskedastic} together with the property \eqref{ineq27} yields
\begin{align}\label{ineq37}
	D_1^0 = \big\|\mathcal{P}_{\sf diag}\big(\widetilde{\bm{M}}\big)\big\| = \left\|\mathcal{P}_{\sf diag}\big(\bm{P}_{\widetilde{\bm{U}}}\widetilde{\bm{M}}\bm{P}_{\widetilde{\bm{U}}}\big)\right\| \leq 4\frac{\mu r}{n_1}\big\|\widetilde{\bm{M}}\big\| = 4\frac{\mu r}{n_1}\widetilde{\sigma}_1^2. 
\end{align}
This together with \eqref{ineq23-1357} further gives
\begin{align}\label{ineq38}
	L_1^0 &\leq D_1^0 + \big\|\mathcal{P}_{\sf off\text{-}diag}\big(\bm{G}_1^0 - \widetilde{\bm{M}}\big)\big\| = D_1^0 +  \big\|\mathcal{P}_{\sf off\text{-}diag}\big(\bm{G}_0 - \widetilde{\bm{M}}\big)\big\|\notag\\
	&\leq 4\frac{\mu r}{n_1}\widetilde{\sigma}_1^2 + \left\|\bm{Z}\right\| \leq 4\frac{\mu r}{n_1}\widetilde{\sigma}_1^2 + 3C_5\left(\sqrt{n_1n_2} + n_1\right)\omega_{\sf max}^2\log^2 n,
\end{align}
where we remind the reader that $\bm{Z} = \mathcal{P}_{\sf off\text{-}diag}\left(\bm{E}\bm{E}^\top - \bm{E}\bm{V}^\star\bm{V}^{\star\top}\bm{E}^\top\right)$.

Next, let us look at the spectrum of the matrices of interest. 
Note that 
\begin{align*}
	\big\|\bm{M}^{\sf oracle} - \widetilde{\bm{M}}\big\| = \left\|\bm{Z}\right\| \quad \text{and} \quad \big\|\mathcal{P}_{\sf diag}\big(\bm{M}^{\sf oracle} - \widetilde{\bm{M}}\big)\big\| = \left\|\mathcal{P}_{\sf diag}\left(\bm{Z}\right)\right\| = 0.
\end{align*}
It comes from Weyl's inequality that, for all $1 \leq i \leq r+1$,
\begin{align}
	\left|\sigma_{i}^\star - \widetilde{\sigma}_i\right| &\leq \left\|\bm{E}\bm{V}^\star\right\| \leq \sqrt{C_5}\sqrt{n_1}\omega_{\sf max}\log n,\label{ineq35a}\\
	\left|\widetilde{\sigma}_i^2 - \lambda_i\left(\bm{M}^{\sf oracle}\right)\right| &\leq \left\|\bm{Z}\right\| \leq 3C_5\left(\sqrt{n_1n_2} + n_1\right)\omega_{\sf max}^2\log^2 n,\label{ineq35b}
\end{align}
where the first line relies on \eqref{ineq26}, and the second line results from \eqref{ineq23-1357}. 
From the assumption~\eqref{assump:snr_two_to_infty}, we can further derive 
\begin{align}\label{ineq43}
	\forall i \in [r_1], \quad \frac{9}{10}\sigma_i^\star \leq \widetilde{\sigma}_i \leq \frac{11}{10}\sigma_i^\star, \quad  \frac{4}{5}\sigma_i^{\star2} \leq \lambda_i\left(\bm{G}_0\right) \leq \frac{61}{50}\sigma_i^{\star2}, \quad \lambda_{r_1+1}\left(\bm{G}_0\right) \leq \frac{\sigma_{r_1}^{\star2}}{100}.
\end{align}
Furthermore, we can easily verify that
\begin{align}\label{ineq36}
	\max\left\{\frac{\sigma_1^{\star2}}{\sigma_{r_1}^{\star2}},\ \frac{\widetilde{\sigma}_1^{2}}{\widetilde{\sigma}_{r_1}^{2}}\right\} \leq 8 \quad \text{and} \quad \min\left\{\frac{\widetilde{\sigma}_{r_1}^2 - \widetilde{\sigma}_{r_1+1}^2}{\widetilde{\sigma}_{r_1}^2},\ \frac{\sigma_{r_1}^{\star2} - \sigma_{r_1+1}^{\star2}}{\sigma_{r_1}^{\star2}}\right\} \geq \frac{1}{2r} > 1 - \left(1 - \frac{1}{4r}\right)^2
\end{align}
and 
\begin{align}\label{ineq39}
	\lambda_{r_1}\left(\bm{M}^{\sf oracle}\right) - \lambda_{r_1 + 1}\left(\bm{M}^{\sf oracle}\right) \asymp \widetilde{\sigma}_{r_1}^2 - \widetilde{\sigma}_{r_1+1}^2 \asymp \sigma_{r_1}^{\star2} - \sigma_{r_1+1}^{\star2} \gg \left\|\bm{Z}\right\|.
\end{align}
Recall that $\bm{U}_1^0$ (resp.~$\bm{U}_1^{\sf oracle}$) is the top-$r_1$ eigenspace of $\bm{G}_0$ (resp.~$\bm{M}^{\sf oracle}$). With the preceding inequalities about the singular values (or eigenvalues) in place, 
invoking the Davis-Kahan theorem \citep[Theorem 2.7]{chen2021spectral} and using \eqref{ineq37}  demonstrate that 
\begin{align}
	\left\|\bm{U}_1^0\bm{U}_1^{0\top} - \bm{U}_1^{\sf oracle}\bm{U}_1^{\sf oracle\top}\right\|	
	&\leq 2\frac{\left\|\bm{G}_0 - \bm{M}^{\sf oracle}\right\|}{\lambda_{r_1}\left(\bm{M}^{\sf oracle}\right) - \lambda_{r_1 + 1}\left(\bm{M}^{\sf oracle}\right)} 
	= 2\frac{D_1^0}{\lambda_{r_1}\left(\bm{M}^{\sf oracle}\right) - \lambda_{r_1 + 1}\left(\bm{M}^{\sf oracle}\right)} \notag\\
	&\lesssim \frac{\frac{\mu r}{n_1}\widetilde{\sigma}_1^2}{\widetilde{\sigma}_{r_1}^2 - \widetilde{\sigma}_{r_1+1}^2} \lesssim \frac{\mu r^2}{n_1} 
	 \ll \sqrt{\frac{\mu r}{n_1}} 
	\leq \frac{1}{16e},
	\label{eq:U10-Uoracle-2inf}
\end{align}
thus validating the claim \eqref{ineq41b} for $t=0$. Here, the first inequality is valid since, according to \eqref{ineq85},   
$$\left\|\bm{G}_0 - \bm{M}^{\sf oracle}\right\| = \left\|\mathcal{P}_{\sf diag}\left(\bm{M}^{\sf oracle}\right)\right\| = D_1^0.$$

Moreover, in view of Theorem \ref{thm:oracle} and \eqref{ineq27}, we can derive
\begin{align}\label{ineq40}
	\left\|\bm{U}^{\sf oracle}\right\|_{2,\infty} &
	=\left\|\bm{U}^{\sf oracle}\bm{U}^{\sf oracle \top}\right\|_{2,\infty} \leq \big\|\bm{U}^{\sf oracle}\bm{U}^{\sf oracle\top} - \widetilde{\bm{U}}\widetilde{\bm{U}}^\top\big\|_{2,\infty} + \big\|\widetilde{\bm{U}}\big\|_{2,\infty} \leq 3\sqrt{\frac{\mu r}{n_1}}, 
\end{align}
where we have also made use of the assumption~\eqref{assump:snr_two_to_infty}. 
Putting \eqref{eq:U10-Uoracle-2inf} and \eqref{ineq40} together leads to 
\begin{align*}
	\left\|\bm{U}_1^0\right\|_{2,\infty} &= \left\|\bm{U}_1^0\bm{U}_1^{0\top}\right\|_{2,\infty} \leq \left\|\bm{U}_1^0\bm{U}_1^{0\top} - \bm{U}_1^{\sf oracle}\bm{U}_1^{\sf oracle\top}\right\| + \left\|\bm{U}_1^{\sf oracle}\right\|_{2,\infty} \leq 4\sqrt{\frac{\mu r}{n_1}} \leq \frac{1}{4e},
\end{align*}
which validates the claim \eqref{ineq41c} when $t=0$. 
We have thus established \eqref{eq:hypothesis-proof-thm2} for the base case.

\paragraph{Step 1.2: induction step for \eqref{ineq41a}-\eqref{ineq41c}.} 
We now move on to the inductive step. 
Suppose that the induction hypotheses \eqref{ineq41a}-\eqref{ineq41c} hold for $t = t'$, 
and we would like to show their validity for $t=t'+1$.

Recalling that the diagonal entries of $\bm{G}_1^{t'+1}$ are equal to the diagonal entries of $\bm{U}_1^{t'}\bm{\Lambda}_1\bm{U}_1^{t'\top} = \bm{P}_{\bm{U}_1^{t'}}\bm{G}_1^{t'}$ and $\bm{U}_1^{t'}$ represents the rank-$r$ leading singular subspace of 
\begin{align*}
	\bm{G}_1^{t'} = \bm{P}_{\widetilde{\bm{U}}_1}\widetilde{\bm{M}} + \big(\bm{G}_1^{t'} - \bm{P}_{\widetilde{\bm{U}}_1}\widetilde{\bm{M}}\big),
\end{align*}
one can  obtain 
\begin{align}\label{ineq44}
	D_{1}^{t'+1} &= \big\|\mathcal{P}_{\sf diag}\big(\bm{G}_1^{t'+1} - \widetilde{\bm{M}}\big)\big\| 
	= \left\|\mathcal{P}_{\sf diag}\big(\bm{P}_{\bm{U}_1^{t'}}\bm{G}_1^{t'} - \widetilde{\bm{M}}\big)\right\|\notag\\
	&\leq \left\|\mathcal{P}_{\sf diag}\left(\bm{P}_{\bm{U}_1^{t'}}\big(\bm{G}_1^{t'} - \widetilde{\bm{M}}\big)\right)\right\| + \left\|\mathcal{P}_{\sf diag}\left(\bm{P}_{\left(\bm{U}_1^{t'}\right)_{\perp}}\widetilde{\bm{M}}\bm{P}_{\widetilde{\bm{U}}}\right)\right\|\notag\\
	&\overset{\mathrm{(i)}}{\leq} \big\|\bm{U}_1^{t'}\big\|_{2,\infty}\big\|\bm{G}_1^{t'} - \widetilde{\bm{M}}\big\| + \big\|\widetilde{\bm{U}}\big\|_{2,\infty}\left\|\big(\bm{U}_1^{t'}\big)_{\perp}\widetilde{\bm{M}}\right\|\notag\\
	&\overset{\mathrm{(ii)}}{\leq} \big\|\bm{U}_1^{t'}\big\|_{2,\infty}L_1^{t'} + 2\sqrt{\frac{\mu r}{n_1}}\left(\left\|\big(\bm{U}_1^{t'}\big)_{\perp}\bm{P}_{\widetilde{\bm{U}}_1}\widetilde{\bm{M}}\right\| + \left\|\bm{P}_{\widetilde{\bm{U}}_{:,r_{1}+1:r}}\widetilde{\bm{M}}\right\|\right)\notag\\
	&\overset{\mathrm{(iii)}}{\leq} \big\|\bm{U}_1^{t'}\big\|_{2,\infty}L_1^{t'} + 2\sqrt{\frac{\mu r}{n_1}}\left(2\left\|\bm{G}_1^{t'} - \bm{P}_{\widetilde{\bm{U}}_1}\widetilde{\bm{M}}\right\| + \left\|\bm{P}_{\widetilde{\bm{U}}_{:,r_{1}+1:r}}\widetilde{\bm{M}}\right\|\right)\notag\\
	&\overset{\mathrm{(iv)}}{\leq} \big\|\bm{U}_1^{t'}\big\|_{2,\infty}L_1^{t'} + 2\sqrt{\frac{\mu r}{n_1}}\left(2\big\|\bm{G}_1^{t'} - \widetilde{\bm{M}}\big\| + 3\left\|\bm{P}_{\widetilde{\bm{U}}_{:,r_{1}+1:r}}\widetilde{\bm{M}}\right\|\right)\notag\\
	&\leq \big\|\bm{U}_1^{t'}\big\|_{2,\infty}L_1^{t'} + 4\sqrt{\frac{\mu r}{n_1}}L_1^{t'} + 6\sqrt{\frac{\mu r}{n}}\widetilde{\sigma}_{r_1 + 1}^2, 
\end{align}
where (i) invokes \citet[Lemma 1]{zhang2022heteroskedastic}, (ii) results from \eqref{ineq27}, (iii) is a consequence of Lemma \ref{lm:projection}, 
and (iv) applies the triangle inequality. 
Recognizing that (see \eqref{eq:defn-Dkt-Lkt-thm2})
\begin{align*}
	L_1^{t'} \leq D_1^{t'} + \big \|\mathcal{P}_{\sf off\text{-}diag}\big(\bm{G}_1^{t'} - \widetilde{\bm{M}}\big) \big\| = D_1^{t'} + \left\|\bm{Z}\right\|,
\end{align*}
one can deduce that
\begin{align}
	D_{1}^{t'+1} &\leq \left(\big\|\bm{U}_1^{t'}\big\|_{2,\infty} +  4\sqrt{\frac{\mu r}{n}}\right)D_1^{t'} + \left(\big\|\bm{U}_1^{t'}\big\|_{2,\infty} +  4\sqrt{\frac{\mu r}{n}}\right)\left\|\bm{Z}\right\| + 6\sqrt{\frac{\mu r}{n_1}}\widetilde{\sigma}_{r_1 + 1}^2\notag\\
	&\stackrel{\eqref{ineq41c}}{\leq} \left(\frac{1}{4e} + \frac{1}{4e}\right)D_1^{t'} + \left(\big\|\bm{U}_1^{t'}\bm{U}_1^{t'\top} - \bm{U}_1^{\sf oracle}\bm{U}_1^{\sf oracle\top}\big\| + \left\|\bm{U}_1^{\sf oracle}\right\|_{2,\infty} +  4\sqrt{\frac{\mu r}{n}}\right)\left\|\bm{Z}\right\| + 6\sqrt{\frac{\mu r}{n_1}}\widetilde{\sigma}_{r_1 + 1}^2\notag\\
	&\stackrel{\eqref{ineq41b} \text{ and } \eqref{ineq40}}{\leq} \frac{1}{2e}D_1^{t'} + \left(2\frac{D_1^{t'}}{\lambda_{r_1}\left(\bm{M}^{\sf oracle}\right) - \lambda_{r_1 + 1}\left(\bm{M}^{\sf oracle}\right)} + 7\sqrt{\frac{\mu r}{n_1}}\right)\left\|\bm{Z}\right\| + 6\sqrt{\frac{\mu r}{n_1}}\widetilde{\sigma}_{r_1 + 1}^2\notag\\
	&\stackrel{\eqref{ineq39}}{\leq} \frac{1}{e}D_1^{t'} +  7\sqrt{\frac{\mu r}{n_1}}\left\|\bm{Z}\right\| + 6\sqrt{\frac{\mu r}{n_1}}\widetilde{\sigma}_{r_1 + 1}^2\label{ineq42}.
\end{align}
This together with the induction hypotheses further leads to
\begin{align*}
	D_{1}^{t'+1} - \left(14\sqrt{\frac{\mu r}{n_1}}\left\|\bm{Z}\right\| + 12\sqrt{\frac{\mu r}{n_1}}\widetilde{\sigma}_{r_1 + 1}^2\right) &\leq \frac{1}{e}\left[D_{1}^{t'} - \left(14\sqrt{\frac{\mu r}{n_1}}\left\|\bm{Z}\right\| + 12\sqrt{\frac{\mu r}{n_1}}\widetilde{\sigma}_{r_1 + 1}^2\right)\right]\\
	&\leq \frac{1}{e^{t'+1}}\left[D_{1}^{0} - \left(14\sqrt{\frac{\mu r}{n_1}}\left\|\bm{Z}\right\| + 12\sqrt{\frac{\mu r}{n_1}}\widetilde{\sigma}_{r_1 + 1}^2\right)\right], 
\end{align*}
thus justifying the induction hypothesis \eqref{ineq41a} for $t = t'+1$.

In addition,  \eqref{ineq42} allows us to derive
\begin{align*}
	\frac{D_{1}^{t'+1}}{\lambda_{r_1}\left(\bm{M}^{\sf oracle}\right) - \lambda_{r_1 + 1}\left(\bm{M}^{\sf oracle}\right)} &\leq \frac{\frac{1}{e}D_1^{t'} +  7\sqrt{\frac{\mu r}{n_1}}\left\|\bm{Z}\right\| + 6\sqrt{\frac{\mu r}{n_1}}\widetilde{\sigma}_{r_1 + 1}^2}{\lambda_{r_1}\left(\bm{M}^{\sf oracle}\right) - \lambda_{r_1 + 1}\left(\bm{M}^{\sf oracle}\right)}\\
	&\leq \frac{1}{e}\cdot \frac{1}{8} + \frac{7\sqrt{\frac{\mu r}{n_1}}\left\|\bm{Z}\right\|}{\lambda_{r_1}\left(\bm{M}^{\sf oracle}\right) - \lambda_{r_1 + 1}\left(\bm{M}^{\sf oracle}\right)} + \frac{C_5\sqrt{\frac{\mu r}{n_1}}\widetilde{\sigma}_{r_1 + 1}^2}{\widetilde{\sigma}_{r_1}^2 - \widetilde{\sigma}_{r_1 + 1}^2}\\
	&\leq \frac{1}{8e} + \frac{1}{80} + C_5\sqrt{\frac{\mu r^3}{n_1}} \leq \frac{1}{16}, 
\end{align*}
where the second line invokes the induction hypothesis \eqref{ineq41b} (when $t=t'$) and \eqref{ineq39}, 
and the last line relies on \eqref{ineq36} and the assumption \eqref{assump:snr_two_to_infty_234}.  

Recalling that $\mathcal{P}_{\sf off\text{-}diag}(\bm{M}^{\sf oracle}) = \mathcal{P}_{\sf off\text{-}diag}\left(\bm{Y}\bm{Y}^\top\right) = \mathcal{P}_{\sf off\text{-}diag}\left(\bm{G}_k^t\right)$ and $\mathcal{P}_{\sf off\text{-}diag}(\bm{M}^{\sf oracle}) = \mathcal{P}_{\sf off\text{-}diag}(\widetilde{\bm{M}})$, one has
\begin{align}\label{ineq83}
	\left\|\bm{G}_k^t - \bm{M}^{\sf oracle}\right\| = \left\|\mathcal{P}_{\sf diag}\left(\bm{G}_k^t - \bm{M}^{\sf oracle}\right)\right\| = \left\|\mathcal{P}_{\sf diag}\left(\bm{G}_k^t - \widetilde{\bm{M}}\right)\right\|  = D_k^t.
\end{align}
Therefore, we can readily apply the Davis-Kahan theorem \citep[Theorem~2.7]{chen2021spectral} to arrive at
\begin{align*}
	\big\|\bm{U}_1^{t'+1}\bm{U}_1^{t'+1\top} - \bm{U}_1^{\sf oracle}\bm{U}_1^{\sf oracle\top}\big\|	&\leq 2\frac{D_1^{t'+1}}{\lambda_{r_1}\left(\bm{M}^{\sf oracle}\right) - \lambda_{r_1 + 1}\left(\bm{M}^{\sf oracle}\right)} \leq \frac{1}{8}. 
\end{align*}
Here, we remind the readers that $\bm{U}_k^t$ (resp.~$\bm{U}_k^{\sf oracle}$) represents the top-$r_k$ eigenspace of $\bm{G}_k^t$ (resp.~$\bm{M}^{\sf oracle}$). 
This establishes the induction hypothesis \eqref{ineq41b} for $t = t'+1$, 
which in turn also validates \eqref{ineq41c} for $t = t'+1$.

Therefore, we have finished the proof for the hypotheses \eqref{ineq41a}-\eqref{ineq41c} when $t=t'+1$, thereby completing the induction step for the first round.

\paragraph{Step 2: bounding $D_k^t = \|\mathcal{P}_{\sf diag}(\bm{G}_k^t - \widetilde{\bm{M}})\|$ for $k > 1$.} 
Having established the desired properties for the first round, 
we would like extend these to accommodate $\{D_k^t\}$ for the $k$-th round with $k>1$. 
More precisely, we would like to further bound $\{\|\mathcal{P}_{\sf diag}(\bm{G}_k^t - \widetilde{\bm{M}})\|\}_{k > 1, t \geq 0}$ by means of a recursive argument.

To begin with, in view of \eqref{ineq23-1357} and \eqref{ineq36}, by choosing $$t_1 \geq \log\left(C\frac{\sigma_1^{\star2}}{\sigma_{r_1 + 1}^{\star2}} \right) \geq \log\Bigg(\frac{C\sqrt{\frac{\mu r}{n_1}}\sigma_1^{\star2}}{3C_5\left(\sqrt{n_1n_2} + n_1\right)\omega_{\sf max}^2\log^2 n + \sigma_{r_1 + 1}^{\star2}} \Bigg),$$
we have
\begin{align*}
	D_2^{0} = D_{1}^{t_1} &\leq 14\sqrt{\frac{\mu r}{n_1}}\left\|\bm{Z}\right\| + 12\sqrt{\frac{\mu r}{n_1}}\widetilde{\sigma}_{r_1 + 1}^2 + 3C_5\sqrt{\frac{\mu r}{n_1}}\left(\sqrt{n_1n_2} + n_1\right)\omega_{\sf max}^2\log^2 n + \sqrt{\frac{\mu r}{n_1}}\widetilde{\sigma}_{r_1 + 1}^2\\
	&\leq 45C_5\sqrt{\frac{\mu r}{n_1}}\left(\sqrt{n_1n_2} + n_1\right)\omega_{\sf max}^2\log^2 n + 13\sqrt{\frac{\mu r}{n_1}}\widetilde{\sigma}_{r_1 + 1}^2.
\end{align*}
Repeating similar arguments as in \eqref{ineq36} and \eqref{ineq39} yields
\begin{align}\label{ineq45}
	\max\left\{\frac{\sigma_{r_1+1}^{\star2}}{\sigma_{r_2}^{\star2}},\ \frac{\widetilde{\sigma}_{r_1+1}^{2}}{\widetilde{\sigma}_{r_2}^{2}}\right\} \leq 8 \quad \text{and} \quad \min\left\{\frac{\widetilde{\sigma}_{r_2}^2 - \widetilde{\sigma}_{r_2+1}^2}{\widetilde{\sigma}_{r_2}^2},\ \frac{\sigma_{r_2}^{\star2} - \sigma_{r_2+1}^{\star2}}{\sigma_{r_2}^{\star2}}\right\} \geq \frac{1}{2r} > 1 - \left(1 - \frac{1}{4r}\right)^2
\end{align}
and 
\begin{align}\label{ineq46}
	\lambda_{r_2}\left(\bm{M}^{\sf oracle}\right) - \lambda_{r_2 + 1}\left(\bm{M}^{\sf oracle}\right) \asymp \widetilde{\sigma}_{r_2}^2 - \widetilde{\sigma}_{r_2+1}^2 \asymp \sigma_{r_2}^{\star2} - \sigma_{r_2+1}^{\star2} \gg \left\|\bm{Z}\right\|.
\end{align}
We can then reach
\begin{align*}
	\frac{D_2^0}{\lambda_{r_2}\left(\bm{M}^{\sf oracle}\right) - \lambda_{r_2 + 1}\left(\bm{M}^{\sf oracle}\right)} &\lesssim \frac{45C_5\sqrt{\frac{\mu r}{n_1}}\left(\sqrt{n_1n_2} + n_1\right)\omega_{\sf max}^2\log^2 n}{\sigma_{r_2}^{\star2} - \sigma_{r_2+1}^{\star2}} + \frac{13\sqrt{\frac{\mu r}{n_1}}\widetilde{\sigma}_{r_1 + 1}^2}{\widetilde{\sigma}_{r_2}^2 - \widetilde{\sigma}_{r_2+1}^2}\\
	&\lesssim \sqrt{\frac{\mu r}{n_1}} + \sqrt{\frac{\mu r}{n_1}}\frac{\widetilde{\sigma}_{r_1 + 1}^2}{\widetilde{\sigma}_{r_2}^2}\cdot\frac{\widetilde{\sigma}_{r_2}^2}{\widetilde{\sigma}_{r_2}^2 - \widetilde{\sigma}_{r_2+1}^2}\\
	&\lesssim \sqrt{\frac{\mu r^3}{n_1}} \ll \frac{1}{8\sqrt{2}}.
\end{align*}
Thus, invoking the Davis-Kahan theorem \citep[Theorem~2.7]{chen2021spectral} and \eqref{ineq83} leads to
\begin{align*}
	\left\|\bm{U}_2^0\bm{U}_2^{0\top} - \bm{U}_2^{\sf oracle}\bm{U}_2^{\sf oracle\top}\right\| \leq \sqrt{2}\frac{D_2^0}{\lambda_{r_2}\left(\bm{M}^{\sf oracle}\right) - \lambda_{r_2 + 1}\left(\bm{M}^{\sf oracle}\right)} \lesssim \sqrt{\frac{\mu r^3}{n_1}} \ll \frac{1}{8},
\end{align*}
where we recall that $\bm{U}_k^t$ (resp.~$\bm{U}_k^{\sf oracle}$) is the top-$r_k$ eigenspace of $\bm{G}_k^t$ (resp.~$\bm{M}^{\sf oracle}$). 
%
Similar to the argument for \eqref{ineq44}, one can obtain 
\begin{align}
	D_2^{t+1} \leq \left\|\bm{U}_2^t\right\|_{2,\infty}L_2^{t} + 4\sqrt{\frac{\mu r}{n_1}}L_2^{t} + 6\sqrt{\frac{\mu r}{n_1}}\widetilde{\sigma}_{r_2 + 1}^2.
	\label{eq:D2-t-plus-1-UB}
\end{align}

Further, repeat similar arguments as in \eqref{ineq:induction_rank}, \eqref{ineq41a}-\eqref{ineq41c}, \eqref{ineq36}, \eqref{ineq39} and \eqref{eq:D2-t-plus-1-UB} 
to yield that: for all $1 \leq k \leq k_{\sf max}$ and $1 \leq t \leq t_k$,  one has the following properties: 
\begin{subequations}
	\begin{align}
		r_k &\in \mathcal{R}_k \quad \text{ where } \mathcal{R}_k \text{ is defined in } \eqref{ineq:induction_rank},\\
		D_k^t - \left(14\sqrt{\frac{\mu r}{n_1}}\left\|\bm{Z}\right\| + 12\sqrt{\frac{\mu r}{n_1}}\widetilde{\sigma}_{r_k + 1}^2\right) &\leq \frac{1}{e^t}\left[D_k^0 - \left(14\sqrt{\frac{\mu r}{n_1}}\left\|\bm{Z}\right\| + 12\sqrt{\frac{\mu r}{n_1}}\widetilde{\sigma}_{r_k + 1}^2\right)\right],\label{ineq45a}\\
		\left\|\bm{U}_k^t\bm{U}_k^{t\top} - \bm{U}_k^{\sf oracle}\bm{U}_k^{\sf oracle\top}\right\|	&\leq 2\frac{D_k^t}{\lambda_{r_k}\left(\bm{M}^{\sf oracle}\right) - \lambda_{r_k + 1}\left(\bm{M}^{\sf oracle}\right)} \leq \frac{1}{8},\label{ineq45b}\\
		\left\|\bm{U}_k^t\right\|_{2,\infty} &\leq \left\|\bm{U}_k^t\bm{U}_k^{t\top} - \bm{U}_k^{\sf oracle}\bm{U}_k^{\sf oracle\top}\right\| + \left\|\bm{U}_k^{\sf oracle}\right\|_{2,\infty} \leq \frac{1}{4},\label{ineq45c}\\
		D_{k+1}^0 = D_k^{t_k} &\leq 45C_5\sqrt{\frac{\mu r}{n_1}}\left(\sqrt{n_1n_2} + n_1\right)\omega_{\sf max}^2\log^2 n + 13\sqrt{\frac{\mu r}{n_1}}\widetilde{\sigma}_{r_k + 1}^2.\label{ineq45d}\\
		\max\left\{\frac{\sigma_{r_k+1}^{\star2}}{\sigma_{r_k}^{\star2}},\ \frac{\widetilde{\sigma}_{r_k+1}^{2}}{\widetilde{\sigma}_{r_k}^{2}}\right\} \leq 8 \quad \text{and} &\quad \min\left\{\frac{\widetilde{\sigma}_{r_k}^2 - \widetilde{\sigma}_{r_k+1}^2}{\widetilde{\sigma}_{r_k}^2},\ \frac{\sigma_{r_k}^{\star2} - \sigma_{r_k+1}^{\star2}}{\sigma_{r_k}^{\star2}}\right\} \geq \frac{1}{2r},\label{ineq45e}\\
		\lambda_{r_k}\left(\bm{M}^{\sf oracle}\right) - \lambda_{r_k + 1}\left(\bm{M}^{\sf oracle}\right) &\asymp \widetilde{\sigma}_{r_k}^2 - \widetilde{\sigma}_{r_k+1}^2 \asymp \sigma_{r_k}^{\star2} - \sigma_{r_k+1}^{\star2} \gg \left(\sqrt{n_1n_2} + n_1\right)\omega_{\sf max}^2\log^2 n,\label{ineq45f} \\
		D_k^{t+1} &\leq \left\|\bm{U}_k^t\right\|_{2,\infty}L_k^{t} + 4\sqrt{\frac{\mu r}{n_1}}L_k^{t} + 6\sqrt{\frac{\mu r}{n_1}}\widetilde{\sigma}_{r_k + 1}^2 , 
		\label{eq:Dk-t-plus-1-UB}
	\end{align}
\end{subequations}
provided that the numbers of iterations $t_i$ satisfy \eqref{ineq:iter1}-\eqref{ineq:iter2}. 
Here, we remind the reader that $\bm{U}_k^{\sf oracle}$ represents the top-$r_k$ eigenspace of $\bm{M}^{\sf oracle}$. 
Given that these can be established using exactly the same arguments as before, 
we omit the details here for the sake of brevity. 

By letting $k = k_{\sf max}$ in \eqref{ineq45a} and \eqref{ineq45d} and recalling \eqref{ineq35b}, we immediately have
\begin{align*}
	D_{k_{\sf max}}^{t_{k_{\sf max}}} \lesssim \sqrt{\frac{\mu r}{n_1}}\left(\sqrt{n_1n_2} + n_1\right)\omega_{\sf max}^2\log^2 n.
\end{align*} 
Then the Davis-Kahan sin$\Theta$ theorem reveals that
\begin{align}
	\left\|\bm{U}\bm{U}^\top - \bm{U}^{\sf oracle}\bm{U}^{\sf oracle\top}\right\| 
	&\lesssim \frac{\| \bm{G}_{k_{\sf max}}^{t_{k_{\sf max}}} - \bm{M}^{\mathsf{oracle}} \|}{\lambda_{r_{k_{\sf max}}}\left(\bm{M}^{\sf oracle}\right) - \lambda_{r_{k_{\sf max}} + 1}\left(\bm{M}^{\sf oracle}\right)} 
	= \frac{D_{k_{\sf max}}^{t_{k_{\sf max}}}}{\lambda_{r_{k_{\sf max}}}\left(\bm{M}^{\sf oracle}\right) - \lambda_{r_{k_{\sf max}} + 1}\left(\bm{M}^{\sf oracle}\right)} \notag\\
	&\lesssim \frac{\sqrt{\frac{\mu r}{n_1}}\left(\sqrt{n_1n_2} + n_1\right)\omega_{\sf max}^2\log^2 n}{\lambda_{r_{k_{\sf max}}}\left(\bm{M}^{\sf oracle}\right) - \lambda_{r_{k_{\sf max}} + 1}\left(\bm{M}^{\sf oracle}\right)} \notag\\
	&\lesssim \frac{\sqrt{\frac{\mu r}{n_1}}\left(\sqrt{n_1n_2} + n_1\right)\omega_{\sf max}^2\log^2 n}{\sigma_{r_{k_{\sf max}}}^{\star2} - \sigma_{r_{k_{\sf max}}+1}^{\star2}} \notag\\
	&\asymp \frac{\sqrt{\frac{\mu r}{n_1}}\left(\sqrt{n_1n_2} + n_1\right)\omega_{\sf max}^2\log^2 n}{\sigma_{r}^{\star2}}, \label{ineq84}
\end{align}
where the first line also applies \eqref{ineq85},  the third line relies on \eqref{ineq45f}, 
and the last line holds since $r_{k_{\sf max}} = r$.

\paragraph{Step 3: bounding $\|\bm{U}\bm{R}_{\bm{U}} - \bm{U}^{\star}\|_{2, \infty}$ and $\|\bm{U}\bm{R}_{\bm{U}} - \bm{U}^{\star}\|$.}
In the final step, we invoke Theorem \ref{thm:oracle} to establish the desired bounds on $\|\bm{U}\bm{R}_{\bm{U}} - \bm{U}^{\star}\|$ and $\|\bm{U}\bm{R}_{\bm{U}} - \bm{U}^{\star}\|_{2, \infty}$. 
To begin with, inequality \eqref{ineq84} taken together with Theorem \ref{thm:oracle} gives 
\begin{align}\label{ineq:two_to_infty_UU}
	\left\|\bm{U}\bm{U}^\top - \bm{U}^{\star}\bm{U}^{\star\top}\right\|_{2,\infty} &\leq \left\|\bm{U}\bm{U}^\top - \bm{U}^{\sf oracle}\bm{U}^{\sf oracle\top}\right\| + \left\|\bm{U}^{\sf oracle}\bm{U}^{\sf oracle\top} - \bm{U}^{\star}\bm{U}^{\star\top}\right\|_{2,\infty}\notag\\
	&\leq \sqrt{\frac{\mu r}{n_1}}\left(\frac{\sqrt{n_1n_2}\omega_{\sf max}^2\log^2 n}{\sigma_r^{\star2}} + \frac{\sqrt{n_1}\omega_{\sf max}\log n}{\sigma_{r}^\star}\right)
\end{align}
and 
\begin{align}\label{ineq:spectral_UU}
	\left\|\bm{U}\bm{U}^\top - \bm{U}^{\star}\bm{U}^{\star\top}\right\| &\leq \left\|\bm{U}\bm{U}^\top - \bm{U}^{\sf oracle}\bm{U}^{\sf oracle\top}\right\| + \left\|\bm{U}^{\sf oracle}\bm{U}^{\sf oracle\top} - \bm{U}^{\star}\bm{U}^{\star\top}\right\|\notag\\
	&\lesssim \frac{\sqrt{\frac{\mu r}{n_1}}\left(\sqrt{n_1n_2} + n_1\right)\omega_{\sf max}^2\log^2 n}{\sigma_{r}^{\star2}} + \frac{\sqrt{n_1n_2}\omega_{\sf max}^2\log^2 n}{\sigma_r^{\star2}} + \frac{\sqrt{n_1}\omega_{\sf max}\log n}{\sigma_{r}^\star}\notag\\
	&\lesssim \frac{\sqrt{n_1n_2}\omega_{\sf max}^2\log^2 n}{\sigma_r^{\star2}} + \frac{\sqrt{n_1}\omega_{\sf max}\log n}{\sigma_{r}^\star}.
\end{align}
As an immediate consequence of \eqref{ineq:two_to_infty_UU} and Definition~\ref{assump:incoherence}, we have
\begin{align}
	\left\|\bm{U}\right\|_{2,\infty} &= \left\|\bm{U}\bm{U}^\top\right\|_{2,\infty} \leq \left\|\bm{U}\bm{U}^\top - \bm{U}^{\star}\bm{U}^{\star\top}\right\|_{2,\infty} + \left\| \bm{U}^{\star}\bm{U}^{\star\top}\right\|_{2,\infty} \notag\\
	&\leq \frac{\sqrt{\frac{\mu r}{n_1}}\left(\sqrt{n_1n_2} + n_1\right)\omega_{\sf max}^2\log^2 n}{\sigma_{r}^{\star2}} + \sqrt{\frac{\mu r}{n_1}} \leq 2\sqrt{\frac{\mu r}{n_1}}.
	\label{eq:U-inc-thm2}
\end{align}
Recalling that $\bm{R}_{\bm{U}} = {\sf sgn}(\bm{U}^\top\bm{U}^{\star})$, one can invoke \citet[Eqn.~(4.123) and Lemma 2.5]{chen2021spectral} to obtain
\begin{align}
	\left\|\bm{R}_{\bm{U}} - \bm{U}^\top\bm{U}^{\star}\right\| \leq \left\|\bm{U}\bm{U}^\top - \bm{U}^{\star}\bm{U}^{\star\top}\right\|^2.
	\label{eq:Ru-UU-ub}
\end{align}
We can then arrive at
\begin{align*}
	\left\|\bm{U}\bm{R}_{\bm{U}} - \bm{U}^{\star}\right\|_{2, \infty} &\leq \left\|\bm{U}\left(\bm{R}_{\bm{U}} - \bm{U}^\top\bm{U}^{\star}\right)\right\|_{2,\infty} + \left\|\bm{U}\bm{U}^\top\bm{U}^{\star} - \bm{U}^{\star}\right\|_{2,\infty}\\
	&\leq \left\|\bm{U}\right\|_{2,\infty}\left\|\bm{R}_{\bm{U}} - \bm{U}^\top\bm{U}^{\star}\right\| + \left\|\left(\bm{U}\bm{U}^\top - \bm{U}^{\star}\bm{U}^{\star\top}\right)\bm{U}^\star\right\|_{2,\infty}\\
	&\leq \left\|\bm{U}\right\|_{2,\infty}\left\|\bm{U}\bm{U}^\top - \bm{U}^{\star}\bm{U}^{\star\top}\right\|^2 + \left\|\bm{U}\bm{U}^\top - \bm{U}^{\star}\bm{U}^{\star\top}\right\|_{2,\infty}\left\|\bm{U}^{\star}\right\|\\
	&\lesssim \sqrt{\frac{\mu r}{n_1}}\left(\frac{\sqrt{n_1n_2}\omega_{\sf max}^2\log^2 n}{\sigma_r^{\star2}} + \frac{\sqrt{n_1}\omega_{\sf max}\log n}{\sigma_{r}^\star}\right)^2\\
	&\quad + \sqrt{\frac{\mu r}{n_1}}\left(\frac{\sqrt{n_1n_2}\omega_{\sf max}^2\log^2 n}{\sigma_r^{\star2}} + \frac{\sqrt{n_1}\omega_{\sf max}\log n}{\sigma_{r}^\star}\right)\\
	&\asymp \sqrt{\frac{\mu r}{n_1}}\left(\frac{\sqrt{n_1n_2}\omega_{\sf max}^2\log^2 n}{\sigma_r^{\star2}} + \frac{\sqrt{n_1}\omega_{\sf max}\log n}{\sigma_{r}^\star}\right), 
\end{align*}
where the third line makes use of \eqref{eq:Ru-UU-ub}, 
the fourth line invokes \eqref{ineq:two_to_infty_UU}, \eqref{ineq:spectral_UU} and \eqref{eq:U-inc-thm2}, 
and the last line results from the assumption~\eqref{assump:snr_two_to_infty}. 
In addition, inequality \eqref{ineq:spectral_UU} and $\|\bm{U}\bm{R}_{\bm{U}} - \bm{U}^{\star}\| \leq \sqrt{2}\|\bm{U}\bm{U}^\top - \bm{U}^{\star}\bm{U}^{\star\top}\|$ (see the proof of \citet[Lemma 2.6]{chen2021spectral}) taken collectively yield 
\begin{align*}
	\left\|\bm{U}\bm{R}_{\bm{U}} - \bm{U}^{\star}\right\| \lesssim \|\bm{U}\bm{U}^\top - \bm{U}^{\star}\bm{U}^{\star\top}\|
	\lesssim \frac{\sqrt{n_1n_2}\omega_{\sf max}^2\log^2 n}{\sigma_r^{\star2}} + \frac{\sqrt{n_1}\omega_{\sf max}\log n}{\sigma_{r}^\star}.
\end{align*}
This concludes the proof.

\subsection{Proof of Theorem \ref{thm:oracle}}\label{proof:thm_oracle}

Let us define the following event:
\begin{align}\label{eq:event}
	\mathcal{E}' := \{\eqref{ineq:error_moment} \text{ and } \eqref{ineq:error_moment2} \text{ hold  for } 0 \leq k \leq \log n\} \cap \{\eqref{ineq26}, \eqref{ineq78}, \eqref{ineq23-1357}, \eqref{ineq34} \text{ and } \eqref{ineq27} \text{ hold}\}.
\end{align}
Then Lemma \ref{lm:error_moment1}, Lemma \ref{lm:error_moment2}, Lemma \ref{lm:collection_event} and the union bound taken collectively imply that
\begin{align}\label{ineq:probability_event}
	\bbP\left(\mathcal{E}'\right) \geq 1 - O\left(n^{-10}\right).
\end{align}
In the rest of the proof, we shall assume that $\mathcal{E}'$ occurs unless otherwise noted.

Recall that $\bm{Z}=\mathcal{P}_{\sf off\text{-}diag}\left(\bm{E}\bm{E}^\top - \bm{E}\bm{V}^\star\bm{V}^{\star\top}\bm{E}^\top\right)$ (see \eqref{eq:M_oracle-n32}) 
and that $\widetilde{\bm{U}}\widetilde{\bm{\Sigma}}\widetilde{\bm{W}}^\top$ denotes the SVD of $\bm{U}^{\star}\bm{\Sigma}^{\star} + \bm{E}\bm{V}^{\star} \in \bbR^{n_1 \times r}$ (cf.~\eqref{eq:SVD-USigmaW}). 
In view of Lemma \ref{lm:representation}, to bound $\|\bm{U}^{\sf oracle}\bm{U}^{\sf oracle\top} - \widetilde{\bm{U}}\widetilde{\bm{U}}^\top\|_{2,\infty}$, 
it suffices to (i) bound each of the terms $\left\|\mathfrak{P}^{-j_1}\bm{Z}\mathfrak{P}^{-j_2}\bm{Z}\cdots\bm{Z}\mathfrak{P}^{-j_{k+1}}\right\|_{2,\infty}$ for $1 \leq k \leq \log n$, where $\bm{j} = [j_1,\cdots,j_{k+1}] \geq \bm{0}$ and $j_1 + \cdots + j_{k+1} = k$; and (ii) show that the total contribution of the remaining terms on the right-hand side of \eqref{ineq:representation} 
is well-controlled. Based on these ideas, our proof consists of four steps below.

\paragraph{Step 1: bounding $\|\bm{Z}^i\widetilde{\bm{U}}\|_{2,\infty}$.} 
We start by bounding a simpler term $\|\bm{Z}^i\widetilde{\bm{U}}\|_{2,\infty}$. It follows from \eqref{ineq26} that
\begin{align}\label{ineq71}
	\big\|\widetilde{\bm{\Sigma}}^{-1}\big\| \leq \frac{1}{\sigma_r^\star - \left\|\bm{E}\bm{V}^\star\right\|} \leq \frac{\sqrt{2}}{\sigma_r^\star}.
\end{align}
It is also observed from \eqref{eq:SVD-USigmaW} that
\begin{align}
	\widetilde{\bm{U}} &= \left(\bm{U}^{\star}\bm{\Sigma}^{\star} + \bm{E}\bm{V}^{\star}\right)\widetilde{\bm{W}}\widetilde{\bm{\Sigma}}^{-1}\\ &= \bm{U}^{\star}\bm{U}^{\star\top}\left(\bm{U}^{\star}\bm{\Sigma}^{\star} + \bm{E}\bm{V}^{\star}\right)\widetilde{\bm{W}}\widetilde{\bm{\Sigma}}^{-1} + \left(\bm{E}\bm{V}^{\star} - \bm{U}^{\star}\bm{U}^{\star\top}\bm{E}\bm{V}^{\star}\right)\widetilde{\bm{W}}\widetilde{\bm{\Sigma}}^{-1} \notag\\
	&= \bm{U}^{\star}\bm{U}^{\star\top}\widetilde{\bm{U}} + \left(\bm{E}\bm{V}^{\star} - \bm{U}^{\star}\bm{U}^{\star\top}\bm{E}\bm{V}^{\star}\right)\widetilde{\bm{W}}\widetilde{\bm{\Sigma}}^{-1}.
	\label{eq:identity-Utilde-decopmose}
\end{align}
As a consequence, $\bm{Z}^i\widetilde{\bm{U}}$ admits the following decomposition:
\begin{align}
	\bm{Z}^i\widetilde{\bm{U}} &= \left[\mathcal{P}_{\sf off\text{-}diag}\left(\bm{E}\bm{E}^\top - \bm{E}\bm{V}^\star\bm{V}^{\star\top}\bm{E}^\top\right)\right]^i\widetilde{\bm{U}} \notag\\
	&= -\sum_{j=0}^{i-1}\left[\mathcal{P}_{\sf off\text{-}diag}\left(\bm{E}\bm{E}^\top\right)\right]^{j}\mathcal{P}_{\sf off\text{-}diag}\left( \bm{E}\bm{V}^\star\bm{V}^{\star\top}\bm{E}^\top\right)\left[\mathcal{P}_{\sf off\text{-}diag}\left(\bm{E}\bm{E}^\top - \bm{E}\bm{V}^\star\bm{V}^{\star\top}\bm{E}^\top\right)\right]^{i-j-1}\widetilde{\bm{U}} \notag\\
	&\quad + \left[\mathcal{P}_{\sf off\text{-}diag}\left(\bm{E}\bm{E}^\top\right)\right]^{i}\widetilde{\bm{U}} \notag\\
	&= -\sum_{j=0}^{i-1}\left[\mathcal{P}_{\sf off\text{-}diag}\left(\bm{E}\bm{E}^\top\right)\right]^{j}\bm{E}\bm{V}^\star\bm{V}^{\star\top}\bm{E}^\top\left[\mathcal{P}_{\sf off\text{-}diag}\left(\bm{E}\bm{E}^\top - \bm{E}\bm{V}^\star\bm{V}^{\star\top}\bm{E}^\top\right)\right]^{i-j-1}\widetilde{\bm{U}} \notag\\
	&\quad + \sum_{j=0}^{i-1}\left[\mathcal{P}_{\sf off\text{-}diag}\left(\bm{E}\bm{E}^\top\right)\right]^{j}\mathcal{P}_{\sf diag}\left( \bm{E}\bm{V}^\star\bm{V}^{\star\top}\bm{E}^\top\right)\left[\mathcal{P}_{\sf off\text{-}diag}\left(\bm{E}\bm{E}^\top - \bm{E}\bm{V}^\star\bm{V}^{\star\top}\bm{E}^\top\right)\right]^{i-j-1}\widetilde{\bm{U}} \notag\\
	&\quad + \left[\mathcal{P}_{\sf off\text{-}diag}\left(\bm{E}\bm{E}^\top\right)\right]^{i}\bm{U}^{\star}\bm{U}^{\star\top}\widetilde{\bm{U}} \notag\\
	&\quad + \left[\mathcal{P}_{\sf off\text{-}diag}\left(\bm{E}\bm{E}^\top\right)\right]^{i}\left(\bm{E}\bm{V}^{\star} - \bm{U}^{\star}\bm{U}^{\star\top}\bm{E}\bm{V}^{\star}\right)\widetilde{\bm{W}}\widetilde{\bm{\Sigma}}^{-1}, 
	\label{eq:ZiU-decompose-long}
\end{align}
where the second identity is valid due to the following relation
\begin{align*}
	\left(\bm{A} + \bm{B}\right)^{i} = \bm{B}^{i} + \sum_{j = 0}^{i-1}\bm{B}^j\bm{A}\left(\bm{A} + \bm{B}\right)^{i-j-1}
\end{align*}
that holds for any matrices $\bm{A}, \bm{B} \in \bbR^{n_1 \times n_1}$, and the third identity in \eqref{eq:ZiU-decompose-long} arises from \eqref{eq:identity-Utilde-decopmose}. 
This allows us to bound $\|\bm{Z}^i\widetilde{\bm{U}}\|_{2,\infty}$, for any $1 \leq i \leq \log n$,  as follows:
\begin{align}\label{ineq28}
	&\big\|\bm{Z}^i\widetilde{\bm{U}}\big\|_{2,\infty}\notag\\ &\quad\leq \sum_{j=0}^{i-1}\left\|\left[\mathcal{P}_{\sf off\text{-}diag}\left(\bm{E}\bm{E}^\top\right)\right]^{j}\bm{E}\bm{V}^\star\right\|_{2,\infty}\left\|\bm{V}^{\star\top}\bm{E}^\top\left[\mathcal{P}_{\sf off\text{-}diag}\left(\bm{E}\bm{E}^\top - \bm{E}\bm{V}^\star\bm{V}^{\star\top}\bm{E}^\top\right)\right]^{i-j-1}\widetilde{\bm{U}}\right\|\notag\\
	&\qquad + \sum_{j=0}^{i-1}\left\|\left[\mathcal{P}_{\sf off\text{-}diag}\left(\bm{E}\bm{E}^\top\right)\right]^{j}\mathcal{P}_{\sf diag}\left( \bm{E}\bm{V}^\star\bm{V}^{\star\top}\bm{E}^\top\right)\left[\mathcal{P}_{\sf off\text{-}diag}\left(\bm{E}\bm{E}^\top - \bm{E}\bm{V}^\star\bm{V}^{\star\top}\bm{E}^\top\right)\right]^{i-j-1}\widetilde{\bm{U}}\right\|\notag\\
	&\qquad + \left\|\left[\mathcal{P}_{\sf off\text{-}diag}\left(\bm{E}\bm{E}^\top\right)\right]^{i}\bm{U}^{\star}\right\|_{2,\infty}\big\|\bm{U}^{\star\top}\widetilde{\bm{U}}\big\| + \left\|\left[\mathcal{P}_{\sf off\text{-}diag}\left(\bm{E}\bm{E}^\top\right)\right]^{i}\bm{E}\bm{V}^{\star}\right\|_{2,\infty}\big\|\widetilde{\bm{W}}\big\|\,\big\|\widetilde{\bm{\Sigma}}^{-1}\big\|\notag\\
	&\qquad + \left\|\left[\mathcal{P}_{\sf off\text{-}diag}\left(\bm{E}\bm{E}^\top\right)\right]^{i}\bm{U}^{\star}\right\|_{2,\infty}\left\|\bm{U}^{\star\top}\bm{E}\bm{V}^{\star}\right\|\big\|\widetilde{\bm{W}}\big\|\,\big\|\widetilde{\bm{\Sigma}}^{-1}\big\|\notag\\
	&\quad \leq \sum_{j=0}^{i-1}\left\|\left[\mathcal{P}_{\sf off\text{-}diag}\left(\bm{E}\bm{E}^\top\right)\right]^{j}\bm{E}\bm{V}^\star\right\|_{2,\infty}\left\|\bm{E}\bm{V}^{\star}\right\|\left\|\mathcal{P}_{\sf off\text{-}diag}\left(\bm{E}\bm{E}^\top - \bm{E}\bm{V}^\star\bm{V}^{\star\top}\bm{E}^\top\right)\right\|^{i-j-1}\notag\\
	&\qquad + \sum_{j=0}^{i-1}\left\|\mathcal{P}_{\sf off\text{-}diag}\left(\bm{E}\bm{E}^\top\right)\right\|^{j}\left\|\bm{E}\bm{V}^\star\right\|_{2,\infty}^2\left\|\mathcal{P}_{\sf off\text{-}diag}\left(\bm{E}\bm{E}^\top - \bm{E}\bm{V}^\star\bm{V}^{\star\top}\bm{E}^\top\right)\right\|^{i-j-1}\notag\\
	&\qquad + \left\|\left[\mathcal{P}_{\sf off\text{-}diag}\left(\bm{E}\bm{E}^\top\right)\right]^{i}\bm{U}^{\star}\right\|_{2,\infty} + \left\|\left[\mathcal{P}_{\sf off\text{-}diag}\left(\bm{E}\bm{E}^\top\right)\right]^{i}\bm{E}\bm{V}^{\star}\right\|_{2,\infty}\big\|\widetilde{\bm{\Sigma}}^{-1}\big\|\notag\\
	&\qquad + \left\|\left[\mathcal{P}_{\sf off\text{-}diag}\left(\bm{E}\bm{E}^\top\right)\right]^{i}\bm{U}^{\star}\right\|_{2,\infty}\left\|\bm{E}\bm{V}^{\star}\right\|\big\|\widetilde{\bm{\Sigma}}^{-1}\big\|\notag\\
	&\quad \leq \sum_{j=0}^{i-1}\left(C_3\sqrt{\mu r}\left(C_3\left(\sqrt{n_1n_2} + n_1\right)\omega_{\sf max}^2\log^2 n\right)^{j}\omega_{\sf max}\log n\right)\cdot C_5\sqrt{n_1}\omega_{\sf max}\log n\notag\\
	&\hspace{1.5cm}\cdot \left(3C_5\left(\sqrt{n_1n_2} + n_1\right)\omega_{\sf max}^2\log^2 n\right)^{i-j-1}\notag\\
	&\qquad + \sum_{j=0}^{i-1}\left(C_5\left(\sqrt{n_1n_2} + n_1\right)\omega_{\sf max}^2\log^2 n\right)^j\left(C_3\sqrt{\mu r}\omega_{\sf max}\log n\right)^2\cdot \left(3C_5\left(\sqrt{n_1n_2} + n_1\right)\omega_{\sf max}^2\log^2 n\right)^{i-j-1}\notag\\
	&\qquad + C_3\sqrt{\frac{\mu r}{n_1}}\left(C_3\left(\sqrt{n_1n_2} + n_1\right)\omega_{\sf max}^2\log^2 n\right)^{i}\notag\\&\qquad + C_3\sqrt{\mu r}\left(C_3\left(\sqrt{n_1n_2} + n_1\right)\omega_{\sf max}^2\log^2 n\right)^{i}\omega_{\sf max}\log n\cdot \frac{2}{\sigma_r^\star}\notag\\
	&\qquad + C_3\sqrt{\frac{\mu r}{n_1}}\left(C_3\left(\sqrt{n_1n_2} + n_1\right)\omega_{\sf max}^2\log^2 n\right)^{i}\cdot C_5\sqrt{n_1}\omega_{\sf max}\log n\cdot \frac{2}{\sigma_r^\star}\notag\\
	&\quad \leq 4C_3\sqrt{\frac{\mu r}{n_1}}\left(C_3\left(\sqrt{n_1n_2} + n_1\right)\omega_{\sf max}^2\log^2 n\right)^{i},
\end{align}
provided that $C_3 \geq 6C_5$ and $\sigma_r^\star \geq C_0\sqrt{n_1}\omega_{\sf max}\log n$. 
Here,  the first inequality relies on \eqref{eq:ZiU-decompose-long} and the triangle inequality, 
the second inequality makes use of $\|\widetilde{\bm{W}}\|=1$, 
whereas the third inequality results from \eqref{ineq:error_moment}, \eqref{ineq:error_moment2}, \eqref{ineq26}, \eqref{ineq23-1357} and \eqref{ineq71}.

\paragraph{Step 2: bounding the sum for small $k$.} 
For any $1 \leq k \leq \log n$ and any $(j_1, \dots, j_{k+1})$ satisfying $j_1, \dots, j_{k+1} \geq 0$ and $j_1 + \cdots + j_{k+1} = k$, let $\ell$ be the smallest $i$ such that $j_i \neq 0$. 
We define the matrices
\begin{equation}
	\widetilde{\mathfrak{P}}^{-j} = \widetilde{\bm{U}}\widetilde{\bm{\Sigma}}^{-2j}\widetilde{\bm{U}}^\top \quad (j\geq 1)
	\qquad \text{and} \qquad \widetilde{\mathfrak{P}}^{0} = \widetilde{\bm{U}}_{\perp}\widetilde{\bm{U}}_{\perp}^\top ,
\end{equation}
where we remind the reader that $\widetilde{\bm{\Sigma}} = {\sf diag}(\widetilde{\sigma}_1, \dots, \widetilde{\sigma}_r)$ is the diagonal matrix containing the nonzero singular values of $\bm{U}^\star\bm{\Sigma}^\star + \bm{E}\bm{V}^\star$. Noting that $\|\widetilde{\mathfrak{P}}^{-j}\| = \|\widetilde{\bm{\Sigma}}^{-1}\|^{2j}$ and $\sum_{i=1}^{k+1}j_i = \sum_{\ell=1}^{k+1}j_i = k$ (using the definition of $\ell$), one has
\begin{align}\label{ineq77}
	\prod_{i = \ell}^{k+1}\big\|\widetilde{\mathfrak{P}}^{-j_{i}}\big\| = \big\|\widetilde{\bm{\Sigma}}^{-1}\big\|^{2k}.
\end{align}
It then follows from \eqref{ineq71}, \eqref{ineq77} and the definition of $\ell$  that
\begin{align}\label{ineq79}
	\left\|\widetilde{\mathfrak{P}}^{-j_{\ell}}\bm{Z}\cdots\bm{Z}\widetilde{\mathfrak{P}}^{-j_{k+1}}\right\| \leq \left\|\bm{Z}\right\|^{k-\ell+1}\prod_{i = \ell}^{k+1}\big\|\widetilde{\mathfrak{P}}^{-j_{i}}\big\| = \left\|\bm{Z}\right\|^{k-\ell+1}\big\|\widetilde{\bm{\Sigma}}^{-1}\big\|^{2k} \leq \left\|\bm{Z}\right\|^{k-\ell+1}\left(\frac{2}{\sigma_r^{\star2}}\right)^k
\end{align}
and for $1 \leq i \leq \ell-1$, 
\begin{align}\label{ineq80}
	\left\|\bm{Z}\widetilde{\mathfrak{P}}^{-j_{i+1}}\bm{Z}\cdots\bm{Z}\widetilde{\mathfrak{P}}^{-j_{k+1}}\right\| \leq \left\|\bm{Z}\right\|^{k-i+1}\big\|\widetilde{\bm{\Sigma}}^{-1}\big\|^{2k} \leq \left\|\bm{Z}\right\|^{k-i+1}\left(\frac{2}{\sigma_r^{\star2}}\right)^k.
\end{align}
We can see from the definition of $\ell$ and $\widetilde{\mathfrak{P}}^{0}$ that
\begin{align}\label{ineq81}
	\widetilde{\mathfrak{P}}^{-j_1}\bm{Z}\widetilde{\mathfrak{P}}^{-j_2}\bm{Z}\cdots\bm{Z}\widetilde{\mathfrak{P}}^{-j_{k+1}} = \bm{Z}^{\ell-1}\widetilde{\mathfrak{P}}^{-j_{\ell}}\bm{Z}\cdots\bm{Z}\widetilde{\mathfrak{P}}^{-j_{k+1}} - \sum_{i=1}^{\ell-1}\bm{Z}^{i-1}\bm{P}_{\widetilde{\bm{U}}}\bm{Z}\widetilde{\mathfrak{P}}^{-j_{i+1}}\bm{Z}\cdots\bm{Z}\widetilde{\mathfrak{P}}^{-j_{k+1}},
\end{align}
which allows us to derive
\begin{align}
	&\left\|\widetilde{\mathfrak{P}}^{-j_1}\bm{Z}\widetilde{\mathfrak{P}}^{-j_2}\bm{Z}\cdots\bm{Z}\widetilde{\mathfrak{P}}^{-j_{k+1}}\right\|_{2,\infty}\notag\\
	&\quad \leq \left\|\bm{Z}^{\ell-1}\widetilde{\mathfrak{P}}^{-j_{\ell}}\bm{Z}\cdots\bm{Z}\widetilde{\mathfrak{P}}^{-j_{k+1}}\right\|_{2,\infty} + \sum_{i=1}^{\ell-1}\left\|\bm{Z}^{i-1}\bm{P}_{\widetilde{\bm{U}}}\bm{Z}\widetilde{\mathfrak{P}}^{-j_{i+1}}\bm{Z}\cdots\bm{Z}\widetilde{\mathfrak{P}}^{-j_{k+1}}\right\|_{2,\infty}\notag\\
	&\quad \leq \big\|\bm{Z}^{\ell-1}\widetilde{\bm{U}}\big\|_{2,\infty}\left\|\widetilde{\mathfrak{P}}^{-j_{\ell}}\bm{Z}\cdots\bm{Z}\widetilde{\mathfrak{P}}^{-j_{k+1}}\right\| + \sum_{i=1}^{\ell-1}\big\|\bm{Z}^{i-1}\widetilde{\bm{U}}\big\|_{2,\infty}\left\|\bm{Z}\widetilde{\mathfrak{P}}^{-j_{i+1}}\bm{Z}\cdots\bm{Z}\widetilde{\mathfrak{P}}^{-j_{k+1}}\right\|\notag\\
	&\quad \leq \left[4C_3\sqrt{\frac{\mu r}{n_1}}\left(C_3\left(\sqrt{n_1n_2} + n_1\right)\omega_{\sf max}^2\log^2 n\right)^{\ell-1}\right]\cdot \frac{\left( 3C_5\left(\sqrt{n_1n_2} + n_1\right)\omega_{\sf max}^2\log^2 n\right)^{k-\ell+1}}{\left(\sigma_r^{\star2}/2\right)^k}\notag\\
	&\qquad + \sum_{i=1}^{\ell-1} \left[4C_3\sqrt{\frac{\mu r}{n_1}}\left(C_3\left(\sqrt{n_1n_2} + n_1\right)\omega_{\sf max}^2\log^2 n\right)^{i-1}\right]\cdot \frac{\left( 3C_5\left(\sqrt{n_1n_2} + n_1\right)\omega_{\sf max}^2\log^2 n\right)^{k-i+1}}{\left(\sigma_r^{\star2}/2\right)^k}\notag\\
	&\quad \leq 4C_3\sqrt{\frac{\mu r}{n_1}}\left(\frac{C_3\left(\sqrt{n_1n_2} + n_1\right)\omega_{\sf max}^2\log^2 n}{\sigma_r^{\star2}}\right)^k\cdot \sum_{i=1}^{\ell}\frac{1}{2^{k-i+1}}\notag\\
	&\quad \leq 8C_3\sqrt{\frac{\mu r}{n_1}}\left(\frac{C_3\left(\sqrt{n_1n_2} + n_1\right)\omega_{\sf max}^2\log^2 n}{\sigma_r^{\star2}}\right)^k\label{ineq75}.
\end{align}
Here, the first inequality comes from \eqref{ineq81} and the triangle inequality, the second inequality holds due to the definition of $\ell$, the basic inequality $\|\bm{A}\bm{B}\|_{2, \infty} \leq \|\bm{A}\|_{2, \infty}\|\bm{B}\|$ and the fact $\|\widetilde{\bm{U}}^\top\bm{C}\| \leq \|\bm{C}\|$, the third inequality is a consequence of \eqref{ineq23-1357}, \eqref{ineq28}, \eqref{ineq79} and \eqref{ineq80}, and the second last inequality is valid as long as $C_3 \geq 12C_5$.

\paragraph{Step 3: bounding the sum for large $k$.} For any $k \geq \lfloor\log n\rfloor + 1$, the signal-to-noise condition \eqref{ineq75a} implies that there exists a large constant $C > 0$ such that
\begin{align*}
	\left(\frac{3C_5\left(\sqrt{n_1n_2} + n_1\right)\omega_{\sf max}^2\log^2 n}{\sigma_r^{\star2}/2}\right)^{k-1} \leq \left(\frac{1}{C^2}\right)^{k-1} \leq \frac{1}{C^k}.
\end{align*}
It is also seen that
\begin{align*}
	\Big|\big\{(j_1, \dots, j_{k+1}): j_1, \dots, j_{k+1} \geq 0\text{ and } j_1 + \cdots + j_{k+1} = k\big\}\Big| = \binom{2k+1}{k} \leq 4^k,
\end{align*}
In view of \eqref{ineq23-1357}, \eqref{ineq71} and \eqref{ineq77}, we have
\begin{align}
	&\sum_{j_1, \dots, j_{k+1} \geq 0\atop j_1 + \cdots + j_{k+1} = k}\left\|\widetilde{\mathfrak{P}}^{-j_1}\bm{Z}\widetilde{\mathfrak{P}}^{-j_2}\bm{Z}\cdots\bm{Z}\widetilde{\mathfrak{P}}^{-j_{k+1}}\right\|_{2,\infty}\notag\\
	&\quad \leq \sum_{j_1, \dots, j_{k+1} \geq 0\atop j_1 + \cdots + j_{k+1} = k}\big\|\widetilde{\mathfrak{P}}^{-j_1}\big\|\left\|\bm{Z}\right\|\big\|\widetilde{\mathfrak{P}}^{-j_2}\big\|\left\|\bm{Z}\right\|\cdots\left\|\bm{Z}\right\|\big\|\widetilde{\mathfrak{P}}^{-j_{k+1}}\big\|\notag\\
	&\quad \leq \sum_{j_1, \dots, j_{k+1} \geq 0\atop j_1 + \cdots + j_{k+1} = k}\left(\frac{3C_5\left(\sqrt{n_1n_2} + n_1\right)\omega_{\sf max}^2\log^2 n}{\sigma_r^{\star2}/2}\right)^k\notag\\
	&\quad \leq \left(\frac{4}{C}\right)^k\frac{6C_5\left(\sqrt{n_1n_2} + n_1\right)\omega_{\sf max}^2\log^2 n}{\sigma_r^{\star2}}\label{ineq76}.
\end{align}

\paragraph{Step 4: bounding $\|\bm{U}^{\sf oracle}\bm{U}^{\sf oracle\top} - \widetilde{\bm{U}}\widetilde{\bm{U}}^\top\|_{2,\infty}$ and $\|\bm{U}^{\sf oracle}\bm{U}^{\sf oracle\top} - \bm{U}^\star\bm{U}^{\star\top}\|$.} By virtue of \eqref{ineq75}, \eqref{ineq76} and Lemma \ref{lm:representation}, we reach
\begin{align}\label{ineq:two_to_infty_representation}
	\big\|\bm{U}^{\sf oracle}\bm{U}^{\sf oracle\top} - \widetilde{\bm{U}}\widetilde{\bm{U}}^\top\big\|_{2,\infty}
	&\quad \leq \sum_{1 \leq k \leq \log n}8C_3\sqrt{\frac{\mu r}{n_1}}\left(\frac{C_3\left(\sqrt{n_1n_2} + n_1\right)\omega_{\sf max}^2\log^2 n}{\sigma_r^{\star2}}\right)^k\notag\\
	&\qquad + \sum_{k \geq \lfloor\log n\rfloor + 1}\left(\frac{4}{C}\right)^k\frac{6C_5\left(\sqrt{n_1n_2} + n_1\right)\omega_{\sf max}^2\log^2 n}{\sigma_r^{\star2}}\notag\\
	&\quad \lesssim \sqrt{\frac{\mu r}{n_1}}\frac{\left(\sqrt{n_1n_2} + n_1\right)\omega_{\sf max}^2\log^2 n}{\sigma_r^{\star2}}.
\end{align}
In addition,  the $\sin\Theta$ theorem \citep[Chapter 2]{chen2021spectral} shows that
\begin{align}
	\big\|\bm{U}^{\star\top}\widetilde{\bm{U}}_{\perp}\big\| & =\big\|\widetilde{\bm{U}}\widetilde{\bm{U}}^{\top}-\bm{U}^{\star}\bm{U}^{\star\top}\big\|\lesssim\frac{\left\Vert \bm{E}\bm{V}^{\star}\right\Vert }{\sigma_{r}^{\star}}\lesssim\frac{\sqrt{n_{1}}\omega_{\sf max}\log n+B\sqrt{\frac{\mu_{2}r}{n_{2}}}\log^{2}n}{\sigma_{r}^{\star}}\notag\\
 & \asymp\frac{\sqrt{n_{1}}\omega_{\sf max}\log n}{\sigma_{r}^{\star}}\label{ineq82}, 
\end{align}
where the first identity makes use of \citet[Lemma 2.5]{chen2021spectral}, the penultimate inequality results from Lemma~\ref{lm:noise}, 
and the last relation comes from Assumption \ref{assump:hetero2} and \eqref{ineq75b}. 
Moreover, applying \eqref{ineq34} and the previous inequality yields that
\begin{align*}
	\big\|\widetilde{\bm{U}}\widetilde{\bm{U}}^\top - \bm{U}^\star\bm{U}^{\star\top}\big\|_{2,\infty} 
	&\leq \big\|\big(\widetilde{\bm{U}} - \bm{U}^\star\bm{U}^{\star\top}\widetilde{\bm{U}}\big)\widetilde{\bm{U}}^\top\big\|_{2,\infty} + \big\|\bm{U}^\star\bm{U}^{\star\top}\widetilde{\bm{U}}\widetilde{\bm{U}}^\top -  \bm{U}^\star\bm{U}^{\star\top}\big\|_{2,\infty}\\ 
	&= \big\|\big(\widetilde{\bm{U}} - \bm{U}^\star\bm{U}^{\star\top}\widetilde{\bm{U}}\big)\widetilde{\bm{U}}^\top\big\|_{2,\infty} + \big\|\bm{U}^\star\bm{U}^{\star\top}\widetilde{\bm{U}}_{\perp}\widetilde{\bm{U}}_{\perp}^\top\big\|_{2,\infty}\\
	&\leq \big\|\widetilde{\bm{U}} - \bm{U}^\star\bm{U}^{\star\top}\widetilde{\bm{U}}\big\|_{2,\infty} + \left\|\bm{U}^\star\right\|_{2,\infty}\big\|\bm{U}^{\star\top}\widetilde{\bm{U}}_{\perp}\big\|\\
	&\lesssim \frac{\sqrt{\mu r}\omega_{\sf max}\log n}{\sigma_r^\star} + \sqrt{\frac{\mu r}{n_1}}\frac{\sqrt{n_1}\omega_{\sf max}\log n}{\sigma_{r}^\star}\\
	&\asymp \sqrt{\frac{\mu r}{n_1}}\frac{\sqrt{n_1}\omega_{\sf max}\log n}{\sigma_{r}^\star}.
\end{align*}
This taken collectively with \eqref{ineq:two_to_infty_representation} gives
\begin{align*}
	\left\|\bm{U}^{\sf oracle}\bm{U}^{\sf oracle\top} - \bm{U}^\star\bm{U}^{\star\top}\right\|_{2,\infty} &\lesssim \sqrt{\frac{\mu r}{n_1}}\left(\frac{\left(\sqrt{n_1n_2} + n_1\right)\omega_{\sf max}^2\log^2 n}{\sigma_r^{\star2}} + \frac{\sqrt{n_1}\omega_{\sf max}\log n}{\sigma_{r}^\star}\right)\\
	&\asymp \sqrt{\frac{\mu r}{n_1}}\left(\frac{\sqrt{n_1n_2}\omega_{\sf max}^2\log^2 n}{\sigma_r^{\star2}} + \frac{\sqrt{n_1}\omega_{\sf max}\log n}{\sigma_{r}^\star}\right),
\end{align*}
where the last relation results from the assumption \eqref{ineq75a}.

Finally, the Davis-Kahan Theorem, \eqref{ineq26} and \eqref{ineq23-1357} together show that
\begin{align*}
	\left\|\bm{U}^{\sf oracle}\bm{U}^{\sf oracle\top} - \bm{U}^\star\bm{U}^{\star\top}\right\| &\leq \big\|\bm{U}^{\sf oracle}\bm{U}^{\sf oracle\top} - \widetilde{\bm{U}}\widetilde{\bm{U}}^\top\big\| + \big\|\widetilde{\bm{U}}\widetilde{\bm{U}}^\top - \bm{U}^\star\bm{U}^{\star\top}\big\|\\
	&\lesssim \frac{\left(\sqrt{n_1n_2} + n_1\right)\omega_{\sf max}^2\log^2 n}{\sigma_r^{\star2}} + \frac{\sqrt{n_1}\omega_{\sf max}\log n}{\sigma_{r}^\star}\\
	&\lesssim \frac{\sqrt{n_1n_2}\omega_{\sf max}^2\log^2 n}{\sigma_r^{\star2}} + \frac{\sqrt{n_1}\omega_{\sf max}\log n}{\sigma_{r}^\star}.
\end{align*}
Here, we have used the triangle inequality in the first inequality, the second inequality comes from \eqref{ineq82}, the Davis-Kahan Theorem, \eqref{ineq78} and \eqref{ineq23-1357}, 
whereas the last inequality holds since
\begin{align*}
	 \frac{n_1\omega_{\sf max}^2\log^2 n}{\sigma_r^{\star2}} \lesssim \frac{\sqrt{n_1}\omega_{\sf max}\log n}{\sigma_{r}^\star}
\end{align*}
under our signal-to-noise condition \eqref{ineq75a}. This concludes the proof.

\subsection{Proof of Lemma \ref{lm:error_moment1}}\label{proof_lm:error_moment1}
To streamline the presentation, we divide the proof into several steps. We shall start by considering the case with bounded noise  (i.e., the case with $|E_{i,j}| \leq B$ deterministically) and develop upper bounds on both $\big\|\left[\mathcal{P}_{\sf off\text{-}diag}\left(\bm{E}\bm{E}^\top\right)\right]^{\ell}\bm{E}\bm{V}^\star\big\|_{2,\infty}$ and $\big\|\bm{E}^\top\left[\mathcal{P}_{\sf off\text{-}diag}\left(\bm{E}\bm{E}^\top\right)\right]^\ell\bm{E}\bm{V}^\star\big\|_{2,\infty}$ via induction. 
We will then move on to the general case and establish the final result by means of a truncation trick.

\subsubsection{The case with bounded noise}

Let us now focus on the case where
\begin{align}\label{ineq64}
	\left|E_{i,j}\right| \leq B \leq C_{\sf b}\omega_{\sf max}\frac{\min\big\{\left(n_1n_2\right)^{1/4}, \sqrt{n_2}\big\}}{\log n}, \quad \forall (i,j) \in [n_1] \times [n_2] 
\end{align}
holds deterministically. We would like to prove, by induction, the following slightly stronger claims: 
suppose that $\bm{E}$ satisfies Conditions~1 and 2 in Assumption~\ref{assump:hetero} and \eqref{ineq64}, 
then for any $0 \leq k \leq \log n$, with probability exceeding $1 - O\big((n+3)^{2k}n^{-C_2\log n}\big)$ one has 
\begin{align}\label{ineq2}
	&\left\|\left[\mathcal{P}_{\sf off\text{-}diag}\left(\bm{E}\bm{E}^\top\right)\right]^{\ell}\bm{E}\bm{V}^\star\right\|_{2,\infty}\leq C_3\sqrt{\frac{\mu r}{n_2}}\left(C_3\left(\sqrt{n_1n_2} + n_1\right)\omega_{\sf max}^2\log^2 n\right)^{\ell}\sqrt{n_2}\omega_{\sf max}\log n 
\end{align}
and
\begin{align}\label{ineq3}
	&\left\|\bm{E}^\top\left[\mathcal{P}_{\sf off\text{-}diag}\left(\bm{E}\bm{E}^\top\right)\right]^\ell\bm{E}\bm{V}^\star\right\|_{2,\infty}\notag\\ &\quad\leq C_4\sqrt{\frac{\mu r}{n_2}}\left(C_3\left(\sqrt{n_1n_2} + n_1\right)\omega_{\sf max}^2\log^2 n\right)^{\ell}\left(\sqrt{n_2}B\omega_{\sf max}\log n + \left(\sqrt{n_1n_2} + n_1\right)\omega_{\sf max}^2\right)\log^2 n 
\end{align}
for all $0 \leq \ell \leq k$. 
Here, $C_3,C_4>0$ are some large numerical constants to be specified shortly.

\paragraph{Step 1: base case.} Let us first look at the base case with $k = 0$. 
It follows from Lemma~\ref{lm:1} and the assumption \eqref{ineq64} that: for any fixed matrices $\bm{W}_1$ with $n_2$ rows and any $\bm{W}_2$ with $n_1$ rows, one has
\begin{subequations}
	\begin{align}
		\max_{i \in [n_1]}\sum_{j \in [n_2]}E_{i,j}^2 &\lesssim B^2\log^2 n + \omega_{\sf row}^2 \lesssim n_2\omega_{\sf max}^2 \label{ineq4a}\\
		\max_{i \in [n_1]}\left\|\bm{E}_{i,:}\bm{W}_1\right\|_2 &\lesssim B\left\|\bm{W}_1\right\|_{2,\infty}\log^2 n + \omega_{\sf max}\left\|\bm{W}_1\right\|_{\F}\log n \lesssim \sqrt{n_2}\omega_{\sf max}\left\|\bm{W}_1\right\|_{2,\infty}\log n \label{ineq4b}\\
		\max_{j \in [n_2]}\sum_{i \in [n_1]}E_{i,j}^2 &\lesssim B^2\log^2 n + \omega_{\sf col}^2 \label{ineq4c}\\
		\max_{j \in [n_2]}\big\|\left(\bm{E}_{:,j}\right)^\top\bm{W}_2\big\|_2 &\lesssim \left(B\log^2 n + \omega_{\sf col}\log n\right)\left\|\bm{W}_2\right\|_{2,\infty}\label{ineq4d}
	\end{align}
\end{subequations}
with probability exceeding $1 - O(n^{-C_4\log n})$ for some numerical constant $C_4>0$. 
Inequality \eqref{ineq4b} combined with Definition~\ref{assump:incoherence} tells us that with probability at least $1 - O(n^{-C_4\log n})$,
\begin{align}\label{ineq5}
	\left\|\bm{E}\bm{V}^\star\right\|_{2,\infty} \lesssim \sqrt{n_2}\omega_{\sf max}\left\|\bm{V}^\star\right\|_{2,\infty}\log n \leq \sqrt{\mu r}\omega_{\sf max}\log n.
\end{align}

In addition, for any $j \in [n_2]$, we can decompose $\bm{E}_{:,j}^\top\bm{E}\bm{V}^\top$ into two terms:
\begin{align}\label{eq:decomposition_EEV}
	\bm{E}_{:,j}^\top\bm{E}\bm{V}^\star = \bm{E}_{:,j}^\top \bm{E}^{(:, -j)}\bm{V}^\star + \bm{E}_{:,j}^\top \bm{E}^{(:, j)}\bm{V}^\star.
\end{align}
Here, $\bm{E}^{(:, -j)}$ and $\bm{E}^{(:, j)}$ are defined as 
\begin{align*}
	\bm{E}^{(:, -j)} = \mathcal{P}_{:, -j}\left(\bm{E}\right) \in \bbR^{n_1 \times n_2} \quad \text{and} \quad \bm{E}^{(:, j)} = \mathcal{P}_{:, j}\left(\bm{E}\right) \in \bbR^{n_1 \times n_2},
\end{align*} 
where $\mathcal{P}_{:, -j}(\cdot)$ (resp.~$\mathcal{P}_{:, j}(\cdot)$) is a projection operator that zeros out the $j$-th column (resp.~all entries except those in the $j$-th column) of a matrix, i.e., for any matrix $\bm{A}$, 
\begin{align}\label{def:projection_column}
\left[\mathcal{P}_{:, -j}(\bm{A})\right]_{i, k} = \begin{cases}
	A_{i, k}, \quad & \text{if } k \neq j,\\
	0, \quad & \text{otherwise}, 
\end{cases} \quad \forall (i, k) \in [n_1] \times [n_2], \quad \text{and} \quad \mathcal{P}_{:, j}(\bm{A}) = \bm{A} - \mathcal{P}_{:, -j}(\bm{A}).
\end{align}
In view of  \eqref{ineq4b} and \eqref{ineq4d}, with probability exceeding $1 - O(n^{-C_2\log n})$,
\begin{align*}
	\big\|\bm{E}_{:,j}^\top \bm{E}^{(:, -j)}\bm{V}^\star\big\|_2 &\lesssim \left(B\log^2 n + \omega_{\sf col}\log n\right)\big\| \bm{E}^{(:, -j)}\bm{V}^\star\big\|_{2,\infty}\\
	&\lesssim \left(B\log n + \omega_{\sf col}\right)\sqrt{\mu r}\omega_{\sf max}\log^2 n,
\end{align*}
where the last inequality can be derived in a way similar to \eqref{ineq5}. 
Recognizing that $(\bm{E}_{:,j}^\top \bm{E}^{(:, j)})^\top$ is a vector with only one nonzero entry $\|\bm{E}_{:,j}\|_2^2$, we know from \eqref{ineq4c} and Definition~\ref{assump:incoherence} that, with probability at least $1 - O(n^{-C_4\log n})$,
\begin{align*}
	\big\|\bm{E}_{:,j}^\top \bm{E}^{(:, j)}\bm{V}^\star\big\|_{2} \leq \left\|\bm{E}_{:,j}\right\|_2^2\left\|\bm{V}^\star\right\|_{2,\infty} \lesssim \left(B^2\log^2 n + \omega_{\sf col}^2\right)\sqrt{\frac{\mu r}{n_2}}.
\end{align*}
Taking the previous two inequalities and \eqref{eq:decomposition_EEV} together and applying  the union bound imply that, with probability at least $1 - O(n^{-C_2\log n})$,
\begin{align*}
	\left\|\bm{E}^\top\bm{E}\bm{V}^\star\right\|_{2,\infty} &\lesssim \left(\sqrt{n_2}B\omega_{\sf max}\log^3 n + \sqrt{n_2}\omega_{\sf col}\omega_{\sf max}\log^2 n + B^2\log^2 n + \omega_{\sf col}^2\right)\sqrt{\frac{\mu r}{n_2}}\\
	&\lesssim \sqrt{\frac{\mu r}{n_2}}\left(\sqrt{n_2}B\omega_{\sf max}\log n + \left(\sqrt{n_1n_2} + n_1\right)\omega_{\sf max}^2\right)\log^2 n, 
\end{align*}
where we have also made use of the assumption \eqref{ineq64}.

Therefore, we have established both \eqref{ineq2} and \eqref{ineq3} for the base case with $k = 0$.

\paragraph{Step 2: inductive step.} 
We now move on to the inductive step. 
Suppose that for any $\bm{E}$ satisfying Conditions~1 and 2 in Assumption \ref{assump:hetero} and \eqref{ineq64}, 
the induction hypotheses \eqref{ineq2} and \eqref{ineq3} hold for all $1 \leq \ell \leq K$ 
with probability exceeding $1 - O((n+2)^{2K}\cdot n^{-C_2\log n})$. 
We intend to justify that these induction hypotheses continue to be valid for $K+1$.

\paragraph{Step 2.1: bounding $\|[\mathcal{P}_{\sf off\text{-}diag}(\bm{E}\bm{E}^\top)]^{K+1}\bm{E}\bm{V}^\star\|_{2,\infty}$.} 
We first look at the quantity of interest in \eqref{ineq2}. 
For any $i \in [n_1]$, define
\begin{align*}
	\bm{E}^{(-i, :)} = \mathcal{P}_{-i, :}(\bm{E}) \in \bbR^{n_1 \times n_2} \quad \text{and} \quad \bm{E}^{(i, :)} = \mathcal{P}_{i, :}(\bm{E}) \in \bbR^{n_1 \times n_2}.
\end{align*}
Here, $\mathcal{P}_{-i, :}(\bm{A})$ (resp.~$\mathcal{P}_{i, :}(\bm{A})$) zeros out the $i$-th row (resp.~all entries except the ones in the $i$-th row) of $\bm{A}$, namely,
\begin{align}\label{def:projection_row}
    \left[\mathcal{P}_{-i, :}(\bm{A})\right]_{j, k} = \begin{cases}
    	A_{j, k}, \quad & \text{if } j \neq i,\\
    	0, \quad & \text{otherwise}, 
    \end{cases} \quad \forall (j, k) \in [n_1] \times [n_2], \quad \text{and} \quad \mathcal{P}_{i, :}(\bm{A}) = \bm{A} - \mathcal{P}_{-i, :}(\bm{A}).
\end{align}
When it comes to $k = K+1$, recognizing the identity
\begin{align*}
	\left[\mathcal{P}_{\sf off\text{-}diag}\left(\bm{E}\bm{E}^\top\right)\right]_{i,:} = \bm{E}_{i,:}\bm{E}^{(-i, :)\top},
\end{align*}
we can derive
\begin{align}\label{ineq12}
	\left[\left[\mathcal{P}_{\sf off\text{-}diag}\left(\bm{E}\bm{E}^\top\right)\right]^{K+1}\bm{E}\bm{V}^\star\right]_{i,:} = \bm{E}_{i,:}\bm{E}^{(-i, :)\top}\left[\mathcal{P}_{\sf off\text{-}diag}\left(\bm{E}\bm{E}^\top\right)\right]^{K}\bm{E}\bm{V}^\star.
\end{align}

We claim for the moment that
\begin{align}\label{ineq11}
	&\left\|\bm{E}_{i,:}\bm{E}^{(-i, :)\top}\left[\mathcal{P}_{\sf off\text{-}diag}\left(\bm{E}\bm{E}^\top\right)\right]^{K}\bm{E}\bm{V}^\star\right\|_{2}\notag\\
	&\quad\leq \underbrace{\left\|\bm{E}_{i,:}\bm{E}^{(-i, :)\top}\big[\mathcal{P}_{\sf off\text{-}diag}\big(\bm{E}^{(-i, :)}\bm{E}^{(-i, :)\top}\big)\big]^K\bm{E}^{(-i, :)}\bm{V}^\star\right\|_{2}}_{=: \tau_1}\notag\\
	&\qquad + \underbrace{\sum_{\ell=0}^{K-1}\big\|\bm{E}_{i,:}\bm{E}^{(-i, :)\top}\big\|_2^2\big\|\mathcal{P}_{\sf off\text{-}diag}\big(\bm{E}^{(-i, :)}\bm{E}^{(-i, :)\top}\big)\big\|^\ell\left\|\left[\mathcal{P}_{\sf off\text{-}diag}\left(\bm{E}\bm{E}^\top\right)\right]^{K-1-\ell}\bm{E}\bm{V}^\star\right\|_{2,\infty}}_{=: \tau_2},
\end{align}
which we shall prove towards the end of the proof for the bounded noise case. 
We define the following event
\begin{align*}
	\mathcal{E}_1 = \bigg\{&\forall 0 \leq \ell \leq K-1,\quad \left\|\left[\mathcal{P}_{\sf off\text{-}diag}\left(\bm{E}\bm{E}^\top\right)\right]^\ell\bm{E}\bm{V}^\star\right\|_{2,\infty}\notag\\ &\qquad \leq C_3\sqrt{\frac{\mu r}{n_2}}\left(C_3\left(\sqrt{n_1n_2} + n_1\right)\omega_{\sf max}^2\log^2 n\right)^{\ell}\sqrt{n_2}\omega_{\sf max}\log n,\\
	&\forall i \in [n_1],\quad \left\|\bm{E}^{(-i, :)\top}\big[\mathcal{P}_{\sf off\text{-}diag}\big(\bm{E}^{(-i, :)}\bm{E}^{(-i, :)\top}\big)\big]^K\bm{E}^{(-i, :)}\bm{V}^\star\right\|_{2,\infty}\\&\qquad \leq C_4\sqrt{\frac{\mu r}{n_2}}\left(C_3\left(\sqrt{n_1n_2} + n_1\right)\omega_{\sf max}^2\log^2 n\right)^{K}\left(\sqrt{n_2}B\omega_{\sf max}\log n + \left(\sqrt{n_1n_2} + n_1\right)\omega_{\sf max}^2\right)\log^2 n\bigg\}.
\end{align*}
Recognizing that $\bm{E}^{(-i, :)}$ satisfies Conditions~1 and 2 in Assumption \ref{assump:hetero} and \eqref{ineq64} as well, 
we learn from our induction hypotheses and the union bound that
\begin{align*}
	\bbP\left(\mathcal{E}_1\right) \geq 1 - \left(n_1 + 1\right)\cdot C_1(n+3)^{2K}n^{-C_2\log n}.
\end{align*}
Moreover, Lemma \ref{lm:off-diag} asserts that with probability exceeding $1 - O(n^{-C_2\log n})$,
\begin{align}\label{ineq4}
	\left\|\mathcal{P}_{\sf off\text{-}diag}\left(\bm{E}\bm{E}^\top\right)\right\| \leq C_5\left(\sqrt{n_1n_2} + n_1\right)\omega_{\sf max}^2\log^2 n.
\end{align}
Given that $\bm{E}_{i,:}\bm{E}^{(-i, :)\top}$ is the $i$-th row of $\mathcal{P}_{\sf off\text{-}diag}(\bm{E}\bm{E}^\top)$ and $\mathcal{P}_{\sf off\text{-}diag}(\bm{E}^{(-i, :)}\bm{E}^{(-i, :)\top})$ is a submatrix of $\mathcal{P}_{\sf off\text{-}diag}(\bm{E}\bm{E}^\top)$, the inequality \eqref{ineq4} implies that 
\begin{align}\label{ineq10}
	\max\left\{\big\|\bm{E}_{i,:}\bm{E}^{(-i, :)\top}\big\|_2, \big\|\mathcal{P}_{\sf off\text{-}diag}\big(\bm{E}^{(-i, :)}\bm{E}^{(-i, :)\top}\big)\big\|\right\} \leq C_5\left(\sqrt{n_1n_2} + n_1\right)\omega_{\sf max}^2\log^2 n.
\end{align}
Armed with these results, we proceed to bound $\tau_1$ and $\tau_2$ in \eqref{ineq11} separately in the sequel.

\begin{itemize}

\item {\em Bounding $\tau_1$.} Note that $\bm{E}^{(-i, :)\top}[\mathcal{P}_{\sf off\text{-}diag}(\bm{E}^{(-i, :)}\bm{E}^{(-i, :)\top})]^K\bm{E}^{(-i, :)}\bm{V}^\star$ is statistically independent of $\bm{E}_{i,:}$. In view of \eqref{ineq4b}, with probability exceeding $1 - O(n^{-C_2\log n})$, one has
\begin{align}\label{ineq9}
	\tau_1
	&\leq C_5B\left\|\bm{E}^{(-i, :)\top}\big[\mathcal{P}_{\sf off\text{-}diag}\big(\bm{E}^{(-i, :)}\bm{E}^{(-i, :)\top}\big)\big]^K\bm{E}^{(-i, :)}\bm{V}^\star\right\|_{2,\infty}\log^2 n\notag\\&\quad+ C_5\omega_{\sf max}\left\|\bm{E}^{(-i, :)\top}\big[\mathcal{P}_{\sf off\text{-}diag}\big(\bm{E}^{(-i, :)}\bm{E}^{(-i, :)\top}\big)\big]^K\bm{E}^{(-i, :)}\bm{V}^\star\right\|_{\F}\log n
\end{align}
for some suitable universal constants $C_2,C_5>0$. 
We have also learned from Lemma \ref{lm:noise} that
\begin{align}\label{ineq7}
	\big\|\bm{E}^{(-i, :)}\bm{V}^\star\big\| \leq \left\|\bm{E}\bm{V}^\star\right\| \leq C_5\left(B\sqrt{\frac{\mu r}{n_2}}\log^2 n + \sqrt{n_1}\omega_{\sf max}\log n\right)
\end{align}
and
\begin{align}\label{ineq8}
	\big\|\bm{E}^{(-i, :)}\big\| \leq \left\|\bm{E}\right\| \leq C_5\big( \sqrt{n_2}\omega_{\sf max} + \sqrt{n_1}\omega_{\sf max}\big)
\end{align}
 with probability exceeding $1 -  O(n^{-C_2\log n})$, provided that $C_5$ is large enough. 
Let $\mathcal{E}_2$ denote the event $\mathcal{E}_2 = \{\eqref{ineq4}, \eqref{ineq9}, \eqref{ineq7} \text{ and } \eqref{ineq8} \text{ hold}\}$. Then $\bbP(\mathcal{E}_2) \geq 1 - O(n^{-C_2\log n})$ and, consequently, 
\begin{align}\label{ineq:probability_1}
	\mathcal{P}\left(\mathcal{E}_1 \cap \mathcal{E}_2\right) \geq 1 - C_1(n_1+2)(n+3)^{2K}n^{-C_2\log n}.
\end{align}
On the event $\mathcal{E}_1$, one has
\begin{align*}
	&\left\|\bm{E}^{(-i, :)\top}\big[\mathcal{P}_{\sf off\text{-}diag}\big(\bm{E}^{(-i, :)}\bm{E}^{(-i, :)\top}\big)\big]^K\bm{E}^{(-i, :)}\bm{V}^\star\right\|_{2,\infty}\\
	&\quad \leq C_4\sqrt{\frac{\mu r}{n_2}}\left(C_3\left(\sqrt{n_1n_2} + n_1\right)\omega_{\sf max}^2\log^2 n\right)^{K}\left(\sqrt{n_2}B\omega_{\sf max}\log n + \left(\sqrt{n_1n_2} + n_1\right)\omega_{\sf max}^2\right)\log^2 n.
\end{align*}
In view of \eqref{ineq10},  \eqref{ineq9}, \eqref{ineq7}, \eqref{ineq8}, the previous inequality and the assumption~\eqref{ineq64}, on the event $\mathcal{E}_1 \cap \mathcal{E}_2$ we have
\begin{align*}
	\tau_1
	&\leq C_5B\cdot C_4\sqrt{\frac{\mu r}{n_2}}\left(C_3\left(\sqrt{n_1n_2} + n_1\right)\omega_{\sf max}^2\log^2 n\right)^{K}\left(\sqrt{n_2}B\omega_{\sf max}\log n + \left(\sqrt{n_1n_2} + n_1\right)\omega_{\sf max}^2\right)\log^4 n\\
	&\quad + \sqrt{r}\cdot C_5\omega_{\sf max}\cdot \big\|\bm{E}^{(-i, :)}\big\|\big\|\mathcal{P}_{\sf off\text{-}diag}\big(\bm{E}^{(-i, :)}\bm{E}^{(-i, :)\top}\big)\big\|^K\big\|\bm{E}^{(-i, :)}\bm{V}^\star\big\|\log n\\
	&\leq C_4\sqrt{\frac{\mu r}{n_2}}\left(C_3\left(\sqrt{n_1n_2} + n_1\right)\omega_{\sf max}^2\log^2 n\right)^{K}\left(C_5C_{\sf b}^2\sqrt{n_2}\sqrt{n_1n_2}\omega_{\sf max}^3 + C_5C_{\sf b}\sqrt{n_2}\left(\sqrt{n_1n_2} + n_1\right)\omega_{\sf max}^3\right)\log^3 n\\
	&\quad + \sqrt{r}\cdot C_5\omega_{\sf max}\cdot C_5\left(\sqrt{n_1} + \sqrt{n_2}\right) \omega_{\sf max} \cdot  \left(C_5\left(\sqrt{n_1n_2} + n_1\right)\omega_{\sf max}^2\log^2 n\right)^K\\&\hspace{1.2cm}\cdot C_5\left(B\sqrt{\frac{\mu r}{n_2}}\log^2 n + \sqrt{n_1}\omega_{\sf max}\log n\right)\log n\\
	&\leq \frac{C_3}{4}\sqrt{\frac{\mu r}{n_2}}\left(C_3\left(\sqrt{n_1n_2} + n_1\right)\omega_{\sf max}^2\log^2 n\right)^{K+1}\sqrt{n_2}\omega_{\sf max}\log n\\
	&\quad + C_5^2\sqrt{r}\cdot \left(\sqrt{n_1} + \sqrt{n_2}\right)\omega_{\sf max}^2\cdot \left(C_5\left(\sqrt{n_1n_2} + n_1\right)\omega_{\sf max}^2\log^2 n\right)^K\cdot C_5(C_{\sf b} + 1)\sqrt{\mu n_1}\omega_{\sf max}\log^2 n\\
	&\leq  \frac{C_3}{2}\sqrt{\frac{\mu r}{n_2}}\left(C_3\left(\sqrt{n_1n_2} + n_1\right)\omega_{\sf max}^2\log^2 n\right)^{K+1}\sqrt{n_2}\omega_{\sf max}\log n,
\end{align*}
provided that $C_3^2 \geq 4C_4(C_5C_{\sf b}^2 + C_5C_{\sf b} + (C_{\sf b} + 1)C_5^2)$. Here, the second and the third inequalities are due to the assumption~\eqref{ineq64}.

\item {\em Bounding $\tau_2$.} By virtue of \eqref{ineq10} and the induction hypotheses, on the same event $\mathcal{E}_1 \cap \mathcal{E}_2$ we have 
\begin{align*}
	\tau_2
	&\leq \sum_{\ell=0}^{K-1}\left[C_5\left(\sqrt{n_1n_2} + n_1\right)\omega_{\sf max}^2\log^2 n\right]^{\ell+2}\cdot C_3\sqrt{\frac{\mu r}{n_2}}\left(C_3\left(\sqrt{n_1n_2} + n_1\right)\omega_{\sf max}^2\log^2 n\right)^{K - 1 - \ell}\sqrt{n_2}\omega_{\sf max}\log n\\
	&\leq \sum_{\ell=0}^{K-1}\frac{1}{2^{\ell + 2}}C_3\sqrt{\frac{\mu r}{n_2}}\left(C_3\left(\sqrt{n_1n_2} + n_1\right)\omega_{\sf max}^2\log^2 n\right)^{K+1}\sqrt{n_2}\omega_{\sf max}\log n\\
	&\leq \frac{C_3}{2}\sqrt{\frac{\mu r}{n_2}}\left(C_3\left(\sqrt{n_1n_2} + n_1\right)\omega_{\sf max}^2\log^2 n\right)^{K+1}\sqrt{n_2}\omega_{\sf max}\log n,
\end{align*}
with the proviso that $C_3 \geq 2C_5$.

\end{itemize}

\noindent 
Putting the previous bounds on $\tau_1$ and $\tau_2$ together with \eqref{ineq12} and \eqref{ineq11}, we arrive at the following result that holds on the event $\mathcal{E}_1 \cap \mathcal{E}_2$: 
\begin{align}\label{ineq:induction_step_1}
	&\left\|\left[\mathcal{P}_{\sf off\text{-}diag}\left(\bm{E}\bm{E}^\top\right)\right]^{K+1}\bm{E}\bm{V}^\star\right\|_{2,\infty}\notag\\
	&\quad\leq C_3\sqrt{\frac{\mu r}{n_2}}\left(C_3\left(\sqrt{n_1n_2} + n_1\right)\omega_{\sf max}^2\log^2 n\right)^{K+1}\sqrt{n_2}\omega_{\sf max}\log n,
\end{align}
provided that $C_3^2 \geq 4C_4(C_5C_{\sf b}^2 + C_5C_{\sf b} + (C_{\sf b} + 1)C_5^2)$ and $C_3 \geq 2C_5$.

\paragraph{Step 2.2: bounding $\|\bm{E}^\top[\mathcal{P}_{\sf off\text{-}diag}(\bm{E}\bm{E}^\top)]^{K+1}\bm{E}\bm{V}^\star\|_{2,\infty}$.}

We then move on to the quantity of interest in \eqref{ineq3}.
For any $j \in [n_2]$, it can be easily verified that
\begin{align}\label{ineq13}
	\left(\bm{E}^\top\left[\mathcal{P}_{\sf off\text{-}diag}\left(\bm{E}\bm{E}^\top\right)\right]^{K+1}\bm{E}\bm{V}^\star\right)_{j, :} = \bm{E}_{:, j}^\top\left[\mathcal{P}_{\sf off\text{-}diag}\left(\bm{E}\bm{E}^\top\right)\right]^{K+1}\bm{E}\bm{V}^\star.
\end{align}
Recalling that $ \bm{E}^{(:, -j)} = \mathcal{P}_{:, -j}(\bm{E})$ and $ \bm{E}^{(:, j)} = \mathcal{P}_{:, j}(\bm{E})$, we have 
\begin{align*}
	\mathcal{P}_{\sf off\text{-}diag}\left(\bm{E}\bm{E}^\top\right) = \mathcal{P}_{\sf off\text{-}diag}\big( \bm{E}^{(:, j)} \bm{E}^{(:, j)\top}\big) + \mathcal{P}_{\sf off\text{-}diag}\big( \bm{E}^{(:, -j)}\bm{E}^{(:, -j)\top}\big).
\end{align*}
For any matrices $\bm{A}, \bm{B} \in \bbR^{n_1 \times n_1}$, it is straightforward to show that
\begin{align*}
	\left(\bm{A} + \bm{B}\right)^{K+1} = \bm{B}^{K+1} + \sum_{\ell = 0}^{K}\bm{B}^\ell\bm{A}\left(\bm{A} + \bm{B}\right)^{K-\ell},
\end{align*}
and consequently one has
\begin{align*}
	\left[\mathcal{P}_{\sf off\text{-}diag}\left(\bm{E}\bm{E}^\top\right)\right]^{K+1} &= \big[\mathcal{P}_{\sf off\text{-}diag}\big( \bm{E}^{(:, -j)} \bm{E}^{(:, -j)\top}\big)\big]^{K+1}\\ &\quad + \sum_{\ell = 0}^{K}\big[\mathcal{P}_{\sf off\text{-}diag}\big( \bm{E}^{(:, -j)} \bm{E}^{(:, -j)\top}\big)\big]^{\ell}\mathcal{P}_{\sf off\text{-}diag}\big( \bm{E}^{(:, j)}\bm{E}^{(:, j)\top}\big)\left[\mathcal{P}_{\sf off\text{-}diag}\left(\bm{E}\bm{E}^\top\right)\right]^{K-\ell}.
\end{align*}
As a result, we can express $\bm{E}_{:, j}^\top[\mathcal{P}_{\sf off\text{-}diag}(\bm{E}\bm{E}^\top)]^{K+1}\bm{E}\bm{V}^\star$ in terms of a sum of vectors as follows:
\begin{align}\label{ineq14}
	&\bm{E}_{:, j}^\top\left[\mathcal{P}_{\sf off\text{-}diag}\left(\bm{E}\bm{E}^\top\right)\right]^{K+1}\bm{E}\bm{V}^\star\notag\\
	&\quad=\bm{E}_{:, j}^\top\big[\mathcal{P}_{\sf off\text{-}diag}\big( \bm{E}^{(:, -j)}\bm{E}^{(:, -j)\top}\big)\big]^{K+1}\bm{E}\bm{V}^\star\notag\\
	&\qquad+\sum_{\ell = 0}^{K}\bm{E}_{:, j}^\top\big[\mathcal{P}_{\sf off\text{-}diag}\big( \bm{E}^{(:, -j)}\bm{E}^{(:, -j)\top}\big)\big]^{\ell}\mathcal{P}_{\sf off\text{-}diag}\big( \bm{E}^{(:, j)} \bm{E}^{(:, j)\top}\big)\left[\mathcal{P}_{\sf off\text{-}diag}\left(\bm{E}\bm{E}^\top\right)\right]^{K-\ell}\bm{E}\bm{V}^\star\notag\\
	&\quad=\underbrace{\bm{E}_{:, j}^\top\big[\mathcal{P}_{\sf off\text{-}diag}\big( \bm{E}^{(:, -j)}\bm{E}^{(:, -j)\top}\big)\big]^{K+1} \bm{E}^{(:, -j)}\bm{V}^\star}_{=:\bm{b}_1}\notag\\
	&\qquad + \underbrace{\bm{E}_{:, j}^\top\big[\mathcal{P}_{\sf off\text{-}diag}\big( \bm{E}^{(:, -j)}\bm{E}^{(:, -j)\top}\big)\big]^{K+1} \bm{E}^{(:, j)}\bm{V}^\star}_{=:\bm{b}_2}\notag\\
	&\qquad + \underbrace{\sum_{\ell = 0}^{K}\bm{E}_{:, j}^\top\big[\mathcal{P}_{\sf off\text{-}diag}\big( \bm{E}^{(:, -j)}\bm{E}^{(:, -j)\top}\big)\big]^{\ell} \bm{E}^{(:, j)} \bm{E}^{(:, j)\top}\left[\mathcal{P}_{\sf off\text{-}diag}\left(\bm{E}\bm{E}^\top\right)\right]^{K-\ell}\bm{E}\bm{V}^\star}_{=:\bm{b}_3}\notag\\
	&\qquad - \underbrace{\sum_{\ell = 0}^{K}\bm{E}_{:, j}^\top\big[\mathcal{P}_{\sf off\text{-}diag}\big( \bm{E}^{(:, -j)}\bm{E}^{(:, -j)\top}\big)\big]^{\ell}\mathcal{P}_{\sf diag}\big( \bm{E}^{(:, j)} \bm{E}^{(:, j)\top}\big)\left[\mathcal{P}_{\sf off\text{-}diag}\left(\bm{E}\bm{E}^\top\right)\right]^{K-\ell}\bm{E}\bm{V}^\star}_{=:\bm{b}_4}, 
\end{align}
thus motivating us to bound each of these terms $\|\bm{b}_1\|_2, \|\bm{b}_2\|_2, \|\bm{b}_3\|_2$ and $\|\bm{b}_4\|_2$ separately. Let $\mathcal{E}_3$ denote the following event:
\begin{align}\label{ineq15}
	\mathcal{E}_3 &= \bigg\{\forall 0 \leq \ell \leq K, \left\|\bm{E}^\top\left[\mathcal{P}_{\sf off\text{-}diag}\left(\bm{E}\bm{E}^\top\right)\right]^\ell\bm{E}\bm{V}^\star\right\|_{2,\infty}\notag\\ &\hspace{2cm}\leq C_4\sqrt{\frac{\mu r}{n_2}}\left(C_3\left(\sqrt{n_1n_2} + n_1\right)\omega_{\sf max}^2\log^2 n\right)^{\ell}\left(\sqrt{n_2}B\omega_{\sf max}\log n + \left(\sqrt{n_1n_2} + n_1\right)\omega_{\sf max}^2\right)\log^2 n,\notag\\
	&\hspace{1cm}\left\|\left[\mathcal{P}_{\sf off\text{-}diag}\left(\bm{E}\bm{E}^\top\right)\right]^{\ell}\bm{E}\bm{V}^\star\right\|_{2,\infty}\leq C_3\sqrt{\frac{\mu r}{n_2}}\left(C_3\left(\sqrt{n_1n_2} + n_1\right)\omega_{\sf max}^2\log^2 n\right)^{\ell}\sqrt{n_2}\omega_{\sf max}\log n,\notag\\
	&\qquad\forall j \in [n_2], \left\|\big[\mathcal{P}_{\sf off\text{-}diag}\big( \bm{E}^{(:, -j)}\bm{E}^{(:, -j)\top}\big)\big]^{K+1} \bm{E}^{(:, -j)}\bm{V}^\star\right\|_{2,\infty}\notag\\
	& \hspace{3cm}\leq C_3\sqrt{\frac{\mu r}{n_2}}\left(C_3\left(\sqrt{n_1n_2} + n_1\right)\omega_{\sf max}^2\log^2 n\right)^{K+1}\sqrt{n_2}\omega_{\sf max}\log n\bigg\}. 
\end{align}
The induction hypotheses and  \eqref{ineq:induction_step_1} taken together with the union bound indicate that 
\begin{align*}
	\bbP\left(\mathcal{E}_3\right) \geq 1 - C_1(n_2+1)(n_1+2)(n+3)^{2K}n^{-C_2\log n}.
\end{align*}
By virtue of \eqref{ineq4c}, \eqref{ineq4d} and the independence between $[\mathcal{P}_{\sf off\text{-}diag}( \bm{E}^{(:, -j)}\bm{E}^{(:, -j)\top})]^{K+1} \bm{E}^{(:, -j)}\bm{V}^\star$ and $\bm{E}_{:, j}$, one has, with probability exceeding $1 - O(n^{-C_2\log n})$,
\begin{align}\label{ineq16}
	\max_{j \in [n_2]}\left\|\bm{E}_{:,j}\right\|_2^2 \leq C_5\left(B^2\log^2 n + \omega_{\sf col}^2\right)
\end{align}
and
\begin{align}\label{ineq17}
	\left\|\bm{b}_1\right\|_2 \leq C_5\left(B\log^2 n + \omega_{\sf col}\log n\right)\left\|\big[\mathcal{P}_{\sf off\text{-}diag}\big( \bm{E}^{(:, -j)}\bm{E}^{(:, -j)\top}\big)\big]^{K+1} \bm{E}^{(:, -j)}\bm{V}^\star\right\|_{2,\infty}.
\end{align}
Applying Lemma \ref{lm:off-diag} and the union bound yields that with probability exceeding $1 - O(n^{-C_2\log n})$, 
\begin{align}\label{ineq18}
	\big\|\mathcal{P}_{\sf off\text{-}diag}\big( \bm{E}^{(:, -j)}\bm{E}^{(:, -j)\top}\big)\big\| \leq C_5\left(\sqrt{n_1n_2} + n_1\right)\omega_{\sf max}^2\log^2 n
\end{align}
for all $j \in [n_2]$. 
Let $\mathcal{E}_4 = \{\eqref{ineq16}, \eqref{ineq17} \text{ and } \eqref{ineq18} \text{ hold}\}$ and $\mathcal{E}_5 = \mathcal{E}_3 \cap \mathcal{E}_4$. Thus, $\bbP(\mathcal{E}_4) \geq 1 - O(n^{-C_2\log n})$, and as a result, 
\begin{align*}
	\bbP\left(\mathcal{E}_5\right) \geq 1 - C_1(n+2)^2(n+3)^{2K}n^{-C_2\log n}.
\end{align*}
Armed with these events, we shall bound $\bm{b}_1,\ldots,\bm{b}_5$ separately in what follows.

\begin{itemize}

\item {\em Bounding $\left\|\bm{b}_1\right\|_2$.} In view of \eqref{ineq17}, \eqref{ineq15} and Assumption \ref{assump:hetero2},  we know that on the event $\mathcal{E}_5$, 
\begin{align}\label{ineq:beta_1}
	\left\|\bm{b}_1\right\|_2 &\leq C_5\left(B\log^2 n + \omega_{\sf col}\log n\right)\cdot C_3\sqrt{\frac{\mu r}{n_2}}\left(C_3\left(\sqrt{n_1n_2} + n_1\right)\omega_{\sf max}^2\log^2 n\right)^{K+1}\sqrt{n_2}\omega_{\sf max}\log n\notag\\
	&\leq  C_3C_5\sqrt{\frac{\mu r}{n_2}}\left(C_3\left(\sqrt{n_1n_2} + n_1\right)\omega_{\sf max}^2\log^2 n\right)^{K+1}\left(\sqrt{n_2}B\omega_{\sf max}\log n + \sqrt{n_1n_2}\omega_{\sf max}^2\right)\log^2 n\notag\\
	&\leq \frac{C_4}{4}\sqrt{\frac{\mu r}{n_2}}\left(C_3\left(\sqrt{n_1n_2} + n_1\right)\omega_{\sf max}^2\log^2 n\right)^{K+1}\left(\sqrt{n_2}B\omega_{\sf max}\log n + \sqrt{n_1n_2}\omega_{\sf max}^2\right)\log^2 n,
\end{align}
as long as $C_4 \geq 4C_3C_5$.

\item {\em Bounding $\left\|\bm{b}_2\right\|_2$.} Turning to $\bm{b}_2$, we recognize that $\bm{E}_{:, j}^\top[\mathcal{P}_{\sf off\text{-}diag}( \bm{E}^{(:, -j)}\bm{E}^{(:, -j)\top})]^{K+1} \bm{E}^{(:, j)}$ is a vector with only one nonzero entry $\bm{E}_{:, j}^\top[\mathcal{P}_{\sf off\text{-}diag}( \bm{E}^{(:, -j)}\bm{E}^{(:, -j)\top})]^{K+1}\bm{E}_{:, j}$. 
	By virtue of \eqref{ineq16}, \eqref{ineq18} and the assumption \eqref{ineq64}, one sees that on the event $\mathcal{E}_5$, 
\begin{align}\label{ineq:beta_2}
	 &\left\|\bm{b}_2\right\|_2 \leq \left|\bm{E}_{:, j}^\top\big[\mathcal{P}_{\sf off\text{-}diag}\big( \bm{E}^{(:, -j)}\bm{E}^{(:, -j)\top}\big)\big]^{K+1}\bm{E}_{:, j}\right|\left\|\bm{V}^\star\right\|_{2,\infty}\notag\\
	&\leq \left\|\bm{E}_{:, j}\right\|_2^2\big\|\mathcal{P}_{\sf off\text{-}diag}\big( \bm{E}^{(:, -j)}\bm{E}^{(:, -j)\top}\big)\big\|^{K+1}\sqrt{\frac{\mu r}{n_2}}\notag\\
	&\leq C_5\left(B^2\log^2 n + \omega_{\sf col}^2\right)\cdot \left(C_5\left(\sqrt{n_1n_2} + n_1\right)\omega_{\sf max}^2\log^2 n\right)^{K+1}\sqrt{\frac{\mu r}{n_2}}\notag\\
	&\leq C_5(C_{\sf b} + 1)\sqrt{\frac{\mu r}{n_2}}\left(C_3\left(\sqrt{n_1n_2} + n_1\right)\omega_{\sf max}^2\log^2 n\right)^{K+1}\left(\sqrt{n_2}B\omega_{\sf max}\log n + n_1\omega_{\sf max}^2\right)\log^2 n\notag\\
	&\leq \frac{C_4}{4}\sqrt{\frac{\mu r}{n_2}}\left(C_3\left(\sqrt{n_1n_2} + n_1\right)\omega_{\sf max}^2\log^2 n\right)^{K+1}\left(\sqrt{n_2}B\omega_{\sf max}\log n + \left(\sqrt{n_1n_2} + n_1\right)\omega_{\sf max}^2\right)\log^2 n,
\end{align}
provided that $C_4 \geq 4C_5(C_{\sf b} + 1)$.

\item {\em Bounding $\left\|\bm{b}_3\right\|_2$.} With regards to $\bm{b}_3$, repeating a similar argument as for \eqref{ineq:beta_2} shows that on the same event, it holds that
\begin{align}\label{ineq:beta_3}
	 &\left\|\bm{b}_3\right\|_2 \leq \sum_{\ell = 0}^{K}\left|\bm{E}_{:, j}^\top\big[\mathcal{P}_{\sf off\text{-}diag}\big( \bm{E}^{(:, -j)}\bm{E}^{(:, -j)\top}\big)\big]^{\ell}\bm{E}_{:, j}\right|\left\|\bm{E}^\top\left[\mathcal{P}_{\sf off\text{-}diag}\left(\bm{E}\bm{E}^\top\right)\right]^{K-\ell}\bm{E}\bm{V}^\star\right\|_{2,\infty}\notag\\
	&\leq \sum_{\ell = 0}^{K}\left\|\bm{E}_{:, j}\right\|_2^2\big\|\mathcal{P}_{\sf off\text{-}diag}\big( \bm{E}^{(:, -j)}\bm{E}^{(:, -j)\top}\big)\big\|^{\ell}\cdot C_4\sqrt{\frac{\mu r}{n_2}}\left(C_3\left(\sqrt{n_1n_2} + n_1\right)\omega_{\sf max}^2\log^2 n\right)^{K-\ell}\notag\\&\hspace{1.2cm}\cdot\left(\sqrt{n_2}B\omega_{\sf max}\log n + \left(\sqrt{n_1n_2} + n_1\right)\omega_{\sf max}^2\right)\log^2 n\notag\\
	&\leq \sum_{\ell = 0}^{K}C_5\left(B^2\log^2 n + \omega_{\sf col}^2\right)\cdot\left(C_5\left(\sqrt{n_1n_2} + n_1\right)\omega_{\sf max}^2\log^2 n\right)^\ell\cdot C_4\sqrt{\frac{\mu r}{n_2}}\left(C_3\left(\sqrt{n_1n_2} + n_1\right)\omega_{\sf max}^2\log^2 n\right)^{K-\ell}\notag\\&\hspace{1.2cm}\cdot\left(\sqrt{n_2}B\omega_{\sf max}\log n + \left(\sqrt{n_1n_2} + n_1\right)\omega_{\sf max}^2\right)\log^2 n\notag\\
	&\leq \sum_{\ell = 0}^{K}\frac{1}{2^\ell}C_4C_5\left(C_{\sf b}^2\sqrt{n_1n_2}\omega_{\sf max}^2 + n_1\omega_{\sf max}^2\right)\sqrt{\frac{\mu r}{n_2}}\left(C_3\left(\sqrt{n_1n_2} + n_1\right)\omega_{\sf max}^2\log^2 n\right)^{K}\notag\\&\hspace{1.2cm}\cdot\left(\sqrt{n_2}B\omega_{\sf max}\log n + \left(\sqrt{n_1n_2} + n_1\right)\omega_{\sf max}^2\right)\log^2 n\notag\\
	&\leq 2C_4C_5\left(C_{\sf b}^2+1\right)C_3^{-1}\sqrt{\frac{\mu r}{n_2}}\left(C_3\left(\sqrt{n_1n_2} + n_1\right)\omega_{\sf max}^2\log^2 n\right)^{K+1}\notag\\&\hspace{1cm}\cdot\left(\sqrt{n_2}B\omega_{\sf max}\log n + \left(\sqrt{n_1n_2} + n_1\right)\omega_{\sf max}^2\right)\log^2 n\notag\\
	&\leq \frac{C_4}{4}\sqrt{\frac{\mu r}{n_2}}\left(C_3\left(\sqrt{n_1n_2} + n_1\right)\omega_{\sf max}^2\log^2 n\right)^{K+1}\left(\sqrt{n_2}B\omega_{\sf max}\log n + \left(\sqrt{n_1n_2} + n_1\right)\omega_{\sf max}^2\right)\log^2 n,
\end{align}
with the proviso that $C_3 \geq 8C_5(C_{\sf b}^2+1)$.

\item {\em Bounding $\left\|\bm{b}_4\right\|_2$.} Regarding $\bm{b}_4$, using the elementary bound $\|\bm{a}^\top\bm{B}\|_2 \leq \|\bm{a}\|_1\|\bm{B}\|_{2,\infty}$ for any vector $\bm{a}$ and matrix $\bm{B}$ and applying \eqref{ineq16}, \eqref{ineq18} and \eqref{ineq15}, we can demonstrate that on the event $\mathcal{E}_5$,
\begin{align}\label{ineq:beta_4}
	\left\|\bm{b}_4\right\|_2 &\leq \sum_{\ell = 0}^{K}\left\|\bm{E}_{:, j}^\top\big[\mathcal{P}_{\sf off\text{-}diag}\big( \bm{E}^{(:, -j)}\bm{E}^{(:, -j)\top}\big)\big]^{\ell}\mathcal{P}_{\sf diag}\big( \bm{E}^{(:, j)} \bm{E}^{(:, j)\top}\big)\right\|_1\left\|\left[\mathcal{P}_{\sf off\text{-}diag}\left(\bm{E}\bm{E}^\top\right)\right]^{K-\ell}\bm{E}\bm{V}^\star\right\|_{2,\infty}\notag\\
	&\leq \sum_{\ell = 0}^{K}\left\|\bm{E}_{:, j}^\top\big[\mathcal{P}_{\sf off\text{-}diag}\big( \bm{E}^{(:, -j)}\bm{E}^{(:, -j)\top}\big)\big]^{\ell}\right\|_2\left\|\bm{E}_{:, j}\right\|_2^2\left\|\left[\mathcal{P}_{\sf off\text{-}diag}\left(\bm{E}\bm{E}^\top\right)\right]^{K-\ell}\bm{E}\bm{V}^\star\right\|_{2,\infty}\notag\\&\leq \sum_{\ell = 0}^{K}\sqrt{C_5}\left(B\log n + \omega_{\sf col}\right)\cdot \left(C_5\left(\sqrt{n_1n_2} + n_1\right)\omega_{\sf max}^2\log^2 n\right)^\ell\cdot C_5\left(B^2\log^2 n + \omega_{\sf col}^2\right)\notag\\
	&\hspace{1cm}\cdot C_3\sqrt{\frac{\mu r}{n_2}}\left(C_3\left(\sqrt{n_1n_2} + n_1\right)\omega_{\sf max}^2\log^2 n\right)^{K-\ell}\sqrt{n_2}\omega_{\sf max}\log n\notag\\
	&\leq \sum_{\ell = 0}^{K}\frac{1}{2^{\ell+1}}C_3\sqrt{C_5}\left(\sqrt{n_2}B\omega_{\sf max}\log n + \sqrt{n_1n_2}\omega_{\sf max}^2\right)\sqrt{\frac{\mu r}{n_2}}\left(C_3\left(\sqrt{n_1n_2} + n_1\right)\omega_{\sf max}^2\log^2 n\right)^{K+1}\log n\notag\\
	&\leq \frac{C_4}{4}\sqrt{\frac{\mu r}{n_2}}\left(C_3\left(\sqrt{n_1n_2} + n_1\right)\omega_{\sf max}^2\log^2 n\right)^{K+1}\left(\sqrt{n_2}B\omega_{\sf max}\log n + \left(\sqrt{n_1n_2} + n_1\right)\omega_{\sf max}^2\right)\log^2 n,
\end{align}
provided that $C_3 \geq 2C_5$ and $C_4 \geq 4C_3\sqrt{C_5}$.
\end{itemize}

Combine \eqref{ineq14}, \eqref{ineq:beta_1}, \eqref{ineq:beta_2}, \eqref{ineq:beta_3} and \eqref{ineq:beta_4} to reach that: on the $\mathcal{E}_5$ one has 
\begin{align}\label{ineq:induction_step_2}
	&\left\|\bm{E}^\top\left[\mathcal{P}_{\sf off\text{-}diag}\left(\bm{E}\bm{E}^\top\right)\right]^{K+1}\bm{E}\bm{V}^\star\right\|_{2,\infty}\notag\\ &\quad= \max_{j \in [n_2]}\left\|\bm{E}_{:, j}^\top\left[\mathcal{P}_{\sf off\text{-}diag}\left(\bm{E}\bm{E}^\top\right)\right]^{K+1}\bm{E}\bm{V}^\star\right\|_2\notag\\
	&\quad \leq C_4\sqrt{\frac{\mu r}{n_2}}\left(C_3\left(\sqrt{n_1n_2} + n_1\right)\omega_{\sf max}^2\log^2 n\right)^{K+1}\left(\sqrt{n_2}B\omega_{\sf max}\log n + \left(\sqrt{n_1n_2} + n_1\right)\omega_{\sf max}^2\right)\log^2 n,
\end{align}
with the proviso that $C_3 \geq 8C_5(C_{\sf b}^2+1)$ and $C_4 \geq 4C_3C_5$.

In summary, if the claim \eqref{ineq11} is valid, then with probability exceeding $1 - C_1(n+2)(n+3)^{2k}n^{-C_2\log n}$, \eqref{ineq:induction_step_1} and \eqref{ineq:induction_step_2} hold simultaneously  as long as $C_4 = 4C_3C_5$ and $C_3 \geq 32C_5^2(C_{\sf b}^2+1)$. We have thus finished the proof of the induction hypotheses \eqref{ineq2} and \eqref{ineq3}, 
as long as  the claim \eqref{ineq11} can be justified; see below.

\paragraph{Proof of the claim~\eqref{ineq11}.} We first make the observation that
\begin{subequations}
	\begin{align}
		\bm{E}^{(i, :)\top}\bm{E}^{(-i, :)} = \bm{E}^{(-i, :)\top}\bm{E}^{(i, :)} = \bm{0},\label{eq1}\\
		\mathcal{P}_{\sf diag}\big(\bm{E}^{(i, :)}\bm{E}^{(-i, :)\top}\big) = \mathcal{P}_{\sf diag}\big(\bm{E}^{(-i, :)}\bm{E}^{(i, :)\top}\big) = \bm{0},\label{eq2}\\
		\bm{E}^{(i, :)\top}\mathcal{P}_{\sf diag}\big(\bm{E}^{(-i, :)}\bm{E}^{(-i, :)\top}\big) = \bm{E}^{(-i, :)\top}\mathcal{P}_{\sf diag}\big(\bm{E}^{(i, :)}\bm{E}^{(i, :)\top}\big) = \bm{0},\label{eq3}\\
		\mathcal{P}_{\sf diag}\big(\bm{E}^{(-i, :)}\bm{E}^{(-i, :)\top}\big)\mathcal{P}_{\sf diag}\big(\bm{E}^{(i, :)}\bm{E}^{(i, :)\top}\big) = \bm{0}\label{eq4}.
	\end{align}
\end{subequations}
The identities \eqref{eq1}, \eqref{eq2} and \eqref{eq3} taken collectively give
\begin{align}\label{eq5}
	\mathcal{P}_{\sf off\text{-}diag}\left(\bm{E}\bm{E}^\top\right) = \mathcal{P}_{\sf off\text{-}diag}\big(\bm{E}^{(-i, :)}\bm{E}^{(-i, :)\top}\big)
	+ \mathcal{P}_{\sf off\text{-}diag}\big(\bm{E}^{(i, :)}\bm{E}^{(i, :)\top}\big)
	+ \bm{E}^{(i, :)}\bm{E}^{(-i, :)\top} + \bm{E}^{(-i, :)}\bm{E}^{(i, :)\top}
\end{align}
and
\begin{align}
	&\bm{E}_{i,:}\bm{E}^{(-i, :)\top}\left[\mathcal{P}_{\sf off\text{-}diag}\left(\bm{E}\bm{E}^\top\right)\right]^{K}\bm{E}\bm{V}^\star \notag\\
	&= \bm{E}_{i,:}\bm{E}^{(-i, :)\top}\big[\mathcal{P}_{\sf off\text{-}diag}\big(\bm{E}^{(-i, :)}\bm{E}^{(-i, :)\top}\big)\big]\left[\mathcal{P}_{\sf off\text{-}diag}\left(\bm{E}\bm{E}^\top\right)\right]^{K-1}\bm{E}\bm{V}^\star \notag\\
	&\quad + \bm{E}_{i,:}\bm{E}^{(-i, :)\top}\bm{E}^{(-i, :)}\bm{E}^{(i, :)\top}\left[\mathcal{P}_{\sf off\text{-}diag}\left(\bm{E}\bm{E}^\top\right)\right]^{K-1}\bm{E}\bm{V}^\star.
	\label{eq:E-complex-identity}
\end{align}
Combining \eqref{eq1}-\eqref{eq4} and \eqref{eq5} then yields
\begin{align*}
	&\big[\mathcal{P}_{\sf off\text{-}diag}\big(\bm{E}^{(-i, :)}\bm{E}^{(-i, :)\top}\big)\big]\left[\mathcal{P}_{\sf off\text{-}diag}\left(\bm{E}\bm{E}^\top\right)\right]\\
	&= \big[\mathcal{P}_{\sf off\text{-}diag}\big(\bm{E}^{(-i, :)}\bm{E}^{(-i, :)\top}\big)\big]^2 + \big[\mathcal{P}_{\sf off\text{-}diag}\big(\bm{E}^{(-i, :)}\bm{E}^{(-i, :)\top}\big)\big]\bm{E}^{(-i, :)}\bm{E}^{(i, :)\top}\\
	&\quad + \mathcal{P}_{\sf off\text{-}diag}\big(\bm{E}^{(-i, :)}\bm{E}^{(-i, :)\top}\big)\mathcal{P}_{\sf off\text{-}diag}\big(\bm{E}^{(i, :)}\bm{E}^{(i, :)\top}\big)\\
	&\quad + \big[\bm{E}^{(-i, :)}\bm{E}^{(-i, :)\top} - \mathcal{P}_{\sf diag}\big(\bm{E}^{(-i, :)}\bm{E}^{(-i, :)\top}\big)\big]\bm{E}^{(i, :)}\bm{E}^{(-i, :)\top}\\
	&= \big[\mathcal{P}_{\sf off\text{-}diag}\big(\bm{E}^{(-i, :)}\bm{E}^{(-i, :)\top}\big)\big]^2 +\big[\mathcal{P}_{\sf off\text{-}diag}\big(\bm{E}^{(-i, :)}\bm{E}^{(-i, :)\top}\big)\big]\bm{E}^{(-i, :)}\bm{E}^{(i, :)\top}.
\end{align*}
As a consequence, we can deduce that
\begin{align*}
	&\bm{E}_{i,:}\bm{E}^{(-i, :)\top}\left[\mathcal{P}_{\sf off\text{-}diag}\left(\bm{E}\bm{E}^\top\right)\right]^{K}\bm{E}\bm{V}^\star\\
	&= \bm{E}_{i,:}\bm{E}^{(-i, :)\top}\big[\mathcal{P}_{\sf off\text{-}diag}\big(\bm{E}^{(-i, :)}\bm{E}^{(-i, :)\top}\big)\big]^2\left[\mathcal{P}_{\sf off\text{-}diag}\left(\bm{E}\bm{E}^\top\right)\right]^{K-2}\bm{E}\bm{V}^\star\\
	&\quad + \bm{E}_{i,:}\bm{E}^{(-i, :)\top}\bm{E}^{(-i, :)}\bm{E}^{(i, :)\top}\left[\mathcal{P}_{\sf off\text{-}diag}\left(\bm{E}\bm{E}^\top\right)\right]^{K-1}\bm{E}\bm{V}^\star\\
	&\quad +  \bm{E}_{i,:}\bm{E}^{(-i, :)\top}\big[\mathcal{P}_{\sf off\text{-}diag}\big(\bm{E}^{(-i, :)}\bm{E}^{(-i, :)\top}\big)\big]\bm{E}^{(-i, :)}\bm{E}^{(i, :)\top}\left[\mathcal{P}_{\sf off\text{-}diag}\left(\bm{E}\bm{E}^\top\right)\right]^{K-2}\bm{E}\bm{V}^\star.
\end{align*}
Repeating the same argument yields
\begin{align}\label{eq6}
	&\bm{E}_{i,:}\bm{E}^{(-i, :)\top}\left[\mathcal{P}_{\sf off\text{-}diag}\left(\bm{E}\bm{E}^\top\right)\right]^{K}\bm{E}\bm{V}^\star\notag\\
	&=\bm{E}_{i,:}\bm{E}^{(-i, :)\top}\big[\mathcal{P}_{\sf off\text{-}diag}\big(\bm{E}^{(-i, :)}\bm{E}^{(-i, :)\top}\big)\big]^K\bm{E}\bm{V}^\star\notag\\
	&\quad + \sum_{\ell = 0}^{K-1}\bm{E}_{i,:}\bm{E}^{(-i, :)\top}\big[\mathcal{P}_{\sf off\text{-}diag}\big(\bm{E}^{(-i, :)}\bm{E}^{(-i, :)\top}\big)\big]^{\ell}\bm{E}^{(-i, :)}\bm{E}^{(i, :)\top}\left[\mathcal{P}_{\sf off\text{-}diag}\left(\bm{E}\bm{E}^\top\right)\right]^{K-1-\ell}\bm{E}\bm{V}^\star\notag\\
	&=\bm{E}_{i,:}\bm{E}^{(-i, :)\top}\big[\mathcal{P}_{\sf off\text{-}diag}\big(\bm{E}^{(-i, :)}\bm{E}^{(-i, :)\top}\big)\big]^K\bm{E}^{(-i, :)}\bm{V}^\star\notag\\
	&\quad + \sum_{\ell = 0}^{K-1}\bm{E}_{i,:}\bm{E}^{(-i, :)\top}\big[\mathcal{P}_{\sf off\text{-}diag}\big(\bm{E}^{(-i, :)}\bm{E}^{(-i, :)\top}\big)\big]^{\ell}\bm{E}^{(-i, :)}\bm{E}^{(i, :)\top}\left[\mathcal{P}_{\sf off\text{-}diag}\left(\bm{E}\bm{E}^\top\right)\right]^{K-1-\ell}\bm{E}\bm{V}^\star.
\end{align}
Since $\bm{E}_{i,:}\bm{E}^{(-i, :)\top}[\mathcal{P}_{\sf off\text{-}diag}(\bm{E}^{(-i, :)}\bm{E}^{(-i, :)\top})]^{\ell}\bm{E}^{(-i, :)}\bm{E}^{(i, :)\top}$ is a vector with only one nonzero entry $$\bm{E}_{i,:}\bm{E}^{(-i, :)\top}[\mathcal{P}_{\sf off\text{-}diag}(\bm{E}^{(-i, :)}\bm{E}^{(-i, :)\top})]^{\ell}\bm{E}^{(-i, :)}\bm{E}_{i,:}^\top,$$ for any $0 \leq \ell \leq K-1$, one can immediately derive
\begin{align*}
	&\left\|\bm{E}_{i,:}\bm{E}^{(-i, :)\top}\big[\mathcal{P}_{\sf off\text{-}diag}\big(\bm{E}^{(-i, :)}\bm{E}^{(-i, :)\top}\big)\big]^{\ell}\bm{E}^{(-i, :)}\bm{E}^{(i, :)\top}\left[\mathcal{P}_{\sf off\text{-}diag}\left(\bm{E}\bm{E}^\top\right)\right]^{K-1-\ell}\bm{E}\bm{V}^\star\right\|_{2}\\
	&\quad \leq \left|\bm{E}_{i,:}\bm{E}^{(-i, :)\top}\big[\mathcal{P}_{\sf off\text{-}diag}\big(\bm{E}^{(-i, :)}\bm{E}^{(-i, :)\top}\big)\big]^{\ell}\bm{E}^{(-i, :)}\bm{E}_{i,:}^\top\right|\left\|\left[\mathcal{P}_{\sf off\text{-}diag}\left(\bm{E}\bm{E}^\top\right)\right]^{K-1-\ell}\bm{E}\bm{V}^\star\right\|_{2,\infty}\\
	&\quad \leq\big\|\bm{E}_{i,:}\bm{E}^{(-i, :)\top}\big\|_2^2\big\|\mathcal{P}_{\sf off\text{-}diag}\big(\bm{E}^{(-i, :)}\bm{E}^{(-i, :)\top}\big)\big\|^\ell\left\|\left[\mathcal{P}_{\sf off\text{-}diag}\left(\bm{E}\bm{E}^\top\right)\right]^{K-1-\ell}\bm{E}\bm{V}^\star\right\|_{2,\infty}.
\end{align*}
Taking this together with \eqref{eq6} and the triangle inequality establishes the advertised result \eqref{ineq11}.

\subsubsection{The general case}\label{sec:general_error_moment1}

Having established the claim for the bounded noise case, 
we can readily turn attention to the more  general case with the noise matrix $\bm{E}$ satisfying Assumption \ref{assump:hetero2}. 
To tackle this scenario, we introduce a properly truncated version $\widetilde{\bm{E}}=[\widetilde{E}_{i,j}]_{(i,j)\in[n_1]\times [n_2]}$, 
which is a zero-mean matrix with entries given by
\begin{align}\label{ineq65}
	\widetilde{E}_{i,j} =  E_{i,j}\mathbbm{1}_{\left\{\left|E_{i,j}\right| \leq B\right\}} - \bbE\left[E_{i,j}\mathbbm{1}_{\left\{\left|E_{i,j}\right| \leq B\right\}}\right].
\end{align}
It is clearly seen that
\begin{align*}
	\mathsf{Var}\big[\widetilde{E}_{i,j}\big] \leq \bbE\left[ \left(E_{i,j}\mathbbm{1}_{\left\{\left|E_{i,j}\right| \leq B\right\}}\right)^2\right] \leq \bbE\left[E_{i,j}^2\right] \leq \omega_{\sf max}^2
\end{align*}
and 
\begin{align*}
	\big|\widetilde{E}_{i,j}\big| \leq 2B.
\end{align*}
Then \eqref{ineq2} and \eqref{ineq3} tell us that with probability $1 - O(n^{-c_1\log n})$, 
\begin{align}\label{ineq66}
	&\left\|\big[\mathcal{P}_{\sf off\text{-}diag}\big(\widetilde{\bm{E}}\widetilde{\bm{E}}^\top\big)\big]^{k}\widetilde{\bm{E}}\bm{V}^\star\right\|_{2,\infty}\leq C_3\sqrt{\frac{\mu r}{n_2}}\left(C_3\left(\sqrt{n_1n_2} + n_1\right)\omega_{\sf max}^2\log^2 n\right)^{k}\sqrt{n_2}\omega_{\sf max}\log n
\end{align}
holds for all $0 \leq k \leq \log n$.

Let $\overline{\bm{E}}$ be another matrix whose entries are given by
\begin{align*}
	\overline{E}_{i,j} = E_{i,j}\mathbbm{1}_{\left\{\left|E_{i,j}\right| \leq B\right\}}.
\end{align*}
In view of the Cauchy-Schwarz inequality, one can derive
\begin{align}\label{ineq67}
	\big\|\widetilde{\bm{E}} - \overline{\bm{E}}\big\| \leq \big\|\widetilde{\bm{E}} - \overline{\bm{E}}\big\|_{\F} \leq \sqrt{n_1n_2}\max_{i,j}\left|\bbE\left[E_{i,j}\mathbbm{1}_{\left\{\left|E_{i,j}\right| \leq B\right\}}\right]\right| \leq \sqrt{n_1n_2}\left(\bbE\left[E_{i,j}^2\right]\bbP\left(\left|E_{i,j}\right| \leq B\right)\right)^{1/2} \leq \frac{\omega_{\sf max}}{n^5}.
\end{align}
By virtue of Lemmas \ref{lm:noise} and \ref{lm:off-diag}, we can see that, with probability exceeding $1 - O(n^{-10})$,
\begin{align}\label{ineq68}
	\big\|\widetilde{\bm{E}}\big\| \lesssim B\sqrt{\log n} + \omega_{\sf col} + \omega_{\sf row} \lesssim \sqrt{n}\omega_{\sf max}
\end{align}
and
\begin{align}\label{ineq69}
	\big\|\mathcal{P}_{\sf off\text{-}diag}\big(\widetilde{\bm{E}}\widetilde{\bm{E}}^\top\big)\big\| \lesssim B^2\log^2n + \omega_{\sf col}\left(\omega_{\sf row} + \omega_{\sf col}\right)\log n.
\end{align}
Combining the above results reveals that, with probability exceeding $1 - O(n^{-10})$, 
\begin{align*}
	\left\|\overline{\bm{E}}\right\| \leq \big\|\widetilde{\bm{E}}\big\| + \big\|\widetilde{\bm{E}} - \overline{\bm{E}}\big\| \leq \sqrt{n}\omega_{\sf max},
\end{align*}
\begin{align*}
	\big\|\mathcal{P}_{\sf off\text{-}diag}\big(\overline{\bm{E}}\,\overline{\bm{E}}^\top\big) - \mathcal{P}_{\sf off\text{-}diag}\big(\widetilde{\bm{E}}\widetilde{\bm{E}}^\top\big)\big\| \leq 2\big\|\overline{\bm{E}}\,\overline{\bm{E}}^\top - \widetilde{\bm{E}}\widetilde{\bm{E}}^\top\big\| \leq 4\big\|\widetilde{\bm{E}} - \overline{\bm{E}}\big\|\big\|\widetilde{\bm{E}}\big\| + 2\big\|\widetilde{\bm{E}} - \overline{\bm{E}}\big\|^2 \lesssim \frac{\omega_{\sf max}^2}{n^{4.5}},
\end{align*}
and for all $0 \leq k \leq \log n$,
\begin{align*}
	&\left\|\big[\mathcal{P}_{\sf off\text{-}diag}\big(\overline{\bm{E}}\,\overline{\bm{E}}^\top\big)\big]^{k}\overline{\bm{E}}\bm{V}^\star - \big[\mathcal{P}_{\sf off\text{-}diag}\big(\widetilde{\bm{E}}\widetilde{\bm{E}}^\top\big)\big]^{k}\widetilde{\bm{E}}\bm{V}^\star\right\|_{2,\infty}\\
	&\quad \leq \sum_{\ell=0}^{k-1}\left\|\left[\mathcal{P}_{\sf off\text{-}diag}\big(\widetilde{\bm{E}}\widetilde{\bm{E}}^\top\big)^\ell\big(\mathcal{P}_{\sf off\text{-}diag}\big(\overline{\bm{E}}\,\overline{\bm{E}}^\top\big) - \mathcal{P}_{\sf off\text{-}diag}\big(\widetilde{\bm{E}}\widetilde{\bm{E}}^\top\big)\big)\right]\left[\mathcal{P}_{\sf off\text{-}diag}\big(\overline{\bm{E}}\,\overline{\bm{E}}^\top\big)\right]^{k-1 -\ell}\overline{\bm{E}}\bm{V}^\star\right\|\\
	&\qquad + \left\|\big[\mathcal{P}_{\sf off\text{-}diag}\big(\widetilde{\bm{E}}\widetilde{\bm{E}}^\top\big)\big]^{k}\big(\overline{\bm{E}} - \widetilde{\bm{E}}\big)\bm{V}^\star\right\|\\
	&\quad \leq \sum_{\ell = 0}^{k-1}\left[C_3\left(\sqrt{n_1n_2} + n_1\right)\omega_{\sf max}^2\log n\right]^\ell\cdot C_3\frac{\omega_{\sf max}^2}{n^{4.5}}\cdot\left[C_3\left(\sqrt{n_1n_2} + n_1\right)\omega_{\sf max}^2\log n\right]^{k-1-\ell}\cdot C_3\sqrt{n}\omega_{\sf max}\\
	&\qquad + \left[C_3\left(\sqrt{n_1n_2} + n_1\right)\omega_{\sf max}^2\log n\right]^k\cdot \frac{\omega_{\sf max}}{n^5}\\
	&\quad \leq (k+1)\left[C_3\left(\sqrt{n_1n_2} + n_1\right)\omega_{\sf max}^2\log n\right]^k\cdot \frac{\omega_{\sf max}}{n^4}\\
	&\quad \ll \sqrt{\frac{\mu r}{n_2}}\left(C_3\left(\sqrt{n_1n_2} + n_1\right)\omega_{\sf max}^2\log^2 n\right)^{k}\sqrt{n_2}\omega_{\sf max}\log n.
\end{align*}
Taking this collectively with \eqref{ineq66} implies that, with probability exceeding $1 - O(n^{-10})$, 
\begin{align}\label{ineq70}
	\left\|\big[\mathcal{P}_{\sf off\text{-}diag}\big(\overline{\bm{E}}\,\overline{\bm{E}}^\top\big)\big]^{k}\overline{\bm{E}}\bm{V}^\star\right\|_{2,\infty} \lesssim \sqrt{\frac{\mu r}{n_2}}\left(C_3\left(\sqrt{n_1n_2} + n_1\right)\omega_{\sf max}^2\log^2 n\right)^{k}\sqrt{n_2}\omega_{\sf max}\log n
\end{align} 
holds for all $0 \leq k \leq \log n$.

To finish up, note that the union bound tell us that with probability exceeding $1 - O(n^{-10})$, 
\begin{align*}
	\overline{\bm{E}} = \bm{E}.
\end{align*}
This combined with inequality \eqref{ineq70} establishes the desired result for the general case.

\subsection{Proof of Lemma \ref{lm:error_moment2}}\label{proof:lm_error_moment2}

We first study the case with bounded noise (i.e., the case that \eqref{ineq64} always holds). Akin to the proof of Lemma \ref{lm:error_moment1}, we first intend to show that the following statement holds: for any $0 \leq k \leq \log n$ and any noise matrix $\bm{E}$ satisfying Condition 1 in Assumption \ref{assump:hetero} and \eqref{ineq64}, with probability exceeding $1 - O((n+3)^{2k}n^{-C_2\log n})$ one has 
\begin{align}\label{ineq19}
	&\left\|\left[\mathcal{P}_{\sf off\text{-}diag}\left(\bm{E}\bm{E}^\top\right)\right]^{\ell}\bm{U}^\star\right\|_{2,\infty}\leq C_3\sqrt{\frac{\mu r}{n_1}}\left(C_3\left(\sqrt{n_1n_2} + n_1\right)\omega_{\sf max}^2\log^2 n\right)^{\ell}
\end{align}
and
\begin{align}\label{ineq20}
	&\left\|\bm{E}^\top\left[\mathcal{P}_{\sf off\text{-}diag}\left(\bm{E}\bm{E}^\top\right)\right]^\ell\bm{U}^\star\right\|_{2,\infty}\notag\\ &\quad\leq C_4\sqrt{\frac{\mu r}{n_1}}\left(C_3\left(\sqrt{n_1n_2} + n_1\right)\omega_{\sf max}^2\log^2 n\right)^{\ell}\left(B\log^2 n + \sqrt{n_1}\omega_{\sf max}\log n\right) 
\end{align}
simultaneously for all $l$ obeying $0 \leq \ell \leq k$.

Regarding the base case with $k = 0$, it is self-evident that \eqref{ineq19} and \eqref{ineq20} hold with probability exceeding $1 - O(n^{-C_2\log n})$ due to Assumption \ref{assump:incoherence} and \eqref{ineq4d}. Suppose now that with probability exceeding $1 - O((n+3)^{2K}n^{-C_2\log n})$, \eqref{ineq19} and \eqref{ineq20} hold for all $0 \leq \ell \leq K$, 
and we would like to extend the results to $k = K+1$. Similar to \eqref{ineq11} and \eqref{ineq14}, one has
\begin{align}\label{ineq21}
	&\left\|\bm{E}_{i,:}\bm{E}^{(-i, :)\top}\left[\mathcal{P}_{\sf off\text{-}diag}\left(\bm{E}\bm{E}^\top\right)\right]^{K}\bm{U}^\star\right\|_{2}\notag\\
	&\leq \left\|\bm{E}_{i,:}\bm{E}^{(-i, :)\top}\big[\mathcal{P}_{\sf off\text{-}diag}\big(\bm{E}^{(-i, :)}\bm{E}^{(-i, :)\top}\big)\big]^K\bm{U}^\star\right\|_{2}\notag\\
	&\quad + \sum_{\ell=0}^{K-1}\big\|\bm{E}_{i,:}\bm{E}^{(-i, :)\top}\big\|_2^2\big\|\mathcal{P}_{\sf off\text{-}diag}\big(\bm{E}^{(-i, :)}\bm{E}^{(-i, :)\top}\big)\big\|^\ell\left\|\left[\mathcal{P}_{\sf off\text{-}diag}\left(\bm{E}\bm{E}^\top\right)\right]^{K-1-\ell}\bm{U}^\star\right\|_{2,\infty}
\end{align}
and 
\begin{align}\label{ineq22}
	&\left\|\bm{E}_{:, j}^\top\left[\mathcal{P}_{\sf off\text{-}diag}\left(\bm{E}\bm{E}^\top\right)\right]^{K+1}\bm{U}^\star\right\|_2\notag\\
	&\leq \left\|\bm{E}_{:, j}^\top\big[\mathcal{P}_{\sf off\text{-}diag}\big( \bm{E}^{(:, -j)}\bm{E}^{(:, -j)\top}\big)\big]^{K+1}\bm{U}^\star\right\|_2\notag\\
	&\quad +\sum_{\ell = 0}^{K}\left|\bm{E}_{:, j}^\top\big[\mathcal{P}_{\sf off\text{-}diag}\big( \bm{E}^{(:, -j)}\bm{E}^{(:, -j)\top}\big)\big]^{\ell}\bm{E}_{:, j}\right|\left\|\bm{E}^\top\left[\mathcal{P}_{\sf off\text{-}diag}\left(\bm{E}\bm{E}^\top\right)\right]^{K-\ell}\bm{U}^\star\right\|_{2,\infty}\notag\\
	&\quad+ \sum_{\ell = 0}^{K}\left\|\bm{E}_{:, j}^\top\big[\mathcal{P}_{\sf off\text{-}diag}\big( \bm{E}^{(:, -j)}\bm{E}^{(:, -j)\top}\big)\big]^{\ell}\mathcal{P}_{\sf diag}\big( \bm{E}^{(:, j)} \bm{E}^{(:, j)\top}\big)\right\|_1\notag\\&\hspace{1.5cm}\cdot\left\|\left[\mathcal{P}_{\sf off\text{-}diag}\left(\bm{E}\bm{E}^\top\right)\right]^{K-\ell}\bm{U}^\star\right\|_{2,\infty}.
\end{align}
In view of \eqref{ineq6}, \eqref{ineq4d} and Lemma \ref{lm:noise}, for any $\bm{E}$ satisfying Condition 1 in Assumption \ref{assump:hetero} and \eqref{ineq64}, with probability $1 - O(n^{-C_1\log n})$, for all $i \in [n_1]$, one has
\begin{align*}
	\big\|\bm{E}_{i,:}\bm{E}^{(-i, )\top}\bm{U}^\star\big\|_2 &\leq C_5\left(B\big\|\bm{E}^{(-i, )\top}\bm{U}^\star\big\|_{2, \infty}\log^2 n + \omega_{\sf max}\big\|\bm{E}^{(-i, )\top}\bm{U}^\star\big\|_{\F}\log n\right)\\
	&\leq C_5\left(B^2\log^4 n + B\omega_{\sf col}\log^3 n\right)\sqrt{\frac{\mu r}{n_1}} + C_5\sqrt{r}\omega_{\sf max}\left(B\sqrt{\frac{\mu r}{n_1}}\log n + \sqrt{n_2}\omega_{\sf max}\right)\log^2 n\\
	&\leq \left[C_5\left(C_{\sf b}^2\sqrt{n_1n_2} + C_{\sf b}\sqrt{n_1n_2}\right)\omega_{\sf max}^2\sqrt{\frac{\mu r}{n_1}} + C_5\omega_{\sf max}^2\sqrt{\frac{\mu r}{n_1}}\left(C_{\sf b}\sqrt{rn_2} + \sqrt{n_1n_2}\right)\right]\log^2 n\\
	&\leq C_5\left(C_{\sf b} + 1\right)^2\sqrt{n_1n_2}\omega_{\sf max}^2\sqrt{\frac{\mu r}{n_1}}\log^2 n.
\end{align*}
As a result, with the same probability, we have
\begin{align*}
	\left\|\left[\mathcal{P}_{\sf off\text{-}diag}\left(\bm{E}\bm{E}^\top\right)\right]\bm{U}^\star\right\|_{2,\infty} &= \max_{1 \leq i \leq n_1}\big\|\bm{E}_{i,:}\bm{E}^{(-i, )\top}\bm{U}^\star\big\|_2 \leq C_5\left(C_{\sf b} + 1\right)^2\sqrt{n_1n_2}\omega_{\sf max}^2\sqrt{\frac{\mu r}{n_1}}\log^2 n
\end{align*}
and
\begin{align*}
	\left\|\left[\mathcal{P}_{\sf off\text{-}diag}\left(\bm{E}\bm{E}^\top\right)\right]\bm{U}^\star\right\|_{\F} &\leq \sqrt{n_1}\left\|\left[\mathcal{P}_{\sf off\text{-}diag}\left(\bm{E}\bm{E}^\top\right)\right]\bm{U}^\star\right\|_{2,\infty} \leq C_5\left(C_{\sf b} + 1\right)^2\sqrt{n_1n_2}\omega_{\sf max}^2\sqrt{\mu r}\log^2 n.
\end{align*}
Equipped with the previous two inequalities, we can carry out the induction step using a similar argument of Lemma \ref{lm:error_moment1}. 

For the general case where the noise matrix $\bm{E}$ satisfies Assumption \ref{assump:hetero2}, one can get the desired result by using the same truncation trick as in Section \ref{sec:general_error_moment1}.

\subsection{Proof of Lemma \ref{lm:collection_event}}\label{proof: lm_collection_event}
\paragraph{Bounding the spectrum of $\widetilde{\bm{\Sigma}}$.} 
Let us first develop an upper bound (resp.~lower bound) on the singular value perturbation $|\widetilde{\sigma}_{i} - \sigma_{i}^\star|$ (resp.~the spectral gap $\widetilde{\sigma}_{r'}^2 - \widetilde{\sigma}_{r'+1}^2$ for any $r'\in \mathcal{R}'$ defined in \eqref{eq:defn-Rprime-collection-event}). Weyl's inequality tell us that, for all $1 \leq i \leq r$,
\begin{align*}
	\left|\widetilde{\sigma}_{i} - \sigma_{i}^\star\right| \leq \left\|\bm{E}\bm{V}^{\star}\right\| 
	&\lesssim B\sqrt{\frac{\mu r}{n_2}}\log^2 n + \left(r\omega_{{\sf max}}^2 + n_1\omega_{{\sf max}}^2\right)^{1/2}\log n\notag\\ 
	&\lesssim \sqrt{\frac{\mu r}{n_2}}\omega_{{\sf max}}\frac{\sqrt{n_2}}{\log n}\log^2 n + \sqrt{n_1}\omega_{{\sf max}}\log n\notag\\
	&\leq \sqrt{C_5}\sqrt{n_1}\omega_{\sf max}\log n \leq \frac{\sigma_r^\star}{40r}
\end{align*}
holds with probability at least $1 - O(n^{-10})$ for some constant $C_5>0$. 
Here, the first line invokes Lemma~\ref{lm:noise}, the second line relies on Assumption \ref{assump:hetero2}, 
and the last line makes use of the assumption~\eqref{assump:snr_two_to_infty_234}.  
Consequently, 
\begin{align*}
	\widetilde{\sigma}_{r'} - \widetilde{\sigma}_{r'+1} \geq \sigma_{r'}^\star - \sigma_{r'+1}^\star - \frac{\sigma_r^\star}{20r} 
	\geq \frac{4\left(\sigma_{r'}^\star-\sigma_{r'+1}^\star\right)}{5},
\end{align*}
where we have made use of the definition of $\mathcal{R}'$ in \eqref{eq:defn-Rprime-collection-event} and the fact that $\sigma_{r'}^\star - \sigma_{r'+1}^\star = \frac{\sigma_{r'}^{\star2}-\sigma_{r'+1}^{\star2}}{\sigma_{r'}^\star + \sigma_{r'+1}^\star} \geq \frac{\sigma_{r'}^{\star2}-\sigma_{r'+1}^{\star2}}{2\sigma_{r'}^\star}$.
This further gives
\begin{align*}
	\widetilde{\sigma}_{r'}^2 - \widetilde{\sigma}_{r'+1}^2 = \left(\widetilde{\sigma}_{r'} - \widetilde{\sigma}_{r'+1}\right)\left(\widetilde{\sigma}_{r'} + \widetilde{\sigma}_{r'+1}\right) \geq \frac{4\left(\sigma_{r'}^\star-\sigma_{r'+1}^\star\right)}{5}\left(\sigma_{r'}^\star + \sigma_{r'+1}^\star - \frac{\sigma_r^\star}{5r}\right) \geq \frac{1}{2}\left(\sigma_{r'}^{\star2}-\sigma_{r'+1}^{\star2}\right).
\end{align*}

\paragraph{Bounding the noise size $\|\mathcal{P}_{\sf off\text{-}diag}(\bm{E}\bm{E}^\top - \bm{E}\bm{V}^\star\bm{V}^{\star\top}\bm{E}^\top)\|$.} 
We now move on to control $\|\mathcal{P}_{\sf off\text{-}diag}(\bm{E}\bm{E}^\top - \bm{E}\bm{V}^\star\bm{V}^{\star\top}\bm{E}^\top)\|$. 
Towards this end, Lemma~\ref{lm:off-diag} tells us that, with probability at least $1 - O(n^{-10})$, 
\begin{align}\label{ineq74}
	\left\|\mathcal{P}_{\sf off\text{-}diag}\left(\bm{E}\bm{E}^\top\right)\right\| &\lesssim B^2\log^4 n + \sqrt{n_1}\omega_{\sf max}\left(\sqrt{n_1} + \sqrt{n_2}\right)\omega_{\sf max}\log^2 n\notag\\
	&\lesssim \left(\frac{(n_1n_2)^{1/4}}{\log n}\omega_{\sf max}\right)^2\log^4 n + \left(\sqrt{n_1n_2} + n_1\right)\omega_{\sf max}^2\log^2 n\notag\\
	&\asymp \left(\sqrt{n_1n_2} + n_1\right)\omega_{\sf max}^2\log^2 n, 
\end{align}
where the second line results from Assumption~\ref{assump:hetero2}. 
In view of \eqref{ineq26} and \eqref{ineq74}, with probability exceeding $1 - O(n^{-10})$, we have
\begin{align*}
	\left\|\mathcal{P}_{\sf off\text{-}diag}\left(\bm{E}\bm{E}^\top - \bm{E}\bm{V}^\star\bm{V}^{\star\top}\bm{E}^\top\right)\right\|
	&\leq \left\|\mathcal{P}_{\sf off\text{-}diag}\left(\bm{E}\bm{E}^\top\right)\right\| + \left\|\mathcal{P}_{\sf off\text{-}diag}\left(\bm{E}\bm{V}^\star\bm{V}^{\star\top}\bm{E}^\top\right)\right\|\notag\\
	&\leq C_5\left(\sqrt{n_1n_2} + n_1\right)\omega_{\sf max}^2\log^2 n + 2\left\|\bm{E}\bm{V}^\star\right\|^2\notag\\
	&\leq C_5\left(\sqrt{n_1n_2} + n_1\right)\omega_{\sf max}^2\log^2 n + 2C_5n_1\omega_{\sf max}^2\log^2 n\notag\\
	&\leq 3C_5\left(\sqrt{n_1n_2} + n_1\right)\omega_{\sf max}^2\log^2 n
\end{align*} 
for some large enough constant $C_5>0$.

\paragraph{Bounding the incoherence concerning $\big\|\widetilde{\bm{U}}\big\|_{2, \infty}$.}

We now turn to the incoherence property w.r.t.~$\widetilde{\bm{U}}$. 
Lemma \ref{lm:noise} together with Assumption~\ref{assump:hetero2} reveals that with probability exceeding $1 - O(n^{-10})$, 
\begin{align}
	\left\|\bm{U}^{\star}\bm{U}^{\star\top}\bm{E}\bm{V}^\star\right\|_{2, \infty} 
	&\leq \left\|\bm{U}^{\star} \right\|_{2,\infty} \left\|\bm{U}^{\star\top}\bm{E}\bm{V}^\star\right\| \leq \sqrt{\frac{\mu r}{n_1}}\left\|\bm{U}^{\star\top}\bm{E}\bm{V}^\star\right\| \notag\\
	& 
	\lesssim \sqrt{\frac{\mu r}{n_1}}\left(B\frac{\mu r}{\sqrt{n_1n_2}}\log^2 n + \sqrt{r}\omega_{\sf max}\log n\right)\notag\\ 
	&\lesssim \sqrt{\frac{\mu r}{n_1}}\left(\frac{\sqrt{n_2}}{\log n}\frac{\mu r}{\sqrt{n_1n_2}}\omega_{{\sf max}}\log^2 n + \sqrt{r}\omega_{\sf max}\log n\right)\notag\\ &\lesssim \sqrt{\frac{\mu r}{n_1}}\sqrt{\mu r}\omega_{\sf max}\log n, 
	\label{ineq73}
\end{align}
where the first line follows from Definition~\ref{assump:incoherence}, 
the third line makes use of Assumption \ref{assump:hetero2}, and 
the last line holds due to the assumption $\mu r^3 \lesssim n_1$. 
Putting \eqref{ineq25}, \eqref{ineq26} and \eqref{ineq73} together, we can demonstrate that with probability exceeding $1 - O(n^{-10})$, 
\begin{align*}
	\big\|\bm{U}^{\star}\bm{U}^{\star\top}\widetilde{\bm{U}} - \widetilde{\bm{U}}\big\|_{2, \infty} &\leq C_5\left(\sqrt{\mu r}\omega_{\sf max}\log n + \sqrt{\frac{\mu r}{n_1}}\sqrt{\mu r}\omega_{\sf max}\log n\right)\frac{2}{\sigma_r^\star}\notag\\
	&\leq \frac{4C_5\sqrt{\mu r}\omega_{\sf max}\log n}{\sigma_r^\star}\leq \sqrt{\frac{\mu r}{n_1}} ,
\end{align*}
where the last inequality follows from the assumptions~\eqref{assump:snr_two_to_infty-all}. 
This in turn indicates that
\begin{align*}
	\big\|\widetilde{\bm{U}}\big\|_{2, \infty} \leq \big\|\bm{U}^{\star}\bm{U}^{\star\top}\widetilde{\bm{U}}\big\|_{2, \infty} + \big\|\bm{U}^{\star}\bm{U}^{\star\top}\widetilde{\bm{U}} - \widetilde{\bm{U}}\big\|_{2, \infty} \leq 2\sqrt{\frac{\mu r}{n_1}}.
\end{align*}

        \section{Proofs for corollaries}\label{sec:proof_corollary}
\subsection{Proof of Corollary \ref{cor:pca}}\label{proof:corollary_pca}
First, by virtue of the standard tail bound of sub-Gaussian random variables (cf.~\citet[Lemma 5.5]{vershynin2010introduction}), we can easily verify that Assumption~\ref{assump:hetero} holds with the following parameters: 
$$\omega_{\sf max} = \omega \qquad \text{and} \qquad B = C_{\mathsf{B}}\omega\log(n+d) \lesssim \omega\frac{\min\left\{(nd)^{1/4}, n^{1/2}\right\}}{\log (n+d)}$$
for some constant $C_{\mathsf{B}}>0$.

Next, let us look at several properties of the matrix $\bm{X} = [\bm{x}_1 \ \dots \ \bm{x}_n] \in \bbR^{d \times n}$. It is seen that 
$$\bm{X} = \bm{U}^\star\bm{\Lambda}^{\star1/2}\bm{F}^\star, \qquad \text{with }\bm{F}^\star = [\bm{f}_1,\ldots,\bm{f}_n] \in \bbR^{r \times n},$$
where $F_{i,j}^\star \overset{\rm i.i.d.}{\sim} \mathcal{N}(0, 1)$ for all $(i,j)\in[r]\times [n]$. 
In view of \citet[Corollary 5.35]{vershynin2010introduction}, we know that with probability exceeding $1 - O\big((n+d)^{-10}\big)$,
\begin{align}
	\sqrt{n}/2 \leq \sqrt{n} - \sqrt{r} - \sqrt{20\log(n+d)} \leq \sigma_r\left(\bm{F}^\star\right) \leq \sigma_1\left(\bm{F}^\star\right) \leq \sqrt{n} + \sqrt{r} + \sqrt{20\log(n+d)} \leq 2\sqrt{n}.
	\label{eq:spectrum-F-UB-LB}
\end{align}
By the min-max principle for singular values, for all $1 \leq i \leq r$, one has
\begin{align}
	\lambda_i^{\star1/2}\sigma_r\left(\bm{F}^\star\right) &= \min_{\bm{S}: {\sf dim}(\bm{S}) = r-i+1}\max_{\bm{x} \in \bm{S}, \|\bm{x}\|_2 = 1}\big\|\bm{x}^\top\bm{\Lambda}^{\star1/2}\big\|\sigma_r\left(\bm{F}^\star\right) \notag\\
	&\leq \sigma_i\left(\bm{X}^\star\right) = \sigma_i\big(\bm{\Lambda}^{\star1/2}\bm{F}^\star\big) \notag\\
	&= \min_{\bm{S}: {\sf dim}(\bm{S}) = r-i+1}\max_{\bm{x} \in \bm{S}, \|\bm{x}\|_2 = 1}\big\|\bm{x}^\top\bm{\Lambda}^{\star1/2}\bm{F}^\star\big\| \notag\\
	&\leq \min_{\bm{S}: {\sf dim}(\bm{S}) = r-i+1}\max_{\bm{x} \in \bm{S}, \|\bm{x}\|_2 = 1}\big\|\bm{x}^\top\bm{\Lambda}^{\star1/2}\big\|\left\|\bm{F}^\star\right\| \notag\\
	&= \lambda_i^{\star1/2}\sigma_1\left(\bm{F}^\star\right).
	\label{eq:sandwich}
\end{align}
Therefore, with probability exceeding $1 - O\big((n+d)^{-10}\big)$, we obtain
\begin{align}
	\label{ineq59}
	\sigma_i\left(\bm{X}^\star\right) \asymp \sqrt{n\lambda_i^\star} \qquad \text{ for all } 1 \leq i \leq r .
\end{align}
In fact, the relation \eqref{eq:sandwich}
 taken together with \eqref{eq:spectrum-F-UB-LB} and \eqref{eq:assumption-PCA-SNR} yields a more concrete lower bound
\begin{align*}
	\sigma_i\left(\bm{X}^\star\right) \geq \sqrt{n\lambda_i^\star}/2 \geq C_0r\left[\left(dn\right)^{1/4} + d^{1/2}\right]\log (n+d)
	 \qquad \text{ for all } 1 \leq i \leq r.
\end{align*}
Hence, the signal-to-noise ratio condition in Theorem \ref{thm:two_to_infty} is satisfied (where we take $n_1 = d$ and $n_2 = n$). 
Additionally, letting $\bm{V}^\star \in \mathcal{O}^{n, r}$ denote the right singular space of $\bm{X}^\star$, we see from the proof of \cite[Corollary 2]{cai2021subspace} that with probability exceeding $1 - O((n+d)^{-10})$,
\begin{align*}
	\left\|\bm{V}^\star\right\|_{2,\infty} \leq \sqrt{\frac{C_2r\log(n+d)}{n}}
\end{align*}
for some constant $C_2>0$. 
Consequently, we have
\begin{align*}
	\mu \leq \mu_{\sf pc} \vee C_2\log(n+d) \lesssim \frac{d}{r^3},
\end{align*}
where $\mu_{\sf pc}$ is defined in \eqref{eq:incoherence-Ustar-pc} and the last inequality arises from the assumption \eqref{eq:assumption-PCA-inc}.

Now, we see that with probability at least $1 - O\big((n+d)^{-10}\big)$, all conditions in Theorem \ref{thm:two_to_infty} are satisfied. 
Thus, apply Theorem \ref{thm:two_to_infty} and \eqref{ineq59} to yield that: with probability exceeding $1 - O\big((n+d)^{-10}\big)$, 
\begin{align*}
	\left\|\bm{U}\bm{R}_{\bm{U}} - \bm{U}^\star\right\| \lesssim \frac{\sqrt{dn}\,\omega^2\log^2 (n+d)}{n\lambda_r^\star} + \frac{\sqrt{d}\,\omega\log (n+d)}{\sqrt{n\lambda_r^\star}} \asymp  \frac{\sqrt{d/n}\,\omega^2\log^2 (n+d)}{\lambda_r^{\star}} + \frac{\sqrt{d/n}\,\omega\log(n+d)}{\sqrt{\lambda_{r}^\star}}
\end{align*}
and
\begin{align*}
	\left\|\bm{U}\bm{R}_{\bm{U}} - \bm{U}^\star\right\|_{2,\infty} \lesssim \sqrt{\frac{\mu_{\sf pc} + \log(n+d)}{d}}\left(\frac{\sqrt{d/n}\,\omega^2\log^2 (n+d)}{\lambda_r^{\star}} + \frac{\sqrt{d/n}\,\omega\log(n+d)}{\sqrt{\lambda_{r}^\star}}\right),
\end{align*}
provided that the number of iterations satisfy \eqref{ineq:iter_pca_1}-\eqref{ineq:iter_pca_2}.

\subsection{Proof of Corollary \ref{cor:tensor_svd}}\label{proof:corollary_tensor}

For notational convenience, we let $\bm{Y}_i \in \bbR^{n_i \times (n_1n_2n_3/n_i)}$ (resp.~$\bm{X}_i^\star$ and $\bm{E}_i$) denote the $i$-th matricization of $\bm{\mathcal{Y}}$ (resp.~$\bm{\mathcal{X}}^\star$ and $\bm{\mathcal{E}}$). 
We need to check that all assumptions in Theorem \ref{thm:two_to_infty} are satisfied for the $i$-th matricization.

Firstly, it  can be easily verifid that Assumption \ref{assump:hetero2} holds for $\bm{E}_i$ with $\omega_{\sf max} = \omega \text{ and } B \asymp \omega\log n.$ 
In addition, taking the assumption $n_1 \asymp n_2 \asymp n_3$ and \eqref{eq:condition:tensor-snr-123} together imply that 
\begin{align*}
	\frac{\sigma_{i, r_i}^\star}{\omega} \geq C_0r\left[(n_1n_2n_3)^{1/4} + n_i^{1/2}\right]\log\big(n_i \vee \left(n_1n_2n_3/n_i\right)\big)
\end{align*}
for some large enough constant $C_0>0$, thus justifying  the SNR condition \eqref{assump:snr_two_to_infty} in Theorem \ref{thm:two_to_infty}.
Next, let $\bm{V}_i^\star \in \mathcal{O}^{n_1n_2n_3/n_i, r_i}$ denote the right singular space of $\bm{X}_i^\star$ and define 
%
\begin{align*}
	\mu\left(\bm{X}_i^\star\right) = \max\left\{\frac{n_i}{r_i}\left\|\bm{U}_i^\star\right\|_{2,\infty}^2, \frac{n_1n_2n_3/n_i}{r_i}\left\|\bm{V}_i^\star\right\|_{2,\infty}^2\right\}
\end{align*}
Given that $\bm{X}_1^\star = \bm{U}_1^\star\mathcal{M}_1(\mathcal{S}^\star)(\bm{U}_3^\star \otimes \bm{U}_2^\star)^\top$, we can invoke \eqref{eq:incoherence-tensor-PCA} and \eqref{eq:condition:tensor-inc-123} to obtain
\begin{align*}
	\left\|\bm{V}_1^\star\right\|_{2,\infty} \leq \left\|\bm{U}_3^\star \otimes \bm{U}_2^\star\right\|_{2,\infty} \leq \left\|\bm{U}_2^\star\right\|_{2,\infty}\left\|\bm{U}_3^\star\right\|_{2,\infty} \leq \sqrt{\frac{\mu^2 r_2r_3}{n_2n_3}}
\end{align*}
and
\begin{align*}
	\mu\left(\bm{X}_1^\star\right) \leq \max\left\{\mu, \mu^2\frac{r_2r_3}{r_1}\right\} \lesssim \frac{n_1}{r_1^3}.
\end{align*}    
Therefore, all conditions and assumptions in Theorem \ref{thm:two_to_infty} are satisfied.  Consequently, invoke Theorem \ref{thm:two_to_infty} to show that, with probability exceeding $1 - O(n^{-10})$,
\begin{align*}
	\left\|\widehat{\bm{U}}_1^{0}\bm{R}_{\widehat{\bm{U}}_1^{0}} - \bm{U}_1^\star\right\|_{2,\infty} &\lesssim \sqrt{\frac{\mu\left(\bm{X}_1^\star\right)r_1}{n_1}}\left(\frac{\sqrt{n_1n_2n_3}\omega^2\log^2 n}{\sigma_{1, r_1}^{\star2}} + \frac{\sqrt{n_1}\omega\log n}{\sigma_{1, r_1}^{\star}}\right)\\
	&\leq \frac{\mu r}{\sqrt{n_1}}\left(\frac{\sqrt{n_1n_2n_3}\omega^2\log^2 n}{\sigma_{\sf min}^{\star2}} + \frac{\sqrt{n_1}\omega\log n}{\sigma_{\sf min}^{\star}}\right)
\end{align*}
and
\begin{align*}
	\Big\|\widehat{\bm{U}}_1^{0}\bm{R}_{\widehat{\bm{U}}_1^{0}} - \bm{U}^\star\Big\| \lesssim \frac{n^{3/2}\omega^2\log^2 n}{\sigma_{\sf min}^{\star2}} + \frac{\sqrt{n}\omega\log n}{\sigma_{\sf min}^{\star}}.
\end{align*}
Similarly, one can show that with probability at least $1 - O(n^{-10})$, \eqref{ineq:tensor_initial_1} and \eqref{ineq:tensor_initial_2} holds for $i = 2$ and $3$,
thereby establishing the first part of Corollary \ref{cor:tensor_svd}.

When it comes to the second part, we can directly use the same argument in the proof of \citet[Theorem 1]{zhang2018tensor} if the following two claims are valid with probability exceeding $1 - O(n^{-10})$:
\begin{align}\label{ineq60}
	\max_{\bm{V}_i \in \bbR^{n_i \times r_i}, \left\|\bm{V}_i\right\| \leq 1}\max\big\{\left\|\bm{E}_1\left(\bm{V}_3 \otimes \bm{V}_2\right)\right\|, \left\|\bm{E}_2\left(\bm{V}_1 \otimes \bm{V}_3\right)\right\|, \left\|\bm{E}_3\left(\bm{V}_2 \otimes \bm{V}_1\right)\right\|\big \} \lesssim \sqrt{nr},
\end{align}
and
\begin{align}\label{ineq61}
	\max\big\{\left\|\bm{E}_1\left(\bm{U}_3^\star \otimes \bm{U}_2^\star\right)\right\|, \left\|\bm{E}_2\left(\bm{U}_1^\star \otimes \bm{U}_3^\star\right)\right\|, \left\|\bm{E}_3\left(\bm{U}_2^\star \otimes \bm{U}_1^\star\right)\right\|\big\} \lesssim \sqrt{n}.
\end{align}
In fact, \eqref{ineq61} is a direct consequence of \citet[Lemma A.2]{zhou2022optimal} (or Lemma 8.2 in its arxiv version) with 
\begin{align*}
	\bm{A} = \bm{I}_{n_1} \text{ (resp.~}\bm{I}_{n_2} \text{ and } \bm{I}_{n_3}) \quad \text{and} \quad \bm{B} = \bm{U}_3^\star \otimes \bm{U}_2^\star \text{ (resp.~}\bm{U}_1^\star \otimes \bm{U}_3^\star \text{ and } \bm{U}_2^\star \otimes \bm{U}_1^\star), 
\end{align*}
whereas \eqref{ineq60} can be proved by combining \citet[Lemma A.2]{zhou2022optimal} and the standard epsilon-net argument in the proof of \citet[Lemma 5]{zhang2018tensor}. 
We omit the details here for the sake of brevity. 

    \section{Technical lemmas}\label{sec:technical_lemmas}

In this section, we collect a couple of useful technical lemmas and provide proofs. Before continuing, we note that Assumption~\ref{assump:hetero} and \ref{assump:hetero2} are subsumed as special cases of the following assumption: 
\begin{assump}\label{assump:hetero3}
	Suppose that the noise components $\{E_{i,j}\}$ satisfy the following conditions:
	\begin{itemize}
		\item[1.] The $E_{i,j}$'s  are statistically independent and zero-mean;
		\item[2.] $\mathsf{Var}[E_{i,j}] = \omega_{i,j}^2 \leq \omega_{\sf max}^2$ for all $(i,j)\in [n_1]\times [n_2]$; 
		\item[3.] For any $(i, j) \in [n_1] \times [n_2]$, one has $\bbP\left(\left|E_{i,j}\right| > B\right) \leq \varepsilon$ for some quantity $B$, where $\varepsilon$ is some quantity within $[0, C_{\sf b}n^{-10}]$ for some universal constant $C_{\sf b} > 0$. 
	\end{itemize}
\end{assump}

Let us begin with several tail bounds regarding the spectral norm of linear functions of $\bm{E}=[E_{i,j}]_{(i,j)\in[n_1]\times [n_2]}$. 
\begin{lemma}\label{lm:noise}
	Suppose that Assumption \ref{assump:hetero3} holds. Then there exists some large (resp.~small) enough constant $C_1 > 0$ ($c_1 > 0$) such that for any $x \geq C_1\sqrt{\log n}$, with probability exceeding $1 - O(e^{-c_1x^2}) - n_1n_2\varepsilon$ one has 
	\begin{subequations}
		\begin{align}
			\left\|\bm{E}\bm{V}^{\star}\right\| &\lesssim B\sqrt{\frac{\mu_2 r}{n_2}}x^2 + \left(\left(\frac{\mu r}{n_2}\omega_{\sf row}^2 \wedge r\omega_{\sf max}^2\right) + \omega_{\sf col}^2\right)^{1/2}x,\label{ineq:noise_a}\\
			\left\|\bm{U}^{\star\top}\bm{E}\bm{V}^{\star}\right\| &\lesssim B\frac{\mu r}{\sqrt{n_1n_2}}x^2 + \left[\left(\sqrt{\frac{\mu r}{n_2}}\omega_{\sf row} + \sqrt{ \frac{\mu r}{n_1}}\omega_{\sf col}\right) \wedge \sqrt{r}\omega_{\sf max}\right]x,\label{ineq:noise_b}\\
			\left\|\bm{E}\bm{V}^{\star}\right\|_{2, \infty} &\lesssim \left(Bx^2 + \omega_{\sf row}x\right)\sqrt{\frac{\mu_2 r}{n_2}},\label{ineq:noise_c}\\
			\left\|\bm{E}\right\| &\lesssim Bx + \left(\omega_{\sf row} + \omega_{\sf col}\right).\label{ineq:noise_d}
		\end{align}
	\end{subequations}
\end{lemma}
\begin{proof}[Proof of Lemma \ref{lm:noise}]
	We start with the case $\varepsilon = 0$, i.e., $|E_{i,j}| \leq B$ holds deterministically (see Assumption~\ref{assump:hetero3}). 
	\begin{itemize}

	\item First,  express $\bm{E}\bm{V}^{\star}$ as a sum of zero-mean independent random matrices as follows
	\begin{align*}
		\bm{E}\bm{V}^{\star} = \sum_{i=1}^{n_1}\sum_{j=1}^{n_2}E_{i,j}\bm{e}_i\bm{V}_{j,:}^{\star}.
	\end{align*}
	From the definition \eqref{eq:defn-omega-set} and the incoherence condition in Definition~\ref{assump:incoherence}, one can verify that
	$$L_1: = \max_{1 \leq i \leq n_1, 1\leq j \leq n_2}\left\|E_{i,j}\bm{e}_i\bm{V}_{j,:}^{\star}\right\| \leq B\sqrt{\frac{\mu_2 r}{n_2}}$$
	and
	\begin{align*}
		V_1&:= \max\left\{\Bigg\|\sum_{i=1}^{n_1}\sum_{j=1}^{n_2}\bbE\left[E_{i,j}^2\right]\left\|\bm{V}_{j,:}^{\star}\right\|_2^2\bm{e}_i\bm{e}_i^\top\Bigg\|, \Bigg\|\sum_{j=1}^{n_2}\sum_{i=1}^{n_1}\bbE\left[E_{i,j}^2\right]\bm{V}_{j,:}^{\star\top}\bm{V}_{j,:}^{\star}\Bigg\|\right\}\\
		&\leq \left(\frac{\mu_2 r}{n_2}\omega_{\sf row}^2 \wedge r\omega_{\sf max}^2\right) + \omega_{\sf col}^2,
	\end{align*}
	where the last line also uses the facts that $\sum_j \|\bm{V}_{j,:}^{\star}\|_2^2 =r $ 
	and $\sum_j \bm{V}_{j,:}^{\star\top}\bm{V}_{j,:}^{\star} = \bm{V}^{\star\top} \bm{V}^{\star} = \bm{I}_r $. 
	Applying the matrix Bernstein inequality \citep{tropp2015introduction} leads to, with probability exceeding $1 - O(e^{-c_1x^2})$,
	\begin{align*}
		\left\|\bm{E}\bm{V}^{\star}\right\| \lesssim L_1x^2 + \sqrt{V_1}x \lesssim B\sqrt{\frac{\mu_2 r}{n_2}}x^2 + \sqrt{\left(\left(\frac{\mu_2 r}{n_2}\omega_{\sf row}^2 \wedge r\omega_{\sf max}^2\right) + \omega_{\sf col}^2\right)}x
	\end{align*}
	for any $x\geq C_1\sqrt{\log n}$, 
	where $c_1, C_1>0$ are some suitable numerical constants.

	\item When it comes to $\bm{U}^{\star\top}\bm{E}\bm{V}^{\star}$, we decompose it into the following zero-mean and independent terms:
	\begin{align*}
		\bm{U}^{\star\top}\bm{E}\bm{V}^{\star} = \sum_{i=1}^{n_1}\sum_{j=1}^{n_2}E_{i,j}\bm{U}_{i,:}^\top\bm{V}_{j,:}^{\star}.
	\end{align*}
	Similar to the above arguments, it follows from \eqref{eq:defn-omega-set} and Definition~\ref{assump:incoherence} that
	\begin{align*}
		L_2 := \max_{1 \leq i \leq n_1, 1\leq j \leq n_2}\left\|E_{i,j}\bm{U}_{i,:}^\top\bm{V}_{j,:}^{\star}\right\| \leq B\sqrt{\frac{\mu r}{n_1}}\sqrt{\frac{\mu r}{n_2}} = B\frac{\mu r}{\sqrt{n_1n_2}}
	\end{align*}
	and
	\begin{align*}
		V_2&:= \max\left\{\Bigg\|\sum_{i=1}^{n_1}\sum_{j=1}^{n_2}\bbE\left[E_{i,j}^2\right]\left\|\bm{V}_{j,:}^{\star}\right\|_2^2\bm{U}_{i,:}^\top\bm{U}_{i,:}\Bigg\|, \Bigg\|\sum_{i=1}^{n_1}\sum_{j=1}^{n_2}\bbE\left[E_{i,j}^2\right]\left\|\bm{U}_{i,:}\right\|_2^2\bm{V}_{j,:}^{\star\top}\bm{V}_{j,:}^{\star}\Bigg\|\right\}\\
		&\leq \left[\frac{\mu r}{n_2}\omega_{\sf row}^2 + \frac{\mu r}{n_1}\omega_{\sf col}^2\right] \wedge r\omega_{\sf max}^2.
	\end{align*}
	The matrix Bernstein inequality reveals that with probability exceeding $1 - O(e^{-c_1x^2})$, 
	\begin{align*}
		\left\|\bm{U}^{\star\top}\bm{E}\bm{V}^{\star}\right\| \lesssim L_2x^2 + \sqrt{V_2}x \lesssim B\frac{\mu r}{\sqrt{n_1n_2}}x^2 + \left[\left(\sqrt{\frac{\mu r}{n_2}}\omega_{\sf row} + \sqrt{ \frac{\mu r}{n_1}}\omega_{\sf col}\right) \wedge \sqrt{r}\omega_{\sf max}\right]x.
	\end{align*}

	\item
	Additionally, \eqref{ineq:noise_c} and \eqref{ineq:noise_d} are direct consequences of \citet[Lemma 12]{cai2021subspace} and \citet[Theorem 3.4]{chen2021spectral}, respectively.
	
	\end{itemize}

	We now move on to the more general case with $\varepsilon > 0$ (see Assumption~\ref{assump:hetero3}).  Denoting by $\widetilde{E}_{i,j}$ the centered truncated noise as follows: 
	\begin{align}\label{def:tilde_E}
		\widetilde{E}_{i,j} = E_{i,j}\mathbbm{1}_{\left\{\left|E_{i,j}\right| \leq B\right\}} - \bbE\left[E_{i,j}\mathbbm{1}_{\left\{\left|E_{i,j}\right| \leq B\right\}}\right].
	\end{align} 
	we see that
	\begin{align*}
		\mathsf{Var}\big(\widetilde{E}_{i,j}\big) \leq \bbE\left[E_{i,j}^2\mathbbm{1}_{\left\{\left|E_{i,j}\right| \leq B\right\}}\right] \leq \bbE\left[E_{i,j}^2\right] = \omega_{i,j}^2
	\end{align*}
	and
	\begin{align*}
		\big|\widetilde{E}_{i,j}\big| \leq B + B = 2B.
	\end{align*}
	The previous argument shows that with probability exceeding $1 - O(e^{-c_1x^2})$, 
	\begin{align}\label{ineq54}
		\text{inequalities}~\eqref{ineq:noise_a}-\eqref{ineq:noise_d} \text{ hold if we replace } \bm{E} \text{ with } \widetilde{\bm{E}}. 
	\end{align}
	Next, let $\overline{\bm{E}}$ denote the matrix with the $(i,j)$-th entry $\overline{E}_{i,j} = E_{i,j}\mathbbm{1}_{\{|E_{i,j}| \leq B\}}$ for all $(i,j)\in [n_1]\times [n_2]$. In view of the Cauchy-Schwarz inequality and the assumption $\mathbb{E}[E_{i,j}]=0$, one has
	\begin{align*}
		\left|\bbE\left[\overline{E}_{i,j}\right]\right| = 
		\left|\bbE\left[E_{i,j}\right] - \bbE\left[E_{i,j}\mathbbm{1}_{\left\{\left|E_{i,j}\right| > B\right\}}\right]\right|
		= \left|\bbE\left[E_{i,j}\mathbbm{1}_{\left\{\left|E_{i,j}\right| > B\right\}}\right]\right| \leq \left(\bbE\left[E_{i,j}^2\right]\bbE\left[\mathbbm{1}_{\left\{\left|E_{i,j}\right| > B\right\}}\right]\right)^{1/2} \leq \omega_{i,j}\sqrt{\varepsilon} ,
	\end{align*}
	and as a result,
	\begin{align}\label{ineq53}
		\left\|\bbE\left[\overline{\bm{E}}\right]\right\| \leq \left\|\bbE\left[\overline{\bm{E}}\right]\right\|_{\rm F} \leq \sqrt{n_1n_2}\max_{i,j}\left|\bbE\left[\overline{E}_{i,j}\right]\right| \leq \omega_{\sf max}\sqrt{n_1n_2\varepsilon} \lesssim \frac{\omega_{\sf max}}{n^{4}} .
	\end{align}
	Assumption \ref{assump:hetero3} and the union bound tell us that with probability at least $1 - n_1n_2\varepsilon$, for all $i, j \in [n_1] \times [n_2]$,
	\begin{align*}
		E_{i,j} = E_{i,j}\mathbbm{1}_{\left\{\left|E_{i,j}\right| \leq B\right\}},
	\end{align*}
	which means
	\begin{align*}
		\bm{E} = \overline{\bm{E}}.
	\end{align*}
	This combined with \eqref{ineq53} yields that with probability exceeding $1 - n_1n_2\varepsilon$,
	\begin{align}\label{ineq55}
		\big\|\bm{E} - \widetilde{\bm{E}}\big\| = \left\|\bbE\left[\overline{\bm{E}}\right]\right\| \lesssim \frac{\omega_{\sf max}}{n^{4}} .
	\end{align}
	On the event $\mathcal{E}_1 =$ \{\eqref{ineq54} and \eqref{ineq55} hold\}, we can apply the triangle inequality to show that 
	\begin{align*}
		\left\|\bm{E}\bm{V}^\star\right\| &\leq \big\|\widetilde{\bm{E}}\bm{V}^\star\big\| + \big\|\big(\bm{E} - \widetilde{\bm{E}}\big)\bm{V}^\star\big\|\\ &\lesssim B\sqrt{\frac{\mu_2 r}{n_2}}x^2 + \left(\left(\frac{\mu_2 r}{n_2}\omega_{\sf row}^2 \wedge r\omega_{\sf max}^2\right) + \omega_{\sf col}^2\right)^{1/2}x + \big\|\bm{E} - \widetilde{\bm{E}}\big\|\\
		&\leq B\sqrt{\frac{\mu_2 r}{n_2}}x^2 + \left(\left(\frac{\mu_2 r}{n_2}\omega_{\sf row}^2 \wedge r\omega_{\sf max}^2\right) + \omega_{\sf col}^2\right)^{1/2}x
	\end{align*}
	for any $x\geq C_1\sqrt{\log n}$. 
	Similarly, one can show that on the same event, \eqref{ineq:noise_b}-\eqref{ineq:noise_d} hold.
\end{proof}
Next, we provide a few more tail bounds concerning the $\ell_{2,\infty}$ norm and sum of squares concerning $\bm{E}$. 
\begin{lemma}\label{lm:1}
	Suppose that Assumption \ref{assump:hetero3} holds. There exists some sufficiently large constant $C_2 > 0$ such that for any fixed matrix $\bm{W}_1$ and $\bm{W}_2$, with probability exceeding $1 - O(n^{-C_2\log n}) - n_1n_2\varepsilon$ one has
	\begin{subequations}
		\begin{align}
			\left\|\bm{E}\bm{W}_1\right\|_{2,\infty} &\lesssim B\left\|\bm{W}_1\right\|_{2,\infty}\log^2 n + \omega_{\sf max}\left\|\bm{W}_1\right\|_{\F}\log n,\label{ineq57a}\\
			\max_{i \in [n_1]}\sum_{j \in [n_2]}E_{i,j}^2 &\lesssim B^2\log^2 n + \omega_{\sf row}^2,\label{ineq57b}\\
			\max_{j \in [n_2]}\sum_{i \in [n_1]}E_{i,j}^2 &\lesssim B^2\log^2 n + \omega_{\sf col}^2,\label{ineq57c}\\
			\max_{j \in [n_2]}\big\|\left(\bm{E}_{:,j}\right)^\top\bm{W}_2\big\|_2 &\lesssim \left(B\log^2 n + \omega_{\sf col}\log n\right)\left\|\bm{W}_2\right\|_{2,\infty}.\label{ineq57d}
		\end{align}
	\end{subequations}
\end{lemma}
\begin{proof}[Proof of Lemma \ref{lm:1}]
	We again consider the case $\varepsilon = 0$ first  (see Assumption~\ref{assump:hetero3}). In this case, \eqref{ineq57b}-\eqref{ineq57d} are basically direct consequences of \cite[Lemma 12]{cai2021subspace}. The only difference is we require a higher probability here ($1-O(n^{-C_2\log n})$ instead of $1-O(n^{-20})$), which leads to an extra $\log n$ factor in our bounds.
	Turning to \eqref{ineq57a}, we note that for any $i \in [n_1]$,  $\bm{E}_{i,:}\bm{W} = \sum_{j \in [n_2]}E_{i,j}\bm{W}_{j,:}$ is a sum of $n_2$ independent zero-mean vectors. 
	In light of the following key quantities:
	\begin{align*}
		L:= \max_{j \in [n_2]}\left\|E_{i,j}\bm{W}_{j,:}\right\|_2 \leq B\left\|\bm{W}\right\|_{2,\infty}
	\end{align*}
	and
	\begin{align*}
		V:= \sum_{j \in [n_2]}\bbE\left[E_{i,j}^2\right]\left\|\bm{W}_{j,:}\right\|_2^2 \leq \omega_{\sf max}^2\sum_{j \in [n_2]}\left\|\bm{W}_{j,:}\right\|_2^2 = \omega_{\sf max}^2\left\|\bm{W}\right\|_{\F}^2,
	\end{align*}
	we can apply the matrix Bernstein inequality to show that: with probability exceeding $1 - n^{-C_3\log n}$, 
	\begin{align}\label{ineq6}
		\left\|\bm{E}_{i,:}\bm{W}\right\|_2 \lesssim L\log^2 n + \sqrt{V}\log n \lesssim B\left\|\bm{W}\right\|_{2,\infty}\log^2 n + \omega_{\sf max}\left\|\bm{W}\right\|_{\F}\log n
	\end{align}
	holds for some numerical constant $C_3>0$. 
	The union bound then shows that with probability exceeding $1 - n\cdot n^{-C_3\log n} \geq 1 - n^{-C_2\log n}$ (for some numerical constant $C_2>0$),
	\begin{align*}
		\left\|\bm{E}\bm{W}\right\|_{2,\infty} = \max_{i \in [n_1]}\left\|\bm{E}_{i,:}\bm{W}\right\|_2 \lesssim B\left\|\bm{W}\right\|_{2,\infty}\log^2 n + \omega_{\sf max}\left\|\bm{W}\right\|_{\F}\log n.
	\end{align*}

	When it comes to the more general case with $\varepsilon > 0$, repeating a similar argument as in the proof of Lemma \ref{lm:noise} immediately helps us finish the proof of Lemma \ref{lm:noise}. 
\end{proof}

The next lemma gathers a spectral norm upper bound on the Gram matrix $\bm{E}\bm{E}^{\top}$ after diagonal deletion. 
\begin{lemma}\label{lm:off-diag}
	Assume that Assumption \ref{assump:hetero3} holds. Then there exists some large (resp.~small) constant $C_1 > 0$ ($c_1 > 0$) such that: for any $x \geq C_1\sqrt{\log n}$, with probability exceeding $1 - O(e^{-c_1x^2}) - n_1n_2\varepsilon$ one has
	\begin{align*}
		\left\|\mathcal{P}_{\sf off\text{-}diag}\left(\bm{E}\bm{E}^\top\right)\right\| \lesssim B^2x^4 + \omega_{\sf col}\left(\omega_{\sf row} + \omega_{\sf col}\right)x^2.
	\end{align*}
\end{lemma}
\begin{proof}[Proof of Lemma \ref{lm:off-diag}]
	In view of \citet[Section~B.2.1]{cai2021subspace} (or more precisely, we use the proof therein but change the probability slightly), we know that with probability $1 - O(e^{-c_1x^2})$,
	\begin{align}\label{ineq56}
		\big\|\mathcal{P}_{\sf off\text{-}diag}\big(\widetilde{\bm{E}}\widetilde{\bm{E}}^\top\big)\big\| \lesssim B^2x^4 + \omega_{\sf col}\left(\omega_{\sf row} + \omega_{\sf col}\right)x^2,
	\end{align}
	where $\widetilde{\bm{E}}$ is defined in \eqref{def:tilde_E}. Let $\mathcal{E}_2$ denote the following event:
	\begin{align*}
		\mathcal{E}_2 := \left\{\eqref{ineq55} \text{ and }\eqref{ineq56} \text{ hold, and }\big\|\widetilde{\bm{E}}\big\| \lesssim Bx + \left(\omega_{\sf row} + \omega_{\sf col}\right)\right\}.
	\end{align*}
	By virtue of \eqref{ineq55}, \eqref{ineq56} and Lemma \ref{lm:noise}, we have 
	\begin{align*}
		\bbP\left(\mathcal{E}_2\right) \geq 1 - O\big(e^{-c_1x^2}\big) - n_1n_2\varepsilon.
	\end{align*}
	On the event $\mathcal{E}_2$, one can obtain 
	\begin{align*}
		\left\|\mathcal{P}_{\sf off\text{-}diag}\left(\bm{E}\bm{E}^\top\right)\right\| &\leq \big\|\mathcal{P}_{\sf off\text{-}diag}\big(\widetilde{\bm{E}}\widetilde{\bm{E}}^\top\big)\big\| + \big\|\mathcal{P}_{\sf off\text{-}diag}\big(\bm{E}\bm{E}^\top - \widetilde{\bm{E}}\widetilde{\bm{E}}^\top\big)\big\|\\
		&\lesssim B^2x^4 + \omega_{\sf col}\left(\omega_{\sf row} + \omega_{\sf col}\right)x^2 + \big\|\bm{E}\bm{E}^\top - \widetilde{\bm{E}}\widetilde{\bm{E}}^\top\big\|\\
		&\leq B^2x^4 + \omega_{\sf col}\left(\omega_{\sf row} + \omega_{\sf col}\right)x^2 + \big\|\big(\bm{E} - \widetilde{\bm{E}}\big)\widetilde{\bm{E}}^\top\big\| + \big\|\widetilde{\bm{E}}\big(\bm{E} - \widetilde{\bm{E}}\big)^\top\big\|\\&\quad + \big\|\big(\bm{E} - \widetilde{\bm{E}}\big)\big(\bm{E} - \widetilde{\bm{E}}\big)^\top\big\|\\
		&\leq B^2x^4 + \omega_{\sf col}\left(\omega_{\sf row} + \omega_{\sf col}\right)x^2 + 2\big\|\bm{E} - \widetilde{\bm{E}}\big\|\big\|\widetilde{\bm{E}}\big\| + \big\|\bm{E} - \widetilde{\bm{E}}\big\|^2\\
		&\lesssim B^2x^4 + \omega_{\sf col}\left(\omega_{\sf row} + \omega_{\sf col}\right)x^2 + \frac{\omega_{\sf max}}{n^{4}}\big(Bx + (\omega_{\sf row} + \omega_{\sf col})\big) + \frac{\omega_{\sf max}^2}{n^{8}}\\
		&\lesssim B^2x^4 + \omega_{\sf col}\left(\omega_{\sf row} + \omega_{\sf col}\right)x^2 + \frac{1}{2}\left(B^2x^2 + \frac{\omega_{\sf max}^2}{n^{8}}\right) + \frac{\omega_{\sf max}}{n^{4}}\left(\omega_{\sf row} + \omega_{\sf col}\right) + \frac{\omega_{\sf max}^2}{n^{8}}\\
		&\lesssim B^2x^4 + \omega_{\sf col}\left(\omega_{\sf row} + \omega_{\sf col}\right)x^2
	\end{align*}
	for any $x\geq C_1\sqrt{\log n}$,  
	where the penultimate line is due to the AM-GM inequality.
\end{proof}

Finally, we make note of a result that controls the projection of $\bm{X}$ onto the subspace spanned by $\widehat{\bm{U}}_{\perp}$ (the orthogonal complement of the leading rank-$r$ left singular subspace of $\bm{Y}$). 
\begin{lemma}[\cite{zhang2018tensor}, Lemma 6]\label{lm:projection}
	Suppose that $\bm{Y} = \bm{X} + \bm{E}$, where $\bm{X}$ is a rank-$r$ matrix and $\bm{E}$ is the noise matrix. Let $\widehat{\bm{U}}$ denote the rank-$r$ leading left singular subspace of $\bm{Y}$, and let $\widehat{\bm{U}}_{\perp}$ represent the orthogonal complement of $\widehat{\bm{U}}$. Then it holds that
	\begin{align*}
		\big\|\mathcal{P}_{\widehat{\bm{U}}_{\perp}}\bm{X}\big\| \leq 2\left\|\bm{E}\right\|.
	\end{align*}
\end{lemma}

\bibliographystyle{apalike}
\bibliography{reference}
\end{document}